\documentclass[10pt]{amsart} 
\usepackage{amssymb}
\usepackage{amscd}
\newcounter{TmpEnumi}
\numberwithin{equation}{section}
\usepackage[colorlinks,linkcolor=black,citecolor=blue]{hyperref}
\usepackage[all]{xy}
\def\today{\number\day\space\ifcase\month\or   January\or February\or
   March\or April\or May\or June\or   July\or August\or September\or
   October\or November\or December\fi\   \number\year}

\theoremstyle{definition}
\newtheorem{thm}{Theorem}[section]
\newtheorem{lem}[thm]{Lemma}
\newtheorem{prp}[thm]{Proposition}
\newtheorem{dfn}[thm]{Definition}
\newtheorem{cor}[thm]{Corollary}

\newtheorem{rmk}[thm]{Remark}
\newtheorem{ntn}[thm]{Notation}

\newtheorem{pbm}[thm]{Problem}

\newtheorem{qst}[thm]{Question}

\newtheorem{cns}[thm]{Construction}

\newcommand{\beq}{\begin{equation}}
\newcommand{\eeq}{\end{equation}}
\newcommand{\beqa}{\begin{eqnarray*}}
\newcommand{\eeqa}{\end{eqnarray*}}
\newcommand{\bal}{\begin{align*}}
\newcommand{\eal}{\end{align*}}
\newcommand{\bi}{\begin{itemize}}
\newcommand{\ei}{\end{itemize}}
\newcommand{\be}{\begin{enumerate}}
\newcommand{\ee}{\end{enumerate}}

\newcommand{\limi}[1]{\lim_{{#1} \to \infty}}

\newcommand{\af}{\alpha}
\newcommand{\bt}{\beta}
\newcommand{\gm}{\gamma}
\newcommand{\dt}{\delta}
\newcommand{\ep}{\varepsilon}
\newcommand{\zt}{\zeta}
\newcommand{\et}{\eta}
\newcommand{\ch}{\chi}
\newcommand{\io}{\iota}
\newcommand{\te}{\theta}
\newcommand{\ld}{\lambda}
\newcommand{\sm}{\sigma}
\newcommand{\kp}{\kappa}
\newcommand{\ph}{\varphi}
\newcommand{\ps}{\psi}
\newcommand{\rh}{\rho}
\newcommand{\om}{\omega}
\newcommand{\ta}{\tau}

\newcommand{\Gm}{\Gamma}

\newcommand{\Ph}{\Phi}
\newcommand{\Ps}{\Psi}

\newcommand{\Q}{{\mathbb{Q}}}
\newcommand{\Z}{{\mathbb{Z}}}
\newcommand{\R}{{\mathbb{R}}}
\newcommand{\N}{{\mathbb{N}}}
\newcommand{\Nz}{{\mathbb{Z}}_{\geq 0}}
\newcommand{\Hi}{{\mathcal{H}}}

\newcommand{\OI}{{\mathcal{O}}_{\infty}}
\newcommand{\OT}{{\mathcal{O}}_{2}}

\newcommand{\cZ}{{\mathcal{Z}}}

\newcommand{\btt}{{\widetilde{\bt}}}

\newcommand{\Lt}{{\mathtt{Lt}}}

\newcommand{\id}{{\operatorname{id}}}
\newcommand{\ev}{{\operatorname{ev}}}

\newcommand{\dist}{{\operatorname{dist}}}
\newcommand{\sa}{{\mathrm{sa}}}

\newcommand{\spec}{{\operatorname{sp}}}

\newcommand{\diag}{{\operatorname{diag}}}

\newcommand{\rank}{{\operatorname{rank}}}

\newcommand{\spn}{{\operatorname{span}}}
\newcommand{\card}{{\operatorname{card}}}
\newcommand{\Aut}{{\operatorname{Aut}}}

\newcommand{\Ad}{{\operatorname{Ad}}}

\newcommand{\Cu}{{\operatorname{Cu}}}
\newcommand{\T}{{\operatorname{T}}}
\newcommand{\QT}{{\operatorname{QT}}}
\newcommand{\rc}{{\operatorname{rc}}}
\newcommand{\W}{{\operatorname{W}}}
\newcommand{\Ker}{{\operatorname{Ker}}}

\newcommand{\cJ}{{\mathcal{J}}}

\newcommand{\dirlim}{\varinjlim}
\newcommand{\invlim}{\varprojlim}

\newcommand{\andeqn}{\qquad {\mbox{and}} \qquad}

\newcommand{\wolog}{without loss of generality}
\newcommand{\Wolog}{Without loss of generality}

\newcommand{\tfae}{the following are equivalent}
\newcommand{\ifo}{if and only if}

\newcommand{\ca}{C*-algebra}

\newcommand{\uca}{unital C*-algebra}

\newcommand{\hm}{homomorphism}

\newcommand{\uhm}{unital homomorphism}

\newcommand{\fd}{finite dimensional}

\newcommand{\tst}{tracial state}

\newcommand{\hsa}{hereditary subalgebra}

\newcommand{\pj}{projection}
\newcommand{\mops}{mutually orthogonal \pj s}
\newcommand{\nzp}{nonzero projection}
\newcommand{\mvnt}{Murray-von Neumann equivalent}

\newcommand{\ct}{continuous}
\newcommand{\cfn}{continuous function}

\newcommand{\chs}{compact Hausdorff space}

\newcommand{\trp}{tracial Rokhlin property}
\newcommand{\trpc}{tracial Rokhlin property with comparison}
\newcommand{\ucp}{unital completely positive}
\newcommand{\cpc}{completely positive contractive}

\renewcommand{\S}{\subseteq}
\newcommand{\ov}{\overline}
\newcommand{\SM}{\setminus}
\newcommand{\I}{\infty}
\newcommand{\E}{\varnothing}

\title[The tracial Rokhlin property with comparison]{Compact
 Group Actions with the Tracial Rokhlin Property}

\author{Javad Mohammadkarimi and N. Christopher Phillips}

\date{14~June 2022}

\address{Department of Pure Mathematics,
 Faculty of Mathematical Sciences,
 Tarbiat Modares University, Tehran, Iran
 \and
 Department of Mathematics, University  of Oregon,
       Eugene OR 97403-1222, USA.}

\thanks{The work of the second author was partially supported by the
Simons Foundation Collaboration Grant for Mathematicians
 \#587103 and by the US National Science Foundation under
  Grant DMS-2055771.}

\begin{document}

\begin{abstract}
We define a ``tracial'' analog of
the Rokhlin property for actions of second countable compact groups
on infinite dimensional simple separable unital C*-algebras.
We prove that fixed point algebras under such actions
(and, in the appropriate cases, crossed products by such actions)
preserve simplicity, Property~(SP), tracial rank zero,
tracial rank at most one,
the Popa property, tracial ${\mathcal{Z}}$-stability,
${\mathcal{Z}}$-stability when the algebra is nuclear,
infiniteness, and pure infiniteness.
We also show that the radius of comparison of the fixed point algebra
is no larger than that of the original algebra.
Our version of the tracial Rokhlin property is an exact
generalization of the tracial Rokhlin property for actions of
finite groups on classifiable C*-algebras
(in the sense of the Elliott program),
but for actions of finite groups on more general C*-algebras
it may be stronger.
We discuss several alternative versions of the tracial Rokhlin property.
We give examples of actions of a totally disconnected
infinite compact group on a UHF~algebra,
and of the circle group on a simple unital AT~algebra
and on ${\mathcal{O}}_{\infty}$,
which have our version of the tracial Rokhlin property,
but do not have the Rokhlin property,
or even finite Rokhlin dimension with commuting towers.
\end{abstract}

\maketitle

\tableofcontents

\section{Introduction}\label{Sec_1X21_Intro}

\indent
Tracially AF \ca{s}, now known as \ca{s} with tracial rank zero
(see~\cite{LnTTR}),
were introduced in~\cite{LnTAF}.
Roughly speaking, a \ca{}  has tracial rank zero if
the local approximation characterization of AF~algebras
holds after cutting out a ``small'' approximately central projection.
The term ``tracial'' comes from the fact that,
in good cases, a projection $p$ is ``small'' if $\ta (p) < \ep$
for every tracial state $\ta$ on~$A$.
Simple \ca{s} with tracial rank zero are much more common
than simple AF~algebras, and
the classification~\cite{Ln15} of simple separable nuclear \ca{s}
with tracial rank zero and satisfying the Universal Coefficient Theorem
can be regarded as a vast generalization of the classification
of AF~algebras.

The use of the Rokhlin property for actions of finite groups
on \ca{s} goes back at least to Herman and Jones
in \cite{HJ1} and~\cite{HJ2}.
It was motivated by earlier work in von Neumann algebras,
such as Jones~\cite{Jns}.
Major advances on classification of Rokhlin actions of finite groups
on \ca{s} which are classifiable in the sense of the Elliott program
appear in~\cite{Izum} and~\cite{Iz2}.
The Rokhlin property can be viewed as
a regularity condition for the group action, which
can be used to show that various structural
properties pass from a \ca{} to its crossed product.
This was first realized in Theorem~2.2 of~\cite{phill23}.
The paper~\cite{OsaPhi_crossed_2012}, and later work, followed up
on this idea.
See Theorem~2.6 of~\cite{phillfree} for a summary of what
was known at the time that survey was written.
For an example of more recent work, see~\cite{San_crossed_2015}.
The papers \cite{HirWin_rokhlin_2007} and~\cite{crossrep},
mentioned below, contain some permanence properties which were
proved for actions of compact groups
without first having been proved for actions of finite groups.

Actions of finite groups with the Rokhlin property are rare
(unlike actions of $\Z$ with the Rokhlin property).
See the discussion in~\cite{phillfree}, especially Section~3 there.
The tracial Rokhlin property for actions of finite groups,
introduced in~\cite{phill23},
is related to the Rokhlin property in roughly the same way
that tracial rank zero is related to the AF~property.
There are many more actions with the tracial Rokhlin property
than there are Rokhlin actions.
See Section~3 of~\cite{phillfree}, especially Example 3.12.
Crossed products by actions of finite groups with
the tracial Rokhlin property still preserve various structural
properties of \ca{s},
although not as many as with the Rokhlin property.
The original result was Theorem~2.6 of~\cite{phill23},
for tracial rank zero.
This result played a key role in the proof in~\cite{ELPW}
that the crossed products of irrational rotation algebras
by the ``standard'' actions of $\Z_3$, $\Z_4$, and~$\Z_6$ are~AF.
Other preservation of structure theorems can be found
in \cite{Arh1}, \cite{OskTry2}
(for a more general setup than group actions),
and~\cite{radifinite} (using the weak tracial Rokhlin property).

In~\cite{HirWin_rokhlin_2007}, Hirshberg and Winter introduced
the Rokhlin property for actions of
second countable compact groups on \ca{s}.
Since then,
crossed products by compact group actions with the Rokhlin property
have been studied by several authors.
In particular, permanence properties
are proved in~\cite{HirWin_rokhlin_2007} and~\cite{crossrep}.
As with finite groups, Rokhlin actions of compact groups are rare,
especially if the group is connected.
For example, by Theorem 3.3(3) of~\cite{GardKKcirc},
if $A$ is unital, $K_0 (A)$ is finitely generated,
and $A$ admits an action of the circle with the Rokhlin property,
then $K_0 (A) \cong K_1 (A)$.

In this paper, we therefore extend the definition of
the tracial Rokhlin property to actions
of second countable compact groups on simple separable unital \ca{s}.
Our property, which we call the \trpc,
is formally stronger than the naive extension.
We were unable to prove the desired permanence properties
using the naive extension, and we doubt that it can be done.
For actions of finite groups on simple separable unital \ca{s}
with strict comparison, and in some other cases, the \trpc{}
is equivalent to the tracial Rokhlin property,
but this seems unlikely to be true in general.
We then prove that fixed point algebras under such actions,
and, in the appropriate cases, crossed products by such actions,
preserve simplicity, Property~(SP), tracial rank zero,
tracial rank at most one,
the Popa property, tracial $\cZ$-stability,
the combination of nuclearity and $\cZ$-stability,
infiniteness, and pure infiniteness.
We also show that the radius of comparison of the fixed point algebra
is no larger than that of the original algebra.
The fact that the crossed product is usually not unital
causes some problems, and makes it more convenient to work with the
fixed point algebra.
In addition, the Popa property is only defined for unital algebras.
We further give examples of actions of both
an infinite totally disconnected
compact group and the circle group~$S^1$ which have the \trpc{}
but don't have the Rokhlin property.
Our examples include an action of $S^1$ on~$\OI$.
If we weaken the definition slightly,
but in a way that does not affect any of the permanence results,
then we get such an action on every
unital purely infinite simple separable nuclear \ca.
Actions of connected compact groups with the Rokhlin property
on AH~algebras without strict comparison will appear elsewhere.

We briefly address a different generalization of the
Rokhlin property, namely finite Rokhlin dimension with commuting towers.
This concept was defined for actions of compact groups
in~\cite{Gar_rokhlin_2017}, generalizing the version for finite
groups in~\cite{HWZ}.
It is clear from~\cite{HrsPh1} that even actions of finite
groups which have finite Rokhlin dimension with commuting towers
are relatively rare.
On the other hand, for finite~$G$,
and under some restrictions on the algebra
(infinite dimensional, simple, separable
[not stated, but used in the proof],
finite, unital, with strict comparison,
and having at most countably many extreme quasitraces),
finite Rokhlin dimension with commuting towers implies
the tracial Rokhlin property, by Theorem~3.4 of~\cite{Rokhdimtracial}.
We hope to address this question for compact groups in future work,
but we point out here that we prove that the examples in this paper
do not have finite Rokhlin dimension with commuting towers
as in~\cite{Gar_rokhlin_2017}.
One of our examples is an action of $S^1$ on~$\OI$,
and, by Corollary 4.23 of \cite{Gar_rokhlin_2017},
there is {\emph{no}} action of $S^1$
on $\OI$ which has finite Rokhlin dimension with commuting towers.

The paper is organized as follows.
In the rest of this section, we present some notation, definitions,
and basic lemmas that we will use throughout the paper.
Section~\ref{Sec_722_CTRAwC_Other_typ} contains the
definition of the tracial Rokhlin property with comparison
for compact groups, and its basic properties.
In particular, we give a first version of a central sequence
formulation of the definition,
and give an averaging process which is the key technical tool
for proofs of permanence properties.

In Section~\ref{Sec_1160_Simpli_Prf},
we prove that the fixed point algebra and the crossed product
of an infinite dimensional simple separable unital \ca{}
by an action of a compact group with the tracial Rokhlin property
with comparison are again simple.
We then give an improved version of a central sequence
formulation of the \trpc.
The rest of our permanence properties are in
Section~\ref{Sec_1951_TRP_Crossed_TRR0}.

Section~\ref{S_2795_N_TRP} discusses the relation between
the \trpc{} and the tracial Rokhlin property.
In particular, when the group is finite and under some
reasonable conditions on the algebra,
we prove that they are equivalent.
Other variants of the tracial Rokhlin property are possible,
and are suggested by the apparent failure of the naive
generalization to do what is wanted.
We discuss two of the most promising variants
in Section~\ref{S_2795_mod_TRP},
and describe what we can prove with them.

In Sections \ref{Sec_3749_Exam_TRPZ2}, \ref{Sec_1908_Exam_TRPS1},
and~\ref{Sec_2114_OI},
we construct two examples of actions which have the
tracial Rokhlin property with comparison.
The first is an action of $(\Z_2)^{\N}$ on the $3^{\I}$~UHF algebra.
The second is an action of the circle group~$S^1$
on a simple unital AT~algebra.
The third is an action of $S^1$ on~$\OI$.
These examples do not have the Rokhlin property,
or even finite Rokhlin dimension with commuting towers.
Indeed, there is no action at all of $S^1$ on~$\OI$
which has finite Rokhlin dimension with commuting towers.
In the first two cases, the actions do have one of the alternate versions
of the tracial Rokhlin property
discussed in Section~\ref{S_2795_mod_TRP}.

In Section~\ref{Sec_1919_NonE} we give an easy nonexistence result
for actions of $S^1$ with even the
weakest form of the tracial Rokhlin property we consider.

Most of this work was done while the first
author was a visiting scholar at the University of
Oregon during the period September 2019 to December 2020.
He wishes to thank that institution for its hospitality.
This paper constitutes a part of first author's Ph.D.\  dissertation.

The first author would like to thank Ilan Hirshberg for
an in person conversation when he was visiting
the second author at the
University of Oregon as well as Eusebio Gardella for
electronic correspondence about the Rokhlin property.
The first author would like to specially
thank and express his gratitude to Massoud Amini.

In the rest of this section, we collect some notation, definitions,
and results that we need.

The \ca{} of $n \times n$ matrices will be denoted by $M_{n}$.
If $C$ is a \ca,
we write $C_{+}$ for the set of positive elements in~$C$
and $C_{\sa}$ for the set of selfadjoint elements in~$C$.
If $C$ is unital, we denote its tracial state space by $\T (C)$,
its set of normalized $2$-quasitraces by $\QT (C)$,
and its unitary group by ${\operatorname{U}} (C)$.
Similarly, the set of unitary operators on a Hilbert space~$\Hi$
is denoted ${\operatorname{U}} (\Hi)$.
We denote the circle group by $S^{1}$, and identify
it with the set of complex numbers of absolute value~$1$.
An action $\alpha \colon G \to \Aut (A)$ of group $G$ on a \ca~$A$
is assumed \ct{} unless stated otherwise.
Also, we denote by $A^{\alpha}$ the fixed point
subalgebra of $A$ under $\alpha$.

We take $\N = \{ 1, 2, \ldots \}$,
and we abbreviate $\N \cup \{ 0 \}$ to $\Nz$.

\begin{ntn}\label{N_1Z31_EF}
If $A$ is a \ca{} and $E, F \S A$ are subsets, then we set
\[
E F = \spn \bigl( \bigl\{ x y \colon
 {\mbox{$x \in E$ and $y \in F$}} \bigr\} \bigr).
\]
\end{ntn}

We emphasize that we take the linear span of the products,
but not its closure: this is written ${\overline{E F}}$.

\begin{ntn}\label{N_1X07_SeqAlgs}
Let $A$ be a \ca.
We define
\[
l^{\I} (\N, A)
  = \left\{ (a_{n})_{n \in \N} \in A^{\N} \colon
     \sup_{n \in \N} \| a_{n} \| < \I\right\},
\]
\[
c_0 (\N, A)
 = \left\{(a_{n})_{n \in \N} \in l^{\I} (\N, A) \colon
      \lim\limits_{n \to \I} \| a_{n} \| = 0 \right\},
\]
and
\[
A_{\I} = l^{\I} (\N, A) / c_0 (\N, A).
\]
We further let $\pi_A \colon l^{\I} (\N, A) \to A_{\I}$
be the quotient map.

Now let $G$ be a topological group
and let $\af \colon G \to \Aut (A)$ be an action of $G$ on~$A$.
There are obvious actions $g \mapsto \af_{g}^{\I}$ of $G$
on $l^{\I} (\N, A)$
and $g \mapsto \af_{\I, g}$ of $G$ on $A_{\I}$,
which need {\emph{not}} be \ct.
We define
\[
l^{\I}_{\af} (\N, A)
 = \bigl\{ a \in l^{\I} (\N, A) \colon
   {\mbox{$g \mapsto \af_{g}^{\I} (a)$ is \ct}} \bigr\}
\]
and
\[
A_{\I, \af}
 = \pi_A \bigl( l^{\I}_{\af} (\N, A) \bigr)
 \S A_{\I}.
\]
The subalgebra $l^{\I}_{\af} (\N, A)$ is $\af^{\I}$-invariant,
so we still write $\af^{\I}$ for the action of $G$ on $A_{\I, \af}$.
By construction, this action {\emph{is}} \ct.
We further still write $g \mapsto \af_{\I, g}$ for the induced action
on $A_{\I, \af}$, which is also obviously \ct.

We identify $A$ in the obvious way
with the subalgebra of $l^{\I}_{\af} (\N, A)$
consisting of constant sequences,
and also with the image of this subalgebra in $A_{\I, \af}$
under $\pi_A$.
Then we can form the relative commutant algebra
$A_{\I, \af} \cap A' \S A_{\I, \af}$.
It is clearly $\af_{\I}$-invariant,
and we also denote the restricted action on this subalgebra
by $\af_{\I}$.
\end{ntn}

We have
\[
\begin{split}
A_{\I, \af} \cap A'
& = \Big\{\pi_{A} ((a_{n})_{n \in \N}) \in A_{\I} \colon
  {\mbox{$(a_{n})_{n \in \N} \in l^{\I}_{\af} (\N, A)$}}
\\
& \hspace*{3em} {\mbox{}}
    {\mbox{and $\lim\limits_{n \to \I} \| a_{n} a - a a_{n} \| = 0$
     for all $a \in A$}} \Big\}.
\end{split}
\]

\begin{dfn}[\cite{HrsPh1}, Definition 1.3]\label{phillhir}
Let $G$ be a compact group, and let $A$ and $D$ be unital \ca{s}.
Let $\alpha \colon G \to \Aut (A)$
and $\gamma \colon G \to \Aut ( D)$
be actions of $G$ on $A$ and $D$.
Let $S \S D$
and $F \S A$ be subsets, and let $\varepsilon > 0$.
A unital completely positive map $\varphi \colon D \to A$
is said to be an {\emph{$(S, F, \varepsilon )$-approximately equivariant
central multiplicative map}} if:
\begin{enumerate}
\item\label{Item_1X07_4}
$\| \varphi (x y) - \varphi (x) \varphi (y) \| < \varepsilon$
for all $x, y \in S$.
\item\label{Item_1X07_5}
$\| \varphi (x) a - a \varphi (x) \| < \varepsilon$
for all $x \in S$ and all $a \in F$.
\item\label{equiapprox}
$\mathrm{sup}_{g \in G} \| \varphi ( \gamma_{g} (x)) -
        \alpha_{g} (\ph (x)) \| < \varepsilon$
for all $x \in S$.
\end{enumerate}
\end{dfn}

Ignoring the actions and omitting condition~(\ref{equiapprox}),
we get the usual definition of
an $(S, F, \varepsilon )$-approximately central multiplicative map.

The following versions of approximate centrality
and approximate multiplicativity are convenient.

\begin{dfn}\label{D_1X10_nSFe}
Let $A$ and $D$ be unital \ca{s}, and let $S \subseteq D$.
A unital completely positive map $\varphi \colon D \to A$
is said to be an
{\emph{$(n, S, \varepsilon )$-approximately multiplicative map}}
if whenever $m \in \{ 1, 2, \ldots, n \}$
and $x_1, x_2, \ldots, x_m \in S$, we have
\[
\bigl\| \varphi (x_1 x_2 \cdots x_m)
  - \varphi (x_1) \varphi (x_2) \cdots \varphi (x_m) \bigr\|
   < \varepsilon.
\]
If also $F \subseteq A$ is given, then $\ph$ is said to be an
{\emph{$(n, S, F, \varepsilon )$-approximately central
multiplicative map}} if, in addition,
$\| \varphi (x) a - a \varphi (x) \| < \varepsilon$
for all $x \in S$ and all $a \in F$.
\end{dfn}

\begin{lem}\label{L_1X09_Get_nSFe}
Let $A$ and $D$ be unital \ca{s}, let $\varepsilon > 0$,
let $S \S D$ be compact, and let $n \in \N$.
\begin{enumerate}
\item\label{Item_1X11_Get_am}
There exist $\dt > 0$ and a compact subset $T \S D$
such that whenever $\ph \colon D \to A$ is \ucp{}
and is a $(T, \dt )$-approximately multiplicative map,
then $\varphi$ is an $(n, S, \ep)$-approximately multiplicative map.
\item\label{Item_1X11_Get_accm}
If in addition $F \S A$ is a compact subset,
then there exist $\dt > 0$ and compact subsets $T \S D$ and $E \S A$
such that whenever $\ph \colon D \to A$ is \ucp{}
and is a $(T, E, \dt )$-approximately central multiplicative map,
then $\varphi$ is an
$(n, S, F, \ep)$-approximately central multiplicative map.
\end{enumerate}
\end{lem}

\begin{proof}
The proof is routine, and is omitted.
\end{proof}

\begin{ntn}\label{N_1X07_Lt}
If $G$ is a locally compact group, we denote by
${\mathtt{Lt}} \colon G \to \Aut (C_0 (G))$
the action of $G$ on $C_0 (G)$ induced by the action of $G$ on itself
by left translation.
\end{ntn}

We also recall the following definitions related to Cuntz comparison.
The first part is originally from~\cite{Cuntz78}.

\begin{dfn}\label{Cuntz.df}
Let $A$ be a \ca.
\begin{enumerate}
\item\label{Cuntz_def_property_a}
For $a, b \in M_{\infty} (A)_{+}$,
we say that $a$ is
\emph{Cuntz subequivalent} to~$b$ in~$A$,
written $a \precsim_{A} b$,
if there is a sequence $(v_{n})_{n = 1}^{\infty}$ in $M_{\infty} (A)$
such that $\limi{n} v_{n} b v_{n}^{*} = a$.
\item\label{Cuntz_def_property_b}
We say that $a$ and $b$ are \emph{Cuntz equivalent} in~$A$,
written $a \sim_{A} b$, if $a \precsim_{A} b$ and $b \precsim_{A} a$.
This relation is an equivalence relation.
\item\label{1X18_Cuntz_def_dtau}
If $A$ is unital, $\ta \in \QT (A)$, and $a \in M_{\infty} (A)_{+}$,
then we define $d_{\ta} (a) = \lim_{n \to \infty} \ta (a^{1 / n})$.
\item\label{1X30_Cuntz_def_rc}
If $A$ is unital, the {\emph{radius of comparison}} $\rc (A)$ is the
infimum of all $\rh > 0$
such that whenever $a, b \in M_{\infty} (A)_{+}$
and $d_{\ta} (a) + \rh < d_{\ta} (b)$ for all $\ta \in \QT (A)$,
then $a \precsim_{A} b$.
\end{enumerate}
\end{dfn}

\begin{lem}[\cite{philar}, Lemma 2.6]\label{lma_L683_baj}
Let $A$ be a simple \ca,
and let $B\subseteq A$ be a nonzero hereditary subalgebra.
Let $n \in \N$, and let
$a_{1}, \ldots, a_{n} \in A_{+} \setminus \{0 \}$.
Then there exists $b \in B_{+} \setminus \{0 \}$ such
that $b \precsim_{A} a_{j}$ for $j = 1, \ldots, n$.
\end{lem}

\begin{lem}[\cite{phill23}, Lemma 1.10]\label{OrthInSP}
Let $A$ be an infinite dimensional simple unital \ca{}
with Property~(SP).
Let $B \S A$ be a nonzero hereditary subalgebra, and let $n \in \N$.
Then there exist nonzero Murray-von
Neumann equivalent mutually orthogonal projections
$p_{1}, p_{2}, \ldots, p_{n} \in B$.
\end{lem}

\section{The tracial Rokhlin property with
 comparison}\label{Sec_722_CTRAwC_Other_typ}

In this section we define the
tracial Rokhlin property with comparison for
actions of compact groups, and prove several equivalent versions
and some basic facts.
We begin with a brief reminder of the tracial Rokhlin property
for actions of finite groups
and the Rokhlin property for actions of compact groups.
These are the properties we combine in this paper.

When restricted to finite groups, our version of the
tracial Rokhlin property for actions of compact groups
is stronger than the usual tracial Rokhlin property
for actions of finite groups.
We discuss this issue in more detail
in Section~\ref{S_2795_N_TRP},
where we also discuss other possible definitions.
We have been unable to prove the expected results just using
the naive generalization of the tracial Rokhlin property
for actions of finite groups.

In Lemma~\ref{mtrpcentral}, we give a version of the definition
(for separable \ca{s}) in terms of central sequence algebras.
It will be improved later, but the proof of the improvement
uses simplicity of the fixed point algebra, whose proof in turn
uses the version in this section.
We end this section with a theorem on existence of
suitable maps from $A$ to $A^{\af}$,
which in many proofs involving the \trpc{} is what we actually use.

\begin{dfn}[\cite{phill23}, Definition 1.2]\label{D_1619_TRP}
Let $A$ be an infinite dimensional simple unital \ca,
let $G$ be a finite group,
and let $\af \colon G \to \Aut (A)$ be an action of $G$ on~$A$.
The action $\af$ has the {\emph{tracial Rokhlin property}}
if for every finite set $F \subseteq A$, every $\ep > 0$, and
every $x \in A_{+}$ with $\| x \| = 1$,
there exist \mops{} $p_g \in A$ for $g \in G$ such that,
with $p = \sum_{g \in G} p_g$, we have:
\begin{enumerate}
\item\label{Item_1619_tRp_p_comm}
$\| p_g a - a p_g \| < \ep$ for all $a \in F$ and all $g \in G$.
\item\label{Item_1619_tRp_CM}
$\| \af_g (p_h) - p_{g h} \| < \ep$ for all $g, h \in G$.
\item\label{Item_1619_tRp_sub_x}
$1 - p \precsim_A x$.
\item\label{Item_1619_tRp_pxp}
$\| p x p \| > 1 - \ep$.
\end{enumerate}
\end{dfn}

Separability is assumed in Definition 1.2 of~\cite{phill23},
but there is no reason to require separability.

We note in passing that, by Proposition 5.26 of~\cite{phieqsemi},
in~(\ref{Item_1619_tRp_CM}), one can require
$\af_g (e_h) = e_{g h}$ for all $g, h \in G$.

Hirshberg and Winter defined the
Rokhlin property for an action of a second countable compact
group in Definition~3.2 of~\cite{HirWin_rokhlin_2007}.
See the explanation before Definition 2.3 in \cite{Gdla}.

\begin{dfn} [\cite{Gdla}, Definition 2.3]\label{df_Rpcpt}
Let $A$ be a separable \uca, let $G$ be a second countable compact
group, and let $\alpha \colon G \to \Aut (A)$ be an
action of $G$ on~$A$.
Recalling Notation~\ref{N_1X07_Lt},
we say that $\alpha$ has the {\emph{Rokhlin property}} if
there is an equivariant unital homomorphism
\[
\varphi \colon
 (C (G), {\mathtt{Lt}}) \to (A_{\I, \alpha} \cap A', \alpha_{\I}).
\]
\end{dfn}

Separability is not assumed in~\cite{Gdla},
but the use of the central sequence algebra means that the
definition is not appropriate for nonseparable \ca{s}.
The right definition for the nonseparable case is the condition
in the following lemma.

\begin{lem}\label{L_1X16_AppDfRk}
Let $A$ be a separable \uca, let $G$ be a second countable compact
group, and let $\alpha \colon G \to \Aut (A)$ be an
action of $G$ on~$A$.
Then $\af$ has the Rokhlin property \ifo{} for every
if for every finite set $F \subseteq A$,
every finite set $S \subseteq C (G)$,
and every $\varepsilon > 0$, there exists a
unital completely positive map $\varphi \colon C (G) \to A$
which is $(F, S, \varepsilon)$-approximately equivariant
central multiplicative (Definition~\ref{phillhir}).
\end{lem}

\begin{proof}
The proof is standard, and is omitted.
\end{proof}

\begin{dfn}\label{traR}
Let $A$ be an infinite dimensional simple unital \ca,
and let $\alpha \colon G \to \Aut (A)$ be
an action of a second countable compact group $G$ on~$A$.
The action $\alpha$ has the
\emph{tracial Rokhlin property with comparison}
if for every finite set $F \subseteq A$,
every finite set $S \subseteq C (G)$,
every $\varepsilon > 0$, every $x \in A_{+}$ with $\| x \| = 1$,
and every $y \in (A^{\alpha})_{+} \setminus \{ 0 \}$,
there exist a projection $p \in A^{\alpha}$ and a
unital completely positive map $\varphi \colon C (G) \to p A p$
such that the following hold.
\begin{enumerate}
\item\label{Item_893_FS_equi_cen_multi_approx}
$\varphi$ is an $(F, S, \varepsilon)$-approximately equivariant
central multiplicative map (Definition~\ref{phillhir}).
\item\label{1_pxcompactsets}
$1 - p \precsim_{A} x$.
\item\label{1_pycompactsets}
$1 - p \precsim_{A^{\alpha}} y$.
\item\label{1_ppcompactsets}
$1 - p \precsim_{A^{\alpha}} p$.
\item\label{Item_902_pxp_TRP}
$\| p x p \| > 1 - \varepsilon$.
\end{enumerate}
\end{dfn}

Definition~\ref{traR} requires that $p$ be $\af$-invariant,
a requirement not present in Definition~\ref{D_1619_TRP}.
But Definition~\ref{D_1619_TRP} is unchanged if this requirement
is added, by Lemma~1.17 of~\cite{phill23}.
Definition~\ref{traR} has two conditions which have no analogs
in Definition~\ref{D_1619_TRP}.
Condition~(\ref{1_pycompactsets}) is automatic for finite groups.
We do not know whether it is automatic in general,
but there is evidence to suggest that it isn't.
Condition~(\ref{1_ppcompactsets}) is automatic for finite groups
under some additional conditions, for example,
if $A$ has strict comparison.
Even for finite groups, we do not know whether it is always automatic.
For this reason, we don't use the term ``tracial Rokhlin property''
in Definition~\ref{traR}.
See Section~\ref{S_2795_N_TRP} for these results, and more.
Condition~(\ref{Item_902_pxp_TRP}) is included to ensure that,
for finite groups,
the \trpc{} implies the tracial Rokhlin property
as originally defined (Definition~\ref{D_1619_TRP}).
It is not needed for any of the permanence properties we prove;
see Remark~\ref{R_2015_nzp} and Remark~\ref{L_2015_S4_nzp}.
It is redundant when $A$ is stably finite;
see Lemma~\ref{L_1X09_NoNorm1}.

When $G$ is finite, and assuming suitable separability and
nuclearity hypotheses, this definition can be reformulated
in terms of the central sequence algebra.
(As far as we know, such a reformulation
has not appeared in the literature.)
In Lemma~\ref{mtrpcentral} and Proposition~\ref{P_1X14_CentSq}
below, we give central sequence algebra formulations
of Definition~\ref{traR}.

\begin{rmk}\label{R_1Z09_Variants}
The main technical result of this section
is Theorem~\ref{thm_ApproxHommtr}, which for a given tolerance
and given compact subsets of $A$ and $A^{\alpha}$
gives an $\af$-invariant projection $p$ which is ``large''
in the senses described in Conditions (\ref{1_pycompactsets}),
(\ref{1_ppcompactsets}), and~(\ref{Item_902_pxp_TRP})
of Definition~\ref{traR}, and a
\ucp{} approximately multiplicative map from $A$ to $p A^{\alpha} p$.
Primarily for our discussion of other versions of Definition~\ref{traR},
we use variants in which some of
these conditions are omitted or modified, and we will need
the corresponding versions of Theorem~\ref{thm_ApproxHommtr}.
For this purpose, we point out here that, in the proof of
Theorem~\ref{thm_ApproxHommtr} and the lemmas leading up to it,
these conditions are treated independently, in some cases only as long
as enough of them are present to ensure that $p \neq 0$.
We explicitly state several versions which we use (after
Theorem~\ref{thm_ApproxHommtr}:
Remark~\ref{R_1Z09_ntr}, Remark~\ref{R_1Z09_nzp},
Remark~\ref{R_1Z09_mod_trp}, and Remark~\ref{R_S_mod_trp}),
and describe how the reasoning below must be modified, but
the method is quite general.
\end{rmk}

For comparison, we give a reformulation of Definition~\ref{traR}
for finite groups
which more closely resembles Definition~\ref{D_1619_TRP}.

\begin{lem}\label{L_1X16_trpc_GFin}
Let $A$ be an infinite dimensional simple unital \ca,
let $G$ be a finite group,
and let $\af \colon G \to \Aut (A)$ be an action of $G$ on~$A$.
Then $\af$ has the \trpc{} \ifo{}
for every finite set $F \subseteq A$, every $\ep > 0$,
every $x \in A_{+}$ with $\| x \| = 1$,
and every $y \in (A^{\af})_{+} \setminus \{ 0 \}$
there exist a projection $p \in A^{\alpha}$
and mutually orthogonal projections $(p_{g})_{g \in G}$
such that the following hold.
\begin{enumerate}
\item\label{Item_1X16_Inv}
$p = \sum_{g \in G} p_{g}$.
\item\label{Item_1X16_Comm}
$\| p_g a - a p_g \| < \ep$ for all $a \in F$ and all $g \in G$.
\item\label{Item_1X16_Prm}
$\| \af_g (p_h) - p_{g h} \| < \ep$ for all $g, h \in G$.
\item\label{Item_1X16_sub_x}
$1 - p \precsim_A x$.
\item\label{Item_1X16_sub_yy}
$1 - p \precsim_{A^{\alpha}} y$.
\item\label{Item_1X16_sub_1mp}
$1 - p \precsim_{A^{\alpha}} p$.
\item\label{Item_1X16_Mnp}
$\| p x p \| > 1 - \ep$.
\setcounter{TmpEnumi}{\value{enumi}}
\end{enumerate}
\end{lem}

\begin{proof}
Set
$S = \bigl\{ \ch_{ \{ g \} } \colon g \in G \bigr\} \S C (G)$.
It is easily seen that Definition~\ref{traR} is equivalent
(with a change in the value of $\ep$) to the same statement
but in which we always use this choice of~$S$.

Assume the conditions of the lemma.
Let $F \subseteq A$ be finite,
let $\varepsilon > 0$, let $x \in A_{+}$ satisfy $\| x \| = 1$,
and let $y \in (A^{\alpha})_{+} \setminus \{ 0 \}$.
Let $p$ and $(p_{g})_{g \in G}$ be as in the condition of the lemma
for these choices.
Then Conditions (\ref{1_pxcompactsets}), (\ref{1_pycompactsets}),
(\ref{1_ppcompactsets}), and~(\ref{Item_902_pxp_TRP})
in Definition~\ref{traR}
follow immediately.
Define a unital \hm{} $\ph \colon C (G) \to p A p$ by
$\ph (f) = \sum_{g \in G} f (g) p_g$ for $f \in C (G)$.
The following calculations then show that
$\varphi$ is $(F, S, \varepsilon)$-approximately equivariant
central, and prove
(\ref{Item_893_FS_equi_cen_multi_approx}) in Definition~\ref{traR}.
First, for $h \in G$, recalling Notation~\ref{N_1X07_Lt},
and by~(\ref{Item_1X16_Prm}) at the second step, we have
\[
\max_{g \in G} \bigl\| (\ph \circ \Lt_g) ( \ch_{ \{ h \} })
       - (\af_g \circ \ph) ( \ch_{ \{ h \} }) \bigr\|
 = \max_{g \in G} \| p_{g h} - \af_g (p_h) \|
 < \ep.
\]
Second, for $g \in G$ and $a \in F$,
by~(\ref{Item_1X16_Comm}) at the second step, we have
\[
\| \ph (\ch_{ \{ g \} }) a - a \ph (\ch_{ \{ g \} }) \|
 = \| p_g a - a p_g \|
 < \ep.
\]

For the other direction, assume that $\af$ has the \trpc.
Set $n = \card (G)$.
Let $F \subseteq A$ be finite,
let $\varepsilon > 0$, let $x \in A_{+}$ satisfy $\| x \| = 1$,
and let $y \in (A^{\alpha})_{+} \setminus \{ 0 \}$.
\Wolog{} $\| a \| \leq 1$ for all $a \in A$.
Set
\[
\ep_0 = \min \biggl( \frac{1}{2 n}, \, \frac{\ep}{3} \biggr).
\]
Choose $\dt > 0$ so small that $\dt \leq \ep_0$ and
whenever $B$ is a \ca{}
and $b_1, b_2, \ldots, b_n \in B$ are selfadjoint
and satisfy $\| b_j^2 - b_j \| < 3 \dt$ for $j = 1, 2, \ldots, n$
and $\| b_j b_k \| < \dt$
for distinct $j, k \in \{ 1, 2, \ldots, n \}$,
then there are \mops{} $e_j \in B$ for $j = 1, 2, \ldots, n$
such that $\| e_j - b_j \| < \ep_0$ for $j = 1, 2, \ldots, n$.

Apply Definition~\ref{traR} with $\dt$ in place of~$\ep$,
and with $F$, $x$, and~$y$ as given,
getting $p \in A^{\af}$ and $\ph \colon C (G) \to p A p$ as there.
In particular:
\begin{enumerate}
\setcounter{enumi}{\value{TmpEnumi}}
\item\label{It_1X21_26_Sq}
$\| \ph (\ch_{ \{ g \} })^2 - \ph (\ch_{ \{ g \} }) \| < \dt$
for all $g \in G$.
\item\label{It_1X21_26_Z}
$\| \ph (\ch_{ \{ g \} }) \ph (\ch_{ \{ h \} }) \| < \dt$
for all $g, h \in G$ with $g \neq h$.
\item\label{It_1X21_26_Comm}
$\| \ph (\ch_{ \{ g \} }) a - a \ph (\ch_{ \{ g \} }) \| < \dt$
for all $g \in G$ and all $a \in F$.
\item\label{It_1X21_26_Tr}
$\| (\af_g \circ \ph) (\ch_{ \{ h \} }) - \ph (\ch_{ \{ g h \} }) \|
   < \dt$
for all $g, h \in G$.
\end{enumerate}
By (\ref{It_1X21_26_Sq}), (\ref{It_1X21_26_Z}),
and the choice of $\dt$, there are \mops{}
$p_g$ for $g \in G$
such that $\| p_g - \ph (\ch_{ \{ g \} }) \| < \ep_0$.
Using~(\ref{It_1X21_26_Comm}), for $g \in G$ and $a \in F$ we get
\[
\| p_g a - a p_g \|
 \leq \| \ph (\ch_{ \{ g \} }) a - a \ph (\ch_{ \{ g \} }) \|
        + 2 \| p_g - \ph (\ch_{ \{ g \} }) \|
 < \dt + 2 \ep_0 \leq \ep.
\]
This is~(\ref{Item_1X16_Comm}).

Using~(\ref{It_1X21_26_Tr}), for $g, h \in G$ we get
\[
\begin{split}
\| \af_g (p_h) - p_{g h} \|
& \leq \| (\af_g \circ \ph) (\ch_{ \{ h \} })
           - \ph (\ch_{ \{ g h \} }) \|
        + \| p_h - \ph (\ch_{ \{ h \} }) \|
        + \| p_{g h} - \ph (\ch_{ \{ g h \} }) \|
\\
& < \dt + 2 \ep_0
  \leq \ep.
\end{split}
\]
This is~(\ref{Item_1X16_Prm}).
Also, since $\ph (1) = 1$,
\[
\biggl\| p - \sum_{g \in G} p_g \biggr\|
 \leq \sum_{g \in G} \| \ph (\ch_{ \{ g \} }) - p_g \|
 < n \ep_0
 < 1,
\]
so, since $\sum_{g \in G} p_g$ is a \pj{},
we get $\sum_{g \in G} p_g = p$.
This is~(\ref{Item_1X16_Inv}).
Conditions (\ref{Item_1X16_sub_x}), (\ref{Item_1X16_sub_yy}),
(\ref{Item_1X16_sub_1mp}), and~(\ref{Item_1X16_Mnp})
are Conditions (\ref{1_pxcompactsets}), (\ref{1_pycompactsets}),
(\ref{1_ppcompactsets}), and~(\ref{Item_902_pxp_TRP})
in Definition~\ref{traR}.
\end{proof}

\begin{prp}\label{P_1X19_RP_to_TRPC}
Let $A$ be an infinite dimensional simple unital \ca,
and let $\alpha \colon G \to \Aut (A)$ be
an action of a second countable compact group $G$ on~$A$.
If $\af$ has the Rokhlin property,
then $\af$ has the tracial Rokhlin property with comparison.
\end{prp}

\begin{proof}
By Lemma~\ref{L_1X16_AppDfRk}, one can take $p = 1$
in Definition~\ref{traR}.
\end{proof}

\begin{lem}\label{PROSPTRPCA}
Let $A$ be an infinite dimensional simple separable unital \ca.
Let $\alpha \colon G \to \Aut (A)$ be
an action of a second countable compact group $G$
on $A$ which has the tracial Rokhlin property with comparison.
Then $A$ has Property~(SP) or $\alpha$ has the Rokhlin property.
\end{lem}

\begin{proof}
Suppose $A$ does not have Property~(SP).
We verify the condition of Lemma~\ref{L_1X16_AppDfRk}.
Let $F \subseteq A$ and $S \subseteq C (G)$ be finite,
and let $\varepsilon > 0$.
Choose $x \in A_{+} \setminus \{ 0 \}$ such that $\| x \| = 1$
and ${\overline{x A x}}$ contains no nonzero \pj, and take $y = 1$.
Apply Definition~\ref{traR}.
Then $p = 1$ by Definition \ref{traR}(\ref{1_pxcompactsets}),
so the condition of Lemma~\ref{L_1X16_AppDfRk} holds.
\end{proof}

Definition~\ref{traR} is unchanged if we merely assume that
$F$ and $S$ are compact instead of finite,
and require that the map $\varphi \colon C (G) \to p A p$
be exactly equivariant and be an
$(n, F, S, \varepsilon)$-approximately central multiplicative map.

\begin{lem}\label{TRPcequi}
Let $A$ be an infinite dimensional simple unital \ca,
and let $\alpha \colon G \to \Aut (A)$ be
an action of a second countable compact group $G$ on~$A$.
Then $\alpha$ has the tracial Rokhlin property with comparison
if and only if for every compact
set $F \subseteq A$, every compact set $S \subseteq C (G)$,
every $\varepsilon > 0$, every $n \in \N$,
every $x \in A_{+}$ with $\| x \| = 1$,
and every $y \in (A^{\alpha})_{+} \setminus \{ 0 \}$,
there exist a projection $p \in A^{\alpha}$ and an
equivariant unital completely positive map
$\varphi \colon C (G) \to p A p$ such that the following hold.
\begin{enumerate}
\item\label{FSexactequi}
$\varphi$ is an
$(n, F, S, \varepsilon)$-approximately central multiplicative map
(Definition~\ref{D_1X10_nSFe}).
\item\label{1_pxcompactsetscpt}
$1 - p \precsim_{A} x$.
\item\label{1_pycompactsetscpt}
$1 - p \precsim_{A^{\alpha}} y$.
\item\label{1_ppcompactsetscpt}
$1 - p \precsim_{A^{\alpha}} p$.
\item\label{pxpcompactsetscpt}
$\| p x p \| > 1 - \varepsilon$.
\end{enumerate}
\end{lem}

\begin{proof}
If in~(\ref{FSexactequi}) we merely ask for an
$(F, S, \varepsilon)$-approximately central multiplicative map,
the argument is the same as in Remark 1.4 of \cite{HrsPh1}.
To get~(\ref{FSexactequi}) as stated,
use Lemma \ref{L_1X09_Get_nSFe}(\ref{Item_1X11_Get_accm}).
\end{proof}

If $A$ is finite, then
Condition~(\ref{Item_902_pxp_TRP}) in Definition~\ref{traR}
and Condition~(\ref{pxpcompactsetscpt}) in Lemma~\ref{TRPcequi}
are redundant.

\begin{lem}\label{L_1X09_NoNorm1}
Let $A$ be a finite infinite dimensional simple unital \ca,
and let $\alpha \colon G \to \Aut (A)$ be
an action of a second countable compact group $G$ on~$A$.
Then $\af$ has the tracial Rokhlin property with comparison
\ifo{} for every finite
set $F \subseteq A$, every finite set $S \subseteq C (G)$,
every $\varepsilon > 0$, every $x \in A_{+} \setminus \{ 0 \}$,
and every $y \in (A^{\alpha})_{+} \setminus \{ 0 \}$,
there exist a projection $p \in A^{\alpha}$ and a
unital completely positive map $\varphi \colon C (G) \to p A p$
such that the following hold.
\begin{enumerate}
\item\label{L_1X09_NoNorm1_ecmm}
$\varphi$ is an $(F, S, \varepsilon)$-approximately equivariant
central multiplicative map.
\item\label{L_1X09_NoNorm1_1mpx}
$1 - p \precsim_{A} x$.
\item\label{L_1X09_NoNorm1_1py}
$1 - p \precsim_{A^{\alpha}} y$.
\item\label{L_1X09_NoNorm1_1mppp}
$1 - p \precsim_{A^{\alpha}} p$.
\end{enumerate}
Moreover, $\af$ has the tracial Rokhlin property with comparison
\ifo{} for every $\varepsilon$, $x$, and $y$ as above,
every compact set $F \subseteq A$,
every compact set $S \subseteq C (G)$,
and every $n \in \N$,
there exist a projection $p \in A^{\alpha}$ and an equivariant
unital completely positive map $\varphi \colon C (G) \to p A p$
such that
conditions (\ref{L_1X09_NoNorm1_1mpx}),
(\ref{L_1X09_NoNorm1_1py}), and~(\ref{L_1X09_NoNorm1_1mppp}) hold, and
$\varphi$ is an $(F, S, \varepsilon)$-approximately
central multiplicative map.
\end{lem}

\begin{proof}
Given Proposition~\ref{P_1X19_RP_to_TRPC},
Lemma~\ref{PROSPTRPCA}, and Lemma~\ref{TRPcequi},
the proof is the same as that of Lemma~1.16 of~\cite{phill23}.
\end{proof}

When $A$ is separable, we characterize
the tracial Rokhlin property with comparison
in terms of central sequences.
(A better version will be given in Proposition~\ref{P_1X14_CentSq}.)
We begin with a definition.

\begin{dfn}\label{smallness_in_centralseq}
Let $A$ be a \ca{} and let $\alpha \colon G \to \Aut (A)$
be an action of a topological group $G$ on~$A$.
We say that a projection $p$ in $A_{\I, \af}$ is {\emph{$\af$-small}}
if whenever $(q_{n})_{n \in \N} \in l^{\I}_{\af} (\N, A)$
is a sequence of projections in $A^{\af}$ which lifts~$p$,
then for every nonzero $x \in A_{+}$
there exist $N \in \N$ such that for every $n \geq N$
we have $q_{n} \precsim_{A} x$.
Without the action (that is, if $G = \{ 1 \}$ and $\af$ is trivial),
we say that $p$, now in $A_{\I}$, is {\emph{small}}.
\end{dfn}

It is enough to consider only one lift instead of all of them.

\begin{lem}\label{L_1X09_SmallPj}
Let $A$ be a \ca{} and let $\alpha \colon G \to \Aut (A)$
be an action of a topological group $G$ on~$A$.
Let $p \in A_{\I}$ be a projection.
Then $p$ is $\af$-small \ifo{} there is a sequence
$(q_{n})_{n \in \N} \in l^{\I}_{\af} (\N, A)$ which lifts~$p$
and such that for every nonzero $x \in A_{+}$
there exists $N \in \N$ such that for every $n \geq N$
we have $q_{n} \precsim_{A} x$.
\end{lem}

\begin{proof}
Only one direction needs proof.
Assume the condition of the lemma,
and let $(e_{n})_{n \in \N} \in l^{\I}_{\af} (\N, A)$
be any other sequence of projections in $A^{\af}$ which lifts~$p$.
Then $\lim_{n \to \infty} \| e_n - q_n \| = 0$.
So there is $N_0 \in \N$ such that for all $n \geq N_0$ the \pj{s}
$q_n$ and $e_n$ are \mvnt.
Then $q_{n} \precsim_{A} x$ \ifo{} $e_{n} \precsim_{A} x$.
So, for every nonzero $x \in A_{+}$,
there exists $N \in \N$ such that for every $n \geq N$
we have $e_{n} \precsim_{A} x$.
\end{proof}

\begin{lem}\label{L_1X09_LiftPj}
Let $A$ be a \uca,
and let $\alpha \colon G \to \Aut (A)$ be
an action of a compact group $G$ on~$A$.
Let $p \in A_{\I, \af}$ be an $\af_{\I}$-invariant \pj.
Then there exists an $\af^{\I}$-invariant \pj{}
$q \in l^{\I}_{\af} (\N, A)$ such that $\pi_A (q) = p$.
\end{lem}

\begin{proof}
Choose any positive element
$a = (a_{n})_{n \in \N} \in l^{\I}_{\af} (\N, A)$
such that $\pi_A (a) = p$.
By averaging over~$G$, we may assume that $a$ is $\af^{\I}$-invariant.
Since $\pi_A (a)$ is a \pj,
we have $\lim_{n \to \I} \| a_n^2 - a_n \| = 0$.
Therefore there is $N \in \N$ such that for all $n \geq N$
we have $\spec (a_n) \cap \bigl( \frac{1}{3}, \frac{2}{3} \bigr) = \E$.
Define
\[
b_n = \begin{cases}
   a_n & \hspace*{1em} n \geq N
        \\
   0   & \hspace*{1em} n < N.
\end{cases}
\]
One easily checks that
$b = (b_{n})_{n \in \N} \in l^{\I}_{\af} (\N, A)$,
that $\pi_A (b) = p$, and that
$\spec (b) \cap \bigl( \frac{1}{3}, \frac{2}{3} \bigr) = \E$.
Set $q = \ch_{[ 1/2, \, \I)} (b)$.
\end{proof}

\begin{lem}\label{L_1X09_CSbPj}
Let $A$ be a \ca, let $a \in A_{+}$, and let $p \in A$ be a \pj.
Suppose there is $v \in A$ such that $\| v^* a v - p \| < 1$.
Then there is $w \in A$ such that $w^* a w = p$,
and $p \precsim_{A} a$.
\end{lem}

\begin{proof}
The second statement follows from the first.

For the first, observe that $\| p v^* a v p - p \| < 1$.
Therefore $p v^* a v p$ has an inverse in $p A p$.
Call it~$c$.
Then $w = v p c^{1/2}$ satisfies $w^* a w = p$.
\end{proof}

\begin{lem}\label{L_1X10_y0y1p}
Let $A$ be a \ca, let $y_0, y_1 \in A_{+}$ satisfy $y_0 y_1 = y_1$,
and let $p \in A$ be a \pj{} such that $p \precsim_A y_1$.
Then there exists $v \in A$
such that $\| v \| \leq 1$ and $v^* y_0 v = p$.
\end{lem}

\begin{proof}
Choose $w_0 \in A$ such that $\| w_0^* y_1 w_0 - p \| < 1$.
Lemma~\ref{L_1X09_CSbPj} provides $w \in A$ such that $w^* y_1 w = p$.
Set $v = y_1^{1/2} w$.
Then $v^* v = p$ so $\| v \| \leq 1$.
(If $p = 0$ then $\| v \| = 0$.)
Moreover, $y_1^{1/2} y_0 = y_1^{1/2}$, so $v^* y_0 v = p$.
\end{proof}

\begin{lem}\label{mtrpcentral}
Let $G$ be a second countable compact group,
let $A$ be a simple separable infinite dimensional \uca, and let
$\alpha \colon G \to \Aut (A)$ be an action of $G$ on~$A$.
Then $\af$ has the tracial Rokhlin property with comparison
\ifo{} for
every $x \in A_{+}$ with $\| x \| = 1$
and every $y \in (A^{\alpha})_{+} \setminus \{ 0 \}$,
there exist a projection
$p \in (A_{\I, \af} \cap A')^{\alpha_{\infty}}$
and a unital equivariant homomorphism
\[
\ps \colon C (G) \to p (A_{\I, \af} \cap A') p
\]
such that the following hold:
\begin{enumerate}
\item\label{1_psmalllnes}
$1 - p$ is $\af$-small in $A_{\I, \af}$.
\item\label{1_psmalllnes_FP}
$1 - p \precsim_{(A^{\af})_{\I}} y$.
\item\label{1_pandpinA_inf}
$1 - p \precsim_{(A^{\af})_{\I}} p$.
\item\label{nonzerop1_var}
$\| p x p\| = 1$.
\setcounter{TmpEnumi}{\value{enumi}}
\end{enumerate}
\end{lem}

We will later improve the statement to replace~(\ref{1_psmalllnes_FP})
with the statement that $1 - p$ is small in $(A^{\af})_{\I}$.
(See Proposition~\ref{P_1X14_CentSq}.)
For this, however, we need to know that $A^{\af}$ is simple,
and the statement here is used in the proof of that fact.
Once we have simplicity of $A^{\af}$,
the proof will be essentially the same, and we will refer
back to this proof for most of the steps.

\begin{proof}[Proof of Lemma~\ref{mtrpcentral}]
We first prove that the condition implies the \trpc.

Let $F \subseteq A$ and $S \subseteq C (G)$ be finite,
let $\varepsilon > 0$, let $x \in A_{+}$ satisfy $\| x \| = 1$,
and let $y \in (A^{\alpha})_{+} \setminus \{ 0 \}$.
Choose $p$ and $\ph$ as in the statement for $x$ and $y$ as given.
Use Lemma~\ref{L_1X09_LiftPj} to choose an $\af^{\I}$-invariant \pj{}
$q = (q_{n})_{n \in \N} \in l^{\I}_{\af} (\N, A)$
such that $\pi_A (q) = p$.
Then $\pi_A$ defines a surjective map
from $q l^{\I}_{\af} (\N, A) q$ to $p A_{\I, \af} p$.
Use the Choi-Effros lifting theorem to choose a lift
$\te = (\te_n)_{n \in \N} \colon C (G) \to q l^{\I}_{\af} (\N, A) q$
of $\ps$,
consisting of \ucp{} maps $\te_n \colon C (G) \to q_{n} A q_{n}$.
We claim that the following properties hold
(recalling Notation~\ref{N_1X07_Lt} for~(\ref{Item_1X07_17})):
\begin{enumerate}
\setcounter{enumi}{\value{TmpEnumi}}
\item\label{Item_1X07_15}
$\lim\limits_{n \to \I}
 \| \te_n (f_{1} f_{2}) - \te_n (f_{1}) \te_n (f_{2}) \| = 0$
for all $f_{1}, f_{2} \in C (G)$.
\item\label{Item_1X07_16}
$\lim\limits_{n \to \I} \|\te_n (f) a - a \te_n (f) \| = 0$
for all $a \in A$ and all $f \in C (G)$.
\item\label{Item_1X07_17}
$\lim\limits_{n \to \I} \sup\limits_{g \in G}
     \| (\te_n \circ \Lt_g) (f) - (\alpha_g \circ \te_n) (f) \| = 0$
for all $f \in C (G)$.
\setcounter{TmpEnumi}{\value{enumi}}
\end{enumerate}

Properties (\ref{Item_1X07_15}) and~(\ref{Item_1X07_16})
follow immediately from the fact that $\pi_A \circ \te = \ps$
is a \hm{} and has range contained in $A_{\I} \cap A'$.
For~(\ref{Item_1X07_17}), let $f \in C (G)$,
and define $\rh_n \colon G \to [0, \I)$
by
\[
\rh_n (g)
 = \sup_{m \geq n}
   \| (\te_m \circ \Lt_g) (f) - (\alpha_g \circ \te_m) (f) \|
\]
for $g \in G$.
We have $\lim_{n \to \I} \rh_n (g) = 0$ for all $g \in G$
because $\pi_A \circ \te = \ps$ is equivariant,
and $\rh_m (g) \leq \rh_n (g)$ whenever $g \in G$ and $m \geq n$.
Further let $\sm \colon l^{\I}_{\af} (\N, A) \to l^{\I}_{\af} (\N, A)$
be the backwards shift,
given by $\sm ( (a_n)_{n \in \N} ) = (a_2, a_3, \ldots)$.
It is easy to see that the range of $\sm$ is
in fact contained in $l^{\I}_{\af} (\N, A)$.
Now
\[
\rh_n (g)
 = \bigl\| \sm^{n - 1} \bigl( (\te \circ \Lt_g) (f)
     - (\alpha_g^{\I} \circ \te) (f) \bigr) \bigr\|,
\]
so $\rh_n$ is \ct.
Therefore~(\ref{Item_1X07_17}) follows from Dini's Theorem.
The claim is proved.

Let $\mu$ be normalized Haar measure on~$G$.
For $n \in \N$ and $f \in C (G)$, define
\[
\gm_n (f)
 = \int_G (\af_g^{-1} \circ \te_n \circ \Lt_g) (f) \, d \mu (g).
\]
Then $\gm_n$ is an equivariant \ucp{} map from $C (G)$ to~$A$,
and, using~(\ref{Item_1X07_17}), one sees that
\begin{enumerate}
\setcounter{enumi}{\value{TmpEnumi}}
\item\label{Eq_1X09_gmte}
$\lim_{n \to \I} \| \gm_n (f) - \te_n (f) \| = 0$ for all $f \in C (G)$.
\end{enumerate}

Using (\ref{Item_1X07_15}), (\ref{Item_1X07_16}),
and~(\ref{Eq_1X09_gmte}),
choose $n_0 \in \N$ such that whenever $n \geq n_0$,
we have $\| \gm_n (f_1 f_2) - \gm_n (f_1) \gm_n (f_2) \| < \ep$
for all $f_{1}, f_{2} \in C (G)$
and $\|\te_n (f) a - a \te_n (f) \| < \ep$
for all $a \in A$ and $f \in C (G)$.
Since $1 - p$ is $\af$-small, there is $n_1 \in \N$
such that for every $n \geq n_1$ we have $1 - q_{n} \precsim_{A} x$.
Since $1 - p \precsim_{(A^{\af})_{\I}} y$,
by Lemma~\ref{L_1X09_CSbPj} there is $c \in (A^{\af})_{\I}$
such that $1 - p = c^* y c$.
Choose $d = (d_{n})_{n \in \N} \in l^{\I}_{\af} (\N, A)$
such that $\pi_A (d) = c$.
Then $\lim_{n \to \infty} \| d_n^* y d_n - (1 - q_n) \| = 0$.
Therefore there is $n_2 \in \N$ such that for every $n \geq n_2$
we have $\| d_n^* y d_n - (1 - q_n) \| < 1$.
Since $1 - p \precsim_{(A^{\af})_{\I}} p$,
there is $v \in (A^{\af})_{\I}$ such that $v^* v = 1 - p$ and
$v v^* \leq p$.
Then $1 - p = v^* p v$.
Choose $w = (w_{n})_{n \in \N} \in l^{\I}_{\af} (\N, A)$
such that $\pi_A (w) = v$.
Then $\lim_{n \to \infty} \| w_n^* q_n w_n - (1 - q_n) \| = 0$.
Therefore there is $n_3 \in \N$ such that for every $n \geq n_3$
we have $\| w_n^* q_n w_n - (1 - q_n) \| < 1$.
Also, there is $n_4 \in \N$
such that for every $n \geq n_4$ we have $\| q_n x q_n \| > 1 - \ep$.
Set $n = \max (n_0, n_1, n_2, n_3, n_4)$, and in Definition~\ref{traR}
take $\ph$ to be $\gm_n$ and take $p$ to be~$q_n$.
All the conditions hold by construction, except that we use
Lemma~\ref{L_1X09_CSbPj} to see that
$1 - q_n \precsim_{A^{\alpha}} y$
and $1 - q_n \precsim_{A^{\alpha}} q_n$.

We now show that the \trpc{} implies the existence of~$\ps$.
Since the Rokhlin property case is already known
(see Theorem~1.7 of~\cite{gardkk-eq} and
Theorem~2.11 of~\cite{crossrep}), we may assume that $\af$ does not have
the Rokhlin property.
Therefore $A$ has Property~(SP) by Lemma~\ref{PROSPTRPCA}.

Let $x \in  A_{+}$ satisfy $\| x \| = 1$
and let $y \in (A^{\af})_{+} \setminus \{ 0 \}$.
We may assume that $\| y \| = 1$.
Choose dense sequences
\begin{equation}\label{Eq_1X14_fnan}
f_1, f_2, \ldots \in C (G),
\qquad
a_1, a_2, \ldots \in A,
\end{equation}
and
\begin{equation}\label{Eq_1X14_xn}
x_1, x_2, \ldots \in \bigl\{ a \in A_{+} \colon \| a \| = 1 \bigr\}.
\end{equation}
For $n \in \N$, use simplicity of~$A$ and Lemma~\ref{lma_L683_baj}
to choose $z_n \in A_{+} \setminus \{ 0 \}$
such that
\begin{equation}\label{Eq_1X10_StSt}
z_n \precsim_{A} \Bigl( x_k - \frac{1}{2} \Bigr)_{+}
\end{equation}
for $k = 1, 2, \ldots, n$.
Define $h_{n}, k_{n} \colon [0, 1] \to [0, 1]$ by
\[
h_n (\ld)
 = \begin{cases}
   \left( 1 - \frac{1}{n + 1} \right)^{-1} \ld
    & \hspace*{1em} 0 \leq \ld \leq 1 - \frac{1}{n + 1}
        \\
   1 & \hspace*{1em} 1 - \frac{1}{n + 1} < \ld \leq 1
\end{cases}
\]
and
\[
k_n (\ld)
 = \begin{cases}
   0 & \hspace*{1em} 0 \leq \ld \leq 1 - \frac{1}{n + 1}
        \\
   (n + 1) \ld - n
     & \hspace*{1em} 1 - \frac{1}{n + 1} < \ld \leq 1.
\end{cases}
\]
Since $\| x \| = 1$, we have $\| k_n (x) \| = 1$,
and there is a nonzero \pj{} $r_n \in {\ov{k_n (x) A k_n (x)}}$.
Then
\begin{equation}\label{Eq_1X10_4St}
h_n (x) r_n = r_n.
\end{equation}
Since $A$ is simple, Lemma~\ref{lma_L683_baj}
provides $d_n \in (r_n A r_n)_{+} \SM \{ 0 \}$ such that
$d_n \precsim_{A} z_n$.
Use Property~(SP) to choose a nonzero \pj{} $e_n \in \ov{d_n A d_n}$.
Then
\begin{equation}\label{Eq_1X10_St}
e_n \leq r_n \andeqn e_n \precsim_{A} d_n \precsim_{A} z_n.
\end{equation}
Also define $y_0 = h_1 (y)$ and $y_1 = k_1 (y)$, getting
\begin{equation}\label{Eq_1X10_StStSt}
\| y_1 \| = 1,
\qquad
y_0 y_1 = y_1,
\andeqn
y_0 \sim y.
\end{equation}
Apply Lemma~\ref{TRPcequi} with
\begin{equation}\label{Eq_1X18_3St}
S_n = \{ f_1, f_2, \ldots, f_n \}
\andeqn
F_n = \{ a_1, a_2, \ldots, a_n \}
\end{equation}
in place of $S$ and~$F$,
with $\ep = \frac{1}{n}$,
with $e_n$ in place of~$x$, and with $y_1$ in place of~$y$.
We get a projection $q_n \in A^{\alpha}$ and an equivariant
unital completely positive map $\ph_n \colon C (G) \to p A p$,
such that the following hold.
\begin{enumerate}
\setcounter{enumi}{\value{TmpEnumi}}
\item\label{Item_1X10_FSE}
$\ph_n$ is an $\bigl( F_n, S_n, \frac{1}{n} \bigr)$-approximately
central multiplicative map.
\item\label{Item_1X10_1mpx}
$1 - q_n \precsim_{A} e_n$.
\item\label{Item_1X10_1py}
$1 - q_n \precsim_{A^{\alpha}} y_1$.
\item\label{Item_1X10_1ppp}
$1 - q_n \precsim_{A^{\alpha}} q_n$.
\item\label{Item_1X10_Nm}
$\left\| q_n e_n q_n \right\| > 1 - \frac{1}{n}$.
\end{enumerate}

Define an $\af$-invariant \pj{} $q \in l^{\I}_{\af} (\N, A)$
by $q = (q_{n})_{n \in \N}$.
Define $\ph \colon C (G) \to l^{\I} (\N, A)$ by
$\ph (f) = (\ph_{n} (f))_{n \in \N}$ for $f \in C (G)$.
Clearly $\ph$ is equivariant and \ucp.
It follows from equivariance that the range of $\ph$ is
contained in $l^{\I}_{\af} (\N, A)$.
Set $p = \pi_A (q)$ and $\ps = \pi_A \circ \ph$.
Obviously $\ps$ is equivariant.
Using the fact that $\| \ph_n \| \leq 1$ for all $n \in \N$,
density of the sequences in~(\ref{Eq_1X14_fnan}),
and~(\ref{Eq_1X18_3St}),
it is easily seen that $\ps$ is a \hm{}
whose range commutes with the image of~$A$ in $A_{\I, \af}$.

We prove~(\ref{1_psmalllnes}).
By Lemma~\ref{L_1X09_SmallPj},
it suffices to prove that for all nonzero $t \in A_{+}$
there exists $N \in \N$ such that for every $n \geq N$
we have $1 - q_{n} \precsim_{A} t$.
We may assume that $\| t \| = 1$.
Choose $N \in \N$ such that $\| x_N - t \| < \frac{1}{2}$.
Then for $n \geq N$ we have,
using (\ref{Item_1X10_1mpx}) at the first step,
(\ref{Eq_1X10_St}) at the second step,
and (\ref{Eq_1X10_StSt}) at the third step,
\[
1 - q_n \precsim_{A} e_n \precsim_{A} z_n
 \precsim_{A} \Bigl( x_N - \frac{1}{2} \Bigr)_{+} \precsim_{A} t.
\]

For~(\ref{1_psmalllnes_FP}), use
(\ref{Eq_1X10_StStSt}), (\ref{Item_1X10_1py}),
and Lemma~\ref{L_1X10_y0y1p}
to choose $w_n \in A^{\af}$ such that $\| w_n \| \leq 1$ and
$w_n^* y_0 w_n = 1 - q_n$.
Then $v = \pi_A ( (w_1, w_2, \ldots ) )$ satisfies $v^* y_0 v = 1 - p$.
Since $y_0 \sim_A y$ (by~(\ref{Eq_1X10_StStSt})),
we conclude that $1 - p \precsim_{(A^{\af})_{\I}} y$.

To prove~(\ref{1_pandpinA_inf}), by~(\ref{Item_1X10_1ppp})
there is $t_n \in A^{\af}$ such that $t_n^* t_n = 1 - q_n$
and $t_n t_n^* \leq q_n$.
Then $s = \pi_A ( (t_1, t_2, \ldots ) )$ satisfies
$s^* s = 1 - p$ and $s s^* \leq p$.

Finally, we prove~(\ref{nonzerop1_var}).
It suffices to show that $\lim_{n \to \infty} \| q_n x q_n \| = 1$.
For $n \in \N$ we have $\| x - h_n (x) \| < \frac{1}{n}$.
Therefore
$\| q_n x q_n \| > \| q_n h_n (x) q_n \| - \frac{1}{n}$.
Now $e_n \leq h_n (x)$ by (\ref{Eq_1X10_St}) and~(\ref{Eq_1X10_4St}),
so, by~(\ref{Item_1X10_Nm}),
\[
\| q_n h_n (x) q_n \| \geq \| q_n e_n q_n \| > 1 - \frac{1}{n}.
\]
Therefore $\| q_n x q_n \| > 1 - \frac{2}{n}$, and~(\ref{nonzerop1_var})
follows.
This completes the proof.
\end{proof}

The following theorem is a combination
of Theorem~1.7 of~\cite{gardkk-eq} and
Theorem~2.11 of~\cite{crossrep}, modified for the \trpc.
It is a key tool for transferring properties
from the original algebra to the fixed point algebra.
The proof follows that of Theorem~2.11 of~\cite{crossrep},
but there is a mistake there: the formula for $\rh_n (f \otimes a)$
isn't additive in~$f$, so can't be used to define a \hm{}
as claimed there.
The fix is taken from the proof of Theorem~1.7 of~\cite{gardkk-eq}.

\begin{thm}\label{thm_ApproxHommtr}
Let $A$ be an infinite dimensional simple separable unital \ca,
let $G$ be a second countable compact group, and let
$\alpha \colon G \to \Aut (A)$ be an action which has the
tracial Rokhlin property with comparison.
For every $\ep > 0$, every $n \in \N$,
every compact subset $F_{1} \S A$,
every compact subset $F_2 \S A^{\alpha}$,
every $x \in A_{+}$ with $\| x \| = 1$,
and every $y \in (A^{\alpha})_{+} \setminus \{ 0 \}$, there exist
a projection $p \in A^{\alpha}$ and a
unital completely positive contractive map
$\psi \colon A \to p A^{\alpha} p$ such that the following hold.
\begin{enumerate}
\item\label{Item_1X07_18}
$\ps$ is an $(n, F_1 \cup F_2, \ep)$-approximately
multiplicative map (Definition~\ref{D_1X10_nSFe}).
\item\label{commute1469}
$\| p a - a p \| < \varepsilon$ for all $a \in F_{1} \cup F_{2}$.
\item\label{Item_1X07_21}
$\| \psi (a) - p a p \| < \ep$ for all $a \in F_2$.
\item\label{Item_1Z11_BigN}
$\| \ps (a) \| > \| a \| - \ep$ for all $a \in F_{1} \cup F_{2}$.
\item\label{Item_1X07_19}
$1 - p \precsim_{A} x$.
\item\label{Item_1X07_20}
$1 - p \precsim_{A^{\alpha}} y$.
\item\label{Item_1X07_20p}
$1 - p \precsim_{A^{\alpha}} p$.
\item\label{Item_1X07_22}
$\| p x p \| > 1 - \varepsilon$.
\setcounter{TmpEnumi}{\value{enumi}}
\end{enumerate}
\end{thm}

\begin{proof}
Let $\ep$, $F_{1}$, $F_2$, $x$, and $y$ be as in the statement.

We first claim that there are an $\af^{\I}$-invariant \pj{}
$q \in l^{\I}_{\af} (\N, A)$ and a \ucp{} map
$\te \colon C (G) \otimes A \to q l^{\I}_{\af} (\N, A) q$
such that, with $r = \pi_A (q)$, the following hold:
\begin{enumerate}
\setcounter{enumi}{\value{TmpEnumi}}
\item\label{Item_1X07_eq}
Recalling Notation~\ref{N_1X07_Lt},
$\te$ is equivariant for the actions $g \mapsto \Lt_g \otimes \af_g$
and $g \mapsto \af^{\I}_g$.
\item\label{Item_1X07_hm}
$\pi_A \circ \te$ is a \hm.
\item\label{Item_1X07_1otimes}
$(\pi_A \circ \te) (1 \otimes a) = r a r$ for all $a \in A$.
\item\label{Item_1X07_afs}
$1 - r$ is $\af$-small in $A_{\I, \af}$.
\item\label{Item_1X07_sm}
$1 - r \precsim_{(A^{\af})_{\I}} y$.
\item\label{Item_1X07_sbp}
$1 - r \precsim_{(A^{\af})_{\I}} r$.
\item\label{Item_1X07_norm1}
$\| r x r \| = 1$.
\end{enumerate}

To prove the claim, apply the ``only if''
part of Lemma~\ref{mtrpcentral} with $x$ and~$y$ as given,
getting $r \in (A_{\I, \af} \cap A')^{\alpha_{\infty}}$
and a unital equivariant homomorphism
\[
\bt \colon C (G) \to r (A_{\I, \af} \cap A') r,
\]
such that Conditions (\ref{1_psmalllnes}), (\ref{1_psmalllnes_FP}),
(\ref{1_pandpinA_inf}), and~(\ref{nonzerop1_var})
of Lemma~\ref{mtrpcentral} hold.
We write $\io \colon A \to r A_{\I, \af} r$
for the inclusion of $A$ in $A_{\I, \af}$ using constant sequences,
followed by the cutdown $b \mapsto r b r$.
This map is a \hm, since $r \in A_{\I, \af} \cap A'$.
The ranges of $\bt$ and $\io$ commute,
so there is a unital \hm{}
$\et \colon C (G) \otimes A \to r A_{\I, \af} r$ such that
\begin{equation}\label{Eq_1X18_NewSt}
\et (f \otimes a) = \bt (f) \io (a)
\end{equation}
for all $f \in C (G)$ and $a \in A$.
Use Lemma~\ref{L_1X09_LiftPj} to choose
an $\af^{\I}$-invariant \pj{}
$q = (q_{m})_{m \in \N} \in l^{\I}_{\af} (\N, A)$
such that $\pi_A (q) = r$.
We now have Conditions (\ref{Item_1X07_afs}), (\ref{Item_1X07_sm}),
(\ref{Item_1X07_sbp}), and~(\ref{Item_1X07_norm1}) of the claim.

We now construct~$\te$.
This requires a lifting of~$\et$, which we obtain
by showing that $\et$ is a pointwise limit of liftable maps.
Using a partition of unity argument and second countability of~$G$,
choose positive integers $k (m)$,
\ucp{} maps $\rh_m \colon C (G) \to {\mathbb{C}}^{k (m)}$,
and \ucp{} maps $\sm_m \colon {\mathbb{C}}^{k (m)} \to C (G)$,
for $m \in \N$,
such that for all $f \in C (G)$ we have
$\lim_{m \to \infty} (\sm_m \circ \rh_m) (f) = f$.

Fix $m \in \N$,
and let $e_1, e_2, \ldots, e_{k (m)}$
be the minimal \pj{s} in~${\mathbb{C}}^{k (m)}$.
For $j = 1, 2, \ldots, k (m) - 1$,
choose $c_j \in [q l^{\I}_{\af} (\N, A) q]_{+}$ such that
$\pi_A (c_j) = (\bt \circ \sm_m) (e_j)$.
Let $c_{k (m)}$ be the positive part of
$q - \sum_{j = 1}^{k (m) - 1} c_j$.
Then $\pi_A (c_{k (m)}) = (\bt \circ \sm_m) (e_{k (m)})$
and, with
$c = \sum_{j = 1}^{k (m)} c_j \in [q l^{\I}_{\af} (\N, A) q]_{+}$,
we have $c \geq q$.
Taking functional calculus in $q l^{\I}_{\af} (\N, A) q$,
for $j = 1, 2, \ldots, k (m)$ define $b_j = c^{- 1/2} c_j c^{- 1/2}$.
Then $b_j \geq 0$ and $\pi_A (b_j) = (\bt \circ \sm_m) (e_j)$
for $j = 1, 2, \ldots, k (m)$, and $\sum_{j = 1}^{k (m)} b_j = q$.
Identify ${\mathbb{C}}^{k (m)} \otimes A$ with $A^{k (m)}$, and define
$\gm_m \colon
 {\mathbb{C}}^{k (m)} \otimes A \to q l^{\I}_{\af} (\N, A) q$
by, identifying $a \in A$ with the corresponding constant sequence
in $l^{\I}_{\af} (\N, A)$,
\[
\gm_m (a_1, a_2, \ldots, a_{k (m)} )
 = \sum_{j = 1}^{k (m)} b_j^{1/2} q a_j q b_j^{1/2}
\]
for $a_1, a_2, \ldots, a_{k (m)} \in A$.
Then one checks that $\gm_m$ is a \ucp{} map such that
$\pi_A \circ \gm_m = \et \circ (\sm_m \otimes \id_A)$.
Therefore
\[
\pi_A \circ \gm_m \circ (\rh_m \otimes \id_A)
 = \et \circ [(\sm_m \circ \rh_m) \otimes \id_A].
\]
This formula shows that the \ucp{} maps
\[
\et \circ [(\sm_m \circ \rh_m) \otimes \id_A] \colon
  C (G) \otimes A \to r A_{\I, \af} r
\]
have lifts to \ucp{} maps
\[
\gm \circ (\rh_m \otimes \id_A)
 \colon C (G) \otimes A \to q l^{\I}_{\af} (\N, A) q.
\]
Since $\sm_m \circ \rh_m \to \id_{C (G)}$ pointwise,
it follows that
\[
\et \circ [(\sm_m \circ \rh_m) \otimes \id_A] \to \et
\]
pointwise.
Now Theorem~6 of~\cite{Ar} shows that $\et$ lifts to a
\ucp{} map $\te_0 \colon C (G) \otimes A \to q l^{\I}_{\af} (\N, A) q$.
Since $\et$ is equivariant, we can average $\te_0$ over~$G$
to get an equivariant \ucp{} map
$\te \colon C (G) \otimes A \to q l^{\I}_{\af} (\N, A) q$
such that $\pi_A \circ \te = \et$.
This gives (\ref{Item_1X07_eq}) and~(\ref{Item_1X07_hm}) of the claim.
Condition~(\ref{Item_1X07_1otimes}) follows from~(\ref{Eq_1X18_NewSt}).
The claim is proved.

Let $\te$ be as in the claim.
For $b \in C (G, A)$, write $\te (b) = (\te_m (b))_{n = 1}^{\I}$.
Then the maps $\te_m \colon C (G, A) \to A$ are
equivariant and \ucp.
It follows from Proposition~2.3 of~\cite{crossrep}
that there is an isomorphism
$\kp \colon A \to C (G, A)^{\Lt \otimes \af}$
such that $\kp (a) = 1 \otimes a$ for all $a \in A^{\af}$.
Then $\pi_A \circ \te \circ \kp$ is a \hm.
Using compactness of $F_1 \cup F_2$ and the fact that $\te \circ \kp$
is \ct, it is easy to see that
there is $m_1 \in \N$ such that, for all $m \geq m_1$
and all $a_1, a_2, \ldots, a_n \in F_1 \cup F_2$,
we have
\[
\bigl\| (\te_m \circ \kp) (a_1 a_2 \cdots a_n)
  - (\te_m \circ \kp) (a_1) (\te_m \circ \kp) (a_2)
      \cdots (\te_m \circ \kp) (a_n) \bigr\|
   < \varepsilon.
\]

Since $r \in A_{\I, \af} \cap A'$,
there is $m_2 \in \N$ such that for all $m \geq m_2$
and all $a \in F_1 \cup F_2$,
we have $\| q_m a - a q_m \| < \ep$.
For $a \in A^{\af}$ we have
\[
(\pi_A \circ \te \circ \kp) (a) = \io (a) = r a r,
\]
so we can choose $m_3 \in \N$ such that for all $m \geq m_3$
and all $a \in F_2$,
we have $\| (\te_m \circ \kp) (a) - q_m a q_m \| < \ep$.
Because $1 - r$ is $\af$-small (by~(\ref{Item_1X07_afs})),
there is $m_4 \in \N$ such that $1 - q_m \precsim_A x$
for all $m \geq m_4$.
Using (\ref{Item_1X07_sm}) and~(\ref{Item_1X07_sbp}),
arguments given after~(\ref{Eq_1X09_gmte})
in the proof of Lemma~\ref{mtrpcentral}
show that there are $m_5, m_6 \in \N$ such that
$1 - q_m \precsim_A y$ for all $m \geq m_5$
and $1 - q_m \precsim_A q_m$ for all $m \geq m_6$.
Use~(\ref{Item_1X07_norm1}) to choose $m_7 \in \N$ such that
for all $m \geq m_7$ we have $\| q_m x q_m \| > 1 - \ep$.

Choose a finite subset $T \subseteq F_1 \cup F_2$ such that
for all $a \in F_1 \cup F_2$
there is $b \in T$ with $\| b - a \| < \frac{\ep}{3}$.
By~(\ref{Item_1X07_norm1}), we have $r \neq 0$.
Therefore the \hm{}
$\pi_A \circ \te \circ \kp \colon A \to r A_{\I, \af} r$ is nonzero.
Since $A$ is simple, this \hm{} is isometric.
It follows that there is $m_8 \in \N$ such that for all $m \geq m_8$
and all $b \in T$,
we have $\| (\te_m \circ \kp) (b) \| >  \| b \| - \frac{\ep}{3}$.
Let $a \in F_1 \cup F_2$ and choose $b \in T$
such that $\| b - a \| < \frac{\ep}{3}$.
Then for $m \geq m_8$ we have
\[
\begin{split}
\| (\te_m \circ \kp) (a) \|
& \geq \| (\te_m \circ \kp) (b) \| - \| b - a \|
\\
& > \| b \| - \frac{\ep}{3} - \frac{\ep}{3}
  \geq \| a \|  - \| b - a \| - \frac{2 \ep}{3}
  > \| a \| - \ep.
\end{split}
\]

Take
\[
m = \max ( m_1, m_2, \ldots, m_8),
\qquad
p = q_m,
\andeqn
\ps = \te_m \circ \kp.
\]
Since $\te_m$ is equivariant
and $\kp (A) \subseteq C (G, A)^{\Lt \otimes \af}$,
it follows that $\te_m \circ \kp$ is a \ucp{} map from
$A$ to $q_m A^{\af} q_m$.
Therefore these choices prove the theorem.
\end{proof}

The next four remarks describe the modifications
of Theorem~\ref{thm_ApproxHommtr} mentioned in
Remark~\ref{R_1Z09_Variants},
to which we refer for the general principle.

\begin{rmk}\label{R_1Z09_ntr}
The naive \trp{} (Definition~\ref{D_1920_NTRP} below)
is obtained from Definition~\ref{traR} by omitting
Conditions (\ref{1_ppcompactsets}) and~(\ref{Item_902_pxp_TRP}).
The corresponding version of Theorem~\ref{thm_ApproxHommtr}
states that if $\af$ is as there, but is only assumed to
have the naive \trp,
then the conclusion of Theorem~\ref{thm_ApproxHommtr} still holds,
except that
Conditions (\ref{Item_1X07_20}) and~(\ref{Item_1X07_20p})
must be omitted.

For the proof this version, first,
the conclusion of Lemma~\ref{PROSPTRPCA} still holds
if one assumes this property,
since Conditions (\ref{1_pycompactsets}) and~(\ref{1_ppcompactsets})
of the definition are not used in the proof.
Next, one uses an analog of Lemma~\ref{TRPcequi}
in which the \trpc{} is replaced by the naive \trp{}
and Conditions (\ref{1_pycompactsetscpt}) and~(\ref{1_ppcompactsetscpt})
there are omitted.
The proof of this analog
is the same as the proof of Lemma~\ref{TRPcequi} as stated,
since Conditions (\ref{1_pycompactsets}) and~(\ref{1_ppcompactsets})
of the definition are simply carried over to the conclusion of
the lemma, and are not used in the rest of the proof.
Lemma~\ref{mtrpcentral} holds for the naive \trp{} in place of
the \trpc{}
when Conditions (\ref{1_psmalllnes_FP}) and~(\ref{1_pandpinA_inf})
there are omitted.
The proof of this analog
is the same as the proof of Lemma~\ref{mtrpcentral} as stated,
except that the parts of the argument dealing with the two omitted
conditions from the definition and from the conclusion of the lemma
are omitted.
Finally, the proof of the modified version
of Theorem~\ref{thm_ApproxHommtr} is again obtained from
the proof of the theorem as stated by simply omitting
the parts involving the omitted conditions,
and noting that the remaining conditions,
specifically~(\ref{Item_1X07_20p}) in the statement
and~(\ref{Item_1X07_sbp}) in the proof,
still ensure that the \pj{} $r$ in the proof is nonzero.
\end{rmk}

\begin{rmk}\label{R_1Z09_nzp}
Consider the modified version of Definition~\ref{traR} in which
Condition~(\ref{Item_902_pxp_TRP}) ($\| p x p \| > 1 - \ep$)
is omitted.
(If we also omitted Definition \ref{traR}(\ref{1_ppcompactsets}),
we would replace
Definition \ref{traR}(\ref{Item_902_pxp_TRP})
with the simpler condition $p \neq 0$, and carry this condition over
throughout the argument.
However,
Definition \ref{traR}(\ref{1_ppcompactsets}) implies $p \neq 0$.)
Then the conclusion of Theorem~\ref{thm_ApproxHommtr} still holds,
except that Condition~(\ref{Item_1X07_22}) must be omitted.
For the proof, the conclusion of Lemma~\ref{PROSPTRPCA} still holds,
with the same proof.
The analog of Lemma~\ref{TRPcequi} for this variant
omits Condition~(\ref{pxpcompactsetscpt}) in the conclusion;
in the proof of Lemma~\ref{TRPcequi} as stated,
this condition is merely carried over from Definition~\ref{traR},
and not used otherwise.
The analog of Lemma~\ref{mtrpcentral} for this variant
omits Condition~(\ref{nonzerop1_var}),
and the proof is an easy modification of the
proof of Lemma~\ref{mtrpcentral} as stated.
The modification to prove this version of Theorem~\ref{thm_ApproxHommtr}
is also easy.
\end{rmk}

\begin{rmk}\label{R_1Z09_mod_trp}
The modified \trp{} (Definition~\ref{moditra} below)
is obtained from Definition~\ref{traR} by omitting
Condition~(\ref{1_pycompactsets}) and strengthening
Condition~(\ref{1_ppcompactsets}) to require that
there be $s \in A^{\alpha}$ such that, in addition to
$s^* s = 1 - p$ and $s s^* \leq p$, we have
$\| s a - a s \| < \varepsilon$ for all $a$ in a given
finite subset $F_2$ of $A^{\alpha}$.
(The original finite subset $F \subseteq A$ is called $F_1$
in this definition.
Theorem~\ref{thm_ApproxHommtr} already uses
compact subsets $F_1 \S A$ and $F_2 \S A^{\af}$.)
The corresponding version of Theorem~\ref{thm_ApproxHommtr}
states that if $\af$ is as there, but is assumed to
have the modified \trp,
then the conclusion of Theorem~\ref{thm_ApproxHommtr} still holds,
except that
Condition~(\ref{Item_1X07_20})
is omitted and Condition~(\ref{Item_1X07_20p}) is strengthened
in the same way as Definition \ref{traR}(\ref{1_ppcompactsets}) was:
there is $s \in A^{\alpha}$ such that
$s^* s = 1 - p$, $s s^* \leq p$, and
$\| s a - a s \| < \varepsilon$ for all $a \in F_2$.

The conclusion of Lemma~\ref{PROSPTRPCA} still holds,
with the same proof.
Lemma~\ref{TRPcequi} for the modified \trp{}
uses compact sets $F_1 \subseteq A$ (in place of $F$)
and $F_2 \subseteq A^{\af}$,
omits Condition~(\ref{1_pycompactsetscpt}), and replaces
Condition~(\ref{1_ppcompactsetscpt}) with the requirement that
there be $s \in A^{\alpha}$ such that
$s^* s = 1 - p$, $s s^* \leq p$, and
$\| s a - a s \| < \varepsilon$ for all $a \in F_2$.
In the analog of Lemma~\ref{mtrpcentral},
Condition~(\ref{1_psmalllnes_FP}) is omitted and
Condition~(\ref{1_pandpinA_inf})
is replaced with $1 - p \precsim_{(A^{\af})_{\I} \cap (A^{\af})'} p$.
The corresponding modification of the forwards implication in the proof
is easy.
(For example, one also needs a dense sequence in~$A^{\af}$,
one uses finite subsets $F_{1, n} \S A$ and $F_{2, n} \S A^{\af}$,
(\ref{Item_1X10_1py}) is omitted,
and~(\ref{Item_1X10_1ppp}) is suitably strengthened.)
The reverse implication uses the same argument as described below
for the proof of the modified version of Theorem~\ref{thm_ApproxHommtr}.

There is additional work to do in the proof of the modified version
of Theorem~\ref{thm_ApproxHommtr}.
At the end of the proof,
start with $t \in (A^{\af})_{\I} \cap (A^{\af})'$
such that $t^{*} t = 1 - p$ and $t t^{*} \leq p$.
Lift $t$ to a sequence $c = (c_{m})_{m \in \N} \in l^{\I} (\N, A^{\af})$
such that $\| c \| \leq 1$.
Then
\[
\lim_{m \to \I} \| c_m^* c_m - (1 - q_m) \| = 0,
\qquad
\lim_{m \to \I} \| q_m c_m c_m^* q_m - c_m c_m^* \| = 0,
\]
and for all $a \in F_2$ we have
$\lim_{m \to \I} \| c_m a - a c_m \| = 0$.
Set $M = 1 + \sup_{a \in F_2} \| a \|$.
Choose $\dt > 0$ such that whenever
$D$ is a \ca, $e, f \in D$ are \pj{s}, and $c \in D$ satisfies
\[
\| c^* c - e \| < \dt \andeqn \| f c c^* f - c c^* \| < \dt,
\]
then there is $s \in D$ such that
\[
s^* s = e,
\qquad
s s^* \leq f,
\andeqn
\| s - c \| < \frac{\ep}{3 M}.
\]
Then choose $m_6 \in \N$ such that for all $m \geq m_6$ we have
\[
\| c_m^* c_m - (1 - q_m) \| < \dt,
\qquad
\| q_m c_m c_m^* q_m - c_m c_m^* \| < \dt,
\]
and $\| c_m a - a c_m \| < \frac{\ep}{3}$ for all $a \in F_2$.
For $m \geq m_6$ let $s_m \in A^{\af}$ be the partial isometry
gotten from the choice of~$\dt$,
with $c = c_m$, $e = 1 - q_m$, and $f = q_m$.
Then $s_m^* s_m = 1 - q_m$, $s_m s_m^* \leq q_m$,
and for all $a \in F_2$ we have
\[
\| s_m a - a s_m \|
 \leq 2 \| a \| \| s_m - c_m \| + \| c_m a - a c_m \|
 < \frac{2 \ep \| a \|}{3 M} + \frac{\ep}{3}
 < \ep.
\]
In the final step, we take (omitting $m_5$ in the definition of~$m$)
\[
m = \max ( m_1, m_2, m_3, m_4, m_6, m_7, m_8),
\quad
p = q_m,
\quad
s = s_m,
\quad {\mbox{and}} \quad
\ps = \te_m \circ \kp.
\]
\end{rmk}

\begin{rmk}\label{R_S_mod_trp}
The strong modified \trp{} (Definition~\ref{D_1824_AddToTRP_Mod} below)
is like the modified \trp,
but requires that $s$ approximately commute
with elements of a given finite subset of~$A$
instead of with elements of a given finite subset of~$A^{\af}$.
The modifications are like those in Remark~\ref{R_1Z09_mod_trp},
but using $A$ in place of $A^{\af}$.
We state only two of the changes.
In Lemma~\ref{mtrpcentral},
omit Condition~(\ref{1_psmalllnes_FP}) and replace
Condition~(\ref{1_pandpinA_inf})
with $1 - p \precsim_{(A^{\af})_{\I} \cap A'} p$.
In Theorem~\ref{thm_ApproxHommtr},
omit Condition~(\ref{Item_1X07_20})
and strengthen Condition~(\ref{Item_1X07_20p})
to say that there is $s \in A^{\alpha}$ such that
$s^* s = 1 - p$, $s s^* \leq p$, and
$\| s a - a s \| < \varepsilon$ for all $a \in F_1$.
\end{rmk}

\section{Simplicity of the fixed point algebra
  and the crossed product}\label{Sec_1160_Simpli_Prf}

In this section we prove simplicity of the fixed point algebra
and crossed product
by an action which has the tracial Rokhlin property with comparison.

For finite groups, the tracial Rokhlin property implies pointwise
outerness by Lemma 1.5 in \cite{phill23}, so that the
crossed product is simple by Theorem 3.1 in \cite{kish.sim}.
Theorem 3.1 in \cite{kish.sim} is valid for general discrete groups.
But, without discreteness of the group, even assuming compactness,
pointwise outerness of the action does not imply simplicity of
the crossed product.
For example, consider the gauge action $\gamma$ of the circle group
$S^{1}$ on $\mathcal{O}_{\infty}$, which in terms of the standard
generators $s_j$ is given by $\gamma_{\zeta} (s_{j}) = \zeta s_{j}$
for all $\zeta$ in $S^{1}$ and all $j \in \N$.
Then $\gamma$ is pointwise outer by Theorem 4 of~\cite{outer_gua}.
However, its strong Connes spectrum is $\Nz$
(Remark 5.2 of~\cite{Kis_simpleCP}), so, by
Theorem 3.5 of \cite{Kis_simpleCP}, its crossed product is not simple.

Instead, we use the \trpc,
in the form of Theorem~\ref{thm_ApproxHommtr}, to prove that the
fixed point algebra is simple.
(This part does not need
Condition~(\ref{1_ppcompactsets}) in Definition~\ref{traR}.)
We also use the \trpc{} to show that the action is saturated,
so that results of~\cite{GooLazPel} give simplicity
of the crossed product.

Simplicity of the fixed point algebra also allows us to give
a nicer formulation of Lemma~\ref{mtrpcentral}.
This is in Proposition~\ref{P_1X14_CentSq}.

\begin{rmk}\label{R_2015_nzp}
Except for Proposition~\ref{P_1X14_CentSq},
all results of this section which have the \trpc{}
(Definition~\ref{traR}) in their hypotheses
are still valid if in the definition of the \trpc{}
we omit Condition~(\ref{Item_902_pxp_TRP}).
Whenever we need to know $p \neq 0$ in an application of
Theorem~\ref{thm_ApproxHommtr},
we deduce it from Theorem \ref{thm_ApproxHommtr}(\ref{Item_1X07_20p}).
This claim then follows from Remark~\ref{R_1Z09_nzp}.
\end{rmk}

\begin{thm}\label{thm_simple12}
Let $A$ be a simple separable unital infinite dimensional \ca,
let $G$ be a second countable compact group, and let
$\alpha \colon G \to \Aut (A)$ be an action which has the
tracial Rokhlin property with comparison.
Then the fixed point algebra $A^{\alpha}$ is simple.
\end{thm}

\begin{proof}
Let $I$ be a nonzero ideal in $A^{\alpha}$.
We claim that $I$ contains an invertible element.
This will prove the theorem.

Following Notation~\ref{N_1Z31_EF}, set $J = \overline{A I A}$.
Then $J$ is a closed ideal in~$A$.
Since $J \neq 0$ and $A$ is simple, $J = A$.
Since $A$ is unital, in fact $A I A = A$.
Therefore there are $m \in \N$,
$a_1, a_2 \ldots, a_{m}, b_1, b_2, \ldots, b_{m} \in A$,
and $x_1, x_2, \ldots, x_{m} \in I$ such that
$\sum_{j = 1}^{m} a_j x_j b_j = 1$.

Fix a nonzero positive element $z \in I$.
Set
\begin{equation}\label{Eq_1X20_Md}
M = 1 + \max\limits_{j = 1, 2, \ldots, m} \max ( \| a_j \|, \| b_j \| )
\andeqn
\delta = \frac{1}{2 (m M^2 + m)}.
\end{equation}
Also set
\begin{equation}\label{Eq_1X20_F1F2}
F_{1} = \{ a_j, b_j, x_j \colon j = 1, 2, \ldots, m \} \cup \{ x \}
\end{equation}
and
\begin{equation}\label{Eq_1X20_F1F2_p2}
F_{2} = \{x_j \colon j = 1, 2, \ldots, m \} \cup \{ x \}.
\end{equation}
Use Theorem~\ref{thm_ApproxHommtr} with $\dt$ in place of $\varepsilon$,
with $n = 3$, with $1$ in place of $x$, with $z$ in place of $y$,
and with $F_{1}$ and $F_{2}$ as given.
We get a projection $p \in A^{\alpha}$ and a \ucp{}
map $\psi \colon A \to p A^{\alpha} p$ such that the following hold.
\begin{enumerate}
\item\label{Item_1577_psiab_muliti_TRPC}
$\| \psi (a b c) - \psi (a) \psi (b) \ps (c) \| < \delta$
for all $a, b, c \in F_{1}$.
\item\label{Item_1578_paap_TRPC}
$\| p a - a p \| < \delta$ for all $a \in F_{1} \cup F_{2}$.
\item\label{Item_1576_psia_iden_TRPC}
$\| \psi (a) - p a p \| < \delta$ for all $a \in F_{2}$.
\item\label{Item_1579_1mp_subz_fix_TRPC}
$1 - p \precsim_{A^{\alpha}} z$.
\end{enumerate}
Define
\[
c = 1 - p + \sum_{j = 1}^{m} \psi (a_{j}) p x_{j} p \psi (b_{j}).
\]
We have $1 - p \in I$ by~(\ref{Item_1579_1mp_subz_fix_TRPC}),
and $x_1, x_2, \ldots, x_{m} \in I$ by construction, so $c \in I$.
It remains to show that $c$ is invertible.

First, using (\ref{Item_1577_psiab_muliti_TRPC})
and~(\ref{Item_1576_psia_iden_TRPC}) at the second step,
\[
\begin{split}
& \bigg\| \psi \bigg( \sum_{j = 1}^{m} a_{j} x_{j} b_{j} \bigg) -
       \sum_{j = 1}^{m} \psi (a_{j}) p x_{j} p \psi (b_{j}) \bigg\|
\\
& \hspace*{1em} {\mbox{}}
  \leq \sum_{j = 1}^{m} \bigl\| \ps ( a_{j} x_{j} b_{j})
              - \psi (a_{j}) \psi (x_{j}) \psi (b_{j}) \bigr\|
      + \sum_{j = 1}^{m}
         \| \psi (a_{j}) \| \| \ps (x_j) - p x_j p \| \| \psi (b_{j}) \|
\\
& \hspace*{1em} {\mbox{}}
  < m \dt + m M^{2} \delta
  = \frac{1}{2}.
\end{split}
\]
Combining this estimate with
\[
\psi \bigg( \sum_{j = 1}^{m} a_{j} x_{j} b_j \bigg) = \ps (1) = p,
\]
we get
\[
\| 1 - c \|
 = \bigg\| p
       - \sum_{j = 1}^{m} \psi (a_{j}) p x_{j} p \psi (b_{j}) \bigg\|
 < \frac{1}{2}.
\]
So $c$ is invertible, as desired.
\end{proof}

\begin{prp}\label{C_1Z11_InfDim}
Let $A$ be a simple separable unital infinite dimensional \ca,
let $G$ be a second countable compact group, and let
$\alpha \colon G \to \Aut (A)$ be an action which has the
tracial Rokhlin property with comparison.
Then
$A^{\alpha}$ is in\fd{} and not of Type~I.
\end{prp}

\begin{proof}
For infinite dimensionality, it is enough to prove that
for every $n \in \N$ there are nonzero orthogonal
positive elements $x_1, x_2, \ldots, x_n \in A^{\alpha}$.
Since $A$ is in\fd{} and simple,
there exist orthogonal $y_1, y_2, \ldots, y_n \in A_{+}$
such that $\| y_k \| = 1$ for $k = 1, 2, \ldots, n$.
Using semiprojectivity of the cone over~${\mathbb{C}}^n$,
choose $\dt > 0$
such that whenever $B$ is a \ca{} and $b_1, b_2, \ldots, b_n \in A_{+}$
satisfy $\| b_k \| \leq 1$ for $k = 1, 2, \ldots, n$
and $\| b_j b_k \| < \dt$ for
$j, k \in \{ 1, 2, \ldots, n \}$ with $j \neq k$,
then there are $c_1, c_2, \ldots, c_n \in A_{+}$
such that $\| c_k - b_k \| < \frac{1}{3}$ for $k = 1, 2, \ldots, n$
and $c_j c_k = 0$ when $j \neq k$.
We may also require $\dt < \frac{1}{3}$.

Apply Theorem~\ref{thm_ApproxHommtr}
with $\dt$ in place of $\varepsilon$, with $n = 2$, with $x = y = 1$,
with $F_{1} = \{ y_1, y_2, \ldots, y_n \}$, and with $F_{2} = \E$.
We get a projection $p \in A^{\alpha}$ and a \ucp{}
map $\psi \colon A \to p A^{\alpha} p$ such that the following hold.
\begin{enumerate}
\item\label{Item_1Z11_2mult}
$\| \psi (y_j y_k) - \psi (y_j) \psi (y_k) \| < \delta$
for $j, k = 1, 2, \ldots, n$.
\item\label{Item_1Z23_BigN}
$\| \psi (y_k) \| > 1 - \delta$ for $k = 1, 2, \ldots, n$.
\end{enumerate}
By~(\ref{Item_1Z11_2mult}),
we have $\| \psi (y_j) \psi (y_k) \| < \dt$ for $j \neq k$.
By the choice of~$\dt$,
there are $x_1, x_2, \ldots, x_n \in (A^{\alpha})_{+}$
such that $\| x_k - \ps (y_k) \| < \frac{1}{3}$
for $k = 1, 2, \ldots, n$
and $x_j x_k = 0$ when $j \neq k$.
Also, for $k = 1, 2, \ldots, n$,
\[
\| x_k \|
 > \| \ps (y_k) \| - \frac{1}{3}
 > 1 - \delta - \frac{1}{3}
 \geq \frac{1}{3},
\]
so $x_k \neq 0$.
Infinite dimensionality of $A^{\alpha}$ is proved.

We know that $A^{\alpha}$ is unital.
It is simple by Theorem~\ref{thm_simple12}.
No simple in\fd{} \ca{} has Type~I,
so $A^{\alpha}$ is not of Type~I.
\end{proof}

We recall the definition of the strong Arveson spectrum of an
action of a compact group on a \ca,
and related concepts.

\begin{dfn}[\cite{GooLazPel}, Definition 1.1(b), the preceding
 discussion, and Definition 1.2(b)]\label{D_1X14_A2pi}
Let $A$ be a \ca, let $\alpha \colon G \to \Aut (A)$ be an
action of a compact group $G$ on~$A$,
and let $\pi \colon G \to {\operatorname{U}} (\Hi_{\pi})$
be a unitary representation
of $G$ on a Hilbert space $\Hi_{\pi}$.
We define
\[
A_2 (\pi)
 = \bigl\{ x \in B (\Hi_{\pi}) \otimes A \colon
   {\mbox{$(\id_{B (\Hi_{\pi})} \otimes \alpha_g) (x)
        = x (\pi (g) \otimes 1_A)$
     for all $g \in G$}} \bigr\}.
\]
Let ${\widehat{G}}$ be a set consisting of exactly
one representation in each unitary equivalence class
of irreducible representations of~$G$.
Define (following Notation~\ref{N_1Z31_EF})
the strong Arveson spectrum by
\[
{\widetilde{\operatorname{Sp}}} (\alpha)
 = \bigl\{ \pi \in {\widehat{G}} \colon
    {\overline{A_{2} (\pi)^{*} A_{2} (\pi)}}
      = (B (\Hi_{\pi}) \otimes A)^{\Ad (\pi) \otimes \alpha}
  \bigr\}.
\]
Finally, the strong Connes spectrum is the intersection over
all nonzero $\af$-invariant \hsa{s}~$B$ of the
strong Arveson spectrum of the restriction of $\af$ to~$B$.
\end{dfn}

\begin{lem}\label{L_1X14_A2Pi_id}
In the situation of Definition~\ref{D_1X14_A2pi}, for any \fd{}
unitary representation $\pi$ of~$G$, the set
${\overline{A_{2} (\pi)^{*} A_{2} (\pi)}}$ is a closed two sided ideal
in $(B (\Hi_{\pi}) \otimes A)^{\Ad (\pi) \otimes \alpha}$.
\end{lem}

\begin{proof}
This is easily checked, and is in the discussion before
Definition 1.1 of~\cite{GooLazPel}.
\end{proof}

Recall (Definition 7.1.4 of~\cite{Phl1})
that an action $\alpha \colon G \to \Aut (A)$
of a compact group $G$ on a \ca~$A$ is saturated if a suitable
completion of $A$ is a Morita equivalence bimodule
from $A^{\af}$ to $C^* (G, A, \af)$.
We recall some essentially known results on saturation,
the strong Arveson spectrum, and simplicity.
We won't actually use Condition~(\ref{Item_1X07_aSimGm});
it is included primarily to give context.

\begin{thm}\label{satunonabel}
Let $A$ be a \ca, and let $\alpha \colon G \to \Aut (A)$ be an
action of a compact group $G$ on~$A$.
Then following are equivalent:
\begin{enumerate}
\item\label{Item_1X07_SimSat}
$A^{\af}$ is simple and $\af$ is saturated.
\item\label{Item_1X07_SinSp}
$A^{\af}$ is simple and
${\widetilde{\operatorname{Sp}}} (\alpha) = \widehat{G}$.
\item\label{Item_1X07_CPSim}
$C^* (G, A, \af)$ is simple.
\item\label{Item_1X07_aSimGm}
$A$ has no nontrivial $G$-invariant ideals and
${\widetilde{\Gm}} (\alpha) = \widehat{G}$.
\end{enumerate}
When these conditions hold, $A^{\af}$
is isomorphic to a full \hsa{} of $C^* (G, A, \af)$,
and is strongly Morita equivalent to $C^* (G, A, \af)$.
\end{thm}

\begin{proof}
The equivalence of (\ref{Item_1X07_SimSat}) and~(\ref{Item_1X07_SinSp})
follows from Theorem 5.10(1) of~\cite{phillfree}, originally from
the remarks after Lemma~3.1 of~\cite{GooLazPel}.

The equivalence of (\ref{Item_1X07_CPSim}) and~(\ref{Item_1X07_aSimGm})
is Theorem~3.4 of~\cite{GooLazPel}.

For the implication from~(\ref{Item_1X07_CPSim})
to~(\ref{Item_1X07_SinSp}),
use the implication
from~(\ref{Item_1X07_CPSim}) to~(\ref{Item_1X07_aSimGm}) and
${\widetilde{\Gm}} (\alpha) \S {\widetilde{\operatorname{Sp}}} (\alpha)$
to get ${\widetilde{\operatorname{Sp}}} (\alpha) = \widehat{G}$.
Use the Corollary in~\cite{Rs} to get simplicity of~$A^{\af}$.
(Warning: the reference there to Corollary~3.8 of~\cite{Rf0}
is to Corollary~3.3 in the published version of~\cite{Rf0},
and there is also a Proposition~3.3 in~\cite{Rf0}.)

To finish, we prove that (\ref{Item_1X07_SinSp})
implies~(\ref{Item_1X07_CPSim}), by comparing the proof of
Corollary 7.1.5 of~\cite{Phl1} with the proof of
the Proposition in~\cite{Rs}.
For $x \in A$ let ${\widetilde{x}} \in C^* (G, A, \af)$ be as in
Definition 7.1.2 of~\cite{Phl1},
and let $p \in M (C^* (G, A, \af))$ be the projection in
the proof of the Proposition in~\cite{Rs}.
A calculation shows that $p {\widetilde{x}} = {\widetilde{x}}$
for all $x \in A$, so also ${\widetilde{x}}^* p = {\widetilde{x}}^*$.
Now, at the first step
as in the proof of Corollary 7.1.5 of~\cite{Phl1},
\[
\bigl\{ \langle x, y \rangle_{C^* (G, A, \af)} \colon x, y \in A \bigr\}
 = \bigl\{ ({\widetilde{x}})^* {\widetilde{y}} \colon x, y \in A \bigr\}
 = \bigl\{ x^* p y \colon x, y \in A \bigr\}.
\]
So the first set spans a dense subspace of $C^* (G, A, \af)$
\ifo{} the last set does too.
That the first spans a dense subspace of $C^* (G, A, \af)$
is the definition of saturation.
Using the Corollary in~\cite{Rs} and simplicity of~$A^{\af}$,
having the last span a dense subspace of $C^* (G, A, \af)$
is equivalent to simplicity of $C^* (G, A, \af)$.

Assuming the conditions, $A^{\af}$
is isomorphic to a full \hsa{} of $C^* (G, A, \af)$
by the Corollary in~\cite{Rs},
and is strongly Morita equivalent to $C^* (G, A, \af)$
by~(\ref{Item_1X07_SimSat}).
\end{proof}

\begin{prp}\label{satpropnoncptp}
Let $A$ be an infinite dimensional simple separable unital \ca.
Let $\alpha \colon G \to \Aut (A)$ be an
action of a second countable compact group $G$ on $A$ which
has the tracial Rokhlin property with comparison.
Then $\alpha$ is saturated.
\end{prp}

\begin{rmk}\label{R_1X20_No14}
In the proof of Proposition~\ref{satpropnoncptp},
we do not use
Conditions (\ref{1_pxcompactsetscpt}), (\ref{1_pycompactsetscpt}),
or~(\ref{pxpcompactsetscpt}) in Lemma~\ref{TRPcequi}.
We also don't use any later results
in Section~\ref{Sec_722_CTRAwC_Other_typ}.
Thus, the modified version of Lemma~\ref{TRPcequi}
described in Remark~\ref{R_1Z09_mod_trp} suffices.
As in Remark~\ref{R_1Z09_mod_trp}, it follows
that Proposition~\ref{satpropnoncptp} still holds
if in its hypotheses the \trpc{}
is replaced with the modified tracial Rokhlin property
of Definition~\ref{moditra} below.
(In fact, we really only need 
Conditions (\ref{Item_893_FS_equi_cen_multi_approx})
and~(\ref{1_ppcompactsets}) in Definition~\ref{traR}.)
\end{rmk}

\begin{proof}[Proof of Proposition~\ref{satpropnoncptp}]
Adopt the notation of Definition~\ref{D_1X14_A2pi}.
Let $\pi \colon G \to {\operatorname{U}} (\Hi_{\pi})$
be in $\widehat{G}$.
We need to prove that
$(B (\Hi_{\pi}) \otimes A)^{\Ad (\pi) \otimes \alpha}
 \S {\overline{ A_{2} (\pi)^{*}  A_{2} (\pi)}}$.
By Lemma~\ref{L_1X14_A2Pi_id},
${\overline{ A_{2} (\pi)^{*}  A_{2} (\pi)}}$ is a
closed two sided ideal in
$(B (\Hi_{\pi}) \otimes A)^{\Ad (\pi) \otimes \alpha}$.
So it is enough to prove that $\overline{ A_{2} (\pi)^{*}  A_{2} (\pi)}$
contains an invertible element.

Let $d_{\pi} = \dim (\Hi_{\pi})$
and identify $\Hi_{\pi}$ with ${\mathbb{C}}^{d_{\pi}}$.
Set $\delta = 1 / (4 d_{\pi}^{3})$.
Equip $C (G)$ with the action $\Lt$ of Notation~\ref{N_1X07_Lt}.
An easy computation shows that the function $g \mapsto \pi (g)^*$,
which we abbreviate to $\pi^*$, is in $C (G)_{2} (\pi)$.
For $j, k = 1, 2, \ldots, d_{\pi}$,
let $\pi_{j, k} \in C (G)$ be the function $g \mapsto \pi (g)_{j, k}$
whose value is the $(j, k)$ matrix entry of~$\pi (g)$.
Apply Lemma~\ref{TRPcequi}
to $\alpha$ with $\delta$ in place of $\varepsilon$, with $n = 2$, with
\[
F = \{ 1_{A} \}
\andeqn
S = \bigl\{ \pi_{j, k}, \, (\pi_{j, k})^* \colon
 j, k = 1, 2, \ldots, d_{\pi} \bigr\},
\]
and with $x = y = 1$.
We get a projection $p \in A^{\alpha}$
such that $1 - p \precsim_{A^{\af}} p$ and an equivariant \ucp{}
$(F, S, \dt)$-approximately central multiplicative map
$\varphi \colon C (G) \to p A p$.
Then
\[
\id_{B (\Hi_{\pi})} \otimes \varphi \colon
  B (\Hi_{\pi}) \otimes C (G)
   \to (1_{B (\Hi_{\pi})} \otimes p) (B (\Hi_{\pi}) \otimes A)
    (1_{B (\Hi_{\pi})} \otimes p)
\]
is an equivariant \ucp{} map.

Let $(e_{j, k})_{j, k = 1, 2, \ldots, n}$
be the standard system of matrix units for $B (\Hi_{\pi})$.
Set
\[
c = ( \id_{B (\Hi_{\pi})} \otimes \varphi ) (\pi^*)
  = \sum_{j, k = 1}^{d_{\pi}}
          e_{j, k} \otimes \varphi ( (\pi_{k, j})^* )
  \in B (\Hi_{\pi}) \otimes p A p.
\]
Since $\id_{B (\Hi_{\pi})} \otimes \varphi$ is equivariant
and $\pi^* \in C (G)_2 (\pi)$,
an easy calculation shows that $c \in A_{2} (\pi)$.
Thus $c^* c \in {\overline{ A_{2} (\pi)^{*} A_{2} (\pi)}}$.
Moreover, since $\pi \pi^* = 1$,
\begin{equation}\label{Eq_1X14_cStc}
\| c^{*} c - 1_{B (\Hi_{\pi})} \otimes p \|
 \leq \sum_{j, k, l = 1}^{\I}
   \bigl\| \ph ( \pi_{j, k}) \ph ( (\pi_{l, k})^* )
           - \ph (\pi_{j, k} (\pi_{l, k})^* ) \bigr\|
 < d_{\pi}^3 \dt.
\end{equation}

Using $1_{A} - p \precsim_{A^{\alpha}} p$, choose $s \in A^{\alpha}$
such that $s^* s = 1_{A} - p$ and $s s^* \leq p$.
Then $s^* p s = 1_{A} - p$.
Clearly
$1_{B (\Hi_{\pi})} \otimes s
 \in (B (\Hi_{\pi}) \otimes A)^{\Ad (\pi) \otimes \alpha}$.
Using~(\ref{Eq_1X14_cStc}), we get
\[
\bigl\| (1_{B (\Hi_{\pi})} \otimes s)^* c^*
       c (1_{B (\Hi_{\pi})} \otimes s)
  - 1_{B (\Hi_{\pi})} \otimes (1 - p) \bigr\|
 < d_{\pi}^3 \dt.
\]
Therefore the element
\[
z = c^{*} c + (1_{B (\Hi_{\pi})} \otimes s)^* c^*
       c (1_{B (\Hi_{\pi})} \otimes s)
\]
satisfies
\begin{equation}\label{Eq_1X14_NEst}
\| z - 1_{B (\Hi_{\pi})} \otimes 1 \|
  < 2 d_{\pi}^3 \dt
  < 1.
\end{equation}
Also, since ${\overline{ A_{2} (\pi)^{*}  A_{2} (\pi)}}$ is an ideal in
$(B (\Hi_{\pi}) \otimes A)^{\Ad (\pi) \otimes \alpha}$,
it follows that
\[
(1_{B (\Hi_{\pi})} \otimes s)^* c^* c (1_{B (\Hi_{\pi})} \otimes s)
  \in {\overline{ A_{2} (\pi)^{*} A_{2} (\pi)}}.
\]
Therefore $z \in {\overline{ A_{2} (\pi)^{*} A_{2} (\pi)}}$.
So~(\ref{Eq_1X14_NEst}) implies that
${\overline{ A_{2} (\pi)^{*} A_{2} (\pi)}}$
contains an invertible element.
This completes the proof.
\end{proof}

\begin{thm}\label{simplecrosscomparison}
Let $A$ be an infinite dimensional simple separable unital \ca,
and let $\alpha \colon G \to \Aut (A)$
be an action of a second countable compact group $G$ on~$A$
which has the tracial Rokhlin property with comparison.
Then the crossed product $C^{*} (G, A, \alpha)$ is simple.
\end{thm}

\begin{proof}
The algebra $A^{\af}$ is simple by Theorem~\ref{thm_simple12}
and $\af$ is saturated by Proposition~\ref{satpropnoncptp}.
So Condition~(\ref{Item_1X07_SimSat}) in Theorem~\ref{satunonabel}
holds.
\end{proof}

\begin{cor}\label{C_1X15_StableIso}
Under the hypotheses of Theorem~\ref{simplecrosscomparison},
the algebras $C^{*} (G, A, \alpha)$ and $A^{\af}$
are strongly Morita equivalent and stably isomorphic.
\end{cor}

\begin{proof}
By Theorem~\ref{simplecrosscomparison}, the condition
in Theorem \ref{satunonabel}(\ref{Item_1X07_CPSim}) holds.
Strong Morita equivalence follows from the last statement
in Theorem~\ref{satunonabel},
and stable isomorphism then follows from separability.
\end{proof}

We can also improve Lemma~\ref{mtrpcentral},
replacing in Condition~(\ref{1_psmalllnes_FP}) there
the requirement $1 - p \precsim_{(A^{\af})_{\I}} y$
with the requirement that $1 - p$ be small in $(A^{\af})_{\I}$.
Similar improvements are valid
for the versions of Lemma~\ref{mtrpcentral} used in
Remarks \ref{R_1Z09_ntr}, \ref{R_1Z09_nzp}, \ref{R_1Z09_mod_trp},
and~\ref{R_S_mod_trp}.
Since we don't use them, we don't state them formally.

\begin{prp}\label{P_1X14_CentSq}
Let $G$ be a second countable compact group,
let $A$ be a simple separable infinite dimensional \uca, and let
$\alpha \colon G \to \Aut (A)$ be an action of $G$ on~$A$.
Then $\af$ has the tracial Rokhlin property with comparison
\ifo{} for
every nonzero positive element $x$ in $A_{+}$ with $\| x \| = 1$,
there exist a projection $p \in (A_{\I, \af} \cap A')^{\alpha_{\infty}}$
and a unital equivariant homomorphism
\[
\ps \colon C (G) \to p (A_{\I, \af} \cap A') p,
\]
such that the following hold:
\begin{enumerate}
\item\label{It_CSq_1X14__1mp_af_sm}
$1 - p$ is $\af$-small in $A_{\I, \af}$.
\item\label{It_CSq_1X14_1mp_sm_Aaf}
$1 - p$ is small in $(A^{\af})_{\I}$.
\item\label{It_CSq_1X14_1mp_p}
$1 - p \precsim_{(A^{\af})_{\I}} p$.
\item\label{It_CSq_1X14_Norm}
Identifying $A$ with its image in $A_{\I, \af}$,
we have $\| p x p\| = 1$.
\setcounter{TmpEnumi}{\value{enumi}}
\end{enumerate}
\end{prp}

\begin{proof}
We only describe the changes from the proof of
Lemma~\ref{mtrpcentral}.

To show that the \trpc{} implies the existence of~$\ps$,
in addition to the dense sequences (\ref{Eq_1X14_fnan})
and~(\ref{Eq_1X14_xn}),
we choose a dense sequence
\[
y_1, y_2, \ldots
 \in \bigl\{ a \in (A^{\af})_{+} \colon \| a \| = 1 \bigr\}.
\]
For $n \in \N$, use simplicity of~$A^{\af}$
(Theorem~\ref{thm_simple12}) and Lemma~\ref{lma_L683_baj}
to choose $d_n \in (A^{\af})_{+} \setminus \{ 0 \}$
such that
\begin{equation}\label{Eq_1X14_NStSt}
d_n \precsim_{A} \Bigl( y_k - \frac{1}{2} \Bigr)_{+}
\end{equation}
for $k = 1, 2, \ldots, n$.
In the application of Lemma~\ref{TRPcequi} in the proof of
Lemma~\ref{mtrpcentral}, use $d_n$ instead of~$y$.
This replaces (\ref{Item_1X10_1py}) in the proof of
Lemma~\ref{mtrpcentral}
with the requirement that $1 - q_n \precsim_{A^{\alpha}} d_n$.
Let $p$ be as in that proof.
To show that $1 - p$ is small in $(A^{\af})_{\I}$,
let $t \in A_{+} \setminus \{ 0 \}$; as there, it is enough to find
$N \in \N$ such that for all $n \geq N$
we have $1 - q_{n} \precsim_{A^{\af}} t$.
We may assume that $\| t \| = 1$.
Choose $N \in \N$ such that $\| y_N - t \| < \frac{1}{2}$.
Then $n \geq N$ implies
\[
1 - q_n \precsim_{A^{\af}} d_n
 \precsim_{A^{\af}} \Bigl( y_N - \frac{1}{2} \Bigr)_{+}
 \precsim_{A} t,
\]
as desired.

For the reverse direction, all parts of the proof of the corresponding
part of Lemma~\ref{mtrpcentral} apply, except that we must use
the hypothesis that $1 - p$ is small in $(A^{\af})_{\I}$,
instead of the hypothesis $1 - p \precsim_{(A^{\af})_{\I}} y$,
to show that for all sufficiently large $n$
we have $q_n \precsim_{A^{\af}} y$.
Accordingly, let $n_0, n_1, n_2, n_3, n_4 \in \N$ be as in the
proof of Lemma~\ref{mtrpcentral}.
Since $1 - p$ is small in $(A^{\af})_{\I}$, there is $n_5 \in \N$
such that for every $n \geq n_5$
we have $1 - q_{n} \precsim_{A^{\af}} y$.
Then take $n = \max (n_0, n_1, \ldots, n_5)$,
and in Definition~\ref{traR}
take $\ph$ to be $\gm_n$ and take $p$ to be~$q_n$.
Lemma~\ref{L_1X09_CSbPj} is still used to prove that
$1 - q_n \precsim_{A^{\alpha}} q_n$.
\end{proof}

\section{Permanence properties}\label{Sec_1951_TRP_Crossed_TRR0}

In this section, we prove that fixed point algebras of actions
of compact groups on infinite dimensional simple separable unital \ca{s}
which have the \trpc{} preserve Property~(SP),
tracial rank zero, tracial rank at most one,
the Popa property, tracial $\cZ$-stability,
the combination of nuclearity and $\cZ$-stability,
infiniteness, and pure infiniteness.
Where appropriate, we prove the same for crossed products.
We also show that the radius of comparison of the fixed point algebra
is no larger than that of the original algebra.

For finite group actions with the tracial Rokhlin property,
the tracial rank zero case is Theorem~2.6 of~\cite{phill23}.
For a second countable compact group and an action with the
Rokhlin property, the tracial rank zero
and tracial rank at most one cases are Theorem~4.5 of~\cite{crossrep}.
Preservation of tracial rank zero
and of tracial rank at most one answers a question posed
after Theorem~4.5 in~\cite{crossrep}.

\begin{rmk}\label{L_2015_S4_nzp}
As in Section~\ref{Sec_1160_Simpli_Prf},
all results of this section which have the \trpc{}
(Definition~\ref{traR}) in their hypotheses
are still valid if in the definition of the \trpc{}
we omit Condition~(\ref{Item_902_pxp_TRP}).
Whenever we need to know $p \neq 0$ in an application of
Theorem~\ref{thm_ApproxHommtr},
we deduce it from Theorem \ref{thm_ApproxHommtr}(\ref{Item_1X07_20p}).
This claim then follows from Remark~\ref{R_1Z09_nzp}.
\end{rmk}

\begin{lem}\label{PROSPTRPCfixed}
Let $A$ be an infinite dimensional simple separable unital \ca.
Let $\alpha \colon G \to \Aut (A)$ be
an action of a second countable compact group
$G$ on $A$ which has the tracial Rokhlin property with comparison,
but does not have the Rokhlin property.
Then $A^{\alpha}$ has Property~(SP).
\end{lem}

\begin{proof}
Suppose $A^{\af}$ does not have Property~(SP).
We verify the condition of Lemma~\ref{L_1X16_AppDfRk}.
Let $F \subseteq A$ and $S \subseteq C (G)$ be finite,
and let $\varepsilon > 0$.
Choose $y \in (A^{\af})_{+} \setminus \{ 0 \}$ such that
and ${\overline{y A^{\af} y}}$ contains no nonzero \pj,
and take $x = 1$.
Apply Definition~\ref{traR}.
Then $p = 1$, so the condition of Lemma~\ref{L_1X16_AppDfRk} holds.
\end{proof}

The following result could have been in~\cite{crossrep}.
However, our proof needs simplicity of~$A$.
Without simplicity,
we don't know how to prove that the \pj{} we construct is nonzero,
and it remains open whether fixed point algebras
and crossed products by actions of
compact groups preserve Property~(SP) in general.

\begin{thm}\label{T_1X20_SP}
Let $A$ be a simple separable infinite dimensional unital \ca,
let $G$ be a second countable compact group, and let
$\alpha \colon G \to \Aut (A)$ be an action
which has the Rokhlin property.
If $A$ has Property~(SP), so do $A^{\alpha}$ and $C^{*} (G, A, \alpha)$.
\end{thm}

\begin{proof}
We prove that $A^{\alpha}$ has Property~(SP).
Let $x \in A^{\alpha}_{+} \setminus \{ 0 \}$;
we need to prove that $\ov{x A^{\alpha} x}$ contains a nonzero \pj.
\Wolog{} $\| x \| = 1$.
Define \cfn{s} $h, h_0 \colon [0, 1] \to [0, 1]$ by
\[
h (\ld) = \begin{cases}
   0 & \hspace*{1em} 0 \leq \ld \leq \frac{3}{4}
        \\
  4 \ld - 3 & \hspace*{1em} \frac{3}{4} \leq \ld \leq 1
\end{cases}
\]
and
\[
h_0 (\ld) = \begin{cases}
   \frac{4}{3} \ld & \hspace*{1em} 0 \leq \ld \leq \frac{3}{4}
        \\
   1 & \hspace*{1em} \frac{3}{4} \leq \ld \leq 1.
\end{cases}
\]
Set $x_0 = h_0 (x)$.
Use Property~(SP) to choose a nonzero \pj{} $e \in \ov{h (x) A h (x)}$.
Then
\begin{equation}\label{Eq_1X20_x0xe}
\| x - x_0 \| \leq \frac{1}{4}
\andeqn
x_0 e = e.
\end{equation}
Since $A$ is simple, there are $m \in \N$ and
$a_1, a_2 \ldots, a_{m}, b_1, b_2, \ldots, b_{m} \in A$ such that
$\sum_{j = 1}^{m} a_j e b_j = 1$.
Define
\[
M = 1 + \max_{j = 1, 2, \ldots, m} \max ( \| a_j \|, \| b_j \| )
\andeqn
\ep_0
 = \min \left( \frac{1}{24}, \, \frac{1}{2 m (M^2 + M + 1)} \right).
\]
Choose $\ep > 0$ so small that $\ep \leq \ep_0$
and whenever $D$ is a \ca{}
and $a \in D_{\sa}$ satisfies $\| a^2 - a \| < \ep$,
then there is a \pj{} $q \in D$ such that $\| q - a \| < \ep_0$.
Set
\begin{equation}\label{Eq_1X20_F1F2_SP}
F_{1}
 = \{ a_j, a_j e, b_j \colon j = 1, 2, \ldots, m \} \cup \{ e, x_0 \}
\andeqn
F_2 = \{ x_0 \}.
\end{equation}
Apply Theorem~2.11 of~\cite{crossrep} with these choices of
$F_1$, $F_2$, and~$\ep$, getting a \ucp{} map $\ps \colon A \to A^{\af}$
such that the following hold.
\begin{enumerate}
\item\label{Item_1X20_SP_mult}
$\| \ps (a b) - \ps (a) \ps (b) \| < \ep$ for all $a, b \in F_1$.
\item\label{Item_1X20_SP_fix}
$\| \ps (x_0) - x_0 \| < \ep$.
\end{enumerate}
We have $\| \ps (e)^2 - \ps (e) \| < \ep$ by~(\ref{Item_1X20_SP_mult}).
Therefore the choice of $\ep$ provides
a \pj{} $q \in A^{\af}$ such that $\| q - \ps (e) \| < \ep_0$.

We claim that $q \neq 0$.
For $j = 1, 2, \ldots, m$ we get
\[
\begin{split}
& \| \ps (a_j e b_j) - \ps (a_j) q \ps (b_j) \|
\\
& \hspace*{3em} {\mbox{}}
  \leq \| \ps (a_j e b_j) - \ps (a_j e) \ps (b_j) \|
    + \| \ps (a_j e) - \ps (a_j) \ps (e) \| \| \ps (b_j) \|
\\
& \hspace*{6em} {\mbox{}}
    + \| \ps (a_j) \| \| \ps (e) - q \| \| \ps (b_j) \|
\\
& \hspace*{3em} {\mbox{}}
  < \ep + M \ep + M^2 \ep_0
  \leq (M^2 + M + 1) \ep_0.
\end{split}
\]
So
\[
\biggl\| 1 - \sum_{j = 1}^{m} \ps (a_j) q \ps (b_j) \biggr\|
  \leq \sum_{j = 1}^{m} \| \ps (a_j e b_j) - \ps (a_j) q \ps (b_j) \|
  < m (M^2 + M + 1) \ep_0
  = \frac{1}{2}.
\]
Therefore $q \neq 0$, as claimed.

We have, using~(\ref{Item_1X20_SP_mult}) and $e x_0 = x_0 e = e$
at the second step,
\[
\begin{split}
& \| \ps (e) x_0 \ps (e) - \ps (e) \|
\\
& \hspace*{3em} {\mbox{}}
  \leq \| x_0 - \ps (x_0) \|
    + \| \ps (e) \ps (x_0) - \ps (e) \| \| \ps (e) \|
    + \| \ps (e)^2 - \ps (e) \|
\\
& \hspace*{3em} {\mbox{}}
  < \ep + \ep + \ep
  = 3 \ep.
\end{split}
\]
So
\[
\begin{split}
\| q x q - q \|
& \leq \| x - x_0 \|
    + 3 \| q - \ps (e) \|
    + \| \ps (e) x_0 \ps (e) - \ps (e) \|
\\
& < \frac{1}{4} + 3 \ep_0 + 3 \ep
  \leq \frac{1}{4} + 6 \ep_0
  \leq \frac{1}{2}.
\end{split}
\]
Therefore $(q x q)^{- 1/2}$ makes sense in $q A^{\af} q$,
and $s = x^{1 / 2} (q x q)^{- 1/2}$ is a partial isometry
with $s^* s = q$.
So $s s^*$ is a nonzero \pj{} in $\ov{x A^{\alpha} x}$.
Thus $A^{\af}$ has Property~(SP).

Corollary~\ref{C_1X15_StableIso} implies that
$C^{*} (G, A, \alpha)$ and $A^{\af}$ are stably isomorphic.
Therefore $C^{*} (G, A, \alpha)$ also has Property~(SP).
\end{proof}

\begin{cor}\label{L_1X17_}
Let $G$ be a second countable compact group,
let $A$ be a simple separable infinite dimensional \uca, and let
$\alpha \colon G \to \Aut (A)$ be an action of $G$ on~$A$
which has the tracial Rokhlin property with comparison.
If $A$ has Property~(SP), so do $A^{\alpha}$ and $C^{*} (G, A, \alpha)$.
\end{cor}

\begin{proof}
If $\af$ has the Rokhlin property, then $A^{\alpha}$ has Property~(SP),
by Theorem~\ref{T_1X20_SP}.
Otherwise, apply Lemma~\ref{PROSPTRPCfixed} to get the same conclusion.

Corollary~\ref{C_1X15_StableIso} implies that
$C^{*} (G, A, \alpha)$ and $A^{\af}$ are stably isomorphic.
Therefore $C^{*} (G, A, \alpha)$ also has Property~(SP).
\end{proof}

We next consider preservation of tracial rank.

\begin{ntn}\label{L_2123_Jk}
For $k = 0, 1$ let $\cJ_k$ be the set of all \ca{s} of the form
\[
\bigoplus_{j = 1}^m C ([0, 1]^{\ep (j)}, \, M_{r (j)} )
\]
for $m , r (1), r(2), \ldots, r (m) \in \N$ and
$\ep (j) \in \Nz$ with $\ep (j) \leq k$ for $j = 1, 2, \ldots, m$.
\end{ntn}

Up to isomorphism, $\cJ_0$ is the class of \fd{} \ca{s}
and $\cJ_1$ is the class of all direct sums $E_0 \oplus C ([0, 1], E_1)$
for \fd{} \ca{s} $E_0$ and $E_1$.
However, the description above is technically more convenient.

We need the following modification of known characterizations of
tracial rank zero and one for simple \uca{s}.
In the known version (Theorem 7.1(2) of \cite{LnTTR}),
the \hm~$\ph$ is required to be injective.

\begin{lem}\label{L_2123_TR01}
Let $A$ be a simple separable \uca, and let $k \in \{ 0, 1 \}$.
Then $A$ has tracial rank at most~$k$,
in the sense of Definition 3.6.2 of~\cite{linbook},
\ifo{} for every $\ep > 0$, every finite subset $F \S A$,
and every $x \in A_{+} \SM \{ 0 \}$, there are a \ca{} $B \in \cJ_k$,
a \nzp{} $p \in A$, and a unital \hm{} $\ph \colon B \to p A p$,
such that the following hold.
\begin{enumerate}
\item\label{I_2123_appa_TRA0}
$\| a p - p a \| < \ep$ for all $a \in F$.
\item\label{I_2123_papB_TRA0}
$\dist (p a p, \, \varphi (B)) < \ep$ for all $a \in F$.
\item\label{I_2123_1_p_x_TRA0}
$1 - p$ is Murray-von Neumann equivalent to a
projection in ${\overline{x A x}}$.
\end{enumerate}
\end{lem}

For convenience, we state the following lemma separately.

\begin{lem}\label{L_2123_SubsetI}
Let $X \S [0, 1]$ be nonempty and closed,
let $\ep > 0$, let $r, n \in \N$,
and let $a_1, a_2, \ldots, a_n \in C (X, M_r)$.
Then there are a \ca{} $B \in \cJ_1$,
an injective unital \hm{} $\ps \colon B \to  C (X, M_r)$,
and $b_1, b_2, \ldots, b_n \in B$,
such that $\| \ps (b_l) - a_l \| < \ep$ for $l = 1, 2, \ldots, n$.
\end{lem}

\begin{proof}
For $x \in X$, $l = 1, 2, \ldots, n$, and $j, k = 1, 2, \ldots, r$,
let $a_{l, j, k} (x)$ be the matrix entries of $a_l (x)$.
Set $\ep_0 = \ep / (2 r^2)$.
Choose $\dt > 0$ such that whenever $x_1, x_2 \in X$
satisfy $| x_1 - x_2 | < 2 \dt$,
then for $l = 1, 2, \ldots, n$ and $j, k = 1, 2, \ldots, r$
we have $| a_{l, j, k} (x_1) - a_{l, j, k} (x_2) | < \ep_0$.
Choose $\nu \in \N$ and $y_0, y_1, \ldots, y_{\nu} \in X$ such that
the set $S = \{ y_0, y_1, \ldots, y_{\nu} \}$ is $\dt$-dense in~$X$ and
$y_0 < y_1 < \cdots < y_{\nu}$.
Then there are $t \in \Nz$
and $\mu (0), \mu (1), \ldots, \mu (t + 1) \in \Z$
such that $0 = \mu (0) < \mu (1) < \cdots < \mu (t + 1) = \nu + 1$,
such that for $s = 0, 1, \ldots, t$ the points
$y_{\mu (s)}, \, y_{\mu (s) + 1}, \, \ldots, \, y_{\mu (s + 1) - 1}$
are all in the same connected component of~$X$,
and such that for $s = 1, 2, \ldots, t$ the points
$y_{\mu (s) - 1}$ and $y_{\mu (s)}$
are not in the same connected component of~$X$.
In particular, for $s = 0, 1, \ldots, t$, the
set $Y_s = [y_{\mu (s)}, \, y_{\mu (s + 1) - 1}]$ is contained in~$X$,
while for $s = 1, 2, \ldots, t$
we have $(y_{\mu (s) - 1}, \, y_{\mu (s)}) \not\S X$.

For $s = 1, 2, \ldots, t$, if $y_{\mu (s)} - y_{\mu (s) - 1} < 2 \dt$
choose some $z_s \in (y_{\mu (s) - 1}, \, y_{\mu (s)})$
with $z_s \not\in X$,
while if $y_{\mu (s)} - y_{\mu (s) - 1} \geq 2 \dt$
set $z_s = y_{\mu (s) - 1} + \dt$.
In the second case, we also have $z_s \not\in X$,
because $S$ is $\dt$-dense in~$X$ but $\dist (z_s, S) = \dt$.
Set $Y = \bigcup_{s = 0}^t Y_s$, which is a closed subset of~$X$.
Define a retraction $h \colon X \to Y$ by
\[
h (x) = \begin{cases}
   x & \hspace*{1em} x \in Y
        \\
   y_{\mu (s) - 1} & \hspace*{1em}
          {\mbox{$s \in \{ 1, 2, \ldots, t \}$
           and $y_{\mu (s) - 1} < x < z_s$}}
       \\
   y_{\mu (s)} & \hspace*{1em}
          {\mbox{$s \in \{ 1, 2, \ldots, t \}$
              and $z_s < x < y_{\mu (s)}$}}.
\end{cases}
\]
We claim that $| h (x) - x | < 2 \dt$ for all $x \in X$.
This is trivial if $x \in Y$.
So suppose $s \in \{ 1, 2, \ldots, t \}$ and
$y_{\mu (s) - 1} < x  < y_{\mu (s)}$.
If $y_{\mu (s)} - y_{\mu (s) - 1} < 2 \dt$,
the claim is immediate.
Otherwise, since $S$ is $\dt$-dense in~$X$,
we have $[ y_{\mu (s) - 1} + \dt, \, y_{\mu (s)} - \dt] \cap X = \E$,
while $h (x) = y_{\mu (s) - 1}$
if $x \in (y_{\mu (s) - 1}, \, y_{\mu (s) - 1} + \dt)$
and $h (x) = y_{\mu (s)}$
if $x \in (y_{\mu (s)} - \dt, \, y_{\mu (s)})$.
The claim follows.

The algebra $B = C (Y, M_r)$ is in $\cJ_1$.
Define $\ps \colon B \to C (X, M_r)$ by $\ps (a) (x) = a (h (x))$
for $x \in X$.
This \hm{} is injective since $h$ is surjective.
For $l = 1, 2, \ldots, n$ set $b_l = a_l |_{Y}$.
The previous claim and the choice of $\dt$ imply
that for $x\in X$, $l = 1, 2, \ldots, n$, and $j, k = 1, 2, \ldots, r$
we have
\[
| a_{l, j, k} (x) - a_{l, j, k} ( h (x)) | < \frac{\ep}{2 r^2}.
\]
Therefore
\[
\| a_l (x) - \ps (b_l) (x) \|
 = \| a_l (x) - a_l ( (h (x)) \|
 \leq \sum_{j, k = 1}^r | a_{l, j, k} (x) - a_{l, j, k} ( h (x)) |
 < \frac{\ep}{2}.
\]
It follows that $\| a_l - \ps (b_l) \| \leq \frac{\ep}{2} < \ep$.
\end{proof}

\begin{proof}[Proof of Lemma~\ref{L_2123_TR01}]
The only difference between the condition in the lemma and 
Definition 3.6.2 of~\cite{linbook} is that in~\cite{linbook}
the \hm{} $\ph$ is required to be injective.

First suppose $k = 0$.
Since in our situation $B / \Ker (\ph)$ is again \fd,
hence in $\cJ_0$,
the two sets of conditions are in fact equivalent.

Now assume $k = 1$.
We assume $A$ satisfies the condition in the lemma, and
prove that one can force the map $\ph$ to be injective.
Accordingly,
let $\ep > 0$, let $F \S A$ be finite, and
let $x \in A_{+} \SM \{ 0 \}$.
Choose a \nzp{} $p \in A$, a \ca{} $D \in \cJ_1$,
and a unital \hm{} $\sm \colon D \to p A p$ satisfying the
conditions stated in the lemma,
but with $\frac{\ep}{4}$ in place of~$\ep$.
Following Notation~\ref{L_2123_Jk}, write
\[
D = \bigoplus_{j = 1}^m C ([0, 1]^{\ep (j)}, \, M_{r (j)} )
\]
with $m , r (1), r(2), \ldots, r (m) \in \N$ and
$\ep (j) \in \{ 0, 1 \}$ for $j = 1, 2, \ldots, m$.
There are closed subsets $X_j \S [0, 1]^{\ep (j)}$ such that
\[
\Ker (\sm)
 = \bigoplus_{j = 1}^m C_0 ([0, 1]^{\ep (j)} \SM X_j, \, M_{r (j)} ).
\]
We make the obvious identification of $B / \Ker (\sm)$
with $\bigoplus_{j = 1}^m C (X_j, M_{r (j)} )$, and we let
\[
\kp \colon D \to \bigoplus_{j = 1}^m C (X_j, M_{r (j)} )
\andeqn
{\ov{\sm}} \colon \bigoplus_{j = 1}^m C (X_j, M_{r (j)} ) \to p A p
\]
be the quotient map and the induced map from the quotient,
so that $\sm = {\ov{\sm}} \circ \kp$.
Discarding summands,
\wolog{} ${\ov{\sm}} |_{C (X_j, M_{r (j)} )} \neq 0$
for $j = 1, 2, \ldots, m$.
Write $F = \{ a_1, a_2, \ldots, a_t \}$.
For $s = 1, 2, \ldots, t$ choose $c_s \in D$ such that
$\| \sm (c_s) - p a_s p \| < \frac{\ep}{2}$.
Write $\kp (c_s) = (d_{s, 1}, d_{s, 2}, \ldots, d_{s, m} )$
with $d_{s, j} \in C (X_j, M_{r (j)} )$ for $j = 1, 2, \ldots, m$.

For $j \in \{ 1, 2, \ldots, m \}$ for which $\ep (j) = 0$,
the map ${\ov{\sm}} |_{C (X_j, M_{r (j)} )}$ is injective.
Set $B_j = C (X_j, M_{r (j)} )$, $\ps_j = \id_{B_j}$,
$\ph_j = {\ov{\sm}} |_{C (X_j, M_{r (j)} )}$, and
$b_{s, j} = d_{s, j}$ for $s = 1, 2, \ldots, t$.
For all other $j \in \{ 1, 2, \ldots, m \}$, since $X_j \neq \E$
we can use Lemma~\ref{L_2123_SubsetI} to choose a \ca{} $B_j \in \cJ_1$,
an injective unital \hm{} $\ps_j \colon B_j \to  C (X_j, M_{r (j)})$,
and, for $s = 1, 2, \ldots, t$, $b_{s, j} \in B_j$
such that $\| \ps_j (b_{s, j}) - d_{s, j} \| < \frac{\ep}{2}$.
Then define $\ph_j = {\ov{\sm}} \circ \ps_j \colon B_j \to p A p$.
Set $B = \bigoplus_{j = 1}^m B_j$, which is in $\cJ_1$,
and define $\ph \colon B \to p A p$ by
$\ph (b_1, b_2, \ldots, b_m) = \sum_{j = 1}^m \ph_j (b_j)$.
Then $\ph$ is an injective unital \hm{}
and, for $s = 1, 2, \ldots, t$,
\[
\begin{split}
& \bigl\| \ph (b_{s, 1}, b_{s, 2}, \ldots, b_{s, m} ) - p a_s p \bigr\|
\\
& \hspace*{3em} {\mbox{}}
  \leq \max_{1 \leq j \leq m} \| \ps_j (b_{s, j}) - d_{s, j} \|
      + \bigl\| {\ov{\sm}} (d_{s, 1}, d_{s, 2}, \ldots, d_{s, m} )
           - p a_s p \bigr\|
\\
& \hspace*{3em} {\mbox{}}
  < \frac{\ep}{2} + \| \sm (c_s) - p a_s p \|
  < \frac{\ep}{2} + \frac{\ep}{2}
  = \ep.
\end{split}
\]
This completes the proof.
\end{proof}

\begin{thm}\label{T_2123_PrsvTR}
Let $G$ be a second countable compact group,
let $A$ be a simple separable infinite dimensional \uca, and let
$\alpha \colon G \to \Aut (A)$ be an action of $G$ on~$A$
which has the tracial Rokhlin property with comparison.
Let $k \in \{ 0, 1 \}$.
If $A$ has tracial rank at most~$k$,
then $A^{\alpha}$ has tracial rank at most~$k$.
\end{thm}

\begin{proof}
Combining Theorem 7.1(2) of \cite{LnTTR}
and Theorem~3.2 of~\cite{Ln_tr1},
we see that $A$ has Property~(SP).
(We do not actually need this step.
If $A$ does not have Property~(SP),
then $\alpha$ has the Rokhlin property by Lemma~\ref{PROSPTRPCA}.
So $A$ has tracial rank at most~$k$ by Theorem~4.5 of~\cite{crossrep}.)

We verify the criterion in Lemma~\ref{L_2123_TR01}.
So let $\ep > 0$, let $F \S A^{\af}$ be finite, and
let $x \in (A^{\af})_{+} \SM \{ 0 \}$.
\Wolog{} $\ep < 1$ and $\| a \| \leq 1$ for all $a \in F$.

The algebra $A^{\af}$ is simple by Theorem~\ref{thm_simple12}
and has Property~(SP) by Theorem~\ref{T_1X20_SP}.
So there is a \nzp{} $q \in {\ov{x A^{\af} x}}$.
Lemma~\ref{OrthInSP} then provides
nonzero mutually orthogonal projections $q_1, q_2 \in q A^{\af} q$.

Apply Lemma~\ref{L_2123_TR01}, getting a \ca{} $B \in \cJ_k$,
a \nzp{} $p_0 \in A$, a unital \hm{} $\ph_0 \colon B \to p_0 A p_0$,
and a partial isometry $s \in A$ (implementing the relation
in Lemma \ref{L_2123_TR01}(\ref{I_2123_1_p_x_TRA0}))
such that the following hold.
\begin{enumerate}
\item\label{I_2123_appa_Pf}
$\| a p_0 - p_0 a \| < \frac{\ep}{4}$ for all $a \in F$.
\item\label{I_2123_papB_Pf}
$\dist (p_0 a p_0, \, \varphi_0 (B)) < \frac{\ep}{4}$ for all $a \in F$.
\item\label{I_2123_1_p_x_Pf}
$s^* s = 1 - p_0$ and $s s^* \leq q_1$.
\setcounter{TmpEnumi}{\value{enumi}}
\end{enumerate}
In particular, there is a finite set $F_0 \S B$ such that
for every $a \in F$ there is $y \in F_0$ with
$\| \ph_0 (y) - p_0 a p_0 \| < \frac{\ep}{2}$.
Further, let $T \S B$ be a finite set which generates $B$ as a \ca,
and such that $\| y \| \leq 1$ for all $y \in T$.

The algebra $B$ is semiprojective, as is seen by combining
Proposition 16.2.1, Theorem 14.2.1,
and Theorem 14.2.2 in~\cite{Lrng_bk}.
One can then easily see that there are $N \in \N$ and $\dt_0 > 0$
such that whenever $D$ is a \uca{} and $\sm \colon B \to D$
is unital, completely positive, satisfies $\| \sm \| \leq 1$, and
is $(N, T, \dt_0)$-approximately multiplicative
(Definition~\ref{D_1X10_nSFe}),
then there is a unital \hm{} $\ph \colon B \to D$
such that $\| \ph (y) - \sm (y) \| < \frac{\ep}{4}$ for all $y \in F_0$.
(There is no reason to think that this can always be done with $N = 2$,
since the relations on $T$ which determine $B$ may
involve products with more than two factors.)

Set
\begin{equation}\label{Eq_2123_dt1}
\dt_1
 = \min \Bigl( \frac{1}{16}, \, \frac{\dt_0}{4 N},
                 \, \frac{\ep}{16} \Bigr).
\end{equation}
Choose $\dt > 0$ so small that
\begin{equation}\label{Eq_2123_dt_l_1}
\dt \leq \dt_1
\end{equation}
and whenever $D$ is a \ca{}
and $c \in D_{\sa}$ satisfies $\| c^2 - c \| < \dt$,
then there is a \pj{} $p \in D$ such that $\| p - c \| < \dt_1$.
Set $n = \max (N, 5)$.
Define finite sets
\begin{equation}\label{Eq_2205_Df}
F_1 = \bigl\{ q, q_1, s, s^*, p_0 \bigr\} \cup F \cup \ph_0 (T)
    \S A
\andeqn
F_2 = \{ q_1 \} \cup F
    \S A^{\af}.
\end{equation}
Apply Theorem~\ref{thm_ApproxHommtr} with $q_2$ in place of $x$ and~$y$,
with $\dt$ in place of~$\ep$,
and with $n$, $F_1$, and $F_2$ as given.
This gives a \pj{} $f \in A^{\af}$
(necessarily nonzero, by (\ref{I_2123_norm}) below),
and a \ucp{} map $\ps \colon A \to f A^{\af} f$,
such that the following hold.
\begin{enumerate}
\setcounter{enumi}{\value{TmpEnumi}}
\item\label{I_2123_am}
Whenever $m \in \{ 1, 2, \ldots, n \}$
and $a_1, a_2, \ldots, a_m \in F_1 \cup F_2$, we have
\[
\bigl\| \ps (a_1 a_2 \cdots a_m)
  - \ps (a_1) \ps (a_2) \cdots \ps (a_m) \bigr\|
   < \dt.
\]
\item\label{I_2123_comm}
$\| f a - a f \| < \dt$ for all $a \in F_{1} \cup F_{2}$.
\item\label{I_2123_faf}
$\| \psi (a) - f a f \| < \dt$ for all $a \in F_2$.
\item\label{I_2123_Big}
$\| \ps (a) \| > \| a \| - \dt$ for all $a \in F_{1} \cup F_{2}$.
\item\label{I_2123_lessq2}
$1 - f \precsim_{A^{\alpha}} q_2$.
\item\label{I_2123_norm}
$1 - f \precsim_{A^{\alpha}} f$.
\setcounter{TmpEnumi}{\value{enumi}}
\end{enumerate}

By (\ref{I_2123_am}) and~(\ref{Eq_2205_Df}),
we have $\| \ps (p_0)^2 - \ps (p_0) \| < \dt$.
Therefore there is a \pj{} $p \in f A^{\af} f$
such that
\begin{equation}\label{Eq_2123_ppsp0}
\| p - \ps (p_0) \| < \dt_1.
\end{equation}
Using (\ref{I_2123_Big}) and~(\ref{Eq_2205_Df}) at the second step
and (\ref{Eq_2123_dt_l_1}) and (\ref{Eq_2123_dt1}) at the third step,
we have
\[
\| p \|
  > \| \ps (p_0) \| - \dt_1
  > 1 - \dt - \dt_1
  > 0,
\]
so $p \neq 0$.

Define a \ucp{} map $\rh \colon B \to p A^{\af} p$ by
\begin{equation}\label{Eq_2123_rh}
\rh (y) = p (\ps \circ \ph_0) (y) p
\end{equation}
for all $y \in B$.
We will find a unital \hm{} $\ph \colon B \to p A^{\af} p$
which is close to~$\rh$,
and we will show that $p$ and~$\ph$ satisfy the conditions
in Lemma~\ref{L_2123_TR01}.
We begin by estimating $\| p a - a p \|$ for $a \in F$
and $a \in (\ps \circ \ph_0) (T)$.

Let $a \in F$.
Since $\| a \| \leq 1$,
and using (\ref{Eq_2123_ppsp0}), (\ref{Eq_2205_Df}) (twice),
(\ref{I_2123_faf}), (\ref{I_2123_am}), and~(\ref{I_2123_appa_Pf})
at the second step,
and (\ref{Eq_2123_dt_l_1}) at the third step,
\[
\begin{split}
\| p f a f - f a f p \|
& \leq 2 \| f a f \| \| p - \ps (p_0) \|
             + 2 \| \ps (p_0) \| \| f a f - \ps (a) \|
\\
& \hspace*{3em} {\mbox{}}
             + \| \ps (p_0) \ps (a) - \ps (p_0 a) \|
             + \| \ps (a) \ps (p_0) - \ps (a p_0) \|
\\
& \hspace*{3em} {\mbox{}}
             + \| \ps \| \| p_0 a - a p_0 \|
\\
& < 2 \dt_1 + 2 \dt + \dt + \dt + \frac{\ep}{4}
  \leq 6 \dt_1 + \frac{\ep}{4}.
\end{split}
\]
Also, using (\ref{I_2123_comm}), (\ref{Eq_2123_dt_l_1}), and $p f = p$,
\[
\| p f a f - p a \|
 \leq \| p \| \| f \| \| a f - f a \|
 < \dt
 \leq \dt_1.
\]
Similarly $\| f a f p - a  p\| < \dt_1$.
Therefore, using~(\ref{Eq_2123_dt1}),
$\| p a - a p \| < 8 \dt_1 + \frac{\ep}{4} < \ep$.
This is Lemma \ref{L_2123_TR01}(\ref{I_2123_appa_TRA0}).

Let $y \in T$.
Then $\| y \| \leq 1$ and $p_0 \ph_0 (y) = \ph_0 (y) = \ph_0 (y) p_0$.
Therefore, using (\ref{Eq_2123_ppsp0}), (\ref{I_2123_am}),
and~(\ref{Eq_2205_Df}) at the second step,
and using (\ref{Eq_2123_dt_l_1}) at the third step,
it follows that
\[
\begin{split}
& \| p (\ps \circ \ph_0) (y) - (\ps \circ \ph_0) (y) p \|
\\
& \hspace*{3em} {\mbox{}}
 \leq 2 \| (\ps \circ \ph_0) (y) \| \| p - \ps (p_0) \|
  + \| \ps (p_0) \ps (\ph_0 (y)) - \ps (p_0 \ph_0 (y)) \|
\\
& \hspace*{6em} {\mbox{}}
  + \| \ps (\ph_0 (y)) \ps (p_0) - \ps (\ph_0 (y) p_0) \|
\\
& \hspace*{3em} {\mbox{}}
 < 2 \dt_1 + \dt + \dt
  \leq 4 \dt_1.
\end{split}
\]

Using this estimate, we prove approximate multiplicativity
of~$\rh$.
Let $m \in \{ 1, 2, \ldots, N \}$ and let $y_1, y_2, \ldots, y_m \in T$.
For $k = 1, 2, \ldots, m - 1$, we have
\[
\begin{split}
& \bigl\| (\ps \circ \ph_0) (y_k) p (\ps \circ \ph_0) (y_{k + 1}) p
  - (\ps \circ \ph_0) (y_k) (\ps \circ \ph_0) (y_{k + 1}) p \bigr\|
\\
& \hspace*{3em} {\mbox{}}
  \leq \| (\ps \circ \ph_0) (y_k) \|
      \| p (\ps \circ \ph_0) (y_{k + 1})
                    - (\ps \circ \ph_0) (y_{k + 1}) p \|
         \| p \|
  < 4 \dt_1.
\end{split}
\]
Recalling~(\ref{Eq_2123_rh}) and $\| y \| \leq 1$ for all $y \in T$,
an inductive argument shows that
\[
\begin{split}
& \bigl\| \rh (y_1) \rh (y_2) \cdots \rh (y_m)
  - p (\ps \circ \ph_0) (y_1) (\ps \circ \ph_0) (y_2)
        \cdots (\ps \circ \ph_0) (y_m) p \bigr\|
\\
& \hspace*{3em} {\mbox{}}
  < 4 (m - 1) \dt_1
  \leq 4 (N - 1) \dt_1.
\end{split}
\]
By~(\ref{I_2123_am}), (\ref{Eq_2205_Df}), and~(\ref{Eq_2123_dt_l_1}),
we have
\[
\bigl\| (\ps \circ \ph_0) (y_1) (\ps \circ \ph_0) (y_2)
           \cdots (\ps \circ \ph_0) (y_m)
      - \ps \bigl( \ph_0 (y_1) \ph_0 (y_2)
           \cdots \ph_0 (y_m) \bigr) \bigr\|
  < \dt
  \leq \dt_1.
\]
Therefore, also using~(\ref{Eq_2123_dt1}),
\[
\bigl\| \rh (y_1) \rh (y_2) \cdots \rh (y_m)
   - \rh (y_1 y_2 \cdots y_m) \bigr\|
 < 4 (N - 1) \dt_1 + \dt_1
 \leq \dt_0.
\]
We have shown that
$\rh$ is $(N, T, \dt_0)$-approximately multiplicative.
By the choices of $N$ and~$\dt_0$,
there is a unital \hm{} $\ph \colon B \to p A^{\af} p$ such that
\begin{equation}\label{Eq_2123_Phrh}
\| \ph (y) - \rh (y) \| < \frac{\ep}{4}
\end{equation}
for all $y \in F_0$.

We claim that the condition
in Lemma \ref{L_2123_TR01}(\ref{I_2123_papB_TRA0}) holds.
Let $a \in F$.
By the choice of~$F_0$, there is $y \in F_0$ such that
\begin{equation}\label{Eq_2123_ya}
\| \ph_0 (y) - p_0 a p_0 \| < \frac{\ep}{2}.
\end{equation}
Using $p f = f p = p$ at the first step,
using (\ref{Eq_2123_ppsp0}), (\ref{Eq_2205_Df}) (twice),
(\ref{I_2123_faf}), and~(\ref{I_2123_am}) at the third step,
using (\ref{Eq_2123_dt_l_1}) at the fourth step,
and using (\ref{Eq_2123_dt1}) at the fifth step,
\[
\begin{split}
\| p a p - p \ps (p_0 a p_0) p \|
& \leq \| p \| \| p f a f p - \ps (p_0 a p_0) \| \| p \|
\\
& \leq 2 \| p - \ps (p_0) \| \| f a f \|
         + \| \ps (p_0) \|^2 \| f a f - \ps (a) \|
\\
& \hspace*{3em} {\mbox{}}
         + \| \ps (p_0) \ps (a) \ps (p_0) - \ps (p_0 a p_0) \|
\\
& < 2 \dt_1 + \dt + \dt
  \leq 4 \dt_1
  \leq \frac{\ep}{4}.
\end{split}
\]
Also,
using (\ref{Eq_2123_ya}) and~(\ref{Eq_2123_Phrh}) at the third step,
\[
\begin{split}
\| p \ps (p_0 a p_0) p - \ph (y) \|
& = \| p [\ps (p_0 a p_0) - \ph (y)] p \|
\\
& \leq \| p \|^2 \| \ps \| \| p_0 a p_0 - \ph_0 (y) \|
    + \| p (\ps \circ \ph_0) (y) p - \ph (y) \|
\\
& < \frac{\ep}{2} + \frac{\ep}{4}.
\end{split}
\]
Therefore $\| p a p - \ph (y) \| < \ep$.
The claim is proved.

It remains to prove the condition
in Lemma \ref{L_2123_TR01}(\ref{I_2123_1_p_x_TRA0}).
Define $t_0 = q_1 \ps (s) (f - p) \in A^{\af}$.
We claim that
\begin{equation}\label{Eq_2123_t0fmp}
\| t_0^* t_0 - (f - p) \| < 1.
\end{equation}
Using (\ref{Eq_2123_ppsp0}) and $\ps (1) = f$, we get
\begin{equation}\label{Eq_2123_ps1mp0}
\| \ps (1 - p_0) - (f - p) \| < \dt_1.
\end{equation}
Since $\ps (s) \in f A^{\af} f$, we can rewrite the definition as
$t_0 = q_1 f \ps (s) (f - p)$.
Therefore, using (\ref{Eq_2123_ps1mp0}), (\ref{I_2123_faf}),
and~(\ref{Eq_2205_Df})
at the second step, and (\ref{Eq_2123_dt_l_1}) at the third step,
\[
\begin{split}
& \bigl\| t_0^* t_0
 - \ps (1 - p_0) \ps (s)^* \ps (q_1) \ps (s) \ps (1 - p_0) \bigr\|
\\
& \hspace*{3em} {\mbox{}}
  \leq 2 \| (f - p) - \ps (1 - p_0) \|
     + \| f q_1 f - \ps (q_1) \|
  < 2 \dt_1 + \dt
  \leq 3 \dt_1.
\end{split}
\]
Also, using $(1 - p_0) s^* q_1 s (1 - p_0) = 1 - p_0$
(which follows from~(\ref{I_2123_1_p_x_Pf})) at the first step,
using (\ref{I_2123_am}), (\ref{Eq_2205_Df}),
$n \geq 5$, and~(\ref{Eq_2123_ps1mp0}) at the second step,
using (\ref{Eq_2123_dt_l_1}) at the third step,
\[
\begin{split}
& \bigl\| \ps (1 - p_0) \ps (s)^* \ps (q_1) \ps (s) \ps (1 - p_0)
  - (f - p) \bigr\|
\\
& \hspace*{3em} {\mbox{}}
  \leq \bigl\| \ps (1 - p_0) \ps (s)^* \ps (q_1) \ps (s) \ps (1 - p_0)
    - \ps \bigl( (1 - p_0) s^* q_1 s (1 - p_0) \bigr) \bigr\|
\\
& \hspace*{6em} {\mbox{}}
       + \| \ps (1 - p_0) - (f - p) \|
\\
& \hspace*{3em} {\mbox{}}
  < \dt + \dt_1
  \leq 2 \dt_1.
\end{split}
\]
Combining the last two estimates and using~(\ref{Eq_2123_dt1}),
we get
\[
\| t_0^* t_0 - (f - p) \| < 5 \dt_1 < 1,
\]
as claimed.

We have $t_0^* t_0 \in (f - p) A^{\af} (f - p)$.
With functional calculus in $(f - p) A^{\af} (f - p)$,
we can therefore define $t = t_0 ( t_0^* t_0)^{- 1/2}$,
which is a partial isometry with $t^* t = f - p$ and $t t^* \leq q_1$.
Thus $f - p \precsim_{A^{\alpha}} q_1$.
Combining this with~(\ref{I_2123_lessq2})
gives $1 - p \precsim_{A^{\alpha}} q_1 + q_2$.
Since $q_1 + q_2 \in {\overline{x A x}}$,
we have verified the condition
in Lemma \ref{L_2123_TR01}(\ref{I_2123_1_p_x_TRA0}).
By Lemma~\ref{L_2123_TR01}, the proof is complete.
\end{proof}

\begin{cor}\label{C_2123_trZero}
Let $G$ be a second countable compact group,
let $A$ be a simple separable infinite dimensional \uca, and let
$\alpha \colon G \to \Aut (A)$ be an action of $G$ on~$A$
which has the tracial Rokhlin property with comparison.
If $A$ has tracial rank zero, then $A^{\alpha}$
and $C^* (G, A, \af)$ have tracial rank zero.
\end{cor}

\begin{proof}
The algebra $A^{\alpha}$ has tracial rank zero
by Theorem~\ref{T_2123_PrsvTR}.
Corollary~\ref{C_1X15_StableIso}
implies that $C^* (G, A, \af)$ is strongly Morita equivalent
to $A^{\alpha}$,
so Theorem A.24 of~\cite{MFNG} implies that
$C^* (G, A, \af)$ has tracial rank zero.
\end{proof}

We don't get an analog of Corollary~\ref{C_2123_trZero}
for tracial rank at most one,
because the analog of Theorem A.24 of~\cite{MFNG} for
tracial rank one has never been proved.
We expect, however, that it is true.

The following definition is based on
a condition in Theorem~1.2 of~\cite{nb}.
That theorem only considers unital algebras.

\begin{dfn}\label{popa}
A simple separable unital C*-algebra $A$ is
called a {\emph{Popa algebra}} if for every finite subset $F \S A$ and
every $\varepsilon > 0$ there
are a nonzero projection $p \in A$, a finite dimensional \ca~$B$,
and a unital homomorphism $\varphi \colon B \to p A p$, such that
the following hold:
\begin{enumerate}
\item\label{Item_2024_appa_popa}
$\| p a - a p \| < \varepsilon$ for all $a \in  F$.
\item\label{Item_2026_papB_popa}
$\dist (p a p, \, \varphi (B)) < \ep$ for all $a \in F$.
\end{enumerate}
\end{dfn}

\begin{thm}\label{spfix}
Let $A$ be a simple, separable unital \ca{}
which is a Popa algebra, let $G$ be a second countable compact group,
and let $\alpha \colon G \to \Aut (A)$ be an action which has
the tracial Rokhlin property with comparison.
Then $A^{\alpha}$ is Popa algebra.
\end{thm}

\begin{proof}
The proof is essentially the same
as that of the case $k = 0$ of Theorem~\ref{T_2123_PrsvTR}.

The algebra $A^{\alpha}$ is simple by Theorem~\ref{thm_simple12}.
So let $F \subseteq A^{\alpha}$ be a finite subset
and let $\ep > 0$.
\Wolog{} $\ep < \frac{1}{3}$.
Apply Definition~\ref{popa},
getting a nonzero projection $p_0 \in A$, a finite dimensional \ca~$B$,
and a unital homomorphism $\varphi \colon B \to p_0 A p_0$,
such that the following hold.
\begin{enumerate}
\item\label{Z_pfItem_202004_appa_popa}
$\| a p_0 - p_0 a \| < \frac{\ep}{4}$ for all $a \in F$.
\item\label{Z_pfItem_202006_papB_popa}
$\dist (p_0 a p_0, \, \varphi (B)) < \frac{\ep}{4}$ for all $a \in F$.
\setcounter{TmpEnumi}{\value{enumi}}
\end{enumerate}
We may clearly assume that $\ph$ is injective.

Let $F_0$, $T$, $N$, $\dt_0$, $\dt_1$, $\dt$, and~$n$
be as in the proof of Theorem~\ref{T_2123_PrsvTR}.
Apply Theorem~\ref{thm_ApproxHommtr}, getting
a projection $f \in A^{\alpha}$ and a
unital completely positive contractive map
$\psi \colon A \to f A^{\alpha} f$ such that the following hold.
\begin{enumerate}
\setcounter{enumi}{\value{TmpEnumi}}
\item\label{Z_Pf_1X07_18}
Whenever $m \in \{ 1, 2, \ldots, n \}$
and $a_1, a_2, \ldots, a_m \in \ph (F_0) \cup \{ p \} \cup F$, we have
\[
\bigl\| \varphi (a_1 a_2 \cdots a_m)
  - \varphi (a_1) \varphi (a_2) \cdots \varphi (a_m) \bigr\|
   < \dt.
\]
\item\label{Z_Pf_commute1469}
$\| f a - a f \| < \dt$ for all $a \in \{ p_0 \} \cup F \cup \ph (T_0)$.
\item\label{Z_Pf_1X07_21}
$\| \psi (a) - f a f \| < \dt$ for all $a \in F$.
\item\label{Z_Pf_1X07_22}
$\| \ps (a) \| > 1 - \dt$
for all $a \in \{ p_0 \} \cup F \cup \ph (T_0)$.
\item\label{I_2123_Pp_norm}
$f \precsim_{A^{\af}} 1 - f$.
\end{enumerate}
(These are essentially the same
as in the proof of Theorem~\ref{T_2123_PrsvTR}.
We have omitted Condition~(\ref{I_2123_lessq2}) there,
and we are using smaller sets in place of $F_1$ and~$F_2$.)
As in the proof of Theorem~\ref{T_2123_PrsvTR},
there is a \nzp{} $p \in A^{\af}$
such that $\| p - \ps (p_0) \| < \dt_1$.

By the same reasoning as in the proof of Theorem~\ref{T_2123_PrsvTR},
we now get $\| a p - p a \| < \ep$ for all $a \in F$,
and a unital \hm{} $\ph \colon B \to p A^{\af} p$ such that
$\dist (p a p, \, \varphi (B)) < \ep$ for all $a \in F$.
The reasoning in the last two paragraphs of that proof
(proving Lemma \ref{L_2123_TR01}(\ref{I_2123_1_p_x_TRA0}))
is simply omitted.
This completes the proof.
\end{proof}

We now turn to pure infiniteness.
It is convenient to state the following lemma separately.

\begin{lem}\label{L_1Z11_pi}
Let $A$ be a unital \ca.
Then \tfae:
\begin{enumerate}
\item\label{Item_1Z11_pi_pi}
$A$ is purely infinite and simple.
\item\label{Item_1Z11_pi_bsxb}
For every $x \in A_{+} \setminus \{ 0 \}$ there is $b \in A$
such that $b^* x b = 1$.
\item\label{Item_1Z11_pi_bsxbm1}
For every $x \in A_{+} \setminus \{ 0 \}$ there is $b \in A$
such that $\| b^* x b - 1 \| < 1$.
\end{enumerate}
\end{lem}

\begin{proof}
Assume~(\ref{Item_1Z11_pi_pi}); we prove~(\ref{Item_1Z11_pi_bsxb}).
Let $x \in A_{+} \setminus \{ 0 \}$.
Choose $c, d \in A$ such that $c x^{1/2} d = 1$.
Then
\[
1 = d^* x^{1/2} c^* c x^{1/2} d \leq \| c \|^2 d^* x d.
\]
So $d^* x d$ is invertible.
Set $b = d (d^* x d)^{- 1/2}$.
Then $b^* x b = 1$, as desired.

For the converse, assume~(\ref{Item_1Z11_pi_bsxb}),
and let $x \in A \setminus \{ 0 \}$.
We need $a, b \in A$ such that $a x b = 1$.
By hypothesis, there is $b \in A$ such that $b^* x^* x b = 1$.
Take $a = b^* x^*$.

It is trivial that (\ref{Item_1Z11_pi_bsxb})
implies~(\ref{Item_1Z11_pi_bsxbm1}).
For the converse, if $x \in A_{+} \setminus \{ 0 \}$
and $\| c^* x c - 1 \| < 1$,
then $b = c (c^* x c)^{- 1/2}$ satisfies $b^* x b = 1$.
\end{proof}

\begin{thm}\label{Thm_2603_pro33}
Let $A$ be a purely infinite simple separable unital \ca,
let $G$ be a second countable compact group, and let
$\alpha \colon G \to \Aut (A)$ be action of $G$ on $A$
which has the tracial Rokhlin property with comparison.
Then $A^{\alpha}$ and $C^{*} (G, A, \alpha)$ are purely infinite.
\end{thm}

\begin{proof}
We first prove that $A^{\alpha}$ is purely infinite.
We verify the condition in
Lemma \ref{L_1Z11_pi}(\ref{Item_1Z11_pi_bsxbm1}).
Let $x \in A_{+} \setminus \{ 0 \}$.
\Wolog{} $\| x \| = 1$.
Set $\dt_0 = \frac{1}{8}$, and
define \cfn{s} $h, h_0 \colon [0, 1] \to [0, 1]$ by
\[
h (\ld) = \begin{cases}
   0 & \hspace*{1em} 0 \leq \ld \leq 1 - \dt_0
        \\
   \dt_0^{-1} (\ld - 1 + \dt_0)
          & \hspace*{1em} 1 - \dt_0 \leq \ld \leq 1
\end{cases}
\]
and
\[
h_0 (\ld) = \begin{cases}
   (1 - \dt_0)^{-1} \ld & \hspace*{1em} 0 \leq \ld \leq 1 - \dt_0
        \\
   1 & \hspace*{1em} 1 - \dt_0 \leq \ld \leq 1.
\end{cases}
\]
Set $y_0 = h_0 (x)$ and $y = h (x)$.
Then
\begin{equation}\label{Eq_2025_yuy}
\| y \| = 1,
\qquad
\| y_0 - x \| \leq \dt_0,
\andeqn
y_0 y = y.
\end{equation}

Since $A^{\alpha}$ is simple (Theorem~\ref{thm_simple12})
and not of Type~I (Proposition~\ref{C_1Z11_InfDim}),
there are orthogonal nonzero positive elements
$z_1, z_2 \in \ov{y A^{\af} y}$.
Lemma \ref{L_1Z11_pi}(\ref{Item_1Z11_pi_bsxb}) provides $a_1 \in A$
such that $a_1^* z_1 a_1 = 1$.
Set
\[
\dt = \frac{1}{4 ( \| a_1 \|^2 + 1)}.
\]
Use Theorem \ref{thm_ApproxHommtr} with
\[
\varepsilon = \delta,
\quad
x = 1,
\quad
y = z_2,
\quad
n = 3,
\quad
F_{1} = \{ a_{1}, a_1^*, z_1 \},
\quad {\mbox{and}} \quad
F_{2} = \{ z_1, z_2 \},
\]
obtaining a projection $p \in A^{\alpha}$ and a \ucp{} map
$\psi \colon A \to p A^{\alpha} p$ which satisfy:
\begin{enumerate}
\item\label{Item_1Z11_Mlt}
$\| \psi (a_1^* z_1 a_1) - \psi (a_1)^* \psi (z_1) \psi (a_1) \|
  < \dt$.
\item\label{Item_1Z11_pap}
$\| \ps (z_j) - p z_j p \| < \dt$ for $j = 1, 2$.
\item\label{Item_1Z11_sub_z2}
$1 - p \precsim_{A^{\af}} z_2$.
\end{enumerate}

Since $\ps (a_1^* z_1 a_1) = \ps (1) = p$
and $\ps (a_1) \in p A^{\af} p$, we get
\begin{equation}\label{Eq_1Z11_b_ps}
\begin{split}
\| \ps (a_1)^* z_1 \ps (a_1) - p \|
& = \bigl\| \ps (a_1)^* p z_1 p \ps (a_1) - \ps (a_1^* z_1 a_1) \bigr\|
\\
& \leq \| \ps (a_1)^* \| \| p z_1 p - \ps (z_1) \| \| \ps (a_1) \|
\\
& \hspace*{3em} {\mbox{}}
      + \| \psi (a_1)^* \psi (z_1) \psi (a_1) - \psi (a_1^* z_1 a_1) \|
\\
& < \dt ( \| a_1 \|^2 + 1)
  = \frac{1}{4}.
\end{split}
\end{equation}
Also, since $1 - p \precsim_{A^{\af}} z_2$,
we can use Lemma~\ref{L_1X09_CSbPj} to find
$a_2 \in A^{\af}$ such that $a_2^* z_2 a_2 = 1 - p$.

Define $b = z_1^{1/2} \ps (a_1) + z_2^{1/2} a_2$.
We have, getting the last two from~(\ref{Eq_2025_yuy})
and $z_1, z_2 \in \ov{y A^{\af} y}$,
\[
z_1^{1/2} z_2^{1/2} = 0,
\qquad
z_1^{1/2} y_0 = z_1^{1/2},
\andeqn
z_2^{1/2} y_0 = z_2^{1/2}.
\]
It follows that
\[
b^* y_0 b
 = b^* b
 = \ps (a_1)^* z_1 p \ps (a_1) + a_2^* z_2 a_2
 = \ps (a_1)^* z_1 p \ps (a_1) + (1 - p).
\]
We deduce from~(\ref{Eq_1Z11_b_ps}) that
$\| b^* y_0 b - 1 \| = \| b^* b - 1 \| < \frac{1}{4}$.
In particular, $\| b^* b \| < 1 + \frac{1}{4} < 2$,
so $\| b \| < \sqrt{2}$.
Now, using~(\ref{Eq_2025_yuy}) at the second step,
\[
\| b^* x b - 1 \|
  \leq \| b^* \| \| x - y_0 \| \| b \|
    + \| b^* y_0 b - 1 \|
  < 2 \dt_0 + \frac{1}{4}
  \leq \frac{1}{2}.
\]
Thus, $\| b^* x b - 1 \| < 1$, as desired,
and $A^{\af}$ is purely infinite.

Corollary~\ref{C_1X15_StableIso} implies that
$C^{*} (G, A, \alpha)$ and $A^{\af}$ are stably isomorphic.
Therefore $C^{*} (G, A, \alpha)$ is also purely infinite.
\end{proof}

\begin{thm}\label{T_1Z11_inf}
Let $A$ be a simple separable unital \ca,
let $G$ be a second countable compact group,
and let $\alpha \colon G \to \Aut (A)$ be an action which has
the tracial Rokhlin property with comparison.
If $A$ is infinite, so are $A^{\alpha}$ and $C^{*} (G, A, \alpha)$.
\end{thm}

\begin{proof}
We find an infinite \pj{} in~$A^{\alpha}$.
This will show that $A^{\alpha}$ is infinite.
Since $A^{\alpha}$
is isomorphic to a subalgebra of $C^{*} (G, A, \alpha)$,
by Theorem~\ref{simplecrosscomparison} and Theorem~\ref{satunonabel},
it will also follow that $C^{*} (G, A, \alpha)$ is infinite

Since $A$ is infinite and unital,
there is $s \in A$ such that $s^* s = 1$ and $s s^* \neq 1$.
Then $\| 1 - s s^* \| = 1$.
Choose $\dt_0 > 0$ so small that whenever $D$ is a \ca{}
and $c \in D$ satisfies $\| c^* c - 1 \| < \dt_0$,
then there is $w \in D$
such that $w^* w = 1$ and $\| w - c \| < \frac{1}{8}$.
Set $\dt = \min \bigl( \frac{1}{8}, \dt_0 \bigr)$.

Use Theorem \ref{thm_ApproxHommtr} with
\[
\varepsilon = \delta,
\quad
x = 1,
\quad
y = 1,
\quad
n = 2,
\quad
F_{1} = \{ s, s^*, s s^*, 1 - s s^* \},
\quad {\mbox{and}} \quad
F_{2} = \E,
\]
obtaining a projection $p \in A^{\alpha}$ and a \ucp{} map
$\psi \colon A \to p A^{\alpha} p$ which satisfy:
\begin{enumerate}
\item\label{Item_1Z11_inf_Mlt}
$\| \psi (a b) - \psi (a) \psi (b) \| < \delta$ for all $a, b \in F$.
\item\label{Item_1Z11_inf_BigN}
$\| \ps (1 - s s^*) \| > 1 - \delta$.
\end{enumerate}
Then $\ps (s) \in p A^{\alpha} p$ and
$\| \psi (s)^* \psi (s) - p \| < \delta \leq \dt_0$.
By the choice of $\dt_0$, there is $v \in p A^{\alpha} p$
such that $v^* v = p$ and $\| v - \psi (s) \| < \frac{1}{8}$.
We have
\[
\begin{split}
\bigl\| (p - v v^*) - (p - \psi (s) \psi (s)^*) \bigr\|
& = \| v v^* - \psi (s) \psi (s)^* \|
\\
& \leq \| v - \psi (s) \| \| v^* \|
     + \| \ps (s) \| \| (v - \psi (s))^* \|
\\
& < \frac{1}{8} + \frac{1}{8}
  = \frac{1}{4}.
\end{split}
\]
Also
\[
\| p - \psi (s) \psi (s)^* \|
  \geq \| \ps (1 - s s^*) \| - \| \psi (s s^*) - \psi (s) \psi (s)^* \|
  > (1 - \dt) - \dt \geq \frac{3}{4}.
\]
Therefore $\| p - v v^* \| > \frac{3}{4} - \frac{1}{4} > 0$.
So $v v^* \neq p$.
Since $v v^* \leq p$, we have shown that $p$ is an infinite \pj.
\end{proof}

It is obvious that if $A$ is (stably) finite, then $A^{\alpha}$
is stably finite, no matter what the action is.
Therefore
Theorem~\ref{Thm_2603_pro33} and Theorem~\ref{T_1Z11_inf}
come close to saying that if $\af$ has the \trpc,
then $A^{\af}$ is one of purely infinite,
infinite but not purely infinite, stably finite,
or finite but not stably finite, \ifo{} $A$ has the same type.
We don't quite get everything: we don't rule out the combinations
$A$ infinite but not purely infinite while $A^{\af}$ is purely infinite,
and $A$ finite but not stably finite while $A^{\af}$ is stably finite.
There is a bigger issue with $C^{*} (G, A, \alpha)$,
since finiteness is not preserved by stable isomorphism,
even for simple \ca{s}~\cite{Rdm8}.
Some hypothesis on the action is needed,
since the gauge action of $S^1$ on the Cuntz algebra ${\mathcal{O}}_2$
has fixed point algebra a UHF algebra.

We next consider $\cZ$-stability and tracial $\cZ$-stability.
We need some preliminaries, which correspond roughly
to Lemma~2.7 and part of the proof of Lemma~2.8 of~\cite{ArBcPh1}.
We first recall the definition of tracial $\cZ$-stability.

\begin{dfn}[Definition 2.1 of~\cite{HrsOv}]\label{TracialZst_HirOr} 
Let $A$ be a unital \ca.
We say that $A$ is {\emph{tracially ${\mathcal{Z}}$-absorbing}}
(or {\emph{tracially ${\mathcal{Z}}$-stable}})
if $A \not\cong {\mathbb{C}}$
and for any $\ep > 0$, any finite set $F \subseteq A$,
any $n \in {\mathbb{N}}$, and any $x \in A_{+} \setminus \{ 0 \}$,
there is a completely positive contractive order zero map
$\varphi \colon M_{n} \to A$ such that the following hold.
\begin{enumerate}
\item\label{Item_TracialZta_320_CTRwC}
$1 - \varphi (1) \precsim_{A} x$.
\item\label{Item_TracialZta_322_CTWwC}
For any $y \in M_{n}$ with $\| y \| = 1$ and any $a \in F$,
we have $\| \varphi (y) a - a \varphi (y) \| < \ep$.
\end{enumerate}
\end{dfn}

The following stable relations statement for \cpc{} order zero maps
is known but seems never to have been formally stated.
We recall from the condition in Corollary 2.6(3)
of~\cite{ArBcPh1} that if
$\rh \colon M_n \to D$ is a \cpc{} order zero map,
then its values $\rh (e_{j, k})$ on the standard matrix units in $M_n$
satisfy the exact version of the relations in this proposition.

\begin{prp}\label{P_1Z12_OrdZ_SmPj}
Let $n \in \N$ and let $\ep > 0$.
Then there is $\dt > 0$ with the following property.
Let $D$ be a \ca{} and let $y_{j, k} \in D$,
for $j, k = 1, 2, \ldots, n$, satisfy (in the unitization of $D$)
\begin{enumerate}
\item\label{L_1Z12_OrdZ_SmPj_CMid}
$\| y_{j, k} y_{k, m} - y_{j, l} y_{l, m} \| < \dt$
for $j, k, l, m = 1, 2, \ldots, n$.
\item\label{L_1Z12_OrdZ_SmPj_Orth}
$\| y_{j, j} y_{k, k} \| < \dt$
for $j, k = 1, 2, \ldots, n$
with $j \neq k$.
\item\label{L_1Z12_OrdZ_SmPj_Adj}
$\| y_{j, k} - y_{k, j}^* \| < \dt$ for $j, k = 1, 2, \ldots, n$.
\item\label{L_1Z12_OrdZ_SmPj_Norm}
$\| y_{j, j} \| < 1 + \dt$ for $j = 1, 2, \ldots, n$.
\item\label{L_1Z12_OrdZ_SmPj_Pos}
$\| 1 - y_{j, j} \| < 1 + \dt$ for $j = 1, 2, \ldots, n$.
\end{enumerate}
Then there is a \cpc{} order zero map $\rh \colon M_n \to D$
such that, with $(e_{j, k})_{j, k = 1, 2, \ldots, n}$
being the standard system of matrix units for~$M_n$,
we have $\| \rh (e_{j, k}) - y_{j, k} \| < \ep$
for $j, k = 1, 2, \ldots, n$.
\end{prp}

\begin{proof}
The case $n = 1$ is easy.
For $n \geq 2$,
this result is essentially contained in the proof of the case $n \geq 2$
of Lemma~2.7 of~\cite{ArBcPh1}.
Specifically, the universal \uca{} on the exact version of these
relations (stated in Lemma~2.5 of~\cite{ArBcPh1}) is, by that lemma,
isomorphic to the unitized cone $(C M_n)^{+}$, with
$y_{j, k}$ corresponding to the function
$f_{j, k} (\ld) = \ld e_{j, k}$ for $\ld \in [0, 1]$.
Therefore, as explained in the proof of Lemma~2.7 of~\cite{ArBcPh1},
the exact versions of the relations are stable.
Thus, as there, there is $\dt > 0$ such that whenever $D$ is a \uca{}
and elements $y_{j, k} \in D$, for $j, k = 1, 2, \ldots, n$,
satisfy (\ref{L_1Z12_OrdZ_SmPj_CMid}), (\ref{L_1Z12_OrdZ_SmPj_Orth}),
(\ref{L_1Z12_OrdZ_SmPj_Adj}),
(\ref{L_1Z12_OrdZ_SmPj_Norm}), and~(\ref{L_1Z12_OrdZ_SmPj_Pos}),
then there is a \hm{} $\ph \colon (C M_n)^{+} \to D^{+}$
such that $\| \ph (f_{j, k}) - y_{j, k} \| < \ep$
for $j, k = 1, 2, \ldots, n$.
The required \cpc{} order zero map is given by, for $z \in M_n$,
taking $\rh (z)$ to be the value of $\ph$ on the function
$\ld \mapsto \ld z$ in $C M_n$.
\end{proof}

\begin{lem}\label{L_1Z22_OrdZ}
Let $D$ and $E$ be \uca{s},
and let $\ps_0 \colon D \to E$ be a \cpc{} order zero map.
Let $\ep > 0$.
Then there is a \cpc{} order zero map $\ps \colon D \to E$
such that $\| \ps - \ps_0 \| \leq \ep$ and
$(1 - \ep) (1 - \ps (1)) = (1 - \ps_0 (1) - \ep)_{+}$.
\end{lem}

\begin{proof}
We will use functional calculus for \cpc{} order zero maps,
as in Corollary~4.2 of~\cite{WntZch09}.
Define $f \colon [0, 1] \to [0, 1]$ by
\[
f (\ld)
 = \begin{cases}
   (1 - \ep)^{- 1} \ld & \hspace*{1em} 0 \leq \ld \leq 1 - \ep
        \\
   1 & \hspace*{1em} 1 - \ep < \ld \leq 1.
\end{cases}
\]
Set $\ps = f (\ps_0)$ as defined in Corollary~4.2 of~\cite{WntZch09}.
Following the notation there, we have
\[
\ps_0 (d) = h \pi_{\ps_0} (d)
\andeqn
\ps (d) = f (h) \pi_{\ps_0} (d)
\]
for all $d \in D$.
The estimate $\| \ps - \ps_0 \| \leq \ep$
follows from the easily verified estimate $\| f (h) - h\| \leq \ep$,
and the relation $(1 - \ep) (1 - \ps (1)) = (1 - \ps_0 (1) - \ep)_{+}$
follows from the easily verified relation
$(1 - \ep) (1 - f (\ld)) = (1 - \ld - \ep)_{+}$
for all $\ld \in [0, 1]$.
\end{proof}

\begin{thm}\label{T_1Z22_TrZStab}
Let $A$ be a simple separable infinite dimensional \uca,
let $G$ be a second countable compact group, and let
$\alpha \colon G \to \Aut (A)$ be an action 
which has the tracial Rokhlin property with comparison. 
If $A$ is tracially ${\mathcal{Z}}$-stable,
then $A^{\af}$ is tracially ${\mathcal{Z}}$-stable.
\end{thm}

\begin{proof}
We verify the condition in Definition~\ref{TracialZst_HirOr}.
So let $\ep > 0$, let $F \subseteq A^{\af}$ be a finite set, 
let $x \in (A^{\af})_{+} \setminus \{ 0 \}$,
and let $n \in {\mathbb{N}}$.
Without loss of generality, we may assume that
$\| a \| \leq 1$ for all $a \in F$ and $\ep < 1$. 
Set
\begin{equation}\label{Eq_1Z23_e0}
\ep_0 = \frac{\ep}{16 n^2}.
\end{equation}
Apply Proposition~\ref{P_1Z12_OrdZ_SmPj} with $\ep_0$ in place of~$\ep$,
getting $\dt > 0$.

The algebra $A^{\af}$ is simple by Theorem~\ref{thm_simple12}
and not of Type~I by Proposition~\ref{C_1Z11_InfDim}. 
So, by Lemma 2.4 of~\cite{philar},
there exist $x_{1}, x_{2} \in (A^{\af})_{+} \setminus \{0\}$ such that
\begin{equation}\label{Eq_361_x1_x2_herAfix}
x_1 x_2 = 0,
\qquad
x_{1} \sim x_{2},
\andeqn
x_{1} + x_{2}  \in {\overline{x A^{\af} x}}.
\end{equation}

Apply Definition~\ref{TracialZst_HirOr}
with $n$ and $F$ as given, with $\ep_0$ in place of~$\ep$,
and with $x_1$ in place of~$a$,
getting $\ph_0 \colon M_n \to A$.
In particular,
\begin{equation}\label{Eq_2115_Star}
\| \ph_0 (z) a - a \ph_0 (z) \| < \ep_0
\end{equation}
for all $a \in F$ and all $z \in M_n$ with $\| z \| = 1$.
Also, $1 - \ph_0 (1) \precsim_A x_1$,
so there is $v \in A$ such that
\begin{equation}\label{Eq_1Z22_CS}
\| v^* x_1 v - (1 - \ph_0 (1)) \| < \frac{\ep}{32}.
\end{equation}
Set
\begin{equation}\label{Eq_1Z23_ep1}
\dt_0 = \min \left( \dt, \, \frac{\ep}{32 ( \| v \|^2 + 1)} \right).
\end{equation}

Let $(e_{j, k})_{j, k = 1, 2, \ldots, n}$
be the standard system of matrix units for~$M_n$.
Set
\begin{equation}\label{Eq_2205_F0}
F_0 = F \cup \bigl\{ \ph_0 (e_{j, k}) \colon
            j, k \in \{ 1, 2, \ldots, n \} \bigr\}
         \cup \{ x_1, v, v^* \}.
\end{equation}
Apply Theorem~\ref{thm_ApproxHommtr}, getting
a projection $p \in A^{\alpha}$ and a
unital completely positive contractive map
$\psi \colon A \to p A^{\alpha} p$ such that the following hold.
\begin{enumerate}
\item\label{I_1Z22_AppM}
$\ps$ is a $(3, F_0, \dt_0)$-approximately
multiplicative map (Definition~\ref{D_1X10_nSFe}).
\item\label{I_1Z22_Commp}
$\| p a - a p \| < \dt_0$ for all $a \in F_0$.
\item\label{I_1Z22_pap}
$\| \psi (a) - p a p \| < \dt_0$ for all $a \in F \cup \{ x_1 \}$.
\item\label{I_1Z22_1mpx2}
$1 - p \precsim_{A^{\alpha}} x_2$.
\end{enumerate}

For $j, k = 1, 2, \ldots, n$,
define $y_{j, k} = (\ps \circ \ph_0) (e_{j, k})$.
Using the condition in Corollary 2.6(3) of~\cite{ArBcPh1},
for $j, k, l, m = 1, 2, \ldots, n$ we have
$\ph_0 (e_{j, k}) \ph_0 (e_{k, m}) = \ph_0 (e_{j, l}) \ph_0 (e_{l, m})$.
It therefore follows from~(\ref{I_1Z22_AppM})
that the elements $y_{j, k}$
satisfy the relations (\ref{L_1Z12_OrdZ_SmPj_CMid})
and~(\ref{L_1Z12_OrdZ_SmPj_Orth}) in Proposition~\ref{P_1Z12_OrdZ_SmPj}.
It is immediate that the following stronger versions
of (\ref{L_1Z12_OrdZ_SmPj_Adj}), (\ref{L_1Z12_OrdZ_SmPj_Norm}),
and~(\ref{L_1Z12_OrdZ_SmPj_Pos}) in Proposition~\ref{P_1Z12_OrdZ_SmPj}
hold:
$y_{j, k} = y_{k, j}^*$ for $j, k = 1, 2, \ldots, n$,
and $\| y_{j, j} \| \leq 1$ and
$\| 1 - y_{j, j} \| \leq 1$ for $j = 1, 2, \ldots, n$.
By the choice of $\dt$ there is a \cpc{} order zero map
$\ph_1 \colon M_n \to p A^{\af} p$
such that
\[
\| \ph_1 (e_{j, k}) - y_{j, k} \| < \ep_0
\]
for $j, k = 1, 2, \ldots, n$.
It follows that
\begin{equation}\label{Eq_1Z22_OZest}
\| \ph_1 (z) - (\ps \circ \ph_0) (z) \| \leq n^2 \ep_0 \| z \|
\end{equation}
for all $z \in M_n$.

We have, at the third step using (\ref{Eq_2205_F0}) (twice),
(\ref{I_1Z22_pap}),
(\ref{I_1Z22_AppM}), (\ref{Eq_1Z22_CS}), and~(\ref{Eq_1Z22_OZest}),
and at the fourth step using
(\ref{Eq_1Z23_ep1}) and~(\ref{Eq_1Z23_e0}),
\[
\begin{split}
& \bigl\| \ps (v)^* x_1 \ps (v) - [p - \ph_1 (1) ] \bigr\|
\\
& \hspace*{3em} {\mbox{}}
 = \bigl\| \ps (v^*) p x_1 p \ps (v) - [p - \ph_1 (1) ] \bigr\|
\\
& \hspace*{3em} {\mbox{}}
 \leq 2 \| \ps (v) \|^2 \| p x_1 p - \ps (x_1) \|
   + \| \ps (v^*) \ps (x_1) \ps (v) - \ps (v^* x_1 v) \|
\\
& \hspace*{6em} {\mbox{}}
   + \| \ps \| \| v^* x_1 v - (1 - \ph_0 (1)) \|
   + \| (\ps \circ \ph_0) (1) - \ph_1 (1) \|
\\
& \hspace*{3em} {\mbox{}}
 < 2 \| \ps (v) \|^2 \dt_0 + \dt_0 + \frac{\ep}{32} + n^2 \ep_0
  \leq \frac{\ep}{8}.
\end{split}
\]
Therefore
\begin{equation}\label{Eq_1Z22_NewSb}
\left( p - \ph_1 (1) - \frac{\ep}{8} \right)_{+}
 \precsim_{A^{\af}} \ps (v)^* x_1 \ps (v)
 \precsim_{A^{\af}} x_1.
\end{equation}
Apply Lemma~\ref{L_1Z22_OrdZ} with $\ph_1$ in place of $\ps_0$
and $\frac{\ep}{8}$ in place of $\ep$,
getting a \cpc{} order zero map $\ph \colon M_n \to p A^{\af} p$
such that
\begin{equation}\label{Eq_1Z23_Eqn}
\| \ph - \ph_1 \| \leq \frac{\ep}{8}
\andeqn
\left( 1 - \frac{\ep}{8} \right) [p - \ph (1)]
 = \left( p - \ph_1 (1) - \frac{\ep}{8} \right)_{+}.
\end{equation}
Combining $\| \ph - \ph_1 \| \leq \frac{\ep}{8}$
with~(\ref{Eq_1Z22_OZest}), and then using~(\ref{Eq_1Z23_e0}),
for all $z \in M_n$ we get
\begin{equation}\label{Eq_1Z22_Nbd}
\| \ph (z) - (\ps \circ \ph_0) (z) \|
 \leq \left( \frac{\ep}{8} + n^2 \ep_0 \right) \| z \|
 \leq \left( \frac{\ep}{4} \right) \| z \|.
\end{equation}

Let $a \in F$ and let $z \in M_n$ satisfy $\| z \| = 1$.
Then, at the second step using (\ref{Eq_2205_F0}) (twice),
(\ref{I_1Z22_pap}), (\ref{Eq_1Z22_Nbd}), and~(\ref{I_1Z22_AppM}),
\[
\begin{split}
& \| p a p \ph (z) - \ps (a \ph_0 (z)) \|
\\
& \hspace*{3em} {\mbox{}}
  \leq \| p a p - \ps (a) \| \| \ph (z) \|
     + \| \ps (a) \| \| \ph (z) - (\ps \circ \ph_0) (z) \|
\\
& \hspace*{6em} {\mbox{}}
     + \| \ps (a) \ps ( \ph_0 (z)) - \ps (a \ph_0 (z)) \|
\\
& \hspace*{3em} {\mbox{}}
  < \dt_0 + \frac{\ep}{4} + \dt_0
  = 2 \dt_0 + \frac{\ep}{4}.
\end{split}
\]
Also, since $p \ph (z) = \ph (z)$ and $p^2 = p$,
and by~(\ref{I_1Z22_Commp}) and~(\ref{Eq_2205_F0}), we have
\[
\| a \ph (z) - p a p \ph (z) \|
 = \| a p^2 \ph (z) - p a p \ph (z) \|
 \leq \| a p - p a \| \| \ps (z) \|
 < \dt_0.
\]
So $\| a \ph (z) - \ps (a \ph_0 (z)) \| < 3 \dt_0 + \frac{\ep}{4}$.
Similarly
$\| \ph (z) a - \ps (\ph_0 (z) a) \| < 3 \dt_0 + \frac{\ep}{4}$.
Combining these two estimates with~(\ref{Eq_2115_Star}),
and using (\ref{Eq_1Z23_e0}) and~(\ref{Eq_1Z23_ep1}), gives
\[
\| a \ph (z) - \ph (z) a \|
 < 2 \left( 3 \dt_0 + \frac{\ep}{4} \right)
      + \| \ps \| \| a \ph_0 (z) - \ph_0 (z) a \|
 < 6 \dt_0 + \frac{\ep}{2} + \ep_0
 < \ep.
\]
This is Condition~(\ref{Item_TracialZta_322_CTWwC})
in Definition~\ref{TracialZst_HirOr}.

For the other condition,
we combine (\ref{Eq_1Z23_Eqn}), (\ref{Eq_1Z22_NewSb}), and $\ep < 1$
to get $p - \ph (1) \precsim_{A^{\af}} x_1$.
Therefore,
also using (\ref{I_1Z22_1mpx2}) and~(\ref{Eq_361_x1_x2_herAfix}),
\[
1 - \ph (1)
 = [1 - p] + [p - \ph (1)]
 \precsim_{A^{\af}} x_2 + x_1
 \precsim_{A^{\af}} x,
\]
as desired.
\end{proof}

\begin{thm}\label{T_1Z22_ZStb}
Let $A$ be a simple separable infinite dimensional nuclear \uca,
let $G$ be a second countable compact group, and let
$\alpha \colon G \to \Aut (A)$ be an action 
which has the tracial Rokhlin property with comparison. 
If $A$ is ${\mathcal{Z}}$-stable,
then $A^{\af}$ and $C^* (G, A, \af)$ are $\mathcal{Z}$-stable.
\end{thm}

\begin{proof}
The algebra $A$ is tracially ${\mathcal{Z}}$-stable
by Proposition~2.2 of~\cite{HrsOv}.
Therefore $A^{\af}$ is tracially ${\mathcal{Z}}$-stable
by Theorem~\ref{T_1Z22_TrZStab}.
Since $A^{\af}$ is nuclear, Theorem~4.1 of~\cite{HrsOv}
now implies that $A$ is ${\mathcal{Z}}$-stable.
Since ${\mathcal{Z}}$-stability is preserved by stable isomorphism
(Corollary~3.2 of~\cite{TmWnt1}), Corollary~\ref{C_1X15_StableIso}
implies that $C^* (G, A, \af)$ is $\mathcal{Z}$-stable.
\end{proof}

Finally, we consider the radius of comparison.
Since it is defined only for unital \ca{s},
we treat only the fixed point algebra.
We start with some preliminaries.

\begin{prp}\label{P_1X25_Mn}
Let $G$ be a second countable compact group,
let $A$ be a stably finite simple separable infinite dimensional \uca,
and let $\alpha \colon G \to \Aut (A)$ be an action of $G$ on~$A$
which has the tracial Rokhlin property with comparison.
Let $n \in \N$.
Then the action $g \mapsto \id_{M_n} \otimes \af_g$
of $G$ on $M_n \otimes A$
has the tracial Rokhlin property with comparison.
\end{prp}

\begin{proof}
We verify the conditions of Lemma~\ref{L_1X09_NoNorm1}.

Let $F \subseteq M_n \otimes A$ and $S \subseteq C (G)$ be finite,
let $\varepsilon > 0$,
let $x \in (M_n \otimes A)_{+} \setminus \{ 0 \}$,
and let $y \in (M_n \otimes A^{\alpha})_{+} \setminus \{ 0 \}$.
Let $(e_{j, k})_{j, k = 1, 2, \ldots, n}$
be the standard system of matrix units for~$M_n$.
Choose a finite set $F_0 \S A$ such that
\[
F \S \Biggl\{ \sum_{j, k = 1}^n e_{j, k} \otimes a_{j, k} \colon
 {\mbox{$a_{j, k} \in F_0$ for $j, k = 1, 2, \ldots, n \}$ }} \Biggr\}.
\]
Set $\ep_0 = n^{- 2} \ep$.
Write $x = \sum_{j, k = 1}^n e_{j, k} \otimes x_{j, k}$
and $y = \sum_{j, k = 1}^n e_{j, k} \otimes y_{j, k}$
with $x_{j, k} \in A$ and $y_{j, k} \in A^{\alpha}$
for $j, k = 1, 2, \ldots, n$.
By Lemma~3.8 of~\cite{phill23},
there are $l, m \in \{ 1, 2, \ldots, n \}$
such that $x_{l, l} \neq 0$ and $y_{m, m} \neq 0$.
It is immediate that
\begin{equation}\label{Eq_1Z22_MtSb}
e_{l, l} \otimes x_{l, l} \precsim_{M_n \otimes A} x
\andeqn
e_{m, m} \otimes y_{m, m} \precsim_{M_n \otimes A^{\alpha}} y.
\end{equation}
Since $A$ is simple and not of Type~I,
and since $A^{\af}$ is simple by Theorem~\ref{thm_simple12}
and not of Type~I by Proposition~\ref{C_1Z11_InfDim},
Lemma~2.4 of~\cite{philar}
implies the existence of $x_0 \in A_{+} \setminus \{ 0 \}$
and $y_0 \in (A^{\alpha})_{+} \setminus \{ 0 \}$ such that
\begin{equation}\label{Eq_1Z22_1Sub}
1_{M_n} \otimes x_0 \precsim_{M_n \otimes A} e_{l, l} \otimes x_{l, l}
\andeqn
1_{M_n} \otimes y_0
 \precsim_{M_n \otimes A^{\alpha}} e_{m, m} \otimes y_{m, m}.
\end{equation}
Apply Lemma~\ref{L_1X09_NoNorm1} using $F_0$, $S$, $\ep_0$, $x_0$,
and $y_0$, getting a projection $p_0 \in A^{\alpha}$ and an equivariant
unital completely positive map $\varphi_0 \colon C (G) \to p_0 A p_0$
such that the following hold.
\begin{enumerate}
\item\label{L_1Z22_ecmm}
$\varphi_0$ is an $(F_0, S, \varepsilon_0)$-approximately
central multiplicative map.
\item\label{L_1Z22_1mpx}
$1 - p_0 \precsim_{A} x_0$.
\item\label{L_1Z22_1py}
$1 - p_0 \precsim_{A^{\alpha}} y_0$.
\item\label{L_1Z22_1mppp}
$1 - p_0 \precsim_{A^{\alpha}} p_0$.
\end{enumerate}
Set $p = 1_{M_n} \otimes p_0$
and define $\ph \colon C (G) \to p (M_n \otimes A) p$
by $\ph (f) = 1_{M_n} \otimes \ph_0 (f)$ for all $f \in C (G)$.
Clearly $\ph$ is equivariant and \ucp.
We verify Conditions (\ref{L_1X09_NoNorm1_ecmm}),
(\ref{L_1X09_NoNorm1_1mpx}), (\ref{L_1X09_NoNorm1_1py}),
and~(\ref{L_1X09_NoNorm1_1mppp}) of Lemma~\ref{L_1X09_NoNorm1}.

For Lemma \ref{L_1X09_NoNorm1}(\ref{L_1X09_NoNorm1_1mpx}),
using (\ref{L_1Z22_1mpx}) above at the second step,
(\ref{Eq_1Z22_1Sub}) at the third step,
and (\ref{Eq_1Z22_MtSb}) at the fourth step, we get
\[
1_{M_n \otimes A} - p
 = 1_{M_n} \otimes (1_A - p_0)
 \precsim_{M_n \otimes A} 1_{M_n} \otimes x_0
 \precsim_{M_n \otimes A} e_{l, l} \otimes x_{l, l}
 \precsim_{M_n \otimes A} x.
\]
Using (\ref{L_1Z22_1py}) above in place of~(\ref{L_1Z22_1mpx}) above,
we also get
\[
1_{M_n \otimes A} - p
 = 1_{M_n} \otimes (1_A - p_0)
 \precsim_{M_n \otimes A^{\alpha}} 1_{M_n} \otimes y_0
 \precsim_{M_n \otimes A^{\alpha}} e_{m, m} \otimes y_{l, l}
 \precsim_{M_n \otimes A^{\alpha}} y,
\]
which is Lemma \ref{L_1X09_NoNorm1}(\ref{L_1X09_NoNorm1_1py}).
Condition~(\ref{L_1X09_NoNorm1_1mppp}) of Lemma~\ref{L_1X09_NoNorm1}
follows immediately
from~(\ref{L_1Z22_1mppp}) above by tensoring with $1_{M_n}$.

It remains to check
Lemma \ref{L_1X09_NoNorm1}(\ref{L_1X09_NoNorm1_ecmm}).
If $f_1, f_2 \in S$,
then
\[
\begin{split}
\| \ph (f_1 f_2) - \ph (f_1) \ph (f_2) \|
& = \bigl\| 1_{M_n} \otimes
     [\ph_0 (f_1 f_2) - \ph_0 (f_1) \ph_0 (f_2)] \bigr\|
\\
& = \| \ph_0 (f_1 f_2) - \ph_0 (f_1) \ph_0 (f_2) \|
  < \ep_0
  \leq \ep.
\end{split}
\]
For $f \in S$ and $a \in F$,
write $a = \sum_{j, k = 1}^n e_{j, k} \otimes a_{j, k}$
with $a_{j, k} \in F_0$ for $j, k = 1, 2, \ldots, n$.
Then
\[
\| \ph (f) a - a \ph (f) \|
  = \biggl\| \sum_{j, k = 1}^n e_{j, k}
     \otimes [\ph_0 (f) a_{j, k} - a_{j, k} \ph_0 (f)] \biggr\|
  < n^2 \ep_0
  = \ep.
\]
This completes the proof.
\end{proof}

\begin{prp}\label{P_1X25_Comp_Fx}
Let $G$ be a second countable compact group,
let $A$ be a stably finite simple separable infinite dimensional \uca,
and let $\alpha \colon G \to \Aut (A)$ be an action of $G$ on~$A$
which has the tracial Rokhlin property with comparison.
Let $n \in \N$, and let $a, b \in M_n (A)_{+}$.
Suppose that $0$ is a limit point of $\spec (b)$.
Then $a \precsim_{A} b$ \ifo{} $a \precsim_{A^{\af}} b$.
\end{prp}

\begin{proof}
We need only prove that
$a \precsim_{A} b$ implies $a \precsim_{A^{\af}} b$.
By Proposition~\ref{P_1X25_Mn}, we may assume $n = 1$.
Also, \wolog{} $\| a \| \leq 1$ and $\| b \| \leq 1$.

By Lemma 1.4(11) of~\cite{philar} (not the original source),
it is enough to let $\ep > 0$ be arbitrary,
and prove that $(a - \ep)_{+} \precsim_{A^{\af}} b$.
Again by Lemma 1.4(11) of~\cite{philar}, there is $\dt > 0$ such that
$\bigl( a - \frac{\ep}{5} \bigr)_{+} \precsim_{A} (b - \dt)_{+}$.
Choose $v \in A$ such that
\begin{equation}\label{Eq_1X28_4St}
\Bigl\| v^* (b - \dt)_{+} v
     - \Bigl( a - \frac{\ep}{5} \Bigr)_{+} \Bigr\|
  < \frac{\ep}{5}.
\end{equation}
By Lemma VI.11 of~\cite{Arcy},
there is $\rh_0 > 0$ such that whenever $D$ is a unital \ca{}
and $c, d \in D_{+}$ satisfy
\[
\| c \| \leq 1,
\qquad
\| d \| \leq 1,
\andeqn
\| c - d \| < \rh_0,
\]
then
\begin{equation}\label{Eq_1X25_Ests}
\Bigl\| \Bigl( c - \frac{\ep}{5} \Bigr)_{+}
     - \Bigl( d - \frac{\ep}{5} \Bigr)_{+} \Bigr\|
  < \frac{\ep}{5}.
\end{equation}
Define
\begin{equation}\label{Eq_1X28_St}
\rh = \min \left( \frac{\rh_0}{2}, \,
        \frac{\ep}{5 (2 + \| v \|^2 )} \right).
\end{equation}
Define $f \colon [0, 1] \to [0, 1]$ by
\[
f (\ld)
 = \begin{cases}
   \ld (\dt - \ld) & \hspace*{1em} 0 \leq \ld \leq \dt
       \\
   0               & \hspace*{1em} \dt < \ld \leq 1.
\end{cases}
\]
Then $f (b) \neq 0$.
Define $F_1 \S A$ and $F_2 \S A^{\af}$ by
\[
F_1 = \Bigl\{ a, \, \Bigl( a - \frac{\ep}{5} \Bigr)_{+}, \,
  (b - \dt)_{+}, \, v, \, v^* \Bigr\}
\andeqn
F_2 = \Bigl\{ \Bigl( a - \frac{\ep}{5} \Bigr)_{+}, \,
  (b - \dt)_{+} \Bigr\}.
\]

Apply Theorem~\ref{thm_ApproxHommtr}, getting
a projection $p \in A^{\alpha}$ and a
unital completely positive contractive map
$\psi \colon A \to p A^{\alpha} p$ such that the following hold.
\begin{enumerate}
\setcounter{enumi}{\value{TmpEnumi}}
\item\label{Pf_1X25_18}
$\| \psi (a b c) - \psi (a) \psi (b) \psi (c) \| < \rh$
for all $a, b, c \in F_1$.
\item\label{Pf_commute1469}
$\| p a - a p \| < \rh$ for all $a \in F_1$.
\item\label{Pf_1X25_21}
$\| \psi (a) - p a p \| < \rh$
for all $a \in F_2$.
\item\label{Pf_1X25_20}
$1 - p \precsim_{A^{\alpha}} f (b)$.
\end{enumerate}
Since
\begin{equation}\label{Eq_2115_StSt}
\bigl\| a - \bigl[ p a p + (1 - p) a (1 - p) \bigr] \bigr\| < 2 \rh
\end{equation}
by~(\ref{Pf_commute1469}),
it follows from~(\ref{Eq_1X28_St}), (\ref{Eq_1X25_Ests}),
and the choice of~$\rh_0$ that
\[
\Bigl\| \Bigl( a - \frac{\ep}{5} \Bigr)_{+}
   - \Bigl( \bigl[ p a p + (1 - p) a (1 - p) \bigr]
       - \frac{\ep}{5} \Bigr)_{+} \Bigr\|
 < \frac{\ep}{5},
\]
whence
\begin{equation}\label{Eq_1X30_CutEst}
\Bigl\| p \Bigl( a - \frac{\ep}{5} \Bigr)_{+} p
   - \Bigl( p a p - \frac{\ep}{5} \Bigr)_{+} \Bigr\|
 <  \frac{\ep}{5}.
\end{equation}

Using, in order, (\ref{Pf_1X25_21}), (\ref{Pf_1X25_18}),
(\ref{Eq_1X28_4St}), (\ref{Pf_1X25_21}), and~(\ref{Eq_1X30_CutEst})
at the third step, and (\ref{Eq_1X28_St}) at the fourth step,
\[
\begin{split}
& \Bigl\| \ps (v)^* (b - \dt)_{+} \ps (v)
  - \Bigl( p a p - \frac{\ep}{5} \Bigr)_{+} \Bigr\|
\\
& \hspace*{3em} {\mbox{}}
  = \Bigl\| \ps (v)^* p (b - \dt)_{+} p \ps (v)
  - \Bigl( p a p - \frac{\ep}{5} \Bigr)_{+} \Bigr\|
\\
& \hspace*{3em} {\mbox{}}
 \leq \| \ps (v)^* \|
       \| p (b - \dt)_{+} p - \ps ( (b - \dt)_{+} ) \| \| \ps (v) \|
\\
& \hspace*{6em} {\mbox{}}
      + \bigl\| \ps (v)^* \ps ( (b - \dt)_{+} ) \ps (v)
          - \ps (v^* (b - \dt)_{+} v) \bigr\|
\\
& \hspace*{6em} {\mbox{}}
      + \| \ps \| \Bigl\| v^* (b - \dt)_{+} v
              - \Bigl( a - \frac{\ep}{5} \Bigr)_{+} \Bigr\|
      + \Bigl\| \ps \Bigl( \Bigl( a - \frac{\ep}{5} \Bigr)_{+} \Bigr)
        - p \Bigl( a - \frac{\ep}{5} \Bigr)_{+} p \Bigr\|
\\
& \hspace*{6em} {\mbox{}}
      + \Bigl\|  p \Bigl( a - \frac{\ep}{5} \Bigr)_{+} p
        - \Bigl( p a p - \frac{\ep}{5} \Bigr)_{+} \Bigr\|
\\
& \hspace*{3em} {\mbox{}}
  < \| v \|^2 \rh + \rh + \frac{\ep}{5} + \rh + \frac{\ep}{5}
  \leq \frac{3 \ep}{5}.
\end{split}
\]
Therefore, by Parts (8) and~(10) of Lemma 1.4 of~\cite{philar},
\[
\Bigl( p a p - \frac{4 \ep}{5} \Bigr)_{+}
    \precsim_{A^{\alpha}} \ps (v)^* (b - \dt)_{+} \ps (v)
    \precsim_{A^{\alpha}} (b - \dt)_{+}.
\]
Also, using $\| a \| \leq 1$ and (\ref{Pf_1X25_20}),
\[
\Bigl( (1 - p) a (1 - p) - \frac{4 \ep}{5} \Bigr)_{+}
  \leq (1 - p) a (1 - p)
  \leq 1 - p
  \precsim_{A^{\alpha}} f (b).
\]
Therefore, since $f (b)$ is orthogonal to $(b - \dt)_{+}$,
\[
\begin{split}
\Bigl( \bigl[ p a p + (1 - p) a (1 - p) \bigr]
       - \frac{4 \ep}{5} \Bigr)_{+}
& = \Bigl( p a p - \frac{4 \ep}{5} \Bigr)_{+}
     + \Bigl( (1 - p) a (1 - p) - \frac{4 \ep}{5} \Bigr)_{+}
\\
& \precsim_{A^{\alpha}} (b - \dt)_{+} + f (b)
  \leq b.
\end{split}
\]
By (\ref{Eq_2115_StSt}),
$2 \rh < \frac{\ep}{5}$ (from~(\ref{Eq_1X28_St})),
and Corollary~1.6 of~\cite{philar},
\[
(a - \ep)_{+}
 \precsim_{A^{\alpha}}
   \Bigl( p a p + (1 - p) a (1 - p) - \frac{4 \ep}{5} \Bigr)_{+}.
\]
So $(a - \ep)_{+} \precsim_{A^{\alpha}} b$, as desired.
\end{proof}

The following result is the analog of Lemma~3.11 of~\cite{radifinite}.
We use the terminology of Definition~3.8 of~\cite{radifinite},
which is based on the discussion before Corollary 2.24 of~\cite{APT11}
and Definition~3.1 of~\cite{philar}.

\begin{prp}\label{WC_plus_injectivity}
Let $A$ be a stably finite simple unital \ca{} which is not of Type~I
and let $\alpha \colon G \to \Aut (A)$
be an action of a second countable compact group $G$ on $A$
which has the tracial Rokhlin property with comparison.
Let $\iota \colon A^{\alpha} \to A$ be the inclusion map.
Then:
\begin{enumerate}
\item\label{WC_plus_injectivity_a}
The map $\W (\iota) \colon \W (A^{\alpha}) \to \W (A)$
induces an isomorphism of ordered semigroups
from $\W_{+} (A^{\alpha}) \cup \{ 0 \}$
to its image in $\W (A)$.
\item\label{WC_plus_injectivity_b}
The map $\Cu (\iota) \colon \Cu (A^{\alpha}) \to \Cu (A)$
induces an isomorphism of ordered semigroups
from $\Cu_{+} (A^{\alpha}) \cup \{ 0 \}$
to its image in $\Cu (A)$.
\end{enumerate}
\end{prp}

\begin{proof}
The proof is essentially the same
as that of Lemma~3.11 of~\cite{radifinite},
with the following changes.
The use of Lemma~3.7 of~\cite{radifinite}
is replaced with the use of Proposition~\ref{P_1X25_Comp_Fx}.
The proof of Lemma 3.10 of~\cite{radifinite} uses
the weak tracial Rokhlin property only to conclude
that $A^{\af}$ is simple and not of Type~I.
Here, $M_n (A^{\af})$ is simple and not of Type~I because
Theorem~\ref{thm_simple12} and Proposition~\ref{C_1Z11_InfDim}
show that this is true of~$A^{\af}$.
\end{proof}

\begin{pbm}\label{Pb_2115_Ran}
In Proposition~\ref{WC_plus_injectivity},
what are the ranges of the maps
$\W_{+} (A^{\alpha}) \cup \{ 0 \} \to \W (A)$
and $\Cu_{+} (A^{\alpha}) \cup \{ 0 \} \to \Cu (A)$?
\end{pbm}

If $G$ is finite, the image of $\Cu_{+} (A^{\alpha}) \cup \{ 0 \}$
is $\Cu_{+} (A)^{\alpha} \cup \{ 0 \}$, the $\af$-fixed points of
$\Cu_{+} (A) \cup \{ 0 \}$, by Theorem~5.5 of~\cite{radifinite}.
If, in addition, $A^{\alpha}$ has stable rank one,
a similar result holds for $\W_{+} (A^{\alpha}) \cup \{ 0 \} \to \W (A)$
by Corollary~5.6 of~\cite{radifinite}.
When we try to adapt the arguments of~\cite{radifinite} to the case
in which $G$ is only assumed to be compact,
we seem to need information on comparison in algebras of the form
$C (G, B)$.
We don't know what happens even if $A = C (G) \otimes B$
and $\af_g = \Lt_g \otimes \id_B$,
an action which even has the Rokhlin property, although the
algebra $A$ here is not simple.

\begin{thm}\label{P_1X25_rc_Fix}
Let $G$ be a second countable compact group,
let $A$ be a stably finite simple separable infinite dimensional \uca,
and let $\alpha \colon G \to \Aut (A)$ be an action of $G$ on~$A$
which has the tracial Rokhlin property with comparison.
Then $\rc (A^{\af}) \leq \rc (A)$.
\end{thm}

\begin{proof}
The proof is essentially the same
as that of Theorem~4.1 of~\cite{radifinite},
with the following changes.
First, there is a minor mistake in~\cite{radifinite}:
the reduction to the case $l = 1$ there is misleading,
because it changes $1_A$ to a \pj{} $e \in A$ such that the direct
sum of $l$ copies of $e$ is equivalent to~$1$.
It seems best to avoid that step.
Then the use of Lemma~3.7 of~\cite{radifinite}
is replaced with the use of Proposition~\ref{P_1X25_Comp_Fx},
and the use of Lemma 3.10 of~\cite{radifinite} is replaced in the same
way as in the proof of Proposition~\ref{WC_plus_injectivity} above.
\end{proof}

\begin{cor}\label{C_1X28_StComp}
Let $G$ be a second countable compact group,
let $A$ be a stably finite simple separable infinite dimensional \uca,
and let $\alpha \colon G \to \Aut (A)$ be an action of $G$ on~$A$
which has the tracial Rokhlin property with comparison.
If $A$ has strict comparison, so does~$A^{\af}$.
\end{cor}

\begin{proof}
Apply Theorem~\ref{P_1X25_rc_Fix}.
\end{proof}

\section{The naive tracial Rokhlin property}\label{S_2795_N_TRP}

Our definition of the \trpc{} (Definition~\ref{traR}),
applied to finite groups, is formally stronger than the
tracial Rokhlin property for finite groups
(Definition~\ref{D_1619_TRP}), as discussed after
Definition~\ref{traR}.
In this section we address the differences.
Definition~\ref{D_1920_NTRP} below is given
only to make discussion easier; it is not intended for general use.
It is what one gets by just copying the definition of the
tracial Rokhlin property for finite groups.
We say at the outset that we know of no examples of actions,
even of infinite compact groups,
which have the naive tracial Rokhlin property but not the \trpc,
although we believe they exist.

For actions of finite groups, we can show that
the tracial Rokhlin property ``almost'' implies the \trpc,

\begin{dfn}\label{D_1920_NTRP}
Let $A$ be an infinite dimensional simple separable unital \ca,
let $G$ be a second countable compact group,
and let $\af \colon G \to \Aut (A)$ be an action of $G$ on~$A$.
The action $\af$ has the
{\emph{naive tracial Rokhlin property}}
if for every finite set $F \subseteq A$,
every finite set $S \subseteq C (G)$, every $\ep > 0$,
and every $x \in A_{+}$ with $\| x \| = 1$,
there exist a projection $p \in A^{\alpha}$
and a unital completely positive
contractive map $\ph \colon C (G) \to p A p$
such that the following hold.
\begin{enumerate}
\item\label{It_1920_NTRP_FSE}
$\ph$ is an $(F, S, \ep)$-equivariant central multiplicative map.
\item\label{It_1920_NTRP_subx}
$1 - p \precsim_A x$.
\item\label{It_1920_NTRP_pxp}
$\| p x p \| > 1 - \ep$.
\end{enumerate}
\end{dfn}

We prove that condition
(\ref{1_pycompactsets}) of Definition \ref{traR} is automatic
when the group is finite, regardless of what $A$ is.
The list of conditions in the next proposition
is the same as in Lemma~\ref{L_1X16_trpc_GFin},
except that (\ref{Item_1X16_sub_1mp}) there has been omitted.

\begin{prp}\label{L_1617_IfGFin}
Let $A$ be an infinite dimensional simple separable unital \ca,
let $G$ be a finite group,
and let $\af \colon G \to \Aut (A)$ be an action of $G$ on~$A$
which has the tracial Rokhlin property.
Then for every finite set $F \subseteq A$, every $\ep > 0$,
every $x \in A_{+}$ with $\| x \| = 1$,
and every $y \in (A^{\alpha})_{+} \SM \{ 0 \}$,
there exist a projection $p \in A^{\alpha}$
and mutually orthogonal projections $(p_{g})_{g \in G}$
such that the following hold.
\begin{enumerate}
\item\label{Item_786}
$p = \sum_{g \in G} p_{g}$.
\item\label{Item_161p_comm}
$\| p_g a - a p_g \| < \ep$ for all $a \in F$ and all $g \in G$.
\item\label{Item_1619_CM}
$\| \af_g (p_h) - p_{g h} \| < \ep$ for all $g, h \in G$.
\item\label{Item_1617sub_x}
$1 - p \precsim_A x$.
\item\label{Item_1617sub_y}
$1 - p \precsim_{A^{\alpha}} y$.
\item\label{Item_1617_GFin_pxp}
$\| p x p \| > 1 - \ep$.
\end{enumerate}
\end{prp}

\begin{proof}
Let $F \subseteq A$ be finite and let $\varepsilon > 0$.
Let $x \in A_{+}$ with $\| x \| = 1$ and
$y \in (A^{\alpha})_{+} \SM \{ 0 \}$ be given.
Corollary 1.6 of~\cite{phill23} implies that $A^{\alpha}$ is simple.
Since $A^{\alpha}$ is unital and not \fd,
it follows that $A^{\alpha}$ is not of Type~I.
Therefore $\overline{y A^{\alpha} y}$ is simple and not of Type~I.
By Lemma 2.1 of \cite{philar} there is a positive element
$z \in \overline{y A^{\alpha} y}$
such that $0$ is a limit point of $\spec (z)$.
Define \cfn{s} $h, h_0 \colon [0, 1] \to [0, 1]$ by
\[
h (\ld) = \begin{cases}
   0 & \hspace*{1em} 0 \leq \ld \leq 1 - \frac{\ep}{2}
        \\
   \frac{2}{\ep} (\ld - 1) + 1
           & \hspace*{1em} 1 - \frac{\ep}{2} \leq \ld \leq 1
\end{cases}
\]
and
\[
h_0 (\ld) = \begin{cases}
   \left( 1 - \frac{\ep}{2} \right)^{-1} \ld
           & \hspace*{1em} 0 \leq \ld \leq 1 - \frac{\ep}{2}
        \\
   1 & \hspace*{1em} 1 - \frac{\ep}{2} \leq \ld \leq 1.
\end{cases}
\]
Then $\| x - h_0 (x) \| \leq \frac{\ep}{2}$.
Also, $h (x) \neq 0$ since $\| x \| = 1$.
Use Lemma \ref{lma_L683_baj} to choose a nonzero positive
element $x_0 \in \ov{h (x) A h (x)}$ such that $x_0 \precsim_A z$.
We may require $\| x_0 \| = 1$.

Now apply Lemma 1.17 of \cite{phill23} to $\alpha$
with $x_0$ in place of~$x$, with $\frac{\ep}{2}$ in place of~$\ep$,
and with $F$ as given.
We obtain mutually orthogonal projections $p_{g} \in A$ for $g \in G$
such that, with $p = \sum_{g \in G} p_{g}$
(so that (\ref{Item_786}) holds),
$p$ is $\af$-invariant;
with $x_0$ in place of~$x$ and $\frac{\ep}{2}$ in place of~$\ep$,
Conditions (\ref{Item_161p_comm}), (\ref{Item_1619_CM})
and~(\ref{Item_1617_GFin_pxp}) are satisfied;
and $1 - p$ is \mvnt{} to a \pj{} in $\ov{x_0 A x_0}$.
In particular, we have (\ref{Item_786}), (\ref{Item_161p_comm}),
and~(\ref{Item_1619_CM}) as stated.
Also,
$1 - p \precsim_A x_0 \precsim_A x$, which is (\ref{Item_1617sub_x}).
Moreover,
\[
1 - p \precsim_A x_0 \precsim_A z,
\qquad
1 - p, z \in A^{\alpha},
\andeqn
0 \in \ov{ \spec (z) \SM \{ 0 \} }.
\]
Since the tracial Rokhlin property implies the weak tracial Rokhlin
property, it follows from Lemma 3.7 of~\cite{radifinite}
that $1 - p \precsim_{A^{\alpha}} z$.
Since $z \in \ov{y A^{\alpha} y}$,
we get $1 - p \precsim_{A^{\alpha}} y$,
which is~(\ref{Item_1617sub_y}).

It remains to prove~(\ref{Item_1617_GFin_pxp}).
Since $h_0 (x) h (x) = h (x)$, we have
\[
x_0 = h_0 (x)^{1 / 2} x_0 h_0 (x)^{1 / 2} \leq h_0 (x).
\]
So, also using $\| p x_0 p \| > 1 - \frac{\ep}{2}$,
\[
\| p x p \|
 \geq \| p h_0 (x) p \| - \| h_0 (x) - x \|
 \geq \| p x_0 p \| - \frac{\ep}{2}
 > 1 - \frac{\ep}{2} - \frac{\ep}{2}
 = 1 - \ep,
\]
as desired.
\end{proof}

We now give some conditions on actions of finite groups under
which Condition (\ref{1_ppcompactsets}) of
Definition \ref{traR} is automatic.

\begin{prp}\label{P_1X17_StComp}
Let $A$ be a stably finite infinite dimensional
simple separable unital \ca.
Let $\af \colon G \to \Aut (A)$
be an action of a finite group $G$ on $A$
which has the tracial Rokhlin property.
If $\rc (A) < 1$ then $\af$ has the \trpc.
\end{prp}

In particular, if $A$ has strict comparison then $\af$ has the \trpc.

\begin{proof}[Proof of Proposition~\ref{P_1X17_StComp}]
The conclusion is similar to but stronger than
that of Proposition~\ref{L_1617_IfGFin}.
We describe the necessary changes and additions to the proof
of that proposition.

We verify the conditions of Lemma~\ref{L_1X16_trpc_GFin}.

By Proposition~\ref{P_1X19_RP_to_TRPC},
we may assume that $\af$ does not have the Rokhlin property.
Let $F \subseteq A$ be finite and let $\varepsilon > 0$.
Let $x \in A_{+}$ with $\| x \| = 1$ and
$y \in (A^{\alpha})_{+} \SM \{ 0 \}$ be given.
As in the proof of Proposition~\ref{L_1617_IfGFin},
$A^{\alpha}$ is simple and not of Type~I,
and there is a positive element $z \in \overline{y A^{\alpha} y}$
such that $0$ is a limit point of $\spec (z)$.

Choose $n \in \N$ such that
\begin{equation}\label{Eq_1X29_n_rc}
n > \frac{2}{1 - \rc (A)}.
\end{equation}
The algebra $A^{\af}$ has Property~(SP) by Lemma~1.13 of~\cite{phill23}.
Lemma~\ref{OrthInSP} provides $n + 1$ nonzero \mops{} in~$A$.
Let $e$ be one of them.
Again as in the proof of Proposition~\ref{L_1617_IfGFin},
there is $d \in (e A^{\af} e)_{+} \SM \{ 0 \}$
such that $0$ is a limit point of $\spec (d)$.
With $h$ as in that proof,
apply Lemma~\ref{lma_L683_baj} to find $x_0 \in \ov{h (x) A h (x)}$
such that
\[
\| x_0 \| = 1,
\qquad
x_0 \precsim_A d,
\andeqn
x_0 \precsim_A y_0.
\]

Now apply Lemma 1.17 of~\cite{phill23} to $\alpha$ with the same
choices as in the proof of Proposition~\ref{L_1617_IfGFin},
getting, as in the second half of the proof,
mutually orthogonal projections $p_{g} \in A$ for $g \in G$
such that, with $p = \sum_{g \in G} p_{g}$,
the conclusion of Proposition~\ref{L_1617_IfGFin} holds,
and $1 - p \precsim_{A} x_0$.

It remains only to show that $1 - p \precsim_{A^{\af}} p$.
We know that $1 - p \precsim_A x_0 \precsim_A d$.
By Lemma 3.7 of~\cite{radifinite},
we have $1 - p \precsim_{A^{\alpha}} d$.
So $1 - p \precsim_{A^{\alpha}} e$.
Let $\ta \in \QT (A^{\af})$.
Then $\ta (1 - p) \leq \ta (e)$, so $\ta (p) \geq \ta (1 - e)$.
The construction of $e$ ensures that
$\ta (e) \leq  \frac{1}{n + 1} < \frac{1}{n}$.
Therefore
\begin{equation}\label{Eq_1X22_Ineq}
\ta (1 - p) + 1 - \frac{2}{n} <  1 - \frac{1}{n} < \ta (p).
\end{equation}
Using Theorem 4.1 of \cite{radifinite} at the first step
and~(\ref{Eq_1X29_n_rc}) at the second,
we have $\rc (A^{\af}) \leq \rc (A) < 1 - \frac{2}{n}$.
Since~(\ref{Eq_1X22_Ineq}) holds for all $\ta \in \QT (A^{\af})$,
we get $1 - p \precsim_{A^{\af}} p$, as desired.
\end{proof}

\begin{cor}\label{C_1X29_Fin_rc}
Let $A$ be a stably finite infinite dimensional
simple separable unital \ca{} such that $\rc (A)$ is finite.
Let $\af \colon G \to \Aut (A)$
be an action of a finite group $G$ on $A$
which has the tracial Rokhlin property.
Then there exists $n \in \N$ such that the action
$g \mapsto \id_{M_n} \otimes \af_g$ on $M_n \otimes A$ has the \trpc.
\end{cor}

\begin{proof}
Choose $n \in \N$ such that $n > \rc (A)$.
Then $\rc (M_n \otimes A) = \frac{1}{n} \rc (A) < 1$.
Also, $g \mapsto \id_{M_n} \otimes \af_g$
has the tracial Rokhlin property by Lemma~3.9 of~\cite{phill23}.
Apply Proposition~\ref{L_1617_IfGFin}.
\end{proof}

To make the argument work in general, one needs to apply to $A^{\af}$
a positive answer to the following question.

\begin{qst}\label{Pb_1X20_ExSm}
Let $B$ be a stably finite simple separable \uca.
Does there exist $z \in B_{+} \setminus \{ 0 \}$ such that
whenever a \pj{} $q \in B$ satisfies $q \precsim_B z$,
then $q \precsim_B 1 - q$?
\end{qst}

This question seems hard, but the answer may well be negative.

We now prove that if $G$ is finite and $A$ is purely infinite simple,
then condition (\ref{1_ppcompactsets}) of
Definition \ref{traR} is automatic.

\begin{prp}\label{P_1X17_PI}
Let $A$ be an infinite dimensional simple separable unital \ca.
Let $\af \colon G \to \Aut (A)$
be an action of a finite group $G$ on $A$
which has the tracial Rokhlin property.
If $A$ is purely infinite then $\af$ has the \trpc.
\end{prp}

\begin{proof}
We verify the condition of Lemma~\ref{L_1X16_trpc_GFin}.
So let $F \subseteq A$ be finite,
let $\varepsilon > 0$, let $x \in A_{+}$ satisfy $\| x \| = 1$,
and let $y \in (A^{\alpha})_{+} \setminus \{ 0 \}$.
\Wolog{} $\ep < 1$.
Apply Lemma 1.17 of~\cite{phill23} with $F$, $\ep$, and $x$ as given,
getting \mops{} $p_g \in A$ for $g \in G$ such that,
in Lemma~\ref{L_1X16_trpc_GFin} and with $p = \sum_{g \in G} p_g$,
Conditions (\ref{Item_1X16_Inv}), (\ref{Item_1X16_Comm}),
(\ref{Item_1X16_Prm}), (\ref{Item_1X16_sub_x}),
and~(\ref{Item_1X16_Mnp}) are satisfied.
The algebra $C^* (G, A, \af)$
is simple by Corollary~1.6 of~\cite{phill23}.
So $A^{\af}$ is simple by Theorem~\ref{satunonabel}.
The action $\af$ is pointwise outer by Lemma~1.5 of~\cite{phill23},
so Theorem~3 of~\cite{Je95}
implies that $C^* (G, A, \af)$ is purely infinite.
Corollary~\ref{C_1X15_StableIso}
now implies that $A^{\af}$ is stably isomorphic to $C^* (G, A, \af)$,
so $A^{\af}$ is purely infinite.
Since $y \neq 0$,
the relation $1 - p \precsim_{A^{\af}} y$ is automatic;
this is Condition (\ref{Item_1X16_sub_yy})
of Lemma~\ref{L_1X16_trpc_GFin}.
Since $p \neq 0$ (from $\| p x p \| > 1 - \ep$),
the relation $1 - p \precsim_{A^{\af}} p$ is automatic;
this is Condition (\ref{Item_1X16_sub_1mp})
of Lemma~\ref{L_1X16_trpc_GFin}.
\end{proof}

\section{The modified tracial Rokhlin property}\label{S_2795_mod_TRP}

The extra conditions in Definition~\ref{traR} seem somewhat
unsatisfactory, partly because there are two of them.
We seem to need Condition~(\ref{1_pycompactsets})
($1 - p \precsim_{A^{\alpha}} y$)
in order to prove preservation of tracial rank when it is zero or one
(see the proof of Theorem~\ref{T_2123_PrsvTR}),
and we seem to need Condition~(\ref{1_ppcompactsets})
($1 - p \precsim_{A^{\alpha}} p$)
in order to prove that the crossed product is simple
(see the proof of Proposition~\ref{satpropnoncptp}).
This has led us to consider other variants.
In this section, we discuss the most promising of these,
which we call the modified tracial Rokhlin property
(again, a name intended for use only in this paper).
There are actually two versions.
We will explain what we can prove with them,
and prove that, in a very special case, they are automatic for
actions of finite groups with the tracial Rokhlin property.
The example we construct in Section~\ref{Sec_3749_Exam_TRPZ2} has
the strong modified tracial Rokhlin property,
and the example in Section~\ref{Sec_1908_Exam_TRPS1}
has the modified tracial Rokhlin property
but probably not the strong modified tracial Rokhlin property.

The difference in the definitions we give below
is in Definition \ref{moditra}(\ref{Item_1X07_30})
and Definition \ref{D_1824_AddToTRP_Mod}(\ref{Im_1824_ATRP_s_Comm_Mod}).
For the strong modified tracial Rokhlin property,
the partial isometry $s$ is required to approximately commute
with the elements of a given finite subset of~$A$,
but for the modified tracial Rokhlin property,
$s$ is only required to approximately commute
with the elements of a given finite subset of~$A^{\af}$.
This definition thus requires two finite sets instead of just one.

\begin{dfn}\label{moditra}
Let $A$ be an infinite dimensional simple separable unital \ca,
let $G$ be a second countable compact group,
and let $\af \colon G \to \Aut (A)$ be an action of $G$ on~$A$.
The action $\af$ has the
{\emph{modified tracial Rokhlin property}}
if for every finite set $F_1 \subseteq A$,
every finite set $F_2 \subseteq A^{\af}$,
every finite set $S \subseteq C (G)$, every $\ep > 0$,
and every $x \in A_{+}$ with $\| x \| = 1$,
there exist a projection $p \in A^{\alpha}$,
a partial isometry $s \in A^{\alpha}$,
and a unital completely positive
contractive map $\ph \colon C (G) \to p A p$,
such that the following hold.
\begin{enumerate}
\item\label{Item_1X07_27}
$\varphi$ is an $(F_1, S, \varepsilon)$-approximately equivariant
central multiplicative map.
\item\label{Item_1X07_28}
$1 - p \precsim_{A} x$.
\item\label{Item_1X07_29}
$s^{*} s = 1 - p$ and $s s^{*} \leq p$.
\item\label{Item_1X07_30}
$\| s a - a s \| < \varepsilon$ for all $a \in F_2$.
\item\label{Item_1X07_31}
$\| p x p \| > 1 - \varepsilon$.
\end{enumerate}
\end{dfn}

\begin{dfn}\label{D_1824_AddToTRP_Mod}
Let $A$ be an infinite dimensional simple separable unital \ca,
let $G$ be a second countable compact group,
and let $\af \colon G \to \Aut (A)$ be an action of $G$ on~$A$.
The action $\af$ has the
{\emph{strong modified tracial Rokhlin property}}
if for every finite set $F \subseteq A$,
every finite set $S \subseteq C (G)$, every $\ep > 0$,
and every $x \in A_{+}$ with $\| x \| = 1$,
there exist a partial isometry $s \in A^{\alpha}$,
a projection $p \in A^{\alpha}$,
and a unital completely positive
contractive map $\ph \colon C (G) \to p A p$,
such that the following hold.
\begin{enumerate}
\item\label{Item_1824_ATRP_Mod_CM}
$\ph$ is an $(F, S, \ep)$-approximately
equivariant central multiplicative map.
\item\label{Item_1824_ATRP_sub_x_Mod}
$1 - p \precsim_A x$.
\item\label{Item_1824_ATRP_SubP_Mod}
$s^* s = 1 - p$ and $s s^* \leq p$.
\item\label{Im_1824_ATRP_s_Comm_Mod}
$\| s a - a s \| < \ep$ for all $a \in F$.
\item\label{Item_1824_ATRP_pxp_Mod}
$\| p x p \| > 1 - \ep$.
\end{enumerate}
\end{dfn}

The modified tracial Rokhlin property
implies simplicity of the fixed point algebra and crossed product.
For simplicity of the fixed point algebra,
we don't need to know that the element $s$ in
Definition \ref{moditra} approximately commutes with anything.

\begin{thm}\label{thm:simple2}
Let $A$ be a simple separable infinite dimensional unital \ca,
let $G$ be a second countable compact group, and let
$\alpha \colon G \to \Aut (A)$ be an action
which has the modified tracial Rokhlin property.
Then $A^{\alpha}$ is simple.
\end{thm}

\begin{proof}
We proceed as in the proof of Theorem~\ref{thm_simple12}.
As there, let $I$ be a nonzero ideal in $A^{\alpha}$.
There are $m \in \N$,
$a_1, a_2 \ldots, a_{m}, b_1, b_2, \ldots, b_{m} \in A$,
and $x_1, x_2, \ldots, x_{m} \in I$ such that
$\sum_{j = 1}^{m} a_j x_j b_j = 1$.
As there, define $M$, $\dt > 0$, $F_1 \S A$, and $F_2 \S A^{\af}$
by (\ref{Eq_1X20_Md}), (\ref{Eq_1X20_F1F2}),
and~(\ref{Eq_1X20_F1F2_p2}).
Apply the modification of Theorem~\ref{thm_ApproxHommtr}
described in Remark~\ref{R_1Z09_mod_trp}
with $\dt$ in place of $\varepsilon$,
with $n = 3$, with $1$ in place of $x$,
and with $F_{1}$ and $F_{2}$ as given.
We get a projection $p$ in $A^{\alpha}$,
a partial isometry $s$ in $A^{\alpha}$, and a \ucp{}
map $\psi \colon A \to p A^{\alpha} p$, such that the following hold.
\begin{enumerate}
\item\label{Item_1577_psiab_muliti_TRPC_MT}
$\| \psi (a b c) - \psi (a) \psi (b) \ps (c) \| < \delta$
for all $a, b, c \in F_{1}$.
\item\label{Item_1578_paap_TRPC_MT}
$\| p a - a p \| < \delta$ for all $a \in F_{1} \cup F_{2}$.
\item\label{Item_1576_psia_iden_TRPC_MT}
$\| \psi (a) - p a p \| < \delta$ for all $a \in F_{2}$.
\item\label{Item_1579_1mp_subz_fix_TRPC_MT}
$s^{*} s = 1 - p$ and $s s^{*} \leq p$.
\end{enumerate}
Define
\[
d = \sum_{j = 1}^{m} \psi (a_{j}) p x_{j} p \psi (b_{j})
  \in p I p
  \S I.
\]
The reasoning of the proof of Theorem~\ref{thm_simple12}
gives $\| p - d \| < \frac{1}{2}$.
Also $s^* d s \in (1 - p) I (1 - p) \S I$ and
\[
\| 1 - p - s^* d s \| = \| s^* (p - d) s \|
  < \frac{1}{2}.
\]
Since $p$ and $1 - p$ are orthogonal, we get
$\| 1 - (d + s^* d s) \| < \frac{1}{2}$.
So $d + s^* d s$ is in $I$ and is invertible.
\end{proof}

\begin{prp}\label{satpropnoncpts}
Let $A$ be an infinite dimensional simple separable unital \ca.
Let $\alpha \colon G \to \Aut (A)$
be an action of a compact group $G$ on $A$ which
has the modified tracial Rokhlin property.
Then $\alpha$ is saturated.
\end{prp}

\begin{proof}
This is Remark~\ref{R_1X20_No14}.
\end{proof}

\begin{thm}\label{simplecross}
Let $A$ be an infinite dimensional simple separable unital \ca,
and let $\alpha \colon G \to \Aut (A)$
be an action of a second countable compact group $G$ on~$A$
which has the modified tracial Rokhlin property.
Then $C^{*} (G, A, \alpha)$ is simple.
\end{thm}

\begin{proof}
The algebra $A^{\af}$ is simple by Theorem~\ref{thm:simple2}
and $\af$ is saturated by Proposition~\ref{satpropnoncpts}.
So Condition~(\ref{Item_1X07_SimSat}) in Theorem~\ref{satunonabel}
holds.
\end{proof}

Crossed products by actions with the modified tracial Rokhlin property
probably also preserve Property~(SP) and pure infiniteness.
We were not able to prove that even
crossed products by actions with the
strong modified tracial Rokhlin property preserve tracial rank zero.

In the rest of this section, we prove that an action
of a finite group $G$ on a UHF~algebra~$A$ which
has the tracial Rokhlin property must in fact have the
strong modified tracial Rokhlin property.
We do not know to what extent this result can be generalized,
even if we require only the modified tracial Rokhlin property.

\begin{ntn}\label{N_1619_Perm}
Let $n \in \N$.
Let $S_n$ denote the symmetric group on $n$ letters,
and let $(e_{j, k})_{j, k = 1, 2, \ldots, n}$
be the standard system of matrix units for~$M_n$.
For $\sm \in S_n$ and $\ep \in \{ - 1, 1 \}^n$,
let $v (\sm, \ep) = \sum_{j = 1}^n \ep_j e_{\sm (j), \, j}$.
We call the matrices $v (\sm, \ep)$
the signed permutation matrices in $M_n$.
\end{ntn}

\begin{lem}\label{L_1619_CommSgnPerm}
Let $A$ be a unital \ca,
let $n \in \N$, and let $\ph \colon M_n \to A$ be a unital \hm.
Following Notation~\ref{N_1619_Perm},
define a map $E \colon A \to A$ by
\[
E (x) = \frac{1}{2^n n!}
  \sum_{\sm \in S_n} \sum_{\ep \in \{ - 1, 1 \}^n}
   \ph (v (\sm, \ep)) x \ph (v (\sm, \ep))^*.
\]
Then $E$ is a conditional expectation from $A$ to the relative
commutant $\ph (M_n)' \cap A$.
Moreover, for any $x \in A$, we have
\[
\| E (x) - x \| \leq 2 \, \dist (x, \, \ph (M_n)' \cap A).
\]
\end{lem}

\begin{proof}
This is well known, and easy.
One checks that the signed permutation matrices are a subgroup
of the unitary group of~$M_n$,
and that they span $M_n$.
This is easily seen to imply
all but the last sentence of the conclusion.
For that, let $x \in A$.
Set $\rh = \dist (x, \, \ph (M_n)' \cap A)$.
Let $\ep > 0$.
Choose $y \in \ph (M_n)' \cap A$
such that $\| y - x \| < \rh + \frac{\ep}{2}$.
Then $E (y) = y$, so
\[
\| E (x) - x \|
 = \| E (x - y) - (x - y) \|
 \leq \| E (x - y) \| + \| x - y \|
 \leq 2 \| x - y \|
 < 2 \rh + \ep.
\]
This completes the proof.
\end{proof}

By the reasoning in the proof of Lemma~\ref{L_1X16_trpc_GFin},
the following result implies that the actions in the
hypotheses have the strong modified tracial Rokhlin property.
We emphasize that the action is not required to be a direct limit
action.

\begin{prp}\label{D_1619_TRP_and_s}
Let $A$ be a UHF~algebra,
let $G$ be a finite group,
and let $\af \colon G \to \Aut (A)$ be an action of $G$ on~$A$
which has the tracial Rokhlin property.
Then for every finite set $F \subseteq A$, every $\ep > 0$,
every $x \in A_{+}$ with $\| x \| = 1$,
there exist a partial isometry $s \in A^{\alpha}$,
a projection $p \in A^{\alpha}$,
and \mops{} $p_g \in A$ for $g \in G$, such that the following hold.
\begin{enumerate}
\item\label{Item_1619_tRp_s_p}
$p = \sum_{g \in G} p_g$.
\item\label{Item_1619_tRp_s_p_comm}
$\| p_g a - a p_g \| < \ep$ for all $a \in F$ and all $g \in G$.
\item\label{Item_1619_tRp_s_CM}
$\| \af_g (p_h) - p_{g h} \| < \ep$ for all $g, h \in G$.
\item\label{Item_1619_tRp_s_sub_x}
$1 - p \precsim_A x$.
\item\label{Item_1619_tRp_s_SubP}
$s^* s = 1 - p$ and $s s^* \leq p$.
\item\label{Item_1619_tRp_s_Comm}
$\| s a - a s \| < \ep$ for all $a \in F$.
\item\label{Item_1619_tRp_s_pxp}
$\| p x p \| > 1 - \ep$.
\setcounter{TmpEnumi}{\value{enumi}}
\end{enumerate}
\end{prp}

\begin{proof}
Write $A = \dirlim_m A_m$ with unital maps and with integers
$d (m) \geq 2$ such that $A_m \cong M_{d (m)}$ for $m \in \Nz$.
Set $n = \card (G)$.
Let $\ta$ be the unique \tst{} on~$A$.
Let $F \subseteq A$ be finite, let $\ep > 0$,
and let $x \in  A_{+}$ satisfy $\| x \| = 1$.
\Wolog{} $\| a \| \leq 1$ for all $a \in F$.
Define
\begin{equation}\label{Eq_1619_B12}
\dt_0 = \min \left( \frac{1}{3}, \,  \frac{\ep}{22 + 2 \sqrt{n}},
 \,  \frac{\ep}{6 n} \right).
\end{equation}
Choose $\dt_1 > 0$ so small that the following three conditions are
satisfied.
First,
\begin{equation}\label{Eq_1619_B11}
\dt_1 \leq \dt_0.
\end{equation}
Second, whenever $B$ is a \ca, $e \in B$ is a \pj, $y \in B$,
$y e = y$, and $\| y^* y - e \| < 41 \dt_1$,
then, with functional calculus evaluated in $e B e$,
the element $z = y (y^* y)^{- 1 / 2}$ exists
and satisfies $z^* z = e$ and $\| z - y \| < \dt_0$.
Finally, whenever $B$ is a \uca,
$e, f \in B$ are \pj{s}, and $\| e - f \| < 2 \dt_1$,
then there is a unitary $w \in B$
such that $w e w^* = f$ and $\| w - 1 \| < \dt_1$.
Choose $\dt_2 > 0$ so small that
\begin{equation}\label{Eq_1619_4}
\dt_2 \leq \min \left( \frac{1}{15 n}, \, \frac{\dt_1}{15 n}, \,
                     \frac{\dt_1}{5 n^{3/2}}, \, \frac{\dt_1^2}{15 n}
         \right),
\end{equation}
and also so small that whenever $B$ is a \ca, $b \in B$ is selfadjoint,
and $\| b^2 - b \| < 60 n \dt_2$,
then there is a \pj{} $e \in B$ such that $\| e - b \| < \dt_1$.
Choose $\dt > 0$ so small that
\begin{equation}\label{Eq_1619_1}
\dt \leq \dt_2,
\end{equation}
and also so small that whenever $B$ is a \ca{}
and $b_1, b_2, \ldots, b_n \in B$ are selfadjoint
and satisfy $\| b_j^2 - b_j \| < 3 \dt$ for $j = 1, 2, \ldots, n$
and $\| b_j b_k \| < 3 \dt$
for distinct $j, k \in \{ 1, 2, \ldots, n \}$,
then there are \mops{} $e_j \in B$ for $j = 1, 2, \ldots, n$
such that $\| e_j - b_j \| < \dt_2$ for $j = 1, 2, \ldots, n$.

Set $F_0 = \bigcup_{g \in G} \af_g (F)$.
Choose $m \in \Nz$ such that for every $a \in F_0$ there is $b \in A_m$
such that $\| a - b \| < \dt_0$.
Let $P$ be the image under some isomorphism $M_{d (m)} \to A_m$
of the set of signed permutation matrices in $M_{d (m)}$
(as in Notation~\ref{N_1619_Perm}).
Define \cfn{s} $h, h_0 \colon [0, 1] \to [0, 1]$ by
\[
h (\ld) = \begin{cases}
   0 & \hspace*{1em} 0 \leq \ld \leq 1 - \dt
        \\
   \dt^{-1} (\ld - 1 + \dt) & \hspace*{1em} 1 - \dt \leq \ld \leq 1
\end{cases}
\]
and
\[
h_0 (\ld) = \begin{cases}
   (1 - \dt)^{-1} \ld & \hspace*{1em} 0 \leq \ld \leq 1 - \dt
        \\
   1 & \hspace*{1em} 1 - \dt \leq \ld \leq 1.
\end{cases}
\]
Then $h (x) \neq 0$ since $\| x \| = 1$.
Since $A$ is a UHF algebra, there is a
nonzero \pj{} $f \in {\ov{h (x) A h (x)}} \S {\ov{x A x}}$
such that
\begin{equation}\label{Eq_1619_1Over}
\ta (f) < \frac{1}{n + 1}.
\end{equation}
Apply Definition~\ref{D_1619_TRP}
with $P$ in place of~$F$, with $\dt$ in place of~$\ep$,
and with $f$ in place of~$x$.
Call the resulting \pj{s} $q_g$ for $g \in G$,
and set $q = \sum_{g \in G} q_g$.
Thus:
\begin{enumerate}
\setcounter{enumi}{\value{TmpEnumi}}
\item\label{Eq_1619_TRP_AppCm}
$\| q_g v - v q_g \| < \dt$ for all $v \in P$ and all $g \in G$.
\item\label{Eq_1619_B90}
$\| \af_g (q_h) - q_{g h} \| < \dt$ for all $g, h \in G$.
\item\label{Eq_1619_B22}
$1 - q \precsim_A f$.
\item\label{Eq_1619_TRP_Norm}
$\| q f q \| > 1 - \dt$.
\end{enumerate}
It follows from~(\ref{Eq_1619_B90}) that
\begin{equation}\label{Eq_1619_q_inv}
\| \af_g (q) - q \| < n \dt
\end{equation}
for all $g \in G$.

We claim that
\begin{equation}\label{Eq_1619_B23}
\| q x q \| > 1 - 3 \dt.
\end{equation}
To prove the claim, first observe that
\[
h_0 (x)^{1/2} h (x) = h (x)
\andeqn
\| h_0 (x) - x \| < 2 \dt.
\]
It follows that $h_0 (x)^{1/2} f = f$.
Therefore
\[
q f q
 = q h_0 (x)^{1/2} f h_0 (x)^{1/2} q
 \leq q h_0 (x) q,
\]
so, using (\ref{Eq_1619_TRP_Norm}) at the last step,
\[
\| q x q \|
 \geq \| q h_0 (x) q \| - \| h_0 (x) - x \|
 > \| q f q \| - 2 \dt
 > 1 - 3 \dt,
\]
as claimed.

By Lemma~\ref{L_1619_CommSgnPerm}, the formula
\[
E (x) = \frac{1}{2^{d (m)} d (m)!} \sum_{v \in P} v x v^*
\]
defines a conditional expectation $E \colon A \to A_m' \cap A$
such that
\begin{equation}\label{Eq_1619_dist}
\| E (x) - x \| \leq 2 \, \dist (x, \, A_m' \cap A)
\end{equation}
for all $x \in A$.
Also clearly
\begin{equation}\label{Eq_1619_dComm}
\| E (x) - x \| \leq \sup_{v \in P} \| v x - x v \|.
\end{equation}
Combining (\ref{Eq_1619_TRP_AppCm}) and~(\ref{Eq_1619_dComm}),
we get
\begin{equation}\label{Eq_2116_TnSt}
\| E (q_g) - q_g \| < \dt
\end{equation}
for all $g \in G$.
Therefore
\[
\begin{split}
& \| E (q_g)^2 - E (q_g) \|
\\
& \hspace*{3em} {\mbox{}}
\leq \| E (q_g) - q_g \| \| E (q_g) \| + \| q_g \| \| E (q_g) - q_g \|
        + \| E (q_g) - q_g \|
 < 3 \dt
\end{split}
\]
and, if $g \neq h$, then
\[
\| E (q_g) E (q_h) \|
 \leq \| E (q_g) - q_g \| \| E (q_h) \| + \| q_g \| \| E (q_h) - q_h \|
 < 2 \dt.
\]
By the choice of~$\dt$,
there exist \mops{} $r_g \in A_m' \cap A$ for $g \in G$
such that $\| r_g - E (q_g) \| < \dt_2$ for $j = 1, 2, \ldots, n$.
Set
\begin{equation}\label{Eq_2116_NNSt}
r = \sum_{g \in G} r_g.
\end{equation}
Using (\ref{Eq_2116_TnSt}) at the second step
and (\ref{Eq_1619_1}) at the last step,
\begin{equation}\label{Eq_1619_B91}
\| r_g - q_g \|
 \leq \| r_g - E (q_g) \| + \| E (q_g) - q_g \|
 < \dt_2 + \dt \leq 2 \dt_2
\end{equation}
for all $g \in G$, so
\begin{equation}\label{Eq_1619_2}
\| r - q \| < 2 n \dt_2.
\end{equation}

Using (\ref{Eq_1619_B90}) and~(\ref{Eq_1619_B91}) at the second step
and (\ref{Eq_1619_1}) at the last step, for $g, h \in G$ we get
\begin{equation}\label{Eq_1619_3}
\begin{split}
\| \af_g (r_h) - r_{g h} \|
& \leq \| \af_g \| \| r_h - q_h \| + \| \af_g (q_h) - q_{g h} \|
       + \| r_{g h} - q_{g h} \|
\\
& < 2 \dt_2 + \dt + 2 \dt_2
  \leq 5 \dt_2.
\end{split}
\end{equation}
Using (\ref{Eq_1619_2}) and (\ref{Eq_1619_q_inv})
at the second step, for $g \in G$,
\begin{equation}\label{Eq_1619_B31}
\begin{split}
\| \af_g (r) - r \|
& \leq \| \af_g (r) - \af_g (q) \| + \| \af_g (q) - q \| + \| q - r \|
\\
& < 2 n \dt_2 + n \dt + 2 n \dt_2
  \leq 5 n \dt_2.
\end{split}
\end{equation}
Since $2 n \dt_2 \leq 1$ by~(\ref{Eq_1619_4}),
it follows from~(\ref{Eq_1619_2})
that the \pj{s} $1 - r$ and $1 - q$ are \mvnt.
Since $1 - q \precsim_A f$ by~(\ref{Eq_1619_B22}),
using (\ref{Eq_1619_1Over}) we get
\[
\ta (1 - r) < \frac{1}{n + 1}
\andeqn
\ta (r) > \frac{n}{n + 1}.
\]
For $g \in G$, since $5 \dt_2 < 1$ by~(\ref{Eq_1619_4}),
it follows from~(\ref{Eq_1619_3}) that $\af_g (r_1) \sim r_g$.
Uniqueness of $\ta$ implies $\ta \circ \af_g = \ta$.
So $\ta (r_g) = \ta (r_1)$.
Hence
\[
\ta (r_1) = \frac{\ta (r)}{n} > \frac{1}{n + 1} > \ta (1 - r).
\]
Since $A_m' \cap A$ is a UHF algebra whose unique \tst{}
is $\ta |_{A_m' \cap A}$,
there is $s_0 \in A_m' \cap A$ such that $s_0^* s_0 = 1 - r$
and $s_0 s_0^* \leq r_1$.
Define
\begin{equation}\label{Eq_1X21_c0Dfn}
c_0 = \frac{1}{\sqrt{n}} \sum_{g \in G} \af_g (s_0).
\end{equation}

For $g \in G$, using (\ref{Eq_1619_3}) at the last step,
\begin{equation}\label{Eq_1619_B54}
\begin{split}
\| r_g \af_g (s_0) - \af_g (s_0) \|
& = \| r_g \af_g (s_0) - \af_g (r_1 s_0) \|
\\
& \leq \| r_g - \af_g (r_1) \| \| \af_g (s_0) \|
  < 5 \dt_2.
\end{split}
\end{equation}
So, for $h \neq g$, since $r_g r_h = 0$,
\begin{equation}\label{Eq_1619_B51}
\| r_g \af_h (s_0) \|
 = \| r_g [ \af_h (s_0) - r_h \af_h (s_0) ] \|
 < 5 \dt_2
\end{equation}
and
\[
\begin{split}
\| \af_g (s_0)^* \af_h (s_0) \|
& \leq \| \af_g (s_0)^* - \af_g (s_0)^* r_g \|
         + \| \af_h (s_0) - r_h \af_h (s_0) \|
\\
& < 5 \dt_2 + 5 \dt_2
 = 10 \dt_2.
\end{split}
\]
Now, using (\ref{Eq_1619_B31}) at the last step,
\[
\biggl\| 1 - r - \frac{1}{n} \sum_{g \in G} \af_g (s_0^* s_0) \biggr\|
  \leq \frac{1}{n} \sum_{g \in G} \| r - \af_g (r) \|
  < 5 n \dt_2,
\]
so
\begin{equation}\label{Eq_1619_B52}
\begin{split}
\| 1 - r - c_0^* c_0 \|
& = \biggl\| 1 - r
   - \frac{1}{n} \sum_{g, h \in G} \af_g (s_0^*) \af_h (s_0) \biggr\|
\\
& \leq \biggl\| 1 - r
           - \frac{1}{n} \sum_{g \in G} \af_g (s_0^* s_0) \biggr\|
    + \frac{1}{n} \sum_{g \neq h} \| \af_g (s_0^*) \af_h (s_0) \|
\\
& < 5 n \dt_2 + 10 n \dt_2
  = 15 n \dt_2.
\end{split}
\end{equation}
Since $15 n \dt_2 \leq 1$ by~(\ref{Eq_1619_4}), we get
\begin{equation}\label{Eq_1619_B55}
\| c_0^* c_0 \| \leq 1 + \| 1 - r \| \leq 2
\andeqn
\| c_0 \| \leq \sqrt{2} < 2.
\end{equation}
It also follows that
\[
\| r c_0^* c_0 r \|
 = \| r [c_0^* c_0 - (1 - r)] r \|
 \leq \| 1 - r - c_0^* c_0 \|
 < 15 n \dt_2,
\]
so, using (\ref{Eq_1619_4}) at the last step,
\begin{equation}\label{Eq_1619_B61}
\| c_0 (1 - r) - c_0 \|
 = \| c_0 r \|
 < \sqrt{15 n \dt_2}
 \leq \dt_1.
\end{equation}
Furthermore, using (\ref{Eq_1X21_c0Dfn}) at the first step,
(\ref{Eq_1619_B54}) and~(\ref{Eq_1619_B51}) at the second last step,
and (\ref{Eq_1619_4}) at the last step,
\begin{equation}\label{Eq_1619_B62}
\begin{split}
\| r c_0 - c_0 \|
& \leq \frac{1}{\sqrt{n}}
    \sum_{g \in G} \| r \af_g (s_0) - \af_g (s_0) \|
\\
& \leq \frac{1}{\sqrt{n}} \sum_{g \in G}
     \biggl( \| r_g \af_g (s_0) - \af_g (s_0) \|
      + \sum_{h \in G \SM \{ g \}} \| r_h \af_g (s_0) \| \biggr)
\\
& < \sqrt{n} (5 \dt_2 + (n - 1) \cdot 5 \dt_2)
 = 5 n^{3 / 2} \dt_2
 \leq \dt_1.
\end{split}
\end{equation}

Let $a \in F$.
By the choice of~$m$, for all $g \in G$ there is $b \in A_m$ such that
$\| \af_g^{-1} (a) - b \| < \dt$,
so $s_0 \in A_m' \cap A$ implies
$\| \af_g^{-1} (a) s_0 - s_0  \af_g^{-1} (a) \| < 2 \dt$.
Thus, by~(\ref{Eq_1X21_c0Dfn}),
\begin{equation}\label{Eq_1619_B63}
\begin{split}
\| a c_0 - c_0 a \|
& \leq \frac{1}{\sqrt{n}}
    \sum_{g \in G} \| a \af_g (s_0) - \af_g (s_0)a  \|
\\
& = \frac{1}{\sqrt{n}}
    \sum_{g \in G} \| \af_g^{-1} (a) s_0 - s_0 \af_g^{-1} (a) \|
  < 2 \sqrt{n} \dt.
\end{split}
\end{equation}

Setting $r_0 = 1 - r$, giving $r_0^2 = r_0$,
and using (\ref{Eq_1619_B52}) and~(\ref{Eq_1619_B55})
at the second last step, we get
\[
\begin{split}
\| (c_0^* c_0)^2 - c_0^* c_0 \|
& \leq \| c_0^* c_0 - r_0 \| \| c_0^* c_0 \|
              + \| r_0 \| \| c_0^* c_0 - r_0 \|
              +  \| c_0^* c_0 - r_0 \|
\\
& < 2 (15 n \dt_2) + 15 n \dt_2 + 15 n \dt_2
 = 60 n \dt_2.
\end{split}
\]
By the choice of~$\dt_2$, there is a \pj{} $p_0 \in A^{\af}$
such that
\begin{equation}\label{Eq_2116_NSt}
\| p_0 - c_0^* c_0 \| < \dt_1.
\end{equation}
Set $p = 1 - p_0$.
Using (\ref{Eq_1619_B52}) at the second last step
and~(\ref{Eq_1619_4}) at the last step, we get
\begin{equation}\label{Eq_1619_B71}
\begin{split}
\| p - r \|
& = \| p_0 - r_0 \|
  \leq \| p_0 - c_0^* c_0 \| + \| c_0^* c_0 - r_0 \|
\\
& < \dt_1 + 15 n \dt_2
  \leq 2 \dt_1.
\end{split}
\end{equation}
Define $c = p c_0 (1 - p)$.
Then
\begin{equation}\label{Eq_1619_B72}
\| c \| \leq 2
\end{equation}
by~(\ref{Eq_1619_B55}).
We have, using (\ref{Eq_1619_B71}), (\ref{Eq_1619_B55}),
and (\ref{Eq_1619_B62}) at the second step,
\[
\| p c_0 - c_0 \|
 \leq \| p - r \| \| c_0 \| + \| r c_0 - c_0 \|
 < 2 \dt_1 \cdot 2 + \dt_1
 = 5 \dt_1,
\]
and, using (\ref{Eq_1619_B61}) in place of~(\ref{Eq_1619_B62}),
\[
\| c_0 (1 - p) - c_0 \|
 \leq \| c_0 \| \| p - r \| + \| c_0 (1 - r) - c_0 \|
 < 2 \cdot 2 \dt_1 + \dt_1
 = 5 \dt_1.
\]
Therefore
\begin{equation}\label{Eq_1619_B77}
\| c - c_0 \|
 \leq \| p \| \| c_0 (1 - p) - c_0 \| + \| p c_0 - c_0 \|
 < 10 \dt_1.
\end{equation}
Using $1 - p = p_0$ at the first step
and (\ref{Eq_2116_NSt}), (\ref{Eq_1619_B77}), (\ref{Eq_1619_B55}),
and (\ref{Eq_1619_B72}) at the second step,
\[
\begin{split}
\| 1 - p - c^* c \|
& \leq \| p_0 - c_0^* c_0 \| + \| c_0^* \| \| c_0 - c \|
    + \| c_0^* - c^* \| \| c \|
\\
& < \dt_1 + 2 \cdot 10 \dt_1 + 10 \dt_1 \cdot 2
 = 41 \dt_1.
\end{split}
\]
By the choice of~$\dt_1$, it makes sense to evaluate $(c^* c)^{- 1/2}$
in $(1 - p) A^{\af} (1 - p)$, and, moreover,
if we define $s = c (c^* c)^{- 1/2} \in A^{\af}$, then
\begin{equation}\label{Eq_1619_B87}
s^* s = 1 - p
\andeqn
\| s - c \| < \dt_0.
\end{equation}

It follows from~(\ref{Eq_1619_B71}) and the choice of~$\dt_1$
that there is a unitary $u \in A$ such that
$u r u^* = p$ and $\| u - 1 \| < \dt_0$.
Define $p_g = u r_g u^*$ for $g \in G$.
We claim that $(p_g)_{g \in G}$, $p$, and~$s$ satisfy the conditions
in the conclusion of the proposition.
Condition~(\ref{Item_1619_tRp_s_p}) is immediate
from (\ref{Eq_2116_NNSt}).
To prove (\ref{Item_1619_tRp_s_p_comm}) and~(\ref{Item_1619_tRp_s_CM}),
we start by observing that for $g \in G$ we have
\begin{equation}\label{Eq_1619_B96}
\| p_g - r_g \| \leq 2 \| u - 1 \| < 2 \dt_0.
\end{equation}
Now let $g \in G$ and let $a \in F$.
Choose $b \in A_m$
such that $\| b - a \| < \dt_0$.
Since $r_g \in A_m' \cap A$ and $\| a \| \leq 1$,
and using (\ref{Eq_1619_B12}) at the last step,
\[
\| p_g a - a p_g \|
 \leq 2 \| p_g - r_g \| + 2 \| a - b \|
 < 4 \dt_0 + 2 \dt_0
 \leq \ep.
\]
This is~(\ref{Item_1619_tRp_s_p_comm}).
Also, for $g, h \in G$, at the second step
using~(\ref{Eq_1619_B96}) on the first and last terms,
(\ref{Eq_1619_B91}) on the second and fourth terms,
and (\ref{Eq_1619_B90}) on the middle term,
at the second last step using
(\ref{Eq_1619_1}), (\ref{Eq_1619_4}), and~(\ref{Eq_1619_B11}),
and at the last step using~(\ref{Eq_1619_B12}),
\[
\begin{split}
& \| \af_g (p_h) - p_{g h} \|
\\
& \hspace*{3em} {\mbox{}}
\leq \| p_h - r_h \| + \| r_h - q_h \| + \| \af_g (q_h) - q_{g h} \|
           + \| q_{g h} - r_{g h} \| + \| r_{g h} - p_{g h} \|
\\
& \hspace*{3em} {\mbox{}}
  < 2 \dt_0 + 2 \dt_2 + \dt + 2 \dt_2 + 2 \dt_0
  \leq 9 \dt_0
  \leq \ep.
\end{split}
\]
Condition~(\ref{Item_1619_tRp_s_CM}) is proved.

For Condition~(\ref{Item_1619_tRp_s_sub_x}),
use (\ref{Eq_1619_B71}) and~(\ref{Eq_1619_2}) at the third step,
(\ref{Eq_1619_4}) at the fourth step,
and
(\ref{Eq_1619_B11}) and~(\ref{Eq_1619_B12}) at the fifth step,
to get
\[
\begin{split}
\| (1 - p) - (1 - q) \|
& = \| p - q \|
  \leq \| p - r \| + \| r - q \|
\\
& < 2 \dt_1 + 2 n \dt_2
  \leq 3 \dt_1
  \leq 1.
\end{split}
\]
Therefore, using~(\ref{Eq_1619_B22})
and $f \in {\ov{x A x}}$ at the second step,
$1 - p \sim 1 - q \precsim_A x$.
To prove Condition~(\ref{Item_1619_tRp_s_pxp}),
observe that, using $\| p -  q \| < 3 \dt_1$
(which is part of the calculation above)
and (\ref{Eq_1619_B23}) at the second step,
(\ref{Eq_1619_1}) and~(\ref{Eq_1619_4}) at the third step,
and (\ref{Eq_1619_B11}) and~(\ref{Eq_1619_B12}) at the last step,
\[
\| p x p \|
 \geq \| q x q \| - 2 \| p - q \|
 > 1 - 3 \dt - 6 \dt_1
 \geq 1 - 7 \dt_1
 \geq 1 - \ep.
\]
For Condition~(\ref{Item_1619_tRp_s_SubP}),
we have $s^* s = 1 - p$ by~(\ref{Eq_1619_B87}),
and $s s^* \leq p$ because $p c = c$ implies $p s = s$.
Finally, we check Condition~(\ref{Item_1619_tRp_s_Comm}).
For $a \in F$, we have,
using $\| a \| \leq 1$ at the first step,
using (\ref{Eq_1619_B87}), (\ref{Eq_1619_B77}),
and~(\ref{Eq_1619_B63}) at the second step,
and using (\ref{Eq_1619_1}), (\ref{Eq_1619_4}),
(\ref{Eq_1619_B11}) and~(\ref{Eq_1619_B12}) at the third step,
\[
\begin{split}
\| s a - a s \|
& \leq 2 \| s - c \| + 2 \| c - c_0 \| + \| c_0 a - a c_0 \|
\\
& < 2 \dt_0 + 20 \dt_1 + 2 \sqrt{n} \dt
  \leq \ep.
\end{split}
\]
This completes the proof.
\end{proof}

\section{An action of a totally disconnected compact
 group on a UHF~algebra}\label{Sec_3749_Exam_TRPZ2}

In this section we construct an action
of a totally disconnected infinite compact group on
a UHF~algebra which has the tracial Rokhlin property with comparison
and the strong modified tracial Rokhlin property,
but does not have the Rokhlin property,
or even finite Rokhlin dimension.
In the next section, we construct an action of $S^1$
on a simple AT~algebra which has the same properties,
except that we only prove the modified tracial Rokhlin property,
and in Section~\ref{Sec_2114_OI} we construct an action of $S^1$
on $\OI$ which has the \trpc.
Also, for any \fd{} second countable compact group~$G$,
there is an example of a simple unital AH~algebra $A$ with $\rc (A) > 0$
and an action of $G$ on~$A$ which has the Rokhlin property.
These examples will appear elsewhere.

We abbreviate $\Z / n \Z$ to $\Z_n$; the $p$-adic integers will
not appear in this paper.
The group is $G = \prod_{n = 1}^{\infty} \mathbb{Z}_{2}$,
and the action is the infinite tensor product of copies of the same
action of $\Z_2$ on the $3^{\infty}$~UHF algebra.
We give the example in Construction~\ref{Cn_1X17_Z2Inf},
and prove its properties in several results afterwards.

\begin{cns}\label{Cn_1X17_Z2Inf}
We start with a slight reformulation
of Example 10.4.8 of~\cite{lecturephill}.
For $k \in \N$, set $r (k) = \frac{1}{2} (3^{k} - 1)$.
Define $w_{k} \in {\operatorname{U}} (M_{3^{k}})$
to be the block unitary
\begin{equation}\label{Eq_1X17_wDfn}
w_{k}
 = \left( \begin{array}{ccc} 0 & 1_{M_{r (k)}} & 0 \\
   1_{M_{r (k)}} & 0 & 0 \\
   0 & 0 & 1_{\mathbb{C}} \end{array} \right) \in M_{3^{k}}.
\end{equation}
Set $B = \bigotimes_{k = 1}^{\infty} M_{3^{k}}$,
which is the $3^{\infty}$~UHF.
Define
\[
\nu = \bigotimes_{k = 1}^{\infty} \Ad (w_{k}) \in \Aut (B),
\]
which is an automorphism of order~$2$.
Let $\gamma \colon \Z_2 \to \Aut (B)$
be the product type action action generated by~$\nu$.

Define $G = \prod_{n = 1}^{\I} \Z_2$
and $A = \bigotimes_{n = 1}^{\infty} B$.
Let $\af \colon G \to \Aut (A)$ be the infinite tensor product action,
determined by
\[
\af_{(h_1, h_2, \ldots)} (b_1 \otimes b_2 \cdots )
 = \gm_{h_1} (b_1) \otimes \gm_{h_2} (b_2) \otimes \cdots
\]
for $h_1, h_2, \ldots \in \Z_2$ and
$b_1, b_2, \ldots \in B$ with $b_n = 1$ for all but finitely many
$n \in \N$.
\end{cns}

To work effectively with this example, we set up some useful notation.

\begin{ntn}\label{N_1X17_Parts}
Given the notation in Construction~\ref{Cn_1X17_Z2Inf}, make
the following further definitions.
For $n \in \N$ set $B_n = B$,
so that $A = \bigotimes_{m = 1}^{\infty} B_m$,
and set $A_n = \bigotimes_{m = 1}^{n} B_m$, so that $A = \dirlim_n A_n$.
For $n, k \in \N$ set $C_{n, k} = M_{3^{k}}$,
so that $B_n = \bigotimes_{k = 1}^{\infty} C_{n. k}$,
and set $B_{n, l} = \bigotimes_{k = 1}^{l} C_{n. k}$,
so that $B_n = \dirlim_k B_{n, k}$.
Further set $A_{n, l} = \bigotimes_{k = 1}^{l} B_{n, l}$.
We identify $A_{n}$ and $A_{n, l}$ with their images in~$A$,
and $B_{n, k}$ with its image in $B_n$.

Treat $G$ similarly: for $n \in \N$ set $H_n = \Z_2$,
so that $G = \prod_{m = 1}^{\I} H_m$,
and set $G_n = \prod_{m = 1}^{n} H_m$, so that $G = \invlim_n G_n$.
This gives
\[
C (G_n) = \bigotimes_{m = 1}^{n} C (H_m),
\andeqn
C (G) = \dirlim_n C (G_n) = \bigotimes_{m = 1}^{\I} C (H_m).
\]
We identify $C (G_n)$ with its image in $C (G)$.
\end{ntn}

As an informal overview, write
\[
A = \bigotimes_{m = 1}^{\infty}
   \left( \bigotimes_{k = 1}^{\infty} M_{3^k} \right)
  = \bigotimes_{m = 1}^{\infty}
   \left( \bigotimes_{k = 1}^{\infty} C_{m, k} \right).
\]
Then:
\begin{itemize}
\item
$C_{n, l}$ uses the $(n, l)$ tensor factor.
\item
$B_n$ uses the $(n, k)$ tensor factors for $k \in \N$.
\item
$B_{n, l}$ uses the $(n, k)$ tensor factors for $k = 1, 2, \ldots, l$.
\item
$A_n$ uses the $(m, k)$ tensor factors
for $m = 1, 2, \ldots, n$ and $k \in \N$.
\item
$A_{n, l}$ uses the $(m, k)$ tensor factors
for $m = 1, 2, \ldots, n$ and $k = 1, 2, \ldots, l$.
\end{itemize}

\begin{lem}\label{L_1X17_Tr_1mp}
Let $n \in \N$, let $A_1, A_2, \ldots, A_n$ be unital \ca{s},
and for $m = 1, 2, \ldots, n$ let $e_m \in A_m$ be a \pj{}
and let $\ta_m$ be a \tst{} on~$A_m$.
Let $A = A_1 \otimes A_2 \otimes \cdots \otimes A_n$
(minimal tensor product), and set
\[
e = e_1 \otimes e_2 \otimes \cdots \otimes e_n \in A
\andeqn
\ta = \ta_1 \otimes \ta_2 \otimes \cdots \otimes \ta_n \in \T (A).
\]
Then
\[
\ta (1 - e) \leq \sum_{m = 1}^{n} \ta_m (1 - e_m).
\]
\end{lem}

\begin{proof}
For $m = 1, 2, \ldots, n$ set $\ld_m = \ta_m (e_m) \in [0, 1]$.
We need to show that
\begin{equation}\label{Eq_1X17_SumProd}
1 - \prod_{m = 1}^{n} \ld_m \leq \sum_{m = 1}^{n} (1 - \ld_m).
\end{equation}
We do this by induction on~$n$.
The case $n = 1$ is immediate.
For $n = 2$, the relation~(\ref{Eq_1X17_SumProd}) becomes
\[
1 - \ld_1 \ld_2 \leq 2 - \ld_1 - \ld_2.
\]
This is equivalent to $(1 - \ld_1) (1 - \ld_2) \geq 0$,
so the case $n = 2$ holds.

Assume now~(\ref{Eq_1X17_SumProd}) holds for some $n \geq 2$,
and $\ld_1, \ld_2, \ldots \ld_{n + 1} \in [0, 1]$.
Set $\mu = \prod_{m = 1}^{n} \ld_m$.
Then $\mu \in [0, 1]$.
Using the case $n = 2$ at the second step
and the induction hypothesis at the third step, we get
\[
1 - \prod_{m = 1}^{n + 1} \ld_m
 = 1 - \mu \ld_{n + 1}
 \leq (1 - \ld_{n + 1}) + (1 - \mu)
 \leq \sum_{m = 1}^{n + 1} (1 - \ld_m).
\]
This completes the proof.
\end{proof}

\begin{lem}\label{L_1X17_1Step}
Let the notation be as in Construction~\ref{Cn_1X17_Z2Inf}.
Let $k \in \N$.
Then there are isomorphisms
\[
(M_{3^{k}})^{\Ad (w_{k})} \cong M_{r (k)} \oplus M_{r (k) + 1}
\]
and
\[
(M_{3^{k}} \otimes M_{3^{k + 1}})^{\Ad (w_{k}) \otimes \Ad (w_{k + 1})}
  \cong M_{r (k^2 + k)} \oplus M_{r (k^2 + k) + 1}.
\]
The first isomorphism sends the \pj{s}
\[
e_0 = \left( \begin{array}{ccc} 1_{M_{r (k)}} & 0 & 0 \\
   0 & 0 & 0 \\
   0 & 0 & 0 \end{array} \right)
\andeqn
e_1 = \left( \begin{array}{ccc} 0 & 0 & 0 \\
   0 & 1_{M_{r (k)}} & 0 \\
   0 & 0 & 0 \end{array} \right).
\]
(using the same block matrix decomposition as in~(\ref{Eq_1X17_wDfn}))
to a \pj{} of rank $r (k)$ in $M_{r (k) + 1}$ and
to the identity of $M_{r (k)}$ respectively.
The map
\[
\rh \colon (M_{3^{k}})^{\Ad (w_{k})}
 \to
 (M_{3^{k}} \otimes M_{3^{k + 1}})^{\Ad (w_{k}) \otimes \Ad (w_{k + 1})}
\]
induced by $a \mapsto a \otimes 1$ induces maps
$\rh_{i, j} \colon M_{r (k) + i} \to M_{r (k^2 + k) + j}$
for $i, j \in \{ 0, 1 \}$,
and the corresponding partial embedding multiplicities $m_k (i, j)$ are
given by
\[
m_k (0, 0) = m_k (1, 1) = r (k + 1) + 1
\andeqn
m_k (0, 1) = m_k (1, 0) = r (k + 1).
\]
\end{lem}

\begin{proof}
For any $k \in \N$,
it is easy to check that $w_{k}$ is unitarily equivalent to
\[
v_{k} = \mathrm{diag} (1, 1, \ldots, 1, -1, -1, \ldots, -1)
 \in M_{3^{k}},
\]
in which the diagonal entry $1$ occurs $r (k) + 1$
times and the diagonal entry $-1$ occurs $r (k)$ times.
Therefore we can prove the lemma with $v_k$ and $v_{k + 1}$
in place of $w_k$ and $w_{k + 1}$.
With this change, for example, the map
$\nu \colon M_{r (k)} \oplus M_{r (k) + 1} \to M_{3^{k}}$,
given by $\nu (a_0, a_1) = \diag (a_1, a_0)$,
is easily seen to be an isomorphism
from $M_{r (k)} \oplus M_{r (k) + 1}$ to $(M_{3^{k}})^{\Ad (w_{k})}$.
The rest of the proof is a computation with diagonal matrices
and the dimensions of their eigenspaces, and is omitted.
\end{proof}

\begin{thm}\label{T_1X17_Has_trpc}
The action $\af \colon G \to \Aut (A)$
of Construction~\ref{Cn_1X17_Z2Inf}
has the tracial Rokhlin property with comparison
and the strong modified tracial Rokhlin property,
using the same choices of $p \in A$ and $\ph \colon G \to p A p$.
\end{thm}

\begin{proof}
Let $F \subseteq A$ and $S \subseteq C (G)$
be finite sets, let $\varepsilon > 0$,
let $x \in A_{+} \setminus \{0 \}$ with $\big\| x \big\| = 1$,
and let $y \in A^{\alpha}_{+} \setminus \{0 \}$.
Without loss of generality we can assume
$\| a \| \leq 1$ for all $a \in F$, $\| f \| \leq 1$ for
all $f \in S$, and $\varepsilon < 1$.
According to Definition \ref{traR}, we need to find
a projection $p \in A^{\alpha}$ and a \ucp{}
$\varphi \colon C (G) \to p A p$ such that the following hold.
\begin{enumerate}
\item\label{Item_3954_approx_FS}
$\varphi$ is an $(F, S, \varepsilon)$-approximately equivariant
 central multiplicative map.
\item\label{Item_3958_1mp_sub_x}
$1 - p \precsim_{A} x$.
\item\label{Item_3956_1mp_sub_y}
$1 - p \precsim_{A^{\alpha}} y$.
\item\label{Item_1mp_fixed_sub_p}
$1 - p \precsim_{A^{\alpha}} p$.
\item\label{Item_pxp_no0}
$\| p x p \| > 1 - \varepsilon$.
\setcounter{TmpEnumi}{\value{enumi}}
\end{enumerate}
According to Definition \ref{moditra}, we also need to find a partial
isometry $s \in A^{\alpha}$ such that the following hold.
\begin{enumerate}
\setcounter{enumi}{\value{TmpEnumi}}
\item\label{Item_3971_s_p_1mp}
$s^{*} s = 1 - p$ and $s s^{*} \leq p$.
\item\label{Item_3973_comm_saas}
$\| s a - a s \| < \varepsilon$ for all $a \in F$.
\end{enumerate}

For the same reason as in the proof of Lemma~\ref{L_1X09_NoNorm1}
(and by Lemma~\ref{L_1X09_NoNorm1} if we are only proving that $\af$
has the \trpc),
we can ignore~(\ref{Item_pxp_no0}).
We can also ignore~(\ref{Item_1mp_fixed_sub_p}),
since it follows from~(\ref{Item_3971_s_p_1mp}).

Since $A$ is a UHF~algebra,
there is a nonzero \pj{} $q_1 \in \ov{x A x}$,
and the unique \tst~$\ta$ on~$A$ satisfies $\ta (q_1) > 0$.
Set $\dt_1 = \ta (q_1)$.

The action $\gm$ has the tracial Rokhlin property by Example 10.4.8
of \cite{lecturephill}.
Therefore $C^* (\Z_2, B, \gm)$ is simple by
Corollary 1.6 of~\cite{phill23},
so $B^{\gm}$ is simple by Theorem~\ref{satunonabel}.
Since $\gm$ is a direct limit action, $B^{\gm}$ is an AF~algebra.
It is easy to check that $A^{\af}$ can be identified with
$\bigotimes_{m = 1}^{\infty} B_m^{\gm}$.
It follows that $A^{\af}$ is an AF~algebra, which is
simple because it is an infinite tensor product of simple \ca{s}.
Therefore there is a nonzero \pj{} $q_2 \in \ov{y A^{\af} y}$,
and the number $\dt_2 = \inf_{\ta \in T (A^{\af})} \ta (q_2)$
satisfies $\dt_2 > 0$.

Following Notation~\ref{N_1X17_Parts},
$\bigcup_{n = 1}^{\I} A_n$ is dense in~$A$ and, for every $n \in \N$,
$\bigcup_{l = 1}^{\I} B_{n, l}$ is dense in $B_n$.
Therefore there are $N_1, L_0 \in \N$
and a finite subset $F_0 \subseteq A_{N_1, L_0} \subseteq A$
such that for every
$a \in F$ there is $b \in F_0$ with $\| a - b \| < \frac{\ep}{4}$,
and also $\| b \| \leq 1$ for all $b \in F_0$.
Similarly, there are $N_2 \in \N$
and a finite subset $S_0 \S C (G_{N_2}) \subseteq C (G)$
such that for every
$f \in F$ there is $c \in F_0$ with $\| f - c \| < \frac{\ep}{4}$
and also $\| c \| \leq 1$ for all $c \in S_0$.
Set $N = \max (N_1, N_2)$, and choose $L \in \N$ so large that
\[
L \geq L_0
\andeqn
\frac{2 N}{3^{L}} < \min \left( \dt_1, \dt_2, \frac{1}{2} \right).
\]

Let $e_0, e_1 \in M_{3^{L + 1}}$ be as in Lemma~\ref{L_1X17_1Step},
with $k = L + 1$.
For $m = 1, 2, \ldots, N$ define the \pj{s}
\[
e_0^{(m)} = 1_{B_{m, L}} \otimes e_0, \, \,
 e_1^{(m)}= 1_{B_{m, L}} \otimes e_1
 \in B_{m, L} \otimes M_{3^{L + 1}}
 = B_{m, L + 1}
 \S B_m.
\]
Identify $H_m = \Z_2 = \{ 0, 1 \}$ with addition modulo~$2$,
and for
$h = (h_1, h_2, \ldots, h_{N}) \in G_{N} = \prod_{m = 1}^{N} H_m$,
set
\[
e_h = e_{h_1}^{(1)} \otimes e_{h_2}^{(2)}
   \otimes \cdots \otimes e_{h_{N}}^{(N)}.
\]
These are \mops.
Define $p = \sum_{h \in G_{N}} e_h \in A_{n, L + 1} \S A$.
There is a unital \hm{}
$\ph_0 \colon C (G_{N}) \to p A_{n, L + 1} p \S p A p$
given by $\ph_0 (f) = \sum_{h \in G_{N}} f (h) e_h$
for $f \in C (G_{N})$.

Set $K = \prod_{m = N + 1}^{\I} H_m$, so that $G = G_{N} \times K$.
Let $\mu$ be normalized Haar measure on~$K$.
Then there is a conditional expectation
$P \colon C (G) \to C (G_{N})$,
given by $P (f) (h) = \int_K f (h, g) \, d \mu (g)$
for $f \in C (G)$ and $h \in G_{N}$.
The identification of $C (G_{N})$ as a subalgebra of $C (G)$
makes it invariant under the action $g \mapsto \Lt_g$
of $G$ on $C (G)$,
and, with this identification, $P$ is equivariant.
Also, it is easily checked that $p \in A^{\af}$ and that
$\ph_0$ is $G$-equivariant.
Therefore $\ph = \ph_0 \circ P \colon C (G) \to p A p$
is an equivariant \ucp{} map.

Since $F_0 \subseteq A_{N, L}$,
it follows that $\ph (f)$ exactly commutes with all elements of $F_0$.
Also, $\ph (c_1 c_2) = \ph (c_1) \ph (c_2)$ for all $c_1, c_2 \in S_0$,
since $S_0 \S C (G_{N}) \subseteq C (G)$.
Let $a \in F$ and let $f_1, f_2 \in S$.
Choose $b \in F_0$ and $c_1, c_2 \in S_0$
such that
\[
\| b - a \| < \frac{\ep}{4},
\qquad
\| c_1 - f_1 \| < \frac{\ep}{4},
\andeqn
\| c_2 - f_2 \| < \frac{\ep}{4}.
\]
Since all elements of $F$, $F_0$, $S$, and $S_0$ have norm at most~$1$,
using $\ph (c_1) b = b \ph (c_1)$ we get
\[
\| \ph (f_1) a - a \ph (f_1) \|
  \leq 2 \| b - a \| + 2 \| f_1 - c_1 \|
  < \ep.
\]
Similarly,
\[
\| c_1 c_2 - f_1 f_2 \| < \frac{\ep}{2}
\andeqn
\| \ph (c_1) \ph (c_2) - \ph (f_1) \ph (f_2) \| < \frac{\ep}{2},
\]
so $\| \ph (f_1 f_2) - \ph (f_1) \ph (f_2) \| < \ep$.
We have proved~(\ref{Item_3954_approx_FS}).

For $m = 1, 2, \ldots, n$, set $p_m = e_0^{(m)} + e_0^{(m)} \in A_m$.
Then $p_m$ is the image in $A_m$
of a \pj{} $z_m \in C_{m, L + 1} = M_{3^{L + 1}}$
such that $1 - z_m$ has rank~$1$.
Therefore the unique \tst{} $\ta_m$ on $A_m$ satisfies
$\ta (1 - p_m) = 3^{- (L + 1)}$.
Since $p = p_1 \otimes p_2 \otimes \cdots \otimes p_N$,
by Lemma~\ref{L_1X17_Tr_1mp} we have
\[
\ta (1 - p) \leq \frac{N}{3^{L + 1}} < \dt_1 = \ta (q_1).
\]
Since UHF~algebras have strict comparison,
we get $1 - p \precsim_A q_1 \precsim_A x$,
which is~(\ref{Item_3958_1mp_sub_x}).

For the remaining conditions,
for convenience set $T = (L + 1) (L + 2)$.
Let $e_0$, $e_1$, and $\rh$ be as in Lemma~\ref{L_1X17_1Step},
with $k = L + 1$, and let the components of $\rh$
in its codomain be
\[
\rh_j \colon M_{r (L + 1)} \oplus M_{r (L + 1) + 1} \to M_{r (T) + j}
\]
for $j = 0, 1$.
Using the ranks and partial embedding multiplicities
given in Lemma~\ref{L_1X17_1Step}, we see that
$\rh_0 (1 - e_0 - e_1) \in M_{r (T)}$ has rank $r (L + 2)$
and $\rh_1 (1 - e_0 - e_1) \in M_{r (T) + 1}$ has rank $r (L + 2) + 1$.
The normalized traces of these are
\begin{equation}\label{Eq_1X21_NTraces}
\frac{r (L + 2)}{r (T)} < \frac{2}{3^{L + 1}}
\andeqn
\frac{r (L + 2) + 1}{r (T) + 1} < \frac{2}{3^{L + 1}}.
\end{equation}

Set
\[
D = (M_{r (L + 1)} \oplus M_{r (L + 1) + 1})^{\otimes N}
\andeqn
E = (M_{r (T)} \oplus M_{r (T) + 1})^{\otimes N},
\]
and consider $\rh^{\otimes N} \colon D \to E$.
We can write
\[
E = \bigoplus_{j \in \{ 0, 1 \}^N}
        M_{r (T) + j_1} \otimes M_{r (T) + j_2} \otimes
          \cdots \otimes M_{r (T) + j_N}.
\]
Call the $j$~tensor factor $E_j$.
For $j = (j_1, j_2, \ldots, j_N) \in \{ 0, 1 \}^N$, let
$d_j$ be the image in $E_j$ of the corresponding summand of
$\rh^{\otimes N} \bigl( (e_0 + e_1)^{\otimes N} \bigr)$.
Thus
\[
d_j = \rh_{j_1} (e_0 + e_1) \otimes \rh_{j_2} (e_0 + e_1) \otimes
         \cdots \otimes \rh_{j_N} (e_0 + e_1).
\]
By Lemma~\ref{L_1X17_Tr_1mp} and~(\ref{Eq_1X21_NTraces}),
$1 - d_j$ has normalized trace less than $2 N \cdot 3^{- L - 1}$.

Use Lemma~\ref{L_1X17_1Step} to identify $D$ and $E$
with the subalgebras
\[
\left( \bigotimes_{m = 1}^N C_{m, L + 1} \right)^{\af}
\andeqn
\left( \bigotimes_{m = 1}^N
  C_{m, L + 1} \otimes C_{m, L + 2} \right)^{\af},
\]
in such a way that $(e_0 + e_1)^{\otimes N}$ is identified with~$p$.
Under this identification, $E$ commutes exactly with all elements
of $A_{N, L}$, hence with all elements of~$F_0$.

Since $2 N \cdot 3^{- L - 1} < \frac{1}{2}$,
we conclude that $1 - d_j \precsim_{E_j} d_j$.
Therefore $1 - p \precsim_E p$, that is,
there is $s \in E$ such that $s^{*} s = 1 - p$ and $s s^{*} \leq p$.
We have $s \in A^{\alpha}$ since $E \S A^{\alpha}$.
Also, $s$ commutes exactly with all elements of~$F_0$,
so $\| a s - s a \| < \frac{\ep}{2}$ for all $a \in F$.
We have verified Conditions (\ref{Item_3971_s_p_1mp})
and~(\ref{Item_3973_comm_saas}).

Since $2 N \cdot 3^{- L - 1} < \dt_2$, for every $\sm \in \T (E)$
we have $\sm (1 - p) < \dt_2$.
Every \tst{} $\ta \in \T (A^{\af})$ restricts to a \tst{} on~$E$,
so $\ta (1 - p) < \dt_2$ for all $\ta \in \T (A^{\af})$.
Since simple AF~algebras have strict comparison,
we get $1 - p \precsim_{A^{\af}} q_2$.
Condition~(\ref{Item_3956_1mp_sub_y}) follows.
\end{proof}

\begin{prp}\label{P_1X17_NoDimR}
The action $\af \colon G \to \Aut (A)$
of Construction~\ref{Cn_1X17_Z2Inf}
does not have finite Rokhlin dimension with commuting towers.
\end{prp}

\begin{proof}
Suppose $\af$ has finite Rokhlin dimension with commuting towers.
Then Proposition~3.10 of~\cite{Gar_rokhlin_2017} implies that
the action on $A$ of the first factor of~$G$,
called $H_1$ in Notation~\ref{N_1X17_Parts},
also has finite Rokhlin dimension with commuting towers.
However, $H_1 \cong \Z_2$, $A$ is the $3^{\I}$~UHF algebra and, by
Corollary 4.8(2) of~\cite{HrsPh1},
there is no action of $\Z_2$ on the $3^{\I}$~UHF algebra
which has finite Rokhlin dimension with commuting towers.
\end{proof}

\section{An action of $S^1$ on a simple AT~algebra}\label{Sec_1908_Exam_TRPS1}

The purpose of this section is to construct a direct limit action
of the group~$S^1$
on a simple unital AT~algebra which has
the tracial Rokhlin property with comparison
but does not have finite Rokhlin dimension with commuting towers.

The general construction, with unspecified partial embedding
multiplicities (which, for properties we want,
need to be chosen appropriately),
is presented in Construction~\ref{Cn_1824_S1An}.
For the purpose of readability, the properties asserted there,
as well as others needed later, are proved in a series of lemmas.

The algebra $A$ in our construction will be a direct limit of
algebras $A_n$ isomorphic to
$C (S^1, M_{N r_0 (n)}) \oplus C (S^1, M_{r_1 (n)})$.
Up to equivariant isomorphism and exterior equivalence,
the action of $\zt \in S^1$ on $C (S^1, M_{N r_0 (n)})$ is rotation
by $\zt^N$ and its action on $C (S^1, M_{r_1 (n)})$
is rotation by $\zt$.
It is technically convenient to present the first summand in a
different way; the description above
is explicit in Lemma~\ref{L_1824_R_T}.
The action has the tracial Rokhlin property with comparison provided
the image of the summand $C (S^1, M_{N r_0 (n)}) \S A_n$ in $A$
can be made ``arbitrarily small in trace'' by choosing $n$ large enough.

The algebras $B_n$ and~$B$ in parts
(\ref{Item_1824_S1_Bn0Bn1}), (\ref{Item_1824_S1_FixPtSys}),
(\ref{Item_1824_S1_FixLim}), and~(\ref{Item_1915_S1_qn})
of Construction~\ref{Cn_1824_S1An}
are a convenient description of the fixed point algebras of $A_n$
and~$A$; see Lemma~\ref{L_1916_FPisB}.

We say here a little more about the motivation for the construction
and possible extensions.
If $G$ is finite, one can construct a direct limit action of $G$ on an
AF~algebra $\dirlim_n A_n$
by taking $A_n = M_{r_0 (n)} \oplus C (G, M_{r_1 (n)})$.
The action on $C (G, M_{r_1 (n)})$ is essentially
translation by group elements.
The partial map from $C (G, M_{r_1 (n)})$ to $M_{r_0 (n + 1)}$
is the direct sum of the evaluations at the points of~$G$.
The action on $M_{r_0 (n + 1)}$ is inner, and must permute the
images of the maps from $C (G, M_{r_1 (n)})$ appropriately;
this leads to a slightly messy inductive construction of inner actions
of $G$ on the algebras $M_{r_0 (n)}$ and inner perturbations
of the translation actions on $C (G, M_{r_1 (n)})$.

When $G$ is not finite, point evaluations can no longer be used,
since equivariance forces one to use all of them or none of them.
The algebra $C (S^1, M_{N r_0 (n)})$
with the action of rotation by $\zt^N$
is the codomain for a usable substitute for point evaluations.
Something similar to the inductive construction of perturbations
from above is needed, but the messiness can be mostly hidden
by instead using the algebra $R$ as in
Construction \ref{Cn_1824_S1An}(\ref{Item_1824_S1_R}).

Construction~\ref{Cn_1824_S1An} can be generalized in several ways.
One can replace $C (S^1, M_{r_1 (n)})$ with $C (X, M_{r_1 (n)})$
for a compact space~$X$ with a free action of $S^1$.
To ensure simplicity, one will need to incorporate additional
partial maps
in the direct system, which can be roughly described as point
evaluations at points of $X / S^1$.
One can increase the complexity of the K-theory and the departure
from the Rokhlin property by taking
\[
A_n = C (S^1, M_{N_2 N_1 r_0 (n)})
   \oplus C (S^1, M_{N_1 r_1 (n)}) \oplus C (S^1, M_{r_2 (n)})
\]
with actions exterior equivalent to rotations by
$\zt^{N_1 N_2}$, $\zt^{N_1}$, and $\zt$.
One can use more summands, even letting the number of them
approach infinity as $n \to \I$.
One can also replace $S^1$ with $(S^1)^m$.
However, it is not clear how to construct an analogous action
with $S^1$ replaced by a nonabelian connected compact Lie group.

We introduce some notation specifically for this section.

\begin{dfn}\label{D_1908_eue}
Let $G$ be a group, let $A$ and $B$ be \ca{s}, with $B$ unital,
let $\af \colon G \to \Aut (A)$ and $\bt \colon G \to \Aut (B)$
be actions of $G$ on $A$ and~$B$,
and let $\ph, \ps \colon A \to B$ be equivariant \hm{s}.
We say that $\ph$ and $\ps$ are
{\emph{equivariantly unitarily equivalent}}, written $\ph \sim \ps$,
if there is a $\bt$-invariant unitary $u \in B$
such that $u \ph (a) u^* = \ps (a)$ for all $a \in A$.
\end{dfn}

\begin{ntn}\label{N_1908_Ampl}
Let $A$ and $B$ be \ca{s}, and let $\ps \colon A \to B$ be a \hm.
We let $\ps^{(k)} \colon A \to M_k \otimes B$ be the map
$a \mapsto 1_{M_k} \otimes \ps (a)$,
and we define
\[
\ps_n = \id_{M_n} \otimes \ps \colon M_n \otimes A \to M_n \otimes B
\]
and
\[
\ps^{(k)}_n = \id_{M_n} \otimes \ps^{(k)} \colon
   M_n \otimes A \to M_{k n} \otimes B.
\]
\end{ntn}

In particular,
the ``amplification map'' from $M_n (A)$ to $M_{k n} (A)$,
given by $a \mapsto 1_{M_k} \otimes a$, is denoted by $(\id_A)^{(k)}_n$.

\begin{cns}\label{Cn_1824_S1An}
We choose and fix $N \in \N$ with $N \geq 2$, $\te \in \R \SM \Q$,
$r (0) = (r_0 (0), r_1 (0)) \in (\N)^2$,
and, for $n \in \Nz$ and $j, k \in \{ 0, 1 \}$,
numbers $l_{j, k} (n) \in \N$.
We suppress them in the notation for the objects we construct,
and, in later results, we will impose additional restrictions on them.

We then define the following \ca{s}, maps, and actions of~$S^1$.
\begin{enumerate}
\item\label{Item_1824_S1_CS1}
Define $\bt \colon S^1 \to \Aut (C (S^1))$
by $\bt_{\zt} (f) (z) = f (\zt^{-1} z)$ for $\zt, z \in S^1$.
Further, for $n \in \N$, identify $C (S^1, M_n)$ and $M_n (C (S^1))$
with $M_n \otimes C (S^1)$ in the obvious way, and
let $\bt_n \colon S^1 \to \Aut (C (S^1, M_n))$ be given by
$\bt_{\zt, n} = \id_{M_n} \otimes \bt_{\zt}$ for $\zt \in S^1$.
(The order of subscripts in $\bt_{\zt, n}$ is chosen to be
consistent with Notation~\ref{N_1908_Ampl}).
To simplify notation,
for $\ld \in \R$ we define
\[
{\widetilde{\bt}}_{\ld} = \bt_{\exp (2 \pi i \ld)}
\andeqn
{\widetilde{\bt}}_{\ld, n} = \bt_{\exp (2 \pi i \ld), \, n}.
\]
\item\label{Item_1824_S1_s_om}
Define
\[
\om = \exp (2 \pi i / N)
\andeqn
s = \left( \begin{matrix}
  0     &  1     &  0     & \cdots &  0     &  0        \\
  0     &  0     &  1     & \cdots &  0     &  0        \\
 \vdots & \vdots & \ddots & \ddots & \vdots & \vdots    \\
 \vdots & \vdots & \ddots & \ddots &  1     &  0        \\
  0     &  0     & \cdots & \cdots &  0     &  1        \\
  1     &  0     & \cdots & \cdots &  0     &  0
\end{matrix} \right)
\in M_N.
\]
\item\label{Item_1824_S1_R}
Define
\[
R = \bigl\{ f \in C (S^1, M_N) \colon
 {\mbox{$f (\om z) = s f (z) s^*$ for all $z \in S^1$}} \bigr\}.
\]
Then $R$ is invariant under the action $\bt_N$ of $S^1$
on $C (S^1, M_N)$ above.
(See Lemma~\ref{L_1824_R_inv} below.)
We define $\gm \colon S^1 \to \Aut (R)$ to be the restriction
of this action.
Further, for $n \in \N$,
let $\gm_n \colon S^1 \to \Aut (M_n \otimes R)$ be the action
$\gm_{\zt, n} = \id_{M_n} \otimes \gm_{\zt}$ for $\zt \in S^1$.
Finally, for $\ld \in \R$ we define
\[
{\widetilde{\gm}}_{\ld} = \gm_{\exp (2 \pi i \ld)}
\andeqn
{\widetilde{\gm}}_{\ld, n} = \gm_{\exp (2 \pi i \ld), \, n}.
\]
\item\label{Item_1824_S1_Crs}
Let $\io \colon R \to C (S^1, M_N)$ be the inclusion.
Define $\xi \colon C (S^1) \to R$ by
\[
\xi (f) (z)
 = {\operatorname{diag}} \bigl( f (z), \, f (\om z),
    \, \ldots, \, f (\om^{N - 1} z) \bigr)
\]
for $f \in C (S^1)$ and $z \in S^1$.
\item\label{Item_1824_S1_ll}
For $n \in \N$ write
\[
l (n) = \left( \begin{matrix}
l_{0, 0} (n)   & l_{0, 1} (n)    \\
N l_{1, 0} (n) & 2N l_{1, 1} (n)
\end{matrix} \right).
\]
For $n \in \N$ inductively define,
starting with $r (0) = (r_0 (0), r_1 (0)) \in (\N)^2$ as at
the beginning of the construction,
\begin{equation}\label{Eq_1901_rn_dfm}
r (n + 1) = l (n) r (n)
\end{equation}
(usual matrix multiplication).
\item\label{Item_1824_S1_An}
For $n \in \Nz$ set
\[
A_{n, 0} = M_{r_0 (n)} (R),
\quad
A_{n, 1} = M_{r_1 (n)} \bigl( C (S^1) \bigr),
\quad {\mbox{and}} \quad
A_n = A_{n, 0} \oplus A_{n, 1}.
\]
Define an action $\af^{(n)} \colon S^1 \to \Aut (A)$
(notation not in line with Notation~\ref{N_1908_Ampl}) by,
now following Notation~\ref{N_1908_Ampl},
$\af^{(n)}_{\zt} = \gm_{\zt, r_0 (n)} \oplus \bt_{\zt, r_1 (n)}$.
\item\label{Item_1824_S1_nu_p}
For $n \in \Nz$ and $j, k \in \{ 0, 1 \}$, define maps
\[
\nu_{n, 0, 0} \colon A_{n, 0} \to M_{l_{0, 0} (n) r_0 (n)} (R),
\qquad
\nu_{n, 0, 1} \colon A_{n, 1} \to M_{l_{0, 1} (n) r_1 (n)} (R),
\]
\[
\nu_{n, 1, 0} \colon A_{n, 0}
 \to M_{N l_{1, 0} (n) r_0 (n)} \bigl( C (S^1) \bigr),
\]
and
\[
\nu_{n, 1, 1} \colon A_{n, 1}
 \to M_{2 N l_{1, 1} (n) r_1 (n)} \bigl( C (S^1) \bigr)
\]
as follows.
Recalling Notation~\ref{N_1908_Ampl}, set
\[
\nu_{n, 0, 0} = (\id_R)_{r_0 (n)}^{l_{0, 0} (n)},
\qquad
\nu_{n, 0, 1} = \xi_{r_1 (n)}^{l_{0, 1} (n)},
\qquad
\nu_{n, 1, 0} = \io_{r_0 (n)}^{l_{1, 0} (n)},
\]
and
\[
\begin{split}
\nu_{n, 1, 1}
& = \diag \left( \btt_{0, \, r_1 (n)}^{l_{1, 1} (n)},
        \, \btt_{1 / N, \, r_1 (n)}^{l_{1, 1} (n)},
        \, \btt_{2 / N, \, r_1 (n)}^{l_{1, 1} (n)},
        \, \ldots,
        \, \btt_{(N - 1) / N, \, r_1 (n)}^{l_{1, 1} (n)}, \right.
\\
& \hspace*{4em} {\mbox{}}
        \left. \, \btt_{\te, \, r_1 (n)}^{l_{1, 1} (n)},
        \, \btt_{\te + 1 / N, \, r_1 (n)}^{l_{1, 1} (n)},
        \, \btt_{\te + 2 / N, \, r_1 (n)}^{l_{1, 1} (n)},
        \, \ldots,
        \, \btt_{\te + (N - 1) / N, \, r_1 (n)}^{l_{1, 1} (n)} \right).
\end{split}
\]
Up to equivariant unitary equivalence, the last one can be written
in the neater form
\[
\begin{split}
& \diag \bigl( \btt_{0},
        \, \btt_{1 / N},
        \, \btt_{2 / N},
        \, \ldots,
        \, \btt_{(N - 1) / N},
\\
& \hspace*{4em} {\mbox{}}
        \, \btt_{\te},
        \, \btt_{\te + 1 / N},
        \, \btt_{\te + 2 / N},
        \, \ldots,
        \, \btt_{\te + (N - 1) / N} \bigr)_{r_1 (n)}^{l_{1, 1} (n)}.
\end{split}
\]
\item\label{Item_1824_S1_nu_A}
Define $\nu_n \colon A_n \to A_{n + 1}$ by
\[
\nu_n (a_0, a_1)
 = \bigl(
  \diag \bigl( \nu_{n, 0, 0} (a_0), \,  \nu_{n, 0, 1} (a_1) \bigr), \,\,
  \diag \bigl( \nu_{n, 1, 0} (a_0), \,  \nu_{n, 1, 1} (a_1) \bigr)
 \bigr)
\]
for $a_0 \in A_{n, 0}$ and $a_1 \in A_{n, 1}$.
For $m, n \in \Nz$ with $m \leq n$ set
\[
\nu_{n, m}
 = \nu_{n - 1} \circ \nu_{n - 2} \circ \cdots \circ \nu_m
 \colon A_m \to A_n.
\]
\item\label{Item_1924_S1_A}
Let $A$ be the direct limit of the system
$\bigl( (A_n)_{n \in \Nz}, \, (\nu_{n, m})_{m \leq n} \bigr)$,
with maps $\nu_{\I, m} \colon A_m \to A$.
Equip $A$ with the direct limit action $\af = \dirlim \af^{(n)}$
of~$S^1$.
This action exists by Lemma~\ref{L_1909_nu_e}.
\item\label{Item_1824_S1_pn}
For $n \in \Nz$ let $p_n = (0, 1) \in A_{n, 0} \oplus A_{n, 1} = A_n$.
\item\label{Item_1824_S1_Bn0Bn1}
For $n \in \Nz$ set
\[
B_{n, 0} = (M_{r_0 (n)})^N = \bigoplus_{k = 0}^{N - 1} M_{r_0 (n)},
\quad
B_{n, 1} = M_{r_1 (n)},
\quad {\mbox{and}} \quad
B_n = B_{n, 0} \oplus B_{n, 1}.
\]
\item\label{Item_1824_S1_FixPtSys}
Let $\mu \colon {\mathbb{C}} \to {\mathbb{C}}^N$ be $\mu (\ld) = (\ld, \ld, \ldots, \ld)$
for $\ld \in {\mathbb{C}}$,
and let $\dt \colon {\mathbb{C}}^N \to M_N$ be
$\dt (\ld_0, \ld_1, \ldots, \ld_{N - 1})
 = \diag (\ld_0, \ld_1, \ldots, \ld_{N - 1})$
for $\ld_0, \ld_1, \ldots, \ld_{N - 1} \in {\mathbb{C}}$.
For $n \in \Nz$ and $j, k \in \{ 0, 1 \}$,
and recalling Notation~\ref{N_1908_Ampl}, define maps
\[
\ch_{n, 0, 0} \colon B_{n, 0} \to M_{l_{0, 0} (n) r_0 (n)} ({\mathbb{C}}^N),
\qquad
\ch_{n, 0, 1} \colon B_{n, 1} \to M_{l_{0, 1} (n) r_1 (n)} ({\mathbb{C}}^N),
\]
\[
\ch_{n, 1, 0} \colon B_{n, 0} \to M_{N l_{1, 0} (n) r_0 (n)},
\andeqn
\ch_{n, 1, 1} \colon B_{n, 1} \to M_{2 N l_{1, 1} (n) r_1 (n)}
\]
by
\[
\ch_{n, 0, 0} = (\id_{{\mathbb{C}}^N})_{r_0 (n)}^{l_{0, 0} (n)},
\qquad
\ch_{n, 0, 1} = \mu_{r_1 (n)}^{l_{0, 1} (n)},
\]
\[
\ch_{n, 1, 0} = \dt_{r_0 (n)}^{l_{1, 0} (n)},
\andeqn
\ch_{n, 1, 1} = (\id_{\mathbb{C}})_{r_1 (n)}^{2 N l_{1, 1} (n)}.
\]
\item\label{Item_1824_S1_FixLim}
Define $\ch_n \colon B_n \to B_{n + 1}$ by
\[
\ch_n (a_0, a_1)
 = \bigl(
  \diag \bigl( \ch_{n, 0, 0} (a_0), \,  \ch_{n, 0, 1} (a_1) \bigr), \,\,
  \diag \bigl( \ch_{n, 1, 0} (a_0), \,  \ch_{n, 1, 1} (a_1) \bigr)
 \bigr)
\]
for $a_0 \in B_{n, 0}$ and $a_1 \in B_{n, 1}$.
For $m, n \in \Nz$ with $m \leq n$ set
\[
\ch_{n, m}
 = \ch_{n - 1} \circ \ch_{n - 2} \circ \cdots \circ \ch_m
 \colon B_m \to B_n.
\]
Let $B$ be the direct limit of the system
$\bigl( (B_n)_{n \in \Nz}, \, (\ch_{n, m})_{m \leq n} \bigr)$,
with maps $\ch_{\I, m} \colon B_m \to B$.
\item\label{Item_1915_S1_qn}
For $n \in \Nz$ let $q_n = (0, 1) \in B_{n, 0} \oplus B_{n, 1} = B_n$.
%
%
%
\end{enumerate}

\end{cns}

\begin{lem}\label{L_1824_R_inv}
Let $R \S C (S^1, M_N)$ be as in
Construction \ref{Cn_1824_S1An}(\ref{Item_1824_S1_R}).
Then $R$ is invariant under the action $\bt_N$ of
Construction \ref{Cn_1824_S1An}(\ref{Item_1824_S1_CS1}),
and map $\io \colon R \to C (S^1, M_N)$
of Construction \ref{Cn_1824_S1An}(\ref{Item_1824_S1_Crs})
is equivariant.
\end{lem}

\begin{proof}
The first part is easy to check from the definitions
of $\bt_N$ and $R$.
Since $\io$ is the inclusion, the second part is immediate.
\end{proof}

\begin{lem}\label{L_1824_phe}
The map $\xi \colon C (S^1) \to R$
of Construction \ref{Cn_1824_S1An}(\ref{Item_1824_S1_Crs})
is well defined and equivariant.
\end{lem}

\begin{proof}
The first part is easy to check just by the definition of $\xi$.
For equivariance, it is enough and immediate
to check on the usual generator of $C (S^1)$.
\end{proof}

\begin{lem}\label{L_1909_nu_e}
The maps $\nu_{n, j, k} \colon A_{n, k} \to A_{n, j}$
of Construction \ref{Cn_1824_S1An}(\ref{Item_1824_S1_nu_p})
and $\nu_n \colon A_n \to A_{n + 1}$
of Construction \ref{Cn_1824_S1An}(\ref{Item_1824_S1_nu_A})
are equivariant.
\end{lem}

\begin{proof}
This is immediate from Lemma~\ref{L_1824_R_inv}, Lemma~\ref{L_1824_phe},
and the fact that $\bt_{\zt_1}$ commutes with $\bt_{\zt_2}$
for $\zt_1, \zt_2 \in S^1$.
\end{proof}

We will need the notation $L_{\Ph}$ from 1.3 of~\cite{DNNP}.
For a \chs~$X$, $m \in \N$, and $x \in X$,
let $\ev_x \colon C (X, M_m) \to M_m$ be evaluation at~$x$.
If also $Y$ is a \chs{} $\Ph \colon C (X, M_m) \to C (Y, M_n)$ is a \hm,
then $L_{\Ph}$ assigns to $y \in Y$ the set of all $x \in X$
such that $\ev_x$ occurs as a summand in the representation
$\ev_y \circ \Ph$.
The definition in 1.3 of~\cite{DNNP} is extended from this case
to \hm{s} between direct sums of algebras of this type.
We refer to that paper for details.

\begin{lem}\label{L_1824_R_T}
Let $R \S C (S^1, M_N)$ be as in
Construction \ref{Cn_1824_S1An}(\ref{Item_1824_S1_R}).
Let ${\overline{\gm}} \colon S^1 \to \Aut (C (S^1))$
be the action
${\overline{\gm}}_{\zt} (f) (z) = f (\zt^{-N} z)$ for $\zt, z \in S^1$.
Then there is an isomorphism $\ps \colon R \to C (S^1, M_N)$
satisfying the following conditions.
\begin{enumerate}
\item\label{Item_824_R_T_Rk1}
For every rank one \pj{} $e \in R \S C (S^1, M_N)$,
the \pj{} $\ps (e) \in C (S^1, M_N)$ has rank one.
\item\label{Item_824_R_T_ExtEq}
The action $\zt \mapsto \ps \circ \gm_{\zt} \circ \ps^{-1}$
is exterior equivalent to the action
$\zt \mapsto {\overline{\gm}}_{\zt, N}$.
\item\label{Item_824_R_T_L}
With $L_{\Ph}$ as defined in 1.3 of~\cite{DNNP},
for $z \in S^1$ we have
\[
L_{\ps \circ \xi} (z)
 = \bigl\{ y \in S^1 \colon y^N = z \bigr\}
\andeqn
L_{\io \circ \ps^{-1}} (z) = \{ z^N \}.
\]
\end{enumerate}
\end{lem}

The isomorphism in this lemma is not equivariant when
$C (S^1, M_N)$ is equipped with the action $\zt \mapsto \bt_{\zt, N}$,
or any action exterior equivalent to it.

\begin{proof}[Proof of Lemma~\ref{L_1824_R_T}]
Choose a \ct{} unitary path $\ld \mapsto s_{\ld}$ in $M_N$,
defined for $\ld \in [0, 1]$,
such that $s_0 = 1$ and $s_1 = s$.
For any $\ld \in \R$,
choose $\ld_0 \in [0, 1)$ such that $\ld - \ld_0 \in \Z$.
Taking $n = \ld - \ld_0$, we then define $s_{\ld} = s^n s_{\ld_0}$.
This function is still \ct.
Moreover, for any $m \in \Z$,
\begin{equation}\label{Eq_1912_ldm}
s_{\ld + m}
 = s^{m + n} s_{\ld_0}
 = s^{m} s^{n} s_{\ld_0} = s^{m} s_{\ld}.
\end{equation}

We claim that there is a well defined \hm{}
$\ps \colon R \to C (S^1, M_N)$ such that, whenever $f \in R$
and $\ld \in \R$, we have
\[
\ps (f) (e^{2 \pi i \ld}) = s_{\ld}^* f (e^{2 \pi i \ld / N}) s_{\ld}.
\]
The only issue is whether $\ps (f) (e^{2 \pi i \ld})$ is well defined.
It is sufficient to prove that if $\ld_0 \in [0, 1)$ and
$n = \ld - \ld_0 \in \Z$, then the formulas for
$\ps (f) (e^{2 \pi i \ld})$ and $\ps (f) (e^{2 \pi i \ld_0})$ agree.
To see this, use the definition of $R$ at the second step to get
\[
s_{\ld}^* f (e^{2 \pi i \ld / N}) s_{\ld}
 = s_{\ld_0}^* s^{- n} f (\om^{n} e^{2 \pi i \ld_0 / N}) s^n s_{\ld}
 = s_{\ld_0}^* f (e^{2 \pi i \ld_0 / N}) s_{\ld_0},
\]
as desired.

The construction of $\ps$ makes Part~(\ref{Item_824_R_T_Rk1}) obvious.
Bijectivity is easy just by checking the definition.
We now prove~(\ref{Item_824_R_T_ExtEq}).
For $\zt \in S^1$, choose $\ta \in \R$ such that
$e^{2 \pi i \ta} = \zt$,
and define a function $v_{\zt} \in C (S^1, M_N)$
by $v_{\zt} (e^{2 \pi i \ld}) = s_{\ld}^* s_{\ld - N \ta}$
for $\ld \in \R$.
We claim that $v_{\zt}$ is well defined.
First, we must show that if $m \in \Z$ then
\[
s_{\ld + m}^* s_{\ld + m - N \ta} = s_{\ld}^* s_{\ld - N \ta}.
\]
This follows directly from~(\ref{Eq_1912_ldm}).
Second, we must show that if $e^{2 \pi i \ta_1} = e^{2 \pi i \ta_2}$,
then
\[
s_{\ld}^* s_{\ld - N \ta_1} = s_{\ld}^* s_{\ld - N \ta_2}.
\]
For this, set $m = \ta_1 - \ta_2 \in \Z$, and use~(\ref{Eq_1912_ldm})
and $s^N = 1$ to see that
\[
s_{\ld - N \ta_2} = s_{\ld - N \ta_1 + N m}
 = s^{N m} s_{\ld - N \ta_1} = s_{\ld - N \ta_1}.
\]
The claim is proved.

It is now easy to check that
$(\zt, \ld) \mapsto v_{\zt} (e^{2 \pi i \ld})$ is \ct,
so that $\zt \mapsto v_{\zt}$ is a \cfn{} from $S^1$
to the unitary group of $C (S^1, M_N)$.

We next claim that
$v_{\zt_1 \zt_2} = v_{\zt_1} \bt_{\zt_1^N, N} (v_{\zt_2})$
for $\zt_1, \zt_2 \in S^1$.
To do this, choose $\ta_1, \ta_2 \in \R$
such that $\zt_1 = e^{2 \pi i \ta_1}$ and $\zt_2 = e^{2 \pi i \ta_2}$.
Then $\zt_1 \zt_2 = e^{2 \pi i (\ta_1 + \ta_2)}$.
So for $\ld \in \R$,
\[
\begin{split}
v_{\zt_1} (e^{2 \pi i \ld})
    \bt_{\zt_1^N, N} (v_{\zt_2}) (e^{2 \pi i \ld})
& = v_{\zt_1} (e^{2 \pi i \ld})
       v_{\zt_2} \bigl( e^{2 \pi i (\ld - N \ta_1)} \bigr)
\\
&
 = s_{\ld}^* \cdot s_{\ld - N \ta_1}
   \cdot s_{\ld - N \ta_1}^* \cdot s_{\ld - N \ta_1 - N \ta_2}
 = v_{\zt_1 \zt_2} (e^{2 \pi i \ld}),
\end{split}
\]
proving the claim.

We have shown that $\zt \mapsto v_{\zt}$ is a cocycle for the
action $\zt \mapsto \bt_{\zt^N, N}$ of $S^1$ on $C (S^1, M_N)$.
Therefore the formula
\[
\rh_{\zt} (g) = v_{\zt} \bt_{\zt^N, N} (g) v_{\zt}^*
\]
defines an action of $S^1$ on $C (S^1, M_N)$ which is exterior
equivalent to $\zt \mapsto \bt_{\zt^N, N} = {\overline{\gm}}_{\zt, N}$.

To finish the proof of~(\ref{Item_824_R_T_ExtEq}),
we show that $\ps$ is equivariant
for the action $\rh$ on $C (S^1, M_N)$.
Let $f \in R$, let $\ld \in \R$, let $\zt \in S^1$, and
choose $\ta \in \R$ such that $\zt = e^{2 \pi i \ta}$.
Then
\[
\begin{split}
(\rh_{\zt} \circ \ps) (f) (e^{2 \pi i \ld})
& = s_{\ld}^* s_{\ld - N \ta} \ps (f) (\zt^{- N} e^{2 \pi i \ld})
                       s_{\ld - N \ta}^* s_{\ld}
\\
& = s_{\ld}^* s_{\ld - N \ta} \bigl[ s_{\ld - N \ta}^*
        f \bigl( e^{2 \pi i (\ld - N \ta) / N} \bigr)
         s_{\ld - N \ta} \bigr] s_{\ld - N \ta}^* s_{\ld}
\\
& = s_{\ld}^* f \bigl(\zt^{-1} e^{2 \pi i \ld / N} \bigr) s_{\ld}
  = (\ps \circ \gm_{\zt}) (f) (e^{2 \pi i \ld}).
\end{split}
\]
This completes the proof of~(\ref{Item_824_R_T_ExtEq}).

For~(\ref{Item_824_R_T_L}),
we first observe that if $f \in C (S^1)$ and $\ld \in \R$ then
\[
(\ps \circ \xi) (f) (e^{2 \pi i \ld})
 = s_{\ld}^*  {\operatorname{diag}} \bigl( f (e^{2 \pi i \ld / N}),
     \, f (\om^{- 1} e^{2 \pi i \ld / N}),
    \, \ldots, \, f (\om^{- N + 1} e^{2 \pi i \ld / N}) \bigr) s_{\ld}.
\]
Therefore
\[
L_{\ps \circ \xi} (e^{2 \pi i \ld})
 = \bigl\{ e^{2 \pi i \ld / N}, \, \om^{- 1} e^{2 \pi i \ld / N},
    \, \ldots, \, \om^{- N + 1} e^{2 \pi i \ld / N} \bigr\}.
\]
This is the same as the description in the statement.

For the second formula,
one checks that for $g \in C (S^1, M_N)$ and $\ld \in \R$, we have
\[
\ps^{-1} (g) (e^{2 \pi i \ld})
 = s_{N \ld} g ( e^{2 \pi i N \ld} ) s_{N \ld}^*
 = s_{N \ld} g \bigl( (e^{2 \pi i \ld})^N  \bigr) s_{N \ld}^*.
\]
Since $\io \colon R \to C (S^1, M_N)$ is just the inclusion,
this gives $L_{\io \circ \ps^{-1}} (z) = \{ z^N \}$ for $z \in S^1$,
as desired.
\end{proof}

\begin{lem}\label{L_1916_Simple}
The algebra $A$ of
Construction \ref{Cn_1824_S1An}(\ref{Item_1924_S1_A})
is a simple AT~algebra.
\end{lem}

\begin{proof}
Using Lemma~\ref{L_1824_R_T}, we can rewrite the direct system in
Construction \ref{Cn_1824_S1An}(\ref{Item_1824_S1_nu_A}) as
\[
A = \dirlim_n
 \bigl[ C (S^1, M_{N r_0 (n)}) \oplus C (S^1, M_{r_1 (n)}) \bigr],
\]
with maps ${\widetilde{\nu}}_{n}$, for $n \in \Nz$,
obtained analogously to
Construction \ref{Cn_1824_S1An}(\ref{Item_1824_S1_nu_A}) from
\[
{\widetilde{\nu}}_{n, 0, 0}
 = (\id_{C (S^1)})_{N r_0 (n)}^{l_{0, 0} (n)},
\qquad
{\widetilde{\nu}}_{n, 0, 1}
 = \ps_{l_{0, 1} (n) r_1 (n)} \circ \nu_{n, 0, 1},
\]
\[
{\widetilde{\nu}}_{n, 1, 0}
 = \nu_{n, 1, 0} \circ (\ps_{l_{1, 0} (n) r_0 (n)})^{-1},
\andeqn
{\widetilde{\nu}}_{n, 1, 1} = \nu_{n, 1, 1},
\]
and with
\[
\begin{split}
{\widetilde{\nu}}_{n, m}
& = {\widetilde{\nu}}_{n - 1}
    \circ {\widetilde{\nu}}_{n - 2}
    \circ \cdots \circ {\widetilde{\nu}}_m \colon
\\
&  C (S^1, M_{N r_0 (m)}) \oplus C (S^1, M_{r_1 (m)})
   \to C (S^1, M_{N r_0 (n)}) \oplus C (S^1, M_{r_1 (n)}).
\end{split}
\]

It is now obvious that $A$ is an AT~algebra.
For simplicity, we use Proposition 2.1 of~\cite{DNNP},
with $L_{\Ph}$ as defined in 1.3 of~\cite{DNNP}.
To make the notation easier, we take $X_n = S^1 \times \{ 0, 1 \}$
(rather than $S^1 \amalg S^1$ as in~\cite{DNNP}),
and for $j \in \{ 0, 1 \}$ identify $S^1 \times \{ j \}$
with the primitive ideal space of $C (S^1, M_{r_j (n)})$.
Moreover, since the spaces $X_n$ are all equal,
we write them all as~$X$.

For $z \in S^1$ it is immediate that
\[
L_{{\widetilde{\nu}}_{n, 0, 0}} (z) = \{ z \}
\]
and, recalling $\om = \exp (2 \pi i / N)$ from
Construction \ref{Cn_1824_S1An}(\ref{Item_1824_S1_s_om})
and using
\[
\{ 1, \om^{- 1}, \ldots, \om^{- N + 1} \}
 = \{ 1, \om, \ldots, \om^{N - 1} \},
\]
also
\[
\begin{split}
L_{{\widetilde{\nu}}_{n, 1, 1}} (z)
& = L_{\nu_{n, 1, 1}} (z)
\\
& = \bigl\{ z, \, \om z, \, \ldots, \om^{N - 1} z, \,
       e^{- 2 \pi i \te} z, \, e^{- 2 \pi i \te} \om z,
        \, \ldots, e^{- 2 \pi i \te} \om^{N - 1} z \bigr\}.
\end{split}
\]
Also, using Lemma \ref{L_1824_R_T}(\ref{Item_824_R_T_L}),
\[
L_{{\widetilde{\nu}}_{n, 0, 1}} (z)
 = L_{\ps \circ \xi} (z)
 = \bigl\{ y \in S^1 \colon y^N = z \bigr\}
\]
and
\[
L_{{\widetilde{\nu}}_{n, 1, 0}} (z)
= L_{\io \circ \ps^{-1}} (z) = \{ z^N \}.
\]
Putting these together, we get
\begin{equation}\label{Eq_1916_Lz_zero}
L_{{\widetilde{\nu}}_{n}} (z, 0)
 = \{ (z, 0) \}
  \cup \bigl\{ (y, 1) \colon {\mbox{$y \in S^1$ and $y^N = z$}} \bigr\}
\end{equation}
and
\begin{equation}\label{Eq_1916_Lz1}
\begin{split}
L_{{\widetilde{\nu}}_{n}} (z, 1)
& = \{ (z^N, 0), \,
     (z, 1), \, (\om z, 1), \, \ldots, (\om^{N - 1} z, 1),
\\
& \hspace*{5em} {\mbox{}}
     (e^{- 2 \pi i \te} z, 1), \, (e^{- 2 \pi i \te} \om z, 1),
        \, \ldots, (e^{- 2 \pi i \te} \om^{N - 1} z, 1) \bigr\}.
\end{split}
\end{equation}
One checks that if $C$, $D$, and $E$ are finite direct sums
of homogeneous \uca{s},
and $\Ph \colon C \to D$ and $\Ps \colon D \to E$ are unital \hm{s},
with primitive ideal spaces $X$, $Y$, and~$Z$,
then for $z \in Z$ we have
\begin{equation}\label{Eq_1916_Comp}
L_{\Ps \circ \Ph} (z) = \bigcup_{y \in L_{\Ps} (z)} L_{\Ph} (y).
\end{equation}
The equations (\ref{Eq_1916_Lz_zero}) and~(\ref{Eq_1916_Lz1})
show that $x \in L_{{\widetilde{\nu}}_{n}} (x)$ for any $x \in X$.
So~(\ref{Eq_1916_Comp}) implies that
for any $l, m, n \in \Nz$ with $n > m > l$,
and any $x \in X$,
\begin{equation}\label{Eq_1916_Contain}
L_{{\widetilde{\nu}}_{n, m}} (x) \cup L_{{\widetilde{\nu}}_{m, l}} (x)
 \S L_{{\widetilde{\nu}}_{n, l}} (x).
\end{equation}
It now suffices to prove that for every $l \in \Nz$ and every $\ep > 0$,
there is $n > l$ such that for every $z \in S^1$ and $j \in \{ 0, 1 \}$,
the set $L_{{\widetilde{\nu}}_{n, l}} (z, j)$ is
$\ep$-dense in $S^1 \times \{ 0, 1 \}$.
Given this, simplicity of $A$ can be deduced
from Proposition~2.1 of~\cite{DNNP}, and the proof will be complete.

Choose $n > l + 2$ such that
\[
\bigl\{ e^{- 2 \pi i k \te} \colon k = 0, 1, \ldots, n - l - 2 \bigr\}
\andeqn
\bigl\{ e^{- 2 \pi i k N \te} \colon k = 0, 1, \ldots, n - l - 2 \bigr\}
\]
are both $\ep$-dense in $S^1$.

We claim that if $z \in S^1$ is arbitrary, then
$L_{{\widetilde{\nu}}_{n - 1, \, l}} (z, 1)$
is $\ep$-dense in $S^1 \times \{ 0, 1 \}$.
To prove the claim,
first use (\ref{Eq_1916_Comp}) repeatedly, (\ref{Eq_1916_Contain}),
and the fact that
$(e^{- 2 \pi i \te} y, 1) \in L_{{\widetilde{\nu}}_{m}} (y, 1)$
for all $m \in \N$ and $y \in S^1$ (by (\ref{Eq_1916_Lz1}))
to see that for $z \in S^1$,
\begin{equation}\label{Eq_1927_Cont}
\bigl\{ (e^{- 2 \pi i k \te} z, \, 1) \colon
     k = 0, 1, \ldots, n - l - 2 \bigr\}
 \S L_{{\widetilde{\nu}}_{n - 1, \, l + 1}} (z, 1).
\end{equation}
By~(\ref{Eq_1916_Contain}),
this set is contained in $L_{{\widetilde{\nu}}_{n - 1, \, l}} (z, 1)$,
and, since since multiplication by~$z$ is isometric,
it is $\ep$-dense in $S^1 \times \{ 1 \}$.
Also, by (\ref{Eq_1927_Cont}) and (\ref{Eq_1916_Contain})
(taking $m = l + 1$), and using~(\ref{Eq_1916_Lz1})
to get $(y^N, 0) \in L_{{\widetilde{\nu}}_{l + 1, l}} (y, 1)$
for all  $y \in S^1$,
\[
\bigl\{ (e^{- 2 \pi i k N \te} z^N, \, 0) \colon
     k = 0, 1, \ldots, n - l - 2 \bigr\}
 \S L_{{\widetilde{\nu}}_{n - 1, \, l}} (z, 1),
\]
and this set is $\ep$-dense in $S^1 \times \{ 0 \}$.
The claim follows.

Now, for any $x \in S^1 \times \{ 0, 1 \}$,
the set $ L_{{\widetilde{\nu}}_{n, n - 1}} (x)$ contains at least
one point in $S^1 \times \{ 1 \}$
by (\ref{Eq_1916_Lz_zero}) and~(\ref{Eq_1916_Lz1}).
Using~(\ref{Eq_1916_Contain}) with $m = n - 1$
and the previous claim, it follows that
$L_{{\widetilde{\nu}}_{n, l}} (x)$
is $\ep$-dense in $S^1 \times \{ 0, 1 \}$, as desired.
\end{proof}

\begin{lem}\label{L_1912_B_Smp}
The algebra $B$ of
Construction \ref{Cn_1824_S1An}(\ref{Item_1824_S1_FixLim})
is a simple unital AF~algebra.
\end{lem}

\begin{proof}
That $B$ is a unital AF~algebra is immediate from its definition.
Simplicity follows from the corollary on page 212 of~\cite{Brtl}.
\end{proof}

\begin{lem}\label{L_1912_R_Fix}
Define
\[
c = \frac{1}{\sqrt{N}} \left( \begin{matrix}
  1    &   1         &   1              & \cdots    &   1          \\
  1    & \om         & \om^2            & \cdots    & \om^{N - 1} \\
  1    & \om^2       & \om^4            & \cdots    & \om^{2 (N - 1)} \\
\vdots & \vdots      & \vdots           & \ddots    &  \vdots      \\
  1    & \om^{N - 1} & \om^{2 (N - 1)}  & \cdots    & \om^{(N - 1)^2}
\end{matrix} \right).
\]
Then the formula
\[
\ep_0 (\ld_0, \ld_1, \ldots, \ld_{N - 1})
 = c^* \diag (\ld_0, \ld_1, \ldots, \ld_{N - 1}) c,
\]
for $\ld_0, \ld_1, \ldots, \ld_{N - 1} \in {\mathbb{C}}$,
defines an isomorphism from ${\mathbb{C}}^N$ to $R^{\gm}$.
\end{lem}

\begin{proof}
One checks that $c$ is unitary and
\begin{equation}\label{Eq_1912_wsw}
c s c^*
 = \diag \bigl(1 , \om^{-1}, \om^{-2}, \ldots, \om^{- (N - 1)} \bigr).
\end{equation}

It is immediate that $R^{\gm}$ is the set of constant functions
in $C (S^1, M_N)$ whose constant value commutes with $s$.
Let $D$ be the set of constant functions
in $C (S^1, M_N)$ whose constant value commutes with $c s c^*$.
Then $a \mapsto c^* a c$ is an isomorphism from $D$ to $R^{\gm}$.
Also, by~(\ref{Eq_1912_wsw}),
\[
(\ld_0, \ld_1, \ldots, \ld_{N - 1})
 \mapsto \diag ( \ld_0, \ld_1, \ldots, \ld_{N - 1} )
\]
is an isomorphism from ${\mathbb{C}}^N$ to~$D$.
\end{proof}

\begin{lem}\label{L_1916_FPisB}
There is a family $(\et_n)_{n \in \Nz}$ of isomorphisms
$\et_n \colon B_n \to (A_n)^{\af^{(n)}}$ such that the following hold.
\begin{enumerate}
\item\label{Item_1925_Comm}
$\et_n \circ \ch_{n, m} = \nu_{n, m} \circ \et_m$
whenever $m, n, \in \Nz$ satisfy $m \leq n$.
\item\label{Item_1925_pn}
With $p_n$ as in Construction \ref{Cn_1824_S1An}(\ref{Item_1824_S1_pn})
and $q_n$ as in Construction \ref{Cn_1824_S1An}(\ref{Item_1915_S1_qn}),
we have $\et_n (q_n) = p_n$.
\item\label{Item_1925_BBnj}
For all $n \in \Nz$ and $j \in \{ 0, 1 \}$, we have
$\et_n (B_{n, j}) = (A_{n, j})^{\af^{(n)}}$.
\item\label{Item_1925_Iso_A}
The family $(\et_n)_{n \in \Nz}$ induces an isomorphism
$\et_{\I} \colon B \to A^{\af}$.
\end{enumerate}
\end{lem}

We warn that the subscript in $\et_n$ does not have the meaning
taken from Notation~\ref{N_1908_Ampl}.

\begin{proof}[Proof of Lemma~\ref{L_1916_FPisB}]
Since $\af$ is a direct limit action,
the inclusions $(A_n)^{\af^{(n)}} \to A_n$ induce an isomorphism
$A^{\af} \to \dirlim_n (A_n)^{\af^{(n)}}$.
Therefore (\ref{Item_1925_Iso_A}) follows from the rest of the
statement of the lemma.

Lemma~\ref{L_1912_R_Fix} implies that
$(\ep_0)_{r_0 (n)}
 \colon M_{r_0 (n)} ({\mathbb{C}}^N) \to (A_{n, 0})^{\af^{(n)}}$
is an isomorphism.
It is immediate that the embedding
$\ep_1 \colon {\mathbb{C}} \to C (S^1)^{\bt}$ as constant functions
is an isomorphism.
Therefore
\[
\et_n^{(0)} = (\ep_0)_{r_0 (n)} \oplus (\ep_1)_{r_1 (n)} \colon
  B_{n, 0} \oplus B_{n, 1} \to (A_{n})^{\af^{(n)}}
\]
is an isomorphism.
Clearly (\ref{Item_1925_pn}) and~(\ref{Item_1925_BBnj})
hold with $\et_n^{(0)}$ in place of~$\et_n$.

Let $\io^{S^1}$ and $\xi^{S^1}$ be the restriction
and corestriction of $\io$ and $\xi$
to the corresponding fixed point algebras,
and similarly define $\nu_n^{S^1}$, $\nu_{n, j, k}^{S^1}$, etc.
Then the inverses of the maps $\et_n^{(0)}$ implement an isomorphism
from the direct system $\bigl( (A_n)^{\af^{(n)}} \bigr)_{n \in \Nz}$
to the direct system $( B_n )_{n \in \Nz}$, with
the maps $\sm_n \colon B_n \to B_{n + 1}$ taken to be
\[
\begin{split}
\sm_n (a_0, a_1)
& = \Bigl( \diag \bigl( (\id_{B_{n, 0}})^{l_{0, 0} (n)} (a_0),
\\
& \hspace*{6em} {\mbox{}}
   \, \bigl( ( (\ep_0)_{l_{0, 1} (n) r_1 (n)})^{-1}
       \circ (\xi^{S^1})_{r_1 (n)}^{l_{0, 1} (n)}
       \circ (\ep_1)_{r_1 (n)} \bigr) (a_1) \bigr),
\\
& \hspace*{2em} {\mbox{}}
 \diag \bigl(  \bigl( ( (\ep_1)_{l_{1, 0} (n) r_1 (n)} )^{-1}
             \circ (\io^{S^1})_{r_0 (n)}^{l_{1, 0} (n)}
             \circ (\ep_0)_{r_1 (n)} \bigr) (a_0),
\\
& \hspace*{6em} {\mbox{}}
        \, (\id_{B_{n, 1}})^{2 N l_{1, 1} (n)} (a_1) \bigr) \Bigr).
\end{split}
\]

The map $(\ep_0)^{-1} \circ \xi^{S^1} \circ \ep_1 \colon {\mathbb{C}} \to {\mathbb{C}}^N$
is a unital homomorphism.
There is only one such unital homomorphism, so this map is equal
to the map $\mu$
in Construction \ref{Cn_1824_S1An}(\ref{Item_1824_S1_FixPtSys}).
The map $(\ep_1)^{-1} \circ \io^{S^1} \circ \ep_0 \colon {\mathbb{C}}^N \to M_N$
is an injective unital homomorphism.
Therefore it must be unitarily equivalent to the map $\dt$
in Construction \ref{Cn_1824_S1An}(\ref{Item_1824_S1_FixPtSys}).
It now follows from the definitions
(see Construction \ref{Cn_1824_S1An}(\ref{Item_1824_S1_FixLim}))
that for $n \in \Nz$
there is a unitary $v_n \in B_{n + 1}$
such that $\ch_n (b) = v_n \sm_n (b) v_n^*$ for all $b \in B_n$.
Inductively define unitaries $w_n \in (A_n)^{\af^{(n)}}$
by $w_0 = 1$ and, given~$w_n$,
setting $w_{n + 1} = \nu_n^{S^1} (w_n) \et_{n + 1}^{(0)} (v_n)^*$.
Then define $\et_n (b) = w_n \et_{n}^{(0)} (b) w_n^*$ for $b \in B_n$.
The conditions (\ref{Item_1925_pn}) and~(\ref{Item_1925_BBnj})
hold as stated because they hold for the maps $\et_{n}^{(0)}$.
For $n \in \Nz$,
using $\nu_n^{S^1} \circ \et_{n}^{(0)} = \et_{n + 1}^{(0)} \circ \sm_n$,
one gets $\nu_n^{S^1} \circ \et_{n} = \et_{n + 1} \circ \ch_n$.
This implies~(\ref{Item_1925_Comm}).
\end{proof}

\begin{lem}\label{L_1917_HasTRP}
In Construction~\ref{Cn_1824_S1An}, assume that $r_0 (0) \leq r_1 (0)$
and that for all $n \in \Nz$
we have $l_{1, 0} (n) \geq l_{0, 0} (n)$
and $l_{1, 1} (n) \geq l_{0, 1} (n)$.
Further assume that
\[
\lim_{n \to \I} \frac{l_{0, 1} (n)}{l_{0, 0} (n)} = \I
\andeqn
\lim_{n \to \I} \frac{l_{1, 1} (n)}{l_{1, 0} (n)} = \I.
\]
Then the action $\af$
of Construction \ref{Cn_1824_S1An}(\ref{Item_1824_S1_nu_A})
has the tracial Rokhlin property with comparison.
\end{lem}

The hypotheses are overkill.
They are chosen to make the proof easy.

\begin{proof}[Proof of Lemma~\ref{L_1917_HasTRP}]
Since $A$ is stably finite,
by Lemma 1.15 in \cite{phill23}, we may disregard
condition~(\ref{Item_902_pxp_TRP})
in Definition~\ref{traR}.

We first claim that $0 < r_0 (n) \leq r_1 (n)$ for all $n \in \N$.
This is true for $n = 0$ by hypothesis.
For any other value of~$n$, using
$r_0 (n - 1) > 0$ and $r_1 (n - 1) > 0$, we have
\[
\begin{split}
r_0 (n)
& = l_{0, 0} (n - 1) r_0 (n - 1) + l_{0, 1} (n - 1) r_1 (n - 1)
\\
&
 \leq l_{1, 0} (n - 1) r_0 (n - 1) + l_{1, 1} (n - 1) r_1 (n - 1)
 = r_1 (n).
\end{split}
\]
Also, since $l_{0, 0} (n - 1) > 0$, the first step of this
calculation implies that $r_0 (n) > 0$.

Next, we claim that for every $n \in \Nz$
and every tracial state $\ta$ on the algebra $B_{n + 1}$
of Construction \ref{Cn_1824_S1An}(\ref{Item_1824_S1_Bn0Bn1}),
with $\ch_n$ as in
Construction \ref{Cn_1824_S1An}(\ref{Item_1824_S1_FixLim})
and $q_n$ as in
Construction \ref{Cn_1824_S1An}(\ref{Item_1915_S1_qn}),
we have
\begin{equation}\label{Eq_1916_TrBd}
\ta (1 - \ch_n (q_n))
 \leq \max \left( \frac{l_{0, 0} (n)}{l_{0, 1} (n)},
   \, \frac{l_{1, 0} (n)}{l_{1, 1} (n)} \right).
\end{equation}
To see this, we first look at the partial embedding multiplicities
in Construction \ref{Cn_1824_S1An}(\ref{Item_1824_S1_FixPtSys})
to see that the rank of $1 - \ch_n (q_n)$ in each of the first
$N$ summands (all equal to $M_{r_0 (n + 1)}$) is $l_{0, 0} (n) r_0 (n)$,
and the rank of $1 - \ch_n (q_n)$ in the last summand
(equal to $M_{r_1 (n + 1)}$) is $l_{1, 0} (n) r_0 (n)$.
Now
\begin{equation}\label{Eq_1916_ZeroCrd}
\frac{l_{0, 0} (n) r_0 (n)}{r_0 (n + 1)}
  = \frac{l_{0, 0} (n) r_0 (n)}{l_{0, 0} (n) r_0 (n)
         + l_{0, 1} (n) r_1 (n)}
  \leq \frac{l_{0, 0} (n) r_0 (n)}{l_{0, 1} (n) r_1 (n)}
  \leq \frac{l_{0, 0} (n)}{l_{0, 1} (n)},
\end{equation}
and
\begin{equation}\label{Eq_1916_1stCrd}
\frac{l_{1, 0} (n) r_0 (n)}{r_1 (n + 1)}
  = \frac{l_{1, 0} (n) r_0 (n)}{
   N l_{1, 0} (n) r_0 (n) + 2 N l_{1, 1} (n) r_1 (n)}
  \leq \frac{l_{1, 0} (n) r_0 (n)}{l_{1, 1} (n) r_1 (n)}
  \leq \frac{l_{1, 0} (n)}{l_{1, 1} (n)}.
\end{equation}
The number $\ta (1 - \ch_n (q_n))$ is a convex combination of
the numbers in (\ref{Eq_1916_ZeroCrd}) and~(\ref{Eq_1916_1stCrd}).
The claim follows.

Now let $F \subseteq A$ and $S \subseteq C (S^1)$ be finite sets,
let $\ep > 0$, let $x \in A_{+} \setminus \{ 0 \}$,
and let $y \in A_{+}^{\af} \setminus \{ 0 \}$.
We may assume that $\| x \| \leq 1$ and $\| y \| \leq 1$,
and that $\| f \| \leq 1$ for all $f \in S$.
Set
\begin{equation}\label{Eq_1917_ep0}
\ep_0 = \frac{1}{2}
 \min \left( \inf_{\ta \in T (A)} \ta (x), \,
     \inf_{\ta \in T (A^{\af})} \ta (y), \, \frac{1}{2} \right).
\end{equation}
Choose $n \in \Nz$ so large that there is a finite subset
$F_0 \subseteq A_n$
with $\dist (a, \, \nu_{\I, n} (F_0) ) < \frac{\ep}{3}$
for all $a \in F$,
and also so large that
\begin{equation}\label{Eq_1917_Ch_n}
\min \left( \frac{l_{0, 1} (n)}{l_{0, 0} (n)},
    \, \frac{l_{1, 1} (n)}{l_{1, 0} (n)} \right)
  > \frac{1}{\ep_0}.
\end{equation}

Let $p \in A$ be $p = \nu_{\I, n} (p_n)$.
Define
\[
\ph_0 \colon
 C (S^1) \to p_n A_n p_n
  = M_{r_1 (n)} \otimes C (S^1)
\]
by $\ph_0 (g) = (0, \, 1 \otimes g)$ for $g \in C (S^1)$.
Define $\ph = \nu_{\I, n} \circ \ph_0 \colon C (S^1) \to p A p$.
Then $\ph$ is an equivariant unital \hm.
In particular, $\ph$ is exactly multiplicative on~$S$.
Further, let $a \in F$ and $f \in S$.
Choose $b \in F_0$
such that $\| a - \nu_{\I, n} (b) \| < \frac{\ep}{3}$.
Then, using $\| f \| \leq 1$ and the fact that $\ph_0 (f)$ commutes
with all elements of~$A_n$, we have
\[
\| \ph (f) a - a \ph (f) \|
  \leq 2 \| a - \nu_{\I, n} (b) \| < \ep.
\]
Part~(\ref{Item_893_FS_equi_cen_multi_approx})
of the definition is verified.

For the remaining three conditions,
let $\ta$ be any \tst{} on either $A$ or $A^{\af}$.
Let $(\et_n)_{n \in \Nz}$ be as in Lemma~\ref{L_1916_FPisB}.
Then $\ta \circ \nu_{\I, n + 1} \circ \et_{n + 1}$
is a \tst{} on~$B_{n + 1}$.
Combining this with (\ref{Eq_1916_TrBd}), (\ref{Eq_1917_Ch_n}),
(\ref{Eq_1917_ep0}),
and $(\nu_{\I, n + 1} \circ \et_{n + 1} \circ \ch_n) (q_n) = p$, we get
$\ta (1 - p) < \ta (x) \leq d_{\ta} (x)$ for all $\ta \in T (A)$
and $\ta (1 - p) < \ta (y) \leq d_{\ta} (y)$
for all $\ta \in T (A^{\alpha})$.
Since simple unital AF~algebras and
simple unital AT~algebras have strict comparison,
it follows that $1 - p \precsim_A x$
and $1 - p \precsim_{A^{\alpha}} y$.
Since $\ep_0 < \frac{1}{2}$,
similar reasoning gives $1 - p \precsim_{A^{\alpha}} p$.
\end{proof}

We use equivariant K-theory to show that $\af$ does not have
finite Rokhlin dimension with commuting towers.
Recall equivariant K-theory from Definition 2.8.1 of~\cite{Phl1}.
For a unital \ca~$A$ with an action $\af \colon G \to \Aut (A)$
of a compact group~$G$, it is the Grothendieck group of the
equivariant isomorphism classes of equivariant finitely generated
projective right modules $E$ over~$A$,
with ``equivariant'' meaning that the module is equipped
with an action of $G$
such that $g \cdot (\xi a) = (g \cdot \xi) \af_g (a)$
for $g \in G$, $\xi \in E$, and $a \in A$.
Further recall the representation ring $R (G)$ of a compact group
from the introduction of~\cite{Sgl1} or Definition 2.1.3 of~\cite{Phl1}
(it is $K_0^G ({\mathbb{C}})$, or the Grothendieck group of the
isomorphism classes of \fd{} representations of~$G$),
its augmentation ideal $I (G)$ from the example before
Proposition~3.8 of~\cite{Sgl1} (where it is called $I_G$)
or the discussion after Definition 2.1.3 of~\cite{Phl1}
(it is the kernel of the dimension map from $R (G)$ to $\Z$),
and, for a \ca{} $A$ with an action $\af \colon G \to \Aut (A)$,
the $R (G)$-module structure on $K_*^G (A)$
from Theorem 2.8.3 and Definition 2.7.8 of~\cite{Phl1}.
In particular, for $G = S^1$,
if we let $\sm \in {\widehat{S^1}}$ be the identity map $S^1 \to S^1$,
then $R (S^1) = \Z [ \sm, \sm^{-1}]$,
the Laurent polynomial ring in one variable over~$\Z$.
(See Example~(ii) at the beginning of Section~3 of~\cite{Sgl1}.)
Moreover, $I (S^1)$ is the ideal generated by $\sm - 1$.

Recall the action $\bt \colon S^1 \to \Aut (C (S^1))$
from Construction \ref{Cn_1824_S1An}(\ref{Item_1824_S1_CS1}),
given by $\bt_{\zt} (f) (z) = f (\zt^{-1} z)$ for $\zt, z \in S^1$,
and ${\overline{\gm}} \colon S^1 \to \Aut (C (S^1))$
from Lemma~\ref{L_1824_R_T},
given by ${\overline{\gm}}_{\zt} (f) (z) = f (\zt^{-N} z)$.
We denote the equivariant K-theory for these actions by
$K_*^{S^1, \bt} (C (S^1))$ and $K_*^{S^1, {\overline{\gm}}} (C (S^1))$,
and similarly for other actions when ambiguity is possible.

We won't actually use the following computation of the
equivariant $K_{1}$-groups, but it is included to give a more
complete description of our example.
As in Section~\ref{Sec_3749_Exam_TRPZ2},
we abbreviate $\Z / n \Z$ to $\Z_n$.

\begin{lem}\label{L_1918_K1}
We have $K_1^{S^1, \bt} (C (S^1)) = 0$ and
(with $R \S C (S^1, M_N)$ as in
Construction \ref{Cn_1824_S1An}(\ref{Item_1824_S1_R}))
$K_*^{S^1} (R) = 0$.
\end{lem}

\begin{proof}
By Theorem 2.8.3(7) of~\cite{Phl1}, we have
\[
K_{1}^{S^{1}, \bt} (C (S^{1}))
 \cong K_{1} \bigr( C^{*} (S^{1}, \, C (S^{1}), \, \bt) \bigr).
\]
Since
\[
C^{*} (S^{1}, \, C (S^{1}), \, \bt) \cong K (L^{2} (S^{1})),
\]
we conclude $K_{1}^{S^{1}, \bt} (C (S^{1})) = 0$.

For $K_{*}^{S^{1}} (R)$, since exterior equivalent actions
of a compact group~$G$ give isomorphic $R (G)$-modules $K_*^G (A)$
(Theorem 2.8.3(5) of~\cite{Phl1}),
by Lemma \ref{L_1824_R_T}(\ref{Item_824_R_T_ExtEq})
it is sufficient to prove this
for the action $\zt \mapsto {\overline{\gm}}_{\zt, N}$.
By stability of equivariant K-theory (Theorem 2.8.3(4) of~\cite{Phl1}),
it suffices to prove this
for action ${\overline{\gm}}$ of $S^{1}$ on $C (S^{1})$.
By~\cite{Jlg} (or Theorem 2.8.3(7) of~\cite{Phl1}), we have
$K_{1}^{S^{1}, {\overline{\gm}}} (C (S^{1}))
  \cong K_{1}
     \bigl( C^{*} (S^{1}, \, C (S^{1}), \, {\overline{\gm}}) \bigr)$.
Corollary 2.10 of~\cite{Grn}, with $G = S^{1}$ and $H = \Z_N$,
tells us that
\[
C^{*} (S^{1}, \, C (S^{1}), \, {\overline{\gm}})
 \cong K (L^{2} (S^{1})) \otimes C^* (\Z_N),
\]
which has trivial $K_{1}$-group.
\end{proof}

\begin{lem}\label{L_1918_EqKbt}
There is an $R (S^1)$-module isomorphism
$K_0^{S^1, \bt} (C (S^1)) \cong \Z$,
with the $R (S^1)$-module structure coming from the
isomorphism $\Z \cong R (S^1) / I (S^1)$,
and such that the class in $K_0^{S^1, \bt} (C (S^1))$ of the rank one
free module is sent to $1 \in \Z$.
\end{lem}

\begin{proof}
In Proposition 2.9.4 of~\cite{Phl1}, take $A = {\mathbb{C}}$, $G = S^1$,
and $H = \{ 1 \}$,
and refer to the description of the map
in the proof of that proposition.
\end{proof}

\begin{lem}\label{L_1918_EqKgmt}
There is an $R (S^1)$-module isomorphism
$K_0^{S^1, {\overline{\gm}}} (C (S^1)) \cong R ( \Z_N)$,
with the $R (S^1)$-module structure coming from the
surjective restriction map $R (S^1) \to R ( \Z_N)$.
Moreover, the classes in $K_0^{S^1, {\overline{\gm}}} (C (S^1))$
of the equivariant finitely generated projective right $C (S^1)$-modules
with underlying nonequivariant module $C (S^1)$
correspond exactly to the elements of
$(\Z_N)^{\wedge} \S R ( \Z_N)$.
\end{lem}

\begin{proof}
In Proposition 2.9.4 of~\cite{Phl1}, take
\[
A = {\mathbb{C}},
\qquad
G = S^1,
\andeqn
H = \{ 1, \om, \ldots, \om^{N - 1} \} \cong \Z_N.
\]
With these choices, $C (G \times_H {\mathbb{C}})$
is the set of $\om$-periodic
functions on $S^1$,
with the action of $\zt \in S^1$ being rotation by~$\zt$.
With ${\overline{\gm}} \colon S^1 \to \Aut (C (S^1))$
as in Lemma~\ref{L_1824_R_T}, this algebra is equivariantly isomorphic
to $(C (S^1), {\overline{\gm}})$ in an obvious way.
{}From Proposition 2.9.4 of~\cite{Phl1}, we get
$K_0^{S^1, {\overline{\gm}}} (C (S^1)) \cong R ( \Z_N)$.
Using the description of the map in the proof of the proposition,
the map sends the class of a ${\overline{\gm}}$-equivariant
finitely generated projective right $C (S^1)$-module~$E$
to the class, as a representation space of~$H$,
of its pushforward under the evaluation map $f \mapsto f (1)$.
If $E$ is nonequivariantly isomorphic to $C (S^1)$,
this pushforward is nonequivariantly isomorphic to~${\mathbb{C}}$.
The only classes in $R (\Z_N)$
with underlying vector space~${\mathbb{C}}$
are those in $(\Z_N)^{\wedge}$.
To check that an element $\ta \in (\Z_N)^{\wedge}$ actually
arises this way, choose $l \in \Z$ such that $\ta (\om) = \om^l$.
Then use the action of $S^1$ on $C (S^1)$ given by
$(\zt \cdot f) (z) = \zt^l f (\zt^{- N} z)$ for $\zt, z \in S^1$
and $f \in C (S^1)$.
One readily checks that this makes $C (S^1)$ a
${\overline{\gm}}$-equivariant right $C (S^1)$-module
whose restriction to $\{ 1 \}$ is ${\mathbb{C}}$
with the representation~$\ta$.
\end{proof}

\begin{lem}\label{L_1918_EqKThR}
There is an $R (S^1)$-module isomorphism
$K_0^{S^1} (R) \cong R ( \Z_N)$,
with the $R (S^1)$-module structure coming from the
surjective restriction map $R (S^1) \to R ( \Z_N)$.
Moreover,
for any given rank one invariant \pj{} $e \in R \S C (S^1, M_N)$,
the isomorphism can be chosen to send $[e]$ to $1 \in R ( \Z_N)$.
\end{lem}

\begin{proof}
Using Lemma \ref{L_1824_R_T}(\ref{Item_824_R_T_Rk1}),
it suffices to prove this for $C (S^1, M_N)$ and the action
$\zt \mapsto \rh_{\zt} = \ps \circ \gm_{\zt} \circ \ps^{-1}$ in
Lemma \ref{L_1824_R_T}(\ref{Item_824_R_T_ExtEq})
in place of $R$ and~$\gm$.

So fix a rank one \pj{} $e \in C (S^1, M_N)$ which is invariant
under $\zt \mapsto \ps \circ \gm_{\zt} \circ \ps^{-1}$.
If the group action is ignored, $e$ is \mvnt{} a constant \pj,
so $E = e C (S^1, M_N)$ is nonequivariantly isomorphic to
$C (S^1, {\mathbb{C}}^N)$ as a right $C (S^1, M_N)$-module.
Let ${\overline{\gm}} \colon S^1 \to \Aut (C (S^1))$
be as in Lemma~\ref{L_1824_R_T},
and write ${\overline{\gm}}_N$ for the action
$\zt \mapsto {\overline{\gm}}_{\zt, N}$ on $C (S^1, M_N)$.
By Lemma \ref{L_1824_R_T}(\ref{Item_824_R_T_ExtEq}),
$\rh$ is exterior equivalent to ${\overline{\gm}}_{N}$.
By Proposition 2.7.4 of~\cite{Phl1},
and the formula in the proof for the isomorphism,
there is a (natural) isomorphism from $K_0^{S^1, \rh} (C (S^1))$
to $K_0^{S^1, {\overline{\gm}}_{N}} (C (S^1, M_N))$
which sends the class of $E$
to the class of the same module with a different action of~$G$.
By stability in equivariant K-theory
(Theorem 2.8.3(4) of~\cite{Phl1}), there is an isomorphism
\[
K_0^{S^1, {\overline{\gm}}_{N}} (C (S^1, M_N))
 \to K_0^{S^1, {\overline{\gm}}} (C (S^1))
\]
which maps the class of $E$ to the class of some equivariant module
whose underlying nonequivariant module is $C (S^1)$.
Combining this with Lemma~\ref{L_1918_EqKgmt},
we have an isomorphism from $K_0^{S^1} (R)$ to $R ( \Z_N)$
which sends $[e]$ to some element $\ta \in (\Z_N)^{\wedge}$.
Multiplying by $\ta^{-1}$ gives an isomorphism
from $K_0^{S^1} (R)$ to $R ( \Z_N)$
which sends $[e]$ to $1 \in (\Z_N)^{\wedge}$.
\end{proof}

\begin{lem}\label{L_1918_EqKSys}
Identify $R (S^1) = \Z [ \sm, \sm^{-1}]$
as before Lemma~\ref{L_1918_K1}.
Let $\io$ and $\xi$
be as in Construction \ref{Cn_1824_S1An}(\ref{Item_1824_S1_Crs}).
There are isomorphisms of $R (S^1)$-modules
$K_0^{S^1, \bt} (C (S^1)) \cong \Z$,
via the surjective ring \hm{} which sends
$\sm \in R (S^1)$ to $1 \in \Z$, and
$K_0^{S^1} (R)
 \cong \Z [ \ov{\sm} ] / \langle \ov{\sm}^N - 1 \rangle$,
via the surjective ring \hm{} which sends
$\sm \in R (S^1)$ to $\ov{\sm} \in \Z$,
with respect to which
$\io_* \colon K_0^{S^1, \bt} (C (S^1)) \to K_0^{S^1} (R)$
is determined by $\io_* (1) = \io_* (\ov{\sm}) = 1$
and $\xi_* \colon K_0^{S^1} (R) \to K_0^{S^1, \bt} (C (S^1))$
is determined by $\xi_* (1) = 1 + \ov{\sm} + \cdots + \ov{\sm}^{N - 1}$.
\end{lem}

\begin{proof}
Recall that $\sm \in {\widehat{S^1}}$ is the identity map $S^1 \to S^1$.
The map $R (S^1) \to R ( \Z_N)$ is well known to be surjective,
and the image $\ov{\sm}$ of $\sm$ in $R ( \Z_N)$
satisfies $\ov{\sm}^N = 1$ but no lower degree polynomial relations,
so
$R (\Z_N) \cong \Z [ \ov{\sm} ] / \langle \ov{\sm}^N - 1 \rangle$.
Now the isomorphism $K_0^{S^1, \bt} (C (S^1)) \cong \Z$ is
Lemma~\ref{L_1918_EqKbt}
and the isomorphism
\begin{equation}\label{Eq_1918_KGR}
K_0^{S^1} (R)
 \cong \Z [ \ov{\sm} ] / \langle \ov{\sm}^N - 1 \rangle
\end{equation}
is Lemma~\ref{L_1918_EqKThR}.
Fix a rank one invariant \pj{} $e \in R \S C (S^1, M_N)$,
gotten from Lemma~\ref{L_1912_R_Fix}.

By Lemma~\ref{L_1918_EqKThR},
the isomorphism~(\ref{Eq_1918_KGR}) can be chosen to send
the class $[e R]$ of the right module $e R$ to~$1$.
We have $\io_* ([e R]) = [e C (S^1, M_N)]$, the class of some
rank one free module, but, by Lemma~\ref{L_1918_EqKbt}, only
one element of $K_0^{S^1, \bt} (C (S^1))$, namely $1 \in \Z$,
comes from a rank one free module.
So $\io_* (1) = 1$.
Since $\ov{\sm}$
is the class of $e R$ with a different action of~$S^1$,
we get $\io_* (\ov{\sm}) = 1$ for the same reason.

By Lemma~\ref{L_1912_R_Fix}, $\xi (1)$ is a sum of $N$ rank one
$\gm$-invariant \pj{s} in~$R$.
It follows from Lemma~\ref{L_1918_EqKThR} that, under the
isomorphism~(\ref{Eq_1918_KGR}),
each corresponds to some element
of $(\Z_N)^{\wedge} \S R (\Z_N)$,
that is, to some power $\ov{\sm}^k$ with $0 \leq k \leq N - 1$.
So there are $m_0, m_1, \ldots, m_{N - 1} \in \{ 0, 1, \ldots, N - 1 \}$
such that $\xi_* (1) = \sum_{j = 0}^{N - 1} \ov{\sm}^{m_j}$.
Since $\sm \cdot 1 = 1$ in $K_0^{S^1, \bt} (C (S^1)) \cong \Z$,
it follows that $\sm \cdot \xi_* (1) = \xi_* (1)$.
The only possibility is then
$\xi_* (1) = 1 + \ov{\sm} + \cdots + \ov{\sm}^{N - 1}$.
\end{proof}

\begin{lem}\label{L_1918_MapRN}
In Construction \ref{Cn_1824_S1An},
assume that for all $n \in \Nz$ the numbers $l_{0, 1} (n)$
and $l_{1, 0} (n)$ are both odd.
Let $A$ and $\af \colon S^1 \to \Aut (A)$ be as in
Construction \ref{Cn_1824_S1An}(\ref{Item_1924_S1_A}).
Then there is an injective $R (S^1)$-module \hm{}
from $\Z [ \ov{\sm} ] / \langle \ov{\sm}^N - 1 \rangle$
to $K_0^{S^1} (A)$.
\end{lem}

The hypotheses are much stronger than necessary,
but this statement is easy to prove.

\begin{proof}[Proof of Lemma~\ref{L_1918_MapRN}]
The inclusion of $R = A_{0, 0}$ in $A_0 = A_{0, 0} \oplus A_{0, 1}$
is injective on equivariant K-theory.
Since equivariant K-theory commutes with direct limits
(Theorem 2.8.3(6) of~\cite{Phl1}),
it is now enough to prove that
$(\nu_n)_* \colon K_0^{S^1} (A_n) \to K_0^{S^1} (A_{n + 1})$
is injective for all $n \in \Nz$.
Identify $K_0^{S^1, \bt} (C (S^1))$ and $K_0^{S^1} (R)$,
as well as $\io_*$ and $\xi_*$, as in Lemma~\ref{L_1918_EqKSys}.
Also, observe that, with these identifications,
for all $t \in \R$ the map
$( {\widetilde{\bt}}_t )_* \colon
  K_0^{S^1} (C (S^1)) \to K_0^{S^1} (C (S^1))$
becomes $\id_{\Z}$.
Therefore injectivity of $(\nu_n)_*$
is the same as injectivity of the map
\[
\Ph \colon \Z [ \ov{\sm} ] / \langle \ov{\sm}^N - 1 \rangle \oplus \Z
       \to \Z [ \ov{\sm} ] / \langle \ov{\sm}^N - 1 \rangle \oplus \Z
\]
which for $m, m_0, m_1, \ldots, m_{N - 1} \in \Z$ is given by
\begin{equation}\label{Eq_1918_Ph}
\begin{split}
& \Ph \left( \sum_{k = 0}^{N - 1} m_k \ov{\sm}^{k}, \, m \right)
\\
& \hspace*{2em} {\mbox{}}
 = \left( \sum_{k = 0}^{N - 1}
          \bigl( l_{0, 0} (n) m_k + l_{0, 1} (n) m \bigr) \ov{\sm}^k, \,
    l_{1, 0} (n) \sum_{k = 0}^{N - 1} m_k + 2 N l_{1, 1} (n) m \right).
\end{split}
\end{equation}

Suppose the right hand side of (\ref{Eq_1918_Ph}) is zero.
For $k = 0, 1, \ldots, N - 1$ we then have
$l_{0, 0} (n) m_k + l_{0, 1} (n) m = 0$.
Therefore
\[
m_0 = m_1 = \cdots = m_{N - 1} = - \frac{l_{0, 1} (n) m}{l_{0, 0} (n)}.
\]
Putting this in the second coordinate gives
\[
0 = 2 N l_{1, 1} (n) m
   - \frac{l_{1, 0} (n) l_{0, 1} (n) m N}{l_{0, 0} (n)}.
\]
Since $N \neq 0$, this says
$2 l_{1, 1} (n) l_{0, 0} (n) = l_{1, 0} (n) l_{0, 1} (n)$
or $m = 0$.
The hypotheses rule out the first, so $m = 0$,
whence also $m_k = 0$ for $k = 0, 1, \ldots, N - 1$.
Thus $(\nu_n)_*$ is injective.
\end{proof}

\begin{cor}\label{C_1918_NoFinRD}
Under the hypotheses of Lemma~\ref{L_1918_MapRN},
the action $\af$ does not have
finite Rokhlin dimension with commuting towers.
\end{cor}

\begin{proof}
Suppose $\af$ has finite Rokhlin dimension with commuting towers.
By Corollary 4.5 of~\cite{Gar_rokhlin_2017},
there is $n \in \N$ such that $I (S^1)^n K_0^{S^1} (A) = 0$.
Lemma~\ref{L_1918_MapRN} implies that
$I (S^1)^n R ( \Z_N) = 0$.
Lemma~\ref{L_1918_EqKgmt} then says that
\[
I (S^1)^n K_0^{S^1, {\overline{\gm}}} (C (S^1)) = 0.
\]
Since the underlying action of $S^1$ on $S^1$ is not free,
this contradicts Theorem 1.1.1 of~\cite{Phl1}.
\end{proof}

It is also not hard to prove directly that
$I (S^1)^n R ( \Z_N) \neq 0$ for all $n \in \N$.
With the notation of Lemma~\ref{L_1918_EqKSys},
there is a \hm{} $f \colon R ( \Z_N) \to {\mathbb{C}}$
such that $f ( \ov{\sm} ) = \om$.
Since $\sm - 1 \in I (S^1)$ and the map $R (S^1) \to R ( \Z_N)$
sends $\sm$ to $\ov{\sm}$, it is enough to show that
$f ( (\ov{\sm} - 1)^n) \neq 0$ for all $n \in \N$.
But ${\mathbb{C}}$ is a field and $f (\ov{\sm} - 1 ) \neq 0$.

\begin{thm}\label{T_1919_Ex}
In Construction~\ref{Cn_1824_S1An}, choose $r_0 (0) = r_1 (0) = 1$
and for $n \in \Nz$ choose
\[
l_{0, 0} (n) = l_{1, 0} (n) = 1
\andeqn
l_{0, 1} (n) = l_{1, 1} (n) = 2 n + 1.
\]
Then $A$ is simple, the action $\af$
of Construction \ref{Cn_1824_S1An}(\ref{Item_1824_S1_nu_A})
has the tracial Rokhlin property with comparison,
but $\af$ does not have
finite Rokhlin dimension with commuting towers.
\end{thm}

\begin{proof}
Simplicity is Lemma~\ref{L_1916_Simple}.
With these choices, both Lemma~\ref{L_1917_HasTRP}
and Corollary~\ref{C_1918_NoFinRD} apply.
\end{proof}

We now address the modified tracial Rokhlin property.

\begin{lem}\label{P_1929_RComm}
Let the notation be as in Construction~\ref{Cn_1824_S1An}.
Let $n \in \Nz$.
Define
\[
D_1 = M_{l_{0, 0} (n + 1) l_{0, 0} (n) + l_{0, 1} (n + 1) l_{1, 0} (n)},
\qquad
D_2 = M_{l_{0, 1} (n + 1) l_{1, 0} (n)},
\]
\[
D_3 = M_{l_{0, 0} (n + 1) l_{0, 1} (n)
           + 2 N l_{0, 1} (n + 1) l_{1, 1} (n)},
\qquad
D_4 = M_{l_{1, 0} (n + 1) l_{0, 0} (n)
           + 2 N l_{1, 1} (n + 1) l_{1, 0} (n)},
\]
and
\[
D_5 = M_{N l_{1, 0} (n + 1) l_{0, 1} (n)
       + 4 N^2 l_{1, 1} (n + 1) l_{1, 1} (n)},
\]
and define
\begin{equation}\label{Eq_1929_mj}
m_1 = N, \qquad
m_2 = N^2 - N, \qquad
m_3 = N, \qquad
m_4 = N, \andeqn
m_5 = 1.
\end{equation}
Set
\begin{equation}\label{Eq_1929_D}
D = (D_1)^{m_1} \oplus (D_2)^{m_2} \oplus (D_3)^{m_3}
  \oplus (D_4)^{m_4} \oplus (D_5)^{m_5},
\end{equation}
and for $k \in \{ 1, 2, 3, 4, 5 \}$ and $j = 1, 2, \ldots, m_j$,
let $\pi_{k, j} \colon D \to D_k$ be the projection to the
$j$~summand of $D_k$ in the definition of~$D$.
As usual, write the relative commutant of
$\nu_{n + 2, \, n} \bigl( (A_n)^{\af^{(n)}} \bigr)$
in $(A_{n + 2})^{\af^{(n + 2)}}$
as
$\nu_{n + 2, \, n} \bigl( (A_n)^{\af^{(n)}} \bigr)'
    \cap (A_{n + 2})^{\af^{(n + 2)}}$.
Then
\[
\nu_{n + 1} (p_{n + 1})
  \in \nu_{n + 2, \, n} \bigl( (A_n)^{\af^{(n)}} \bigr)'
    \cap (A_{n + 2})^{\af^{(n + 2)}},
\]
and there is an isomorphism
\[
\kp \colon \nu_{n + 2, \, n} \bigl( (A_n)^{\af^{(n)}} \bigr)'
    \cap (A_{n + 2})^{\af^{(n + 2)}}
  \to D
\]
such that for $k \in \{ 1, 2, 3, 4, 5 \}$ and $j = 1, 2, \ldots, m_j$,
we have
\[
\rank \bigl( (\pi_{1, j} \circ \kp \circ \nu_{n + 1}) (p_{n + 1}) \bigr)
 = \rank
    \bigl( (\pi_{2, j} \circ \kp \circ \nu_{n + 1}) (p_{n + 1}) \bigr)
 = l_{0, 1} (n + 1) l_{1, 0} (n),
\]
\[
\rank \bigl( (\pi_{3, j} \circ \kp \circ \nu_{n + 1}) (p_{n + 1}) \bigr)
 = 2 N l_{0, 1} (n + 1) l_{1, 1} (n),
\]
\[
\rank \bigl( (\pi_{4, j \circ \kp} \circ \nu_{n + 1}) (p_{n + 1}) \bigr)
 = 2 N l_{1, 1} (n + 1) l_{1, 0} (n),
\]
and
\[
\rank \bigl( (\pi_{5, j} \circ \kp \circ \nu_{n + 1}) (p_{n + 1}) \bigr)
 =  4 N^2 l_{1, 1} (n + 1) l_{1, 1} (n).
\]
\end{lem}

\begin{proof}
Let the notation be as in parts
(\ref{Item_1824_S1_Bn0Bn1}), (\ref{Item_1824_S1_FixPtSys}),
(\ref{Item_1824_S1_FixLim}), and~(\ref{Item_1915_S1_qn})
of Construction~\ref{Cn_1824_S1An}.
By Lemma~\ref{L_1916_FPisB}, it is enough to prove the lemma
with $B_t$ in place of $(A_t)^{\af^{(n)}}$
for $t = n, \, n + 1, \, n + 1$,
with $\ch_{t}$ and $\ch_{t, s}$ in place of $\nu_{t}$ and $\nu_{t, s}$,
and with $q_{n + 1}$ in place of $p_{n + 1}$.

First, since $q_{n + 1}$ is in the center of $B_{n + 1}$,
it commutes with the range of $\ch_{n}$.
So $\ch_{n + 1} (p_{n + 1}) \in \ch_{n + 2, \, n} (B_n)' \cap B_{n + 2}$
is clear.

We change to more convenient notation
for the structure of the algebras~$B_t$.
Write
\[
B_t = B_{t, 0} \oplus B_{t, 1}
   \oplus \cdots \oplus B_{t, N - 1} \oplus B_{t, N},
\]
with
\[
B_{t, 0} = B_{t, 1} = \cdots = B_{t, N - 1} = M_{r_0 (n)}
\andeqn
B_{t, N} = M_{r_1 (n)}.
\]
Thus $B_{t, 0} \oplus B_{t, 1} \oplus \cdots \oplus B_{t, N - 1}$
is what was formerly called $B_{t, 0}$,
and $B_{t, N}$ is what was formerly called $B_{t, 1}$.
The partial multiplicities in $\ch_t$
of the maps $B_{t, j} \to B_{t + 1, k}$ are
\[
m_t (k, j) = \begin{cases}
   l_{0, 0} (t) & \hspace*{1em} j, k \in \{ 0, 1, \ldots, N - 1 \}
        \\
   l_{0, 1} (t) & \hspace*{1em}
      {\mbox{$j = N$ and $k \in \{ 0, 1, \ldots, N - 1 \}$}}
        \\
   l_{1, 0} (t) & \hspace*{1em}
      {\mbox{$j \in \{ 0, 1, \ldots, N - 1 \}$ and $k = N$}}
       \\
   2 N l_{1, 1} (t) & \hspace*{1em} j - k = N.
\end{cases}
\]
An easy calculation now
shows that the partial multiplicities in $\ch_{n + 2, n}$
of the maps $B_{n, j} \to B_{n + 2, k}$ are
\begin{equation}\label{Eq_1929_Cases}
\widetilde{m} (k, j) = \begin{cases}
& \hspace*{-1em}
l_{0, 0} (n + 1) l_{0, 0} (n) + l_{0, 1} (n + 1) l_{1, 0} (n)
\\
& \hspace*{10em} {\mbox{}}
      {\mbox{$j, k \in \{ 0, 1, \ldots, N - 1 \}$ and $j = k$}}
        \\
& \hspace*{-1em}
l_{0, 1} (n + 1) l_{1, 0} (n)
\\
& \hspace*{10em} {\mbox{}}
      {\mbox{$j, k \in \{ 0, 1, \ldots, N - 1 \}$ and $j \neq k$}}
        \\
& \hspace*{-1em}
l_{0, 0} (n + 1) l_{0, 1} (n)
           + 2 N l_{0, 1} (n + 1) l_{1, 1} (n)
\\
& \hspace*{10em} {\mbox{}}
      {\mbox{$j = N$ and $k \in \{ 0, 1, \ldots, N - 1 \}$}}
        \\
& \hspace*{-1em}
l_{1, 0} (n + 1) l_{0, 0} (n)
           + 2 N l_{1, 1} (n + 1) l_{1, 0} (n)
\\
& \hspace*{10em} {\mbox{}}
      {\mbox{$j \in \{ 0, 1, \ldots, N - 1 \}$ and $k = N$}}
       \\
& \hspace*{-1em}
N l_{1, 0} (n + 1) l_{0, 1} (n)
       + 4 N^2 l_{1, 1} (n + 1) l_{1, 1} (n)
\\
& \hspace*{10em} {\mbox{}} j - k = N.
\end{cases}
\end{equation}
Recall that if
\[
\ph \colon M_{r_1} \oplus M_{r_2} \oplus \cdots \oplus M_{r_s}
 \to M_{l_1 r_1 + l_2 r_2 + \cdots + l_2 r_2}
\]
is unital with partial multiplicities $l_1, l_2, \ldots, l_s$,
then the relative commutant of the range of $\ph$ is isomorphic to
$M_{l_1} \oplus M_{l_2} \oplus \cdots \oplus M_{l_s}$,
with the identity of $M_{l_t}$ being,
by abuse of notation, $\ph (1_{M_{l_t}})$.
Therefore the values of ${\widetilde{m}} (k, j)$ are the matrix sizes
of the summands in $\ch_{n + 2, \, n} (B_n)' \cap B_{n + 2}$.
That the exponents $m_j$ in~(\ref{Eq_1929_D}) are given as
in~(\ref{Eq_1929_mj}) follows simply by counting the number of
times each case in~(\ref{Eq_1929_Cases}) occurs.

The rank of the image of $q_{n + 1}$ in each summand is the
contribution to ${\widetilde{m}} (k, j)$ from maps factoring
through $B_{n + 1, N}$ (in the original notation, $B_{n + 1, 1}$),
as opposed to the other summands.
That these numbers are in the statement of the lemma is again
an easy calculation.
\end{proof}

\begin{prp}\label{L_1X05_HasMdTRP}
In Construction~\ref{Cn_1824_S1An}, assume that $r_0 (0) \leq r_1 (0)$
and that for all $n \in \Nz$
we have
\[
l_{1, 0} (n) \geq l_{0, 0} (n),
\qquad
l_{0, 1} (n) \geq l_{0, 0} (n),
\]
\[
l_{1, 1} (n) \geq l_{0, 1} (n),
\andeqn
l_{1, 1} (n) \geq l_{1, 0} (n).
\]
Further assume that
\[
\lim_{n \to \I} \frac{l_{0, 1} (n)}{l_{0, 0} (n)} = \I
\andeqn
\lim_{n \to \I} \frac{l_{1, 1} (n)}{l_{1, 0} (n)} = \I.
\]
Then the action $\af$
of Construction \ref{Cn_1824_S1An}(\ref{Item_1824_S1_nu_A})
has the tracial Rokhlin property with comparison
and has the modified tracial Rokhlin property
as in Definition~\ref{moditra}.
Moreover, given finite sets $F \subseteq A$, $F_0 \subseteq A^{\af}$,
and $S \subseteq C (G)$, as well as $\ep > 0$,
$x \in A_{+}$ with $\| x \| = 1$,
and $y \in (A^{\alpha})_{+} \SM \{ 0 \}$,
it is possible to choose a projection $p \in A^{\alpha}$,
a unital completely positive
contractive map $\ph \colon C (G) \to p A p$,
and a partial isometry $s \in A^{\alpha}$,
such that the conditions of both
Definition~\ref{traR} and Definition~\ref{moditra}
are simultaneously satisfied.
\end{prp}

As usual, the hypotheses are overkill.

\begin{proof}[Proof of Proposition~\ref{L_1X05_HasMdTRP}]
Since $A$ is finite, as usual, the argument of
Lemma~1.16 of~\cite{phill23} applies, and shows that it suffices
to verify this without the condition $\| p x p \| > 1 - \ep$,
simultaneously in both definitions.

The method used for
the proof of Lemma~\ref{L_1917_HasTRP} also applies here.
The key new point is that in every summand of
\[
\nu_{n + 2, \, n} \bigl( (A_n)^{\af^{(n)}} \bigr)'
    \cap (A_{n + 2})^{\af^{(n + 2)}}
\]
as described in Lemma~\ref{P_1929_RComm}, the rank of the
component of $\nu_{n + 1} (p_{n + 1})$
is greater than half the corresponding matrix size.
Therefore the rank of
the component of $1 - \nu_{n + 1} (p_{n + 1})$
is less than the rank of the component of $\nu_{n + 1} (p_{n + 1})$.
Rank comparison implies that there exists
\[
s_{0} \in \nu_{n + 2, \, n} \bigl( (A_n)^{\af^{(n)}} \bigr)'
    \cap (A_{n + 2})^{\af^{(n + 2)}}
\]
such that
\[
1 - \nu_{n + 1} (p_{n + 1}) = s_{0} ^{*} s_{0}
\andeqn
s_{0} s_{0} ^{*} \leq \nu_{n + 1} (p_{n + 1}).
\]
Now set $s = \nu_{\infty, n + 1} (s_{0})$.
The rest of the proof is an easy computation.
\end{proof}

\section{Actions of $S^1$ on Kirchberg algebras}\label{Sec_2114_OI}

\indent
The purpose of this section is to construct a action of $S^1$ on $\OI$
which has the tracial Rokhlin property with comparison.
As observed in Lemma~\ref{2610_No_fin_RD} below
(really just Corollary 4.23 of \cite{Gar_rokhlin_2017}),
there is no action of $S^1$
on $\OI$ which has finite Rokhlin dimension with commuting towers.
Although we don't carry this out,
easy modifications should work for ${\mathcal{O}}_n$ in place of~$\OI$,
and it should also be not too hard to generalize the construction
to cover actions of $(S^1)^m$ and $(S^1)^{\Z}$.

Tensoring with our action gives an action on any unital Kirchberg algebra
which has the tracial Rokhlin property with comparison
(Definition~\ref{traR})
except for condition (\ref{Item_902_pxp_TRP})
(the one which says $\| p x p \| > 1 - \varepsilon$).
Recall (see Remark~\ref{R_2015_nzp} and Remark~\ref{L_2015_S4_nzp})
that this condition is not needed for any of the
permanence properties we prove for fixed point algebras
and crossed products
of actions with the tracial Rokhlin property with comparison.

\begin{lem}\label{L_2521_Exist}
There exist an action $\bt \colon S^1 \to \Aut (\OI)$,
a nonzero \pj{} $p \in \OI^{\bt}$, and a unital \hm{}
$\io \colon C (S^1, \OT) \to p \OI p$ such that $\io$ is equivariant
when $C (S^1, \OT)$ is equipped with the action
(following Notation~\ref{N_1X07_Lt})
$\zt \mapsto \Lt_{\zt} \otimes \id_{\OT}$,
and such that $p$ has the following property.
For every $n \in \N$ and every \pj{} $q \in M_n \otimes \OI^{\bt}$,
there is $t \in M_n \otimes \OI^{\bt}$
such that $\| t \| = 1$ and $t^* (e_{1, 1} \otimes p) t = q$.
\end{lem}

\begin{proof}
Write $s_1, s_2, \ldots$ for the standard generating isometries
in~$\OI$.
Also write $w$ for the standard generating unitary in $C (S^1)$,
which is the function $w (z) = z$ for $z \in S^1$.
When convenient, identify $C (S^1, \OT)$ with $C (S^1) \otimes \OT$.

Let $\gm \colon S^1 \to \Aut (\OI)$ be the quasifree action
(in the sense of~\cite{Kts2})
determined by, for $\zt \in S^1$,
\[
\gm_{\zt} (s_j)
 = \begin{cases}
   s_j & \hspace*{1em} j = 1
        \\
   \zt s_j & \hspace*{1em} j = 2
       \\
   \zt^{-1} s_j & \hspace*{1em} j = 3
       \\
   s_j & \hspace*{1em} j = 4, 5, \ldots,
\end{cases}
\]
and let ${\ov{\gm}} \colon S^1 \to \Aut (\OI)$
be the quasifree action $\zt \mapsto \gm_{\zt}^{- 1}$.
Define an action
$\bt^{(0)} \colon S^1 \to \Aut (\OI \otimes \OI \otimes \OI)$ by
$\bt^{(0)}_{\zt} = \gm_{\zt} \otimes {\ov{\gm}}_{\zt} \otimes \id_{\OI}$
for $\zt \in S^1$.

In $\OI$, define
\[
e = 1 - s_1 s_1^*,
\qquad
v_1 = s_2 e,
\andeqn
v_2 = 1 - s_1 s_1^* - s_2 s_2^* + s_2 s_1 s_2^*.
\]
In $\OI \otimes \OI \otimes \OI$, define
\[
p_0 = e \otimes e \otimes e
\andeqn
u = v_1 \otimes v_2^* \otimes e + v_2 \otimes v_1^* \otimes e.
\]
One checks that $u$ is a unitary in $e (\OI \otimes \OI \otimes \OI) e$,
and that
\begin{equation}\label{Eq_2605_Act_u}
\bt^{(0)}_{\zt} (u) = \zt u
\end{equation}
for all $\zt \in S^1$.
Since $[e] = 0$ in $K_0 (\OI)$, there is a unital \hm{}
$\mu \colon \OT \to e \OI e$.
Since $u$ commutes with $e \otimes e \otimes \mu (a)$ for all $a \in \OT$,
there is a \uhm{}
$\io_0 \colon C (S^1) \otimes \OT \to p_0 (\OI \otimes \OI \otimes \OI) p_0$
such that $\io_0 (w \otimes 1) = u$ and
$\io_0 (1 \otimes a) = e \otimes e \otimes \mu (a)$ for $a \in \OT$.
Then $\io_0$ is equivariant by~(\ref{Eq_2605_Act_u}).

Let $\sm \colon \OI \otimes \OI \otimes \OI \to \OI$ be an isomorphism.
Define $\io = \sm \circ \io_0$,
$p = \sm (p_0)$, and $\bt_{\zt} = \rh \circ \bt^{(0)}_{\zt} \circ \rh^{-1}$
for $\zt \in S^1$.
Then $\bt$ an action of $S^1$ on~$\OI$,
$p$ is a \pj{} in $\OI^{\bt}$, and
$\io$ is an equivariant unital \hm{} from $C (S^1, \OT)$ to $p \OI p$.

It remains to prove the last sentence.
For this purpose, it suffices to use
$\OI \otimes \OI \otimes \OI$ in place of $\OI$,
$\bt^{(0)}$ in place of $\bt$, and $p_0$ in place of~$p$.
We may also assume that $q = 1_{M_n} \otimes 1$ for some $n \in \N$.

We first claim that $\OI^{\gm}$ is purely infinite and simple.
One checks from the definition of $\om$-invariance for subsets of~$\Gm$
(Definition 3.3 of~\cite{Kts2}),
with $\om = (0, 1, -1, 0, 0, \ldots)$,
that, in our case, $\Z$ has no nontrivial invariant subsets.
So Proposition~7.4 of~\cite{Kts2} implies that $C^* (S^1, \OI, \gm)$
is purely infinite and simple.
Now the claim follows from Theorem~\ref{satunonabel}.
Since ${\ov{\gm}}$ has the same fixed point algebra,
it also follows that $\OI^{\ov{\gm}}$ is purely infinite and simple.

Next, define an action
${\widetilde{\bt}} \colon
   S^1 \times S^1 \to \Aut (\OI \otimes \OI \otimes \OI)$
by
\[
{\widetilde{\bt}}_{\zt_1, \zt_2}
 = \gm_{\zt_1} \otimes {\ov{\gm}}_{\zt_2} \otimes \id_{\OI}
\]
for $\zt_1, \zt_2 \in S^1$.
Then
\[
(\OI \otimes \OI \otimes \OI)^{\widetilde{\bt}}
 = \OI^{\gm} \otimes \OI^{\ov{\gm}} \otimes \OI,
\]
which is purely infinite and simple.
Hence so is $M_n \otimes (\OI \otimes \OI \otimes \OI)^{\widetilde{\bt}}$.
Clearly $e_{1, 1} \otimes p_0$ is a nonzero \pj{} in
$M_n \otimes (\OI \otimes \OI \otimes \OI)^{\widetilde{\bt}}$.
Therefore there exists an isometry
$t \in M_n \otimes (\OI \otimes \OI \otimes \OI)^{\widetilde{\bt}}$
such that $t^* t = 1$ and $t t^* \leq e_{1, 1} \otimes p_0$.
It is immediate that
$t \in M_n \otimes (\OI \otimes \OI \otimes \OI)^{\bt^{(0)}}$
and $t^* (e_{1, 1} \otimes p_0) t = q$.
\end{proof}

\begin{cns}\label{Cns_2114_OIA}
We define an equivariant direct system,
and a corresponding direct limit action, as follows.
That we actually get an equivariant direct system is proved afterwards,
in Lemma~\ref{L_2605_Cns_OK}.
\begin{enumerate}
\item\label{I_2114_OIA_A}
Let $\bt \colon S^1 \to \Aut (\OI)$, $p \in \OI^{\bt} \subseteq \OI$, and
$\io \colon C (S^1, \OT) \to p \OI p$
be as in Lemma~\ref{L_2521_Exist}.
Let $(e_{j, k})_{j, k \in \Nz}$ be the standard system of 
matrix units for $K = K (l^2 (\Nz))$.
For $n \in \Nz$ define $f_n \in K \otimes \OI$ by
\[
f_n = e_{0, 0} \otimes 1 + e_{1, 1} \otimes p + e_{2, 2} \otimes p
                                        + \cdots e_{n, n} \otimes p.
\]
Then define
\[
A_{n, 0} = f_{n} ( K \otimes \OI ) f_{n},
\quad
A_{n, 1} = C (S^1) \otimes \OT = C (S^1, \OT),
\quad {\mbox{and}} \quad
A_{n} = A_0 \oplus A_1.
\]
\item\label{I_2521_OIA_Action0}
Define actions 
\[
\af^{(n, 0)} \colon S^{1} \to \Aut (A_{n, 0}),
\quad
\af^{(n, 1)} \colon S^1 \to \Aut (A_{n, 1})
\quad {\mbox{and}} \quad
\af^{(n)} \colon S^1 \to \Aut (A_{n})
\]
as follows.
For $\zeta \in S^1$, take $\af^{(n, 0)}_{\zeta}$ to be the
restriction and corestriction of $\id_K \otimes \bt_{\zt}$
to the subalgebra $f_{n} ( K \otimes \OI ) f_{n}$, and,
using Notation~\ref{N_1X07_Lt} for the first, take
\[
\af^{(n, 1)}_{\zeta} = \Lt_{\zt} \otimes \id_{\OT}
\quad {\mbox{and}} \quad
\af^{(n, 1)}_{\zeta} = \af^{(n, 0)}_{\zeta} \oplus \af^{(n, 1)}_{\zeta}.
\]
\item\label{I_2114_OIA_PreParts}
We give preliminaries for the construction of the maps of the system.
Fix \nzp{s} $q_0, q_1 \in \OT$ such that $q_0 + q_1 = 1$.
For $n \in \Nz$ choose unital \hm{s}
\[
\mu_{n, 0} \colon A_{n, 0} \to q_0 \OT q_0,
\andeqn
\mu_{n, 1} \colon \OT \to q_1 \OT q_1.
\]
(The \hm{} $\mu_{n, 1}$ can be chosen to be independent of~$n$.)
Use Proposition~4.1 of~\cite{Rokhdimtracial} to choose an isomorphism
$\ld_n \colon C (S^1, A_{n, 0}) \to C (S^1, A_{n, 0})$ which is equivariant
when $S^1$ acts on the domain
via $\zt \mapsto \Lt_{\zt} \otimes \af^{(n, 0)}_{\zeta}$
and on the codomain via $\zt \mapsto \Lt_{\zt} \otimes \id_{A_{n, 0}}$.
Let $\kp_n \colon A_{n, 0} \to C (S^1, A_{n, 0})$ be the inclusion of
elements of $A_{n, 0}$ as constant functions.
\item\label{I_2114_OIA_Parts}
For $n \in \Nz$ and $j, k \in \{ 0, 1 \}$ define \hm{s}
$\nu_{n + 1, n}^{(k, j)} \colon A_{n, j} \to A_{n + 1, k}$ as follows.
We let $\nu_{n + 1, n}^{(0, 0)}$ be the inclusion of
$A_{n, 0}$ in $A_{n + 1, 0}$
which comes from the relation $f_n \leq f_{n + 1}$.
For $a \in C (S^1, \OT)$, with $\io$ as in~(\ref{I_2114_OIA_A}),
take $\nu_{n + 1, n}^{(0, 1)} (a) = e_{n + 1, \, n + 1} \otimes \io (a)$.
Set
\[
\nu_{n + 1, n}^{(1, 0)}
 = (\id_{C (S^1)} \otimes \mu_{n, 0}) \circ \ld_n \circ \kp_n
\andeqn
\nu_{n + 1, n}^{(1, 1)} (a)
 = \id_{C (S^1)} \otimes \mu_{n, 1}.
\]
Then define $\nu_{n + 1, n} \colon A_{n} \to A_{n + 1}$ by
\[
\nu_{n + 1, n} (a_0, a_1)
 = \bigl( \nu_{n + 1, n}^{(0, 0)} (a_0) + \nu_{n + 1, n}^{(0, 1)} (a_1),
     \, \nu_{n + 1, n}^{(1, 0)} (a_0) + \nu_{n + 1, n}^{(1, 1)} (a_1) \bigr)
\]
for $a_0 \in A_{n, 0}$ and $a_1 \in A_{n, 1}$.
\item\label{I_2114_OIA_nm}
For $m, n \in \Nz$ with $m \leq n$, define
\[
\nu_{n, m} = \nu_{n, \, n - 1} \circ \nu_{n - 1, \, n - 2}
          \circ \cdots \circ \nu_{m + 1, \, m}.
\]
\item\label{I_2114_OIA_Lim}
Define $A = \dirlim_n A_n$, using the maps $\nu_{n + 1, n}$,
and let $\af \colon S^1 \to \Aut (A)$ be the direct limit action.
\end{enumerate}
\end{cns}
 
\begin{lem}\label{L_2605_Cns_OK}
Adopt the notation of Construction~\ref{Cns_2114_OIA}.
\begin{enumerate}
\item\label{I_L_2605_Cns_OK_Algs}
For $m, n \in \Nz$ with $n \geq m$, the map
$\nu_{n + 1, n} \colon A_{n} \to A_{n + 1}$
is an equivariant unital \hm.
\item\label{I_L_2605_Cns_inj}
For $n \in \Nz$ and $j, k \in \{ 0, 1 \}$,
the map $\nu_{n + 1, n}^{(k, j)}$ in
Construction \ref{Cns_2114_OIA}(\ref{I_2114_OIA_Parts}) is injective.
\item\label{I_L_2605_Cns_OK_Lim}
The action $\af \colon S^1 \to \Aut (A)$ is a well defined \ct{} action.
\end{enumerate}
\end{lem}

\begin{proof}
We prove (\ref{I_L_2605_Cns_OK_Algs}).
We need only consider $\nu_{n + 1, n}$ for $n \in \Nz$.
That $\nu_{n + 1, n}$ is a unital \hm{} follows from the
computations (in which the summands are orthogonal)
\[
\nu_{n + 1, n}^{(0, 0)} (f_n) + \nu_{n + 1, n}^{(0, 1)} (1)
  = f_{n} + e_{n + 1, \, n + 1} \otimes p
  = f_{n + 1}
\]
and
\[
\nu_{n + 1, n}^{(1, 0)} (f_n) + \nu_{n + 1, n}^{(1, 1)} (1)
 = 1 \otimes \mu_{n, 0} (f_n) + 1 \otimes \mu_{n, 1} (1)
 = 1 \otimes q_0 + 1 \otimes q_1
 = 1.
\]
For the actions to be well defined, the only point which needs to be
checked is that $f_n$ is invariant under $\id_K \otimes \bt_{\zt}$.
This follows from invariance of $p$ under $\bt$,
which is a consequence of the choices made using Lemma~\ref{L_2521_Exist}.

For equivariance, it is enough to check equivariance of
the maps $\nu_{n + 1, \, n}^{(k, j)}$ in
Construction \ref{Cns_2114_OIA}(\ref{I_2114_OIA_Parts}).
This is immediate for $\nu_{n + 1, \, n}^{(0, 0)}$
and $\nu_{n + 1, \, n}^{(1, 1)}$,
and by the choices made using Lemma~\ref{L_2521_Exist}
for $\nu_{n + 1, \, n}^{(0, 1)}$.
The map $\nu_{n + 1, \, n}^{(1, 0)}$ is the composition
(writing pairs consisting of an algebra and an action)
\[
\begin{split}
\bigl( A_{n, 0}, \, \af^{(n, 0)} \bigr)
& \stackrel{\kp_n}{\longrightarrow}
\bigl( C (S^1) \otimes A_{n, 0}, \, \Lt \otimes \af^{(n, 0)} \bigr)
\\
& \stackrel{\ld_n}{\longrightarrow}
\bigl( C (S^1) \otimes A_{n, 0}, \, \Lt \otimes \id_{A_{n, 0}} \bigr)
\stackrel{\id_{C (S^1)} \otimes \mu_{n, 0}}{\longrightarrow}
\bigl( C (S^1) \otimes \OT, \, \Lt \otimes \id_{\OT} \bigr).
\end{split}
\]
The maps $\kp_n$ and $\id_{C (S^1)} \otimes \mu_{n, 0}$ are
obviously equivariant, and $\ld_n$ is equivariant by construction.

For~(\ref{I_L_2605_Cns_inj}), in the compositions defining these maps,
$\io$ is injective by the choices made using Lemma~\ref{L_2521_Exist}
and $\ld_n$ is injective by construction.
Injectivity of everything else which appears is immediate.

Part~(\ref{I_L_2605_Cns_OK_Lim})
is immediate from part~(\ref{I_L_2605_Cns_OK_Algs}).
\end{proof}

The following lemma is known, but we have not found a reference.

\begin{lem}\label{L_2521_PI_bound}
Let $A$ be a unital purely infinite simple \ca, let $\rh > 0$,
and let $x \in A$ satisfy $\| x \| > \rh$.
Then there are $a, b \in A$ such that $a x b = 1$,
$\| a \| < \rh^{-1}$, and $ \| b \| \leq 1$.
\end{lem}

\begin{proof}
Choose $\rh_0$ such that $(\rh \| x \|)^{1/2} < \rh_0 < \| x \|$.

Define \cfn{s} $f, f_0 \colon [0, \I) \to [0, \I)$ by
\[
f (\ld)
 = \begin{cases}
   \rh_0^{-2} \ld & \hspace*{1em} 0 \leq \ld \leq \rh_0^2
        \\
   1 & \hspace*{1em} \rh_0^2 < \ld
\end{cases}
\andeqn
f_0 (\ld)
 = \begin{cases}
   0           & \hspace*{1em} 0 \leq \ld \leq \rh_0^2
        \\
   \ld - \rh_0^2 & \hspace*{1em} \rh_0^2 < \ld.
\end{cases}
\]
Then $f_0 (x^* x) \neq 0$ since $\| x^* x \| > \rh_0^2$.
Therefore there are a \nzp{} $p \in {\ov{f_0 (x^* x) A f_0 (x^* x)}}$,
and, by pure infiniteness,
a partial isometry $s \in A$ such that $s^* s = 1$ and $s s^* \leq p$.
We have $f (x^* x) p = p$ and $p s = s$.
Therefore
\[
1 = s^* p s
  = s^* f (x^* x) p s
  = s^* f (x^* x) s
  \leq \rh_0^{-2} s^* x^* x s.
\]
Hence $s^* x^* x s$ is invertible, with
$\| (s^* x^* x s)^{-1} \| \leq \rh_0^{-2}$.
Take $b = s$ and $a = (s^* x^* x s)^{-1} s^* x^*$,
noting that $\| a \| \leq \rh_0^{-2} \| x \| < \rh^{-1}$.
\end{proof}

\begin{lem}\label{L_2521_SLim}
Let $(B_n)_{n \in \Nz}$ be a direct system of unital \ca{s},
with unital maps $\mu_{n, m} \colon B_m \to B_{n}$.
Suppose that for all $n \in \Nz$ we are given a direct sum decomposition
$B_n = B_n^{(0)} \oplus B_n^{(1)}$, in which both summands are nonzero.
For $m, n \in \Nz$ and $j, k \in \{ 0, 1 \}$ let
$\mu_{n, m}^{(k, j)} \colon B_m^{(j)} \to B_n^{(k)}$
be the corresponding partial map.
Assume the following:
\begin{enumerate}
\item\label{I__2521_SLim_pi}
$B_n^{(1)}$ is purely infinite and simple for all $n \in \Nz$.
\item\label{I__2521_SLim_inj}
$\mu_{n + 1, n}^{(k, j)}$ is injective for all $n \in \Nz$
and all $j, k \in \{ 0, 1 \}$.
\item\label{I__2521_SLim_full}
For every $n \in \Nz$ there is $t \in B_{n + 1}^{(0)}$
such that $\| t \| = 1$
and $t^* \mu_{n + 1, n}^{(0, 1)} (1_{B_n^{(1)}}) t$
is the identity of $B_{n + 1}^{(0)}$.
\end{enumerate}
Then $\dirlim_n B_n$ is purely infinite and simple.
\end{lem}

\begin{proof}
Set $B = \dirlim_n B_n$,
and for $m \in \Nz$ let $\mu_{\I, n} \colon B_n \to B$
be the standard map associated with the direct limit.
Also, for $n \in \Nz$ and $j \in \{ 0, 1 \}$ let $p_n^{(j)}$
be the identity of $B_n^{(j)}$.

Let $x \in B \SM \{ 0 \}$.
We need to find $a, b \in B$ such that $a x b = 1$.
\Wolog{} $\| x \| = 1$.
Choose $n \in \Nz$ and $y \in B_n$ such that
$\| \mu_{\I, n} (y) - x \| < \frac{1}{3}$.
Then $\| y \| > \frac{2}{3}$.
Set $z = \mu_{n + 1, n} (y)$.
Write $y = (y_0, y_1)$ and $z = (z_0, z_1)$ with $y_j \in B_n^{(j)}$
and $z_j \in B_{n + 1}^{(j)}$ for $j \in \{ 0, 1 \}$.
We have $\| y_j \|= \| y \|$ for some $j$,
so injectivity of $\mu_{n + 1, n}^{(1, j)}$
implies that $\| z_1 \| > \frac{2}{3}$.

Lemma~\ref{L_2521_PI_bound} provides $r_0, s_0 \in B_{n + 1}^{(1)}$
such that
\[
\| r_0 \| < \frac{3}{2},
\qquad
\| s_0 \| \leq 1,
\andeqn
r_0 z_1 s_0 = p_{n + 1}^{(1)}.
\]
Condition~(\ref{I__2521_SLim_full})
provides $t \in B_{n + 2}^{(0)}$
such that $\| t \| = 1$
and
$t^* \mu_{n + 2, \, n + 1}^{(0, 1)} \bigl( p_{n + 1}^{(1)} \bigr) t
  = p_{n + 2}^{(0)}$.
Define elements of $B_{n + 1}^{(0)}$ by
\[
c_0 = t^* \mu_{n + 2, \, n + 1}^{(0, 1)} (r_0)
\andeqn
d_0 = \mu_{n + 2, \, n + 1}^{(0, 1)} (s_0) t.
\]
Then
\begin{equation}\label{Eq_2605_Zero}
\| c_0 \| < \frac{3}{2},
\qquad
\| d_0 \| \leq 1,
\andeqn
t^* \mu_{n + 2, \, n + 1} \bigl( p_{n + 1}^{(0)} \bigr) t = 0.
\end{equation}
Using the last part of~(\ref{Eq_2605_Zero}),
and regarding everything in the first expression as being in $B_{n + 1}^{(0)}$,
we get
\begin{equation}\label{Eq_2605_3St}
c_0 \mu_{n + 2, \, n + 1} (z) d_0
 = c_0 \mu_{n + 2, \, n + 1}^{(0, 1)} (z_1) d_0
 = t^* \mu_{n + 2, \, n + 1}^{(0, 1)} (r_0 z_1 s_0) t
 = p_{n + 2}^{(0)}.
\end{equation}

Since $\mu_{n + 2, \, n + 1}^{(1, 1)}$ is injective
(by~(\ref{I__2521_SLim_inj}))
and $B_{n + 2}^{(1)}$ is purely infinite and simple,
Lemma~\ref{L_2521_PI_bound} provides $r_1, s_1 \in B_{n + 2}^{(1)}$
such that
\[
\| r_1 \| < \frac{3}{2},
\qquad
\| s_1 \| \leq 1,
\andeqn
r_1 \mu_{n + 2, \, n + 1}^{(1, 1)} (z_1) s_1 = p_{n + 2}^{(1)}.
\]
Define elements of $B_{n + 1}^{(1)}$ by
\[
q = \mu_{n + 2, \, n + 1}^{(1, 1)} \bigl( p_{n + 1}^{(1)} \bigr),
\qquad
c_1 = r_1 q,
\andeqn
d_1 = q s_1.
\]
Then, analogously to~(\ref{Eq_2605_Zero}) and~(\ref{Eq_2605_3St}),
\[
\| c_1 \| < \frac{3}{2},
\qquad
\| d_1 \| \leq 1,
\andeqn
c_1 \mu_{n + 2, \, n + 1} (z) d_1
 = c_1 \mu_{n + 2, \, n + 1}^{(1, 1)} (z_1) d_1
 = p_{n + 2}^{(1)}.
\]

Taking $c = (c_0, c_1)$ and $d = (d_0, d_1)$,
we get
\[
\| c \| < \frac{3}{2},
\qquad
\| d \| \leq 1,
\andeqn
c \mu_{n + 2, \, n + 1} (z) d = 1.
\]
Setting $a = \mu_{\I, \, n + 2} (c)$ and $h = \mu_{\I, \, n + 2} (d)$,
we get $\| a \| < \frac{3}{2}$, $\| h \| \leq 1$, and
$a \mu_{\I, \, n} (y) h = 1$.
Therefore
\[
\| a x h - 1 \|
  \leq \| a \| \| x - \mu_{\I, \, n} (y) \| \| h \|
  < \left( \frac{3}{2} \right) \left( \frac{1}{3} \right)
  = \frac{1}{2}.
\]
So $a x h$ is invertible.
Setting $b = h (a x h)^{-1}$ gives $a x b = 1$, as desired.
\end{proof}
 
\begin{lem}\label{L_2605_Cns_OI}
The algebra $A$ in Construction~\ref{Cns_2114_OIA}(\ref{I_2114_OIA_Lim})
is isomorphic to~$\OI$.
\end{lem}

\begin{proof}
We first claim that $A$ is purely infinite and simple.
We use Lemma~\ref{L_2521_SLim},
with $B_n^{(0)} = A_{n, 1}$, $B_n^{(1)} = A_{n, 0}$, and
$\mu_{n, m} = \nu_{n, m}$.
Thus $\mu_{n + 1, n}^{(k, j)} = \nu_{n + 1, n}^{(1 - k, \, 1 - j)}$
for $n \in \Nz$ and $j, k \in \{ 0, 1 \}$.
Condition~(\ref{I__2521_SLim_pi}) in Lemma~\ref{L_2521_SLim}
is immediate because $A_{n, 0}$ is a corner of $K \otimes \OI$.
The maps $\mu_{n + 1, n}^{(k, j)}$ are all injective by
Lemma \ref{L_2605_Cns_OK}(\ref{I_L_2605_Cns_inj}),
which is condition~(\ref{I__2521_SLim_inj}).
For condition~(\ref{I__2521_SLim_full}),
following the notation of
Construction \ref{Cns_2114_OIA}(\ref{I_2114_OIA_PreParts}),
we have $\mu_{n + 1, n}^{(0, 1)} (f_n) = 1 \otimes q_0$.
Since $q_0$ is a \nzp{} in $\OT$, there is $t_0 \in \OT$
such that $t_0^* t_0 = 1$ and $t_0 t_0^* = q_0$.
Then $t = 1 \otimes t_0$ satisfies
$t^* \mu_{n + 1, n}^{(1, 0)} (f_n) t = 1$.
The claim now follows from Lemma~\ref{L_2521_SLim}.

The algebra $A$ satisfies the Universal Coefficient Theorem because it is
a direct limit, with injective maps,
of algebras which satisfy the Universal Coefficient Theorem.

We have $K_1 (A_n) = 0$ for all $n \in \Nz$, so $K_1 (A) = 0$.
Following the notation of
Construction \ref{Cns_2114_OIA}(\ref{I_2114_OIA_A}),
we have $K_0 (A_{n, 1}) = 0$.
Also $K_0 (A_{n, 0}) \cong \Z$ and $[f_0]$ is a generator,
and, in $K_0 (A_{n, 0})$, $[p] = [\io_* (1)] = \io_* (0) = 0$.
Therefore $[f_n] = [f_0]$.
It follows that $K_0 (A_n) \cong \Z$ for all~$n$,
generated by $[1_{A_n}]$,
and $(\nu_{n + 1, n})_* ([1_{A_n}]) = [1_{A_{n + 1}}]$.
So $K_0 (A) \cong \Z$, generated by $[1_A]$.

The classification theorem for purely infinite simple C*-algebras,
Theorem 4.2.4 of~\cite{Ph_PICls},
now implies that $A \cong \OI$.
\end{proof}

\begin{lem}\label{L_2609_FixedPt}
Let $\af \colon S^1 \to \Aut (A)$ be as in
Construction \ref{Cns_2114_OIA}(\ref{I_2114_OIA_Lim}).
Then $A^{\af}$ is purely infinite and simple.
\end{lem}

\begin{proof}
Following the notation of Construction \ref{Cns_2114_OIA}, for $n \in \Nz$
in Lemma~\ref{L_2521_SLim} we take
\[
B_n = (A_n)^{\af^{(n)}},
\qquad
B_n^{(0)} = (A_{n, 0})^{\af^{(n, 0)}},
\andeqn
B_n^{(1)} = (A_{n, 1})^{\af^{(n, 1)}},
\]
and let
$\mu_{n, m}$ be the restriction of $\nu_{n, m}$ to the fixed point algebra.
Then
$\mu_{n + 1, n}^{(k, j)}
 = \nu_{n + 1, n}^{(k, j)} |_{(A_{n, j})^{\af^{(n, j)}}}$
for $n \in \Nz$ and $j, k \in \{ 0, 1 \}$.

In Lemma~\ref{L_2521_SLim}, 
condition~(\ref{I__2521_SLim_pi})
follows because
$B_n^{(1)} = (C (S^1) \otimes \OT)^{\Lt \otimes \id_{\OT}} \cong \OT$,
and condition~(\ref{I__2521_SLim_inj})
follows from Lemma \ref{L_2605_Cns_OK}(\ref{I_L_2605_Cns_inj})
by restriction.
Condition~(\ref{I__2521_SLim_full}) is a consequence of the
choices in Construction \ref{Cns_2114_OIA}(\ref{I_2114_OIA_A})
made using Lemma~\ref{L_2521_Exist}, and the fact that
$e_{n + 1, \, n + 1} \otimes p$ is \mvnt{} to $e_{0, 0} \otimes p$
in $(K \otimes \OI)^{\id_K \otimes \bt}$.
The conclusion thus follows from Lemma~\ref{L_2521_SLim}.
\end{proof}

\begin{thm}\label{2610_OI_TRP}
Let $\af \colon S^1 \to \Aut (A)$ be as in
Construction \ref{Cns_2114_OIA}(\ref{I_2114_OIA_Lim}).
Then $\af$ has the tracial Rokhlin property with comparison.
\end{thm}

\begin{proof}
We verify the conditions of Definition~\ref{traR}.
Let $F \subseteq A$ and $S \subseteq C (G)$ be finite,
let $\varepsilon > 0$, let $x \in A_{+}$ satisfy $\| x \| = 1$,
and let $y \in (A^{\alpha})_{+} \setminus \{ 0 \}$.
Choose $n \in \Nz$ so large that there is a finite subset
$E \subseteq A_n$
with $\dist (a, \, \nu_{\I, n} (E) ) < \frac{\ep}{2}$
for all $a \in F$ and there is $c \in A_n$
with $\| x - \nu_{\I, n} (c) \| < \frac{\ep}{2}$.
Define $q = (0, 1) \in A_{n + 1}$.
The formula $\ps (f) = f \otimes 1_{\OT}$ defines
an equivariant unital \hm{} $\ps \colon C (S^1) \to A_{n + 1, 1}$,
which we identify with a unital \hm{}
$\ps \colon C (S^1) \to q A_{n + 1} q$.
Now define $p = \nu_{\I, n} (q)$ and
$\ph = \nu_{\I, n + 1} \circ \ps \colon C (S^1) \to A$.
Then $p$ is an $\af$-invariant \pj{} and $\ph$ is
an equivariant unital \hm{} from $C (S^1)$ to~$p A p$.

We claim that for all $a \in F$ and $f \in C (S^1)$ we have
$\| a \ph (f) - \ph (f) a \| < \ep$.
This will verify condition~(\ref{Item_893_FS_equi_cen_multi_approx})
of Definition~\ref{traR}.
Let $a \in F$.
Choose $b \in E$ such that
$\| \nu_{\I, n} (b) - a \| < \frac{\ep}{2}$.
Then $\nu_{n + 1, \, n} (b) \ps (f) = \ps (f) \nu_{n + 1, \, n} (b)$,
so, applying $\nu_{\I, n + 1}$,
\[
\| a \ph (f) - \ph (f) a \|
 \leq 2 \| a - \nu_{\I, n} (b) \| < \ep.
\]
The claim is proved.

We have $1 - p \precsim_A x$ because $x \neq 0$
and $A$ is purely infinite and simple.
Similarly, using Lemma~\ref{L_2609_FixedPt},
$1 - p \precsim_{A^{\af}} y$ and $1 - p \precsim_{A^{\af}} p$.

It remains only to verify condition~(\ref{Item_902_pxp_TRP})
of Definition~\ref{traR}.
Let $y = \nu_{n + 1, \, n} (c) \in A_{n + 1}$ and
write $y = (y_0, y_1)$ with $y_j \in A_{n + 1, \, j}$ for $j = 0, 1$.
The map $A_n \to A_{n + 1, \, 1}$ is injective by
Lemma \ref{L_2605_Cns_OK}(\ref{I_L_2605_Cns_inj}).
Using this at the third step, we have
\[
\begin{split}
\| p x p \|
& > \| p \nu_{\I, \, n} (c) p \| - \frac{\ep}{2}
  = \| q \nu_{n + 1, \, n} (c) q \| - \frac{\ep}{2}
\\
& = \| y_1 \| - \frac{\ep}{2}
  = \| c \| - \frac{\ep}{2}
  > \| x \| - \ep.
\end{split}
\]
This completes the proof.
\end{proof}

\begin{cor}\label{C_2611_EOI}
There exists an action of $S^1$ on $\OI$
which has the tracial Rokhlin property with comparison.
\end{cor}

\begin{proof}
By Lemma~\ref{L_2605_Cns_OI}, the algebra $A$ in Theorem~\ref{2610_OI_TRP}
is isomorphic to~$\OI$.
\end{proof}

\begin{lem}\label{2610_No_fin_RD}
There is no action of $S^1$
on $\OI$ which has finite Rokhlin dimension with commuting towers.
\end{lem}

\begin{proof}
This is part of Corollary 4.23 of \cite{Gar_rokhlin_2017}. 
\end{proof}

For $n \in \{ 3, 4, \ldots \}$,
a construction similar to Construction~\ref{Cns_2114_OIA}
presumably gives an action of $S^1$ on ${\mathcal{O}}_{n}$ which has
the tracial Rokhlin property with comparison.
The only changes needed are in the construction of $u$ in the proof
of Lemma~\ref{L_2521_Exist}.
Corollary 4.23 of \cite{Gar_rokhlin_2017} also implies that
there is no action of $S^1$ on ${\mathcal{O}}_{n}$
which has finite Rokhlin dimension with commuting towers.

\begin{thm}\label{T_2610_Kbg}
Let $B$ be a unital purely infinite simple separable nuclear \ca.
Then there exists an action of $S^1$ on $B$
which has the version of the tracial Rokhlin property with comparison
obtained by omitting
condition (\ref{Item_902_pxp_TRP}) in Definition~\ref{traR}.
\end{thm}

By Remark~\ref{R_2015_nzp} and Remark~\ref{L_2015_S4_nzp},
condition (\ref{Item_902_pxp_TRP}) in Definition~\ref{traR}
is not needed for any of the
permanence properties we prove for fixed point algebras
and crossed products.
Also, we point out that the Universal Coefficient Theorem is not
needed.

\begin{proof}[Proof of Theorem~\ref{T_2610_Kbg}]
By Theorem~3.15 of~\cite{KP1},
it suffices to use $\OI \otimes B$ in place of~$B$.
Let $\af \colon S^1 \to \Aut (\OI)$ be as in
Construction \ref{Cns_2114_OIA}(\ref{I_2114_OIA_Lim}),
recalling that $A \cong \OI$ by Lemma~\ref{L_2605_Cns_OI}.
Set $\bt = \af \otimes \id_B \colon S^1 \to \Aut (\OI \otimes B)$.

We verify the conditions of Definition~\ref{traR},
except for condition (\ref{Item_902_pxp_TRP}).
Let $F \subseteq A$ and $S \subseteq C (G)$ be finite,
let $\varepsilon > 0$, let $x \in A_{+}$ satisfy $\| x \| = 1$,
and let $y \in (A^{\alpha})_{+} \setminus \{ 0 \}$.
By approximation and algebra, we may assume that there are finite sets
$F_1 \subseteq \OI$ and $F_2 \subseteq B$, contained in the closed unit balls
of these algebras, such that
\[
F = \bigl\{ a_1 \otimes a_2 \colon
   {\mbox{$a_1 \in F_1$ and $a_2 \in F_2$}} \bigr\}.
\]
Use Theorem~\ref{2610_OI_TRP}
to choose $p_0 \in \OI$ and $\ph_0 \colon C (S^1) \to \OI$
as in Definition~\ref{traR}, with $F_1$ in place of $F$,
with $S$ as given, and with $x = y = 1$.
Define $p = p_0 \otimes 1$ and define
$\ph \colon C (S^1) \to p (\OI \otimes B) p$
by $\ph (f) = \ph_0 (f) \otimes 1$ for $f \in C (S^1)$.
It is easily checked that $\varphi$
is an $(F, S, \varepsilon)$-approximately equivariant
central multiplicative map.
The algebras $\OI \otimes B \cong B$
and $(\OI \otimes B)^{\bt} = (\OI)^{\af} \otimes A$
are purely infinite and simple
(using Lemma~\ref{L_2609_FixedPt} for the second one),
and $p$, $x$, and $y$ are all nonzero,
so the relations
\[
1 - p \precsim_{\OI \otimes B} x,
\qquad
1 - p \precsim_{(\OI \otimes B)^{\bt}} y,
\andeqn
1 - p \precsim_{(\OI \otimes B)^{\bt}} p
\]
are automatic.
\end{proof}

As mentioned in the introduction to this section,
it should also be not too hard to generalize Construction \ref{Cns_2114_OIA}
to actions of $(S^1)^m$ for $m \in \{ 2, 3, 4, \ldots, \I \}$.
It seems harder to deal with nonabelian groups.

\begin{pbm}\label{Pb_2114_NC}
Find an action of a connected noncommutative second countable
compact group (such as ${\operatorname{SU}} (2)$) on $\OI$
which has the tracial Rokhlin property with comparison.
\end{pbm}

\section{Nonexistence}\label{Sec_1919_NonE}

For actions of $S^1$, one can require in all variants of the
definition the $(F, S, \ep)$-equivariant central multiplicative map
$\ph$ be exactly a \hm{} and exactly equivariant.
For direct limit actions of $S^1$, one can also require that it
take values in some algebra in the direct system.
We state a slightly more general form
(also covering finite abelian groups)
in the following proposition.
The statement about general actions is covered, using the trivial
direct system in which all the algebras are~$A$.

\begin{prp}\label{P_1920_S1DLim}
Let $G$ be a (not necessarily connected) compact abelian Lie group such
that $\dim (G) \leq 1$.
Let $\bigl( (A_n)_{n \in \Nz}, \, (\nu_{n, m})_{m \leq n} \bigr)$
be a direct system of \uca{s} with unital injective maps
$\nu_{n, m} \colon A_m \to A_n$.
Set $A = \dirlim_n A_n$, with maps $\nu_{\I, m} \colon A_m \to A$.
Assume we are given actions $\af^{(n)} \colon G \to \Aut (A_n)$
such that the maps $\nu_{n, m}$ are equivariant,
and let $\af$ be the direct limit action $\af = \dirlim \af^{(n)}$.
\begin{enumerate}
\item\label{Item_1920_S1DL_R}
The action $\af$ has the Rokhlin property
\ifo{} for every $N \in \Nz$, every finite set $F \subseteq A_{N}$,
every finite set $S \subseteq C (G)$, and every $\ep > 0$,
there exists $n \geq N$ and a unital equivariant \hm{}
$\ph \colon C (G) \to A_n$ such that
$\| \ph (f) \nu_{n, N} (a) - \nu_{n, N} (a) \ph (f) \| < \ep$
for all $f \in S$ and $a \in F$.
\item\label{Item_1920_S1DL_TRPC}
Suppose $A$ is simple.
Then $\af$ has the tracial Rokhlin property with comparison
\ifo{} for every $N \in \Nz$, every finite set $F \subseteq A_{N}$,
every finite set $S \subseteq C (G)$, every $\ep > 0$,
every $x \in A_{+}$ with $\| x \| = 1$,
and every $y \in (A^{\alpha})_{+} \SM \{ 0 \}$,
there exist $n \geq N$, a projection $p \in (A_n)^{\alpha^{(n)}}$,
and a unital equivariant \hm{} $\ph \colon C (G) \to p A_n p$
such that the following hold.
\begin{enumerate}
\item\label{Item_1920_S1DL_TRPC_Comm}
$\| \ph (f) \nu_{n, N} (a) - \nu_{n, N} (a) \ph (f) \| < \ep$
for all $f \in S$ and $a \in F$.
\item\label{Item_1920_S1DL_TRPC_Sub}
$1 - \nu_{\I, n} (p) \precsim_A x$,
$1 - \nu_{\I, n} (p) \precsim_{A^{\alpha}} y$,
and $1 - p \precsim_{(A_n)^{\alpha^{(n)}}} p$.
\item\label{Item_1920_S1DL_TRPC_Nm}
$\| \nu_{\I, n} (p) x \nu_{\I, n} (p) \| > 1 - \ep$.
\end{enumerate}
\item\label{Item_1920_S1DL_NTRP}
Suppose $A$ is simple.
Then $\af$ has the naive tracial Rokhlin property
\ifo{} for every $N \in \Nz$, every finite set $F \subseteq A_{N}$,
every finite set $S \subseteq C (G)$, every $\ep > 0$,
and every $x \in A_{+}$ with $\| x \| = 1$,
there exist $n \geq N$, a projection $p \in (A_n)^{\alpha^{(n)}}$,
and a unital equivariant \hm{} $\ph \colon C (G) \to p A_n p$
such that the following hold.
\begin{enumerate}
\item\label{Item_1920_S1DL_NTRP_Comm}
$\| \ph (f) \nu_{n, N} (a) - \nu_{n, N} (a) \ph (f) \| < \ep$
for all $f \in S$ and $a \in F$.
\item\label{Item_1920_S1DL_NTRP_Sub}
$1 - \nu_{\I, n} (p) \precsim_A x$.
\item\label{Item_1920_S1DL_NTRP_Nm}
$\| \nu_{\I, n} (p) x \nu_{\I, n} (p) \| > 1 - \ep$.
\end{enumerate}
\setcounter{TmpEnumi}{\value{enumi}}
\end{enumerate}
\end{prp}

\begin{proof}
In all three parts, the fact that the condition implies the
appropriate property follows from the fact that any finite subset of~$A$
can be approximated arbitrarily well by a finite subset
of $\bigcup_{n = 0}^{\I} \nu_{\I, n} (A_n)$.

The reverse directions are deduced from equivariant semiprojectivity
of $C (G)$,
which is Theorem~4.4 of~\cite{Gdla}.
The proofs are similar.
We only do~(\ref{Item_1920_S1DL_TRPC}), which has the most steps.
We give a full proof, since the steps must be done in the right order.

To simplify notation, we assume that $A_k \S A$ for all $k \in \Nz$,
and that the maps $\nu_{l, k}$ and $\nu_{\I, k}$ are all inclusions.
Thus, $A = {\overline{\bigcup_{l = 1}^{\I} A_l}}$.
Moreover, $\af^{(l)}_g = \af_g |_{A_l}$
for all $g \in G$ and $l \in \N$, and we just write~$\af_g$.

Let $N \in \Nz$, let $F \S A_{N}$ be finite,
let $S \S C (G)$ be finite,
let $\ep > 0$, let $x \in A_{+}$ satisfy $\| x \| = 1$,
and let $y \in (A^{\af})_{+} \SM \{ 0 \}$.
Since $\af$ has the \trpc, Proposition~\ref{P_1X14_CentSq} provides
a \pj{} $e \in (A_{\I, \af} \cap A')^{\af_{\I}}$
and an equivariant unital \hm{}
$\ps \colon C (G) \to e (A_{\I, \af} \cap A') e$
such that the following hold.
\begin{enumerate}
\setcounter{enumi}{\value{TmpEnumi}}
\item\label{It_CSq_2115_1mp_af_sm}
$1 - e$ is $\af$-small in $A_{\I, \af}$.
\item\label{It_CSq_2115_1mp_sm_Aaf}
$1 - e$ is small in $(A^{\af})_{\I}$.
\item\label{It_CSq_2115_1mp_p}
$1 - e \precsim_{(A^{\af})_{\I}} e$.
\item\label{It_CSq_2115_Norm}
Identifying $A$ with its image in $A_{\I, \af}$,
we have $\| e x e\| = 1$.
\setcounter{TmpEnumi}{\value{enumi}}
\end{enumerate}

Choose $c^{(1)}, c^{(2)}, \ldots \in C (G)$ such that
$\bigl\{ c^{(1)}, c^{(2)}, \ldots  \bigr\}$ is dense in $C (G)$.
For $j \in \N$ choose
$d^{(j)} = \bigl( d^{(j)}_m \bigr)_{m \in \N} \in l^{\I}_{\af} (\N, A)$
such that $\pi_A \bigl( d^{(j)} \bigr) = \ps \bigl( c^{(j)} \bigr)$.
Use Lemma~\ref{L_1X09_LiftPj} to lift $e$
to an $\af^{\I}$-invariant \pj{}
$r = (r_m)_{m \in \N} \in l^{\I}_{\af} (\N, A)$.
Since $\af$ is a direct limit action,
$A^{\af} = {\overline{\bigcup_{l = 1}^{\I} (A_l)^{\af^{(l)}} }}$.
Therefore every \pj{} in $A^{\af}$ is a norm limit of \pj{s}
in $\bigcup_{l = 1}^{\I} (A_l)^{\af^{(l)}}$.
For $m \in \N$ choose $l (m) \in \N$ such that
there is a \pj{} $q_m \in (A_{l (m)})^{\af^{(l (m))}}$ satisfying
$\| q_m - r_m \| < \frac{1}{m}$,
and also so large that for $j = 1, 2, \ldots, m$ we have
$\dist \bigl( d^{(j)}_m, \, A_{l (m)} \bigr) < \frac{1}{m}$.
We may assume $l (1) \leq l (2) \leq \cdots$.
Set $q = (q_m)_{m \in \N} \in l^{\I}_{\af} (\N, A)$.
Then $\pi_A (q) = e$.
Let $B \S l^{\I}_{\af} (\N, A)$ be the closed subalgebra consisting
of all sequences $b = (b_m)_{m \in \N} \in l^{\I}_{\af} (\N, A)$
such that for all $m \in \N$ we have $b_m \in q_m A_{l (m)} q_m$.
Let $D = \pi_A (B) \S e A_{\I, \af} e \S A_{\I, \af}$.

We claim that $\ps (C (G)) \S D$.
It suffices to let $j \in \N$
and prove that $\ps \bigl( c^{(j)} \bigr) \in D$.
By construction, for all $m \geq j$ there is $b_m \in A_{l (m)}$
such that $\bigl\| b_m - d^{(j)}_m \bigr\| < \frac{1}{m}$.
Setting $b = (b_m)_{m \in \N} \in l^{\I} (\N, A)$,
we get $\pi_A (b) = \ps \bigl( c^{(j)} \bigr)$.
Since $c_0 (\N, A) \S l^{\I}_{\af} (\N, A)$
and $d^{(j)} \in l^{\I}_{\af} (\N, A)$, this implies that
$b \in l^{\I}_{\af} (\N, A)$.
Therefore also
\[
\ps \bigl( c^{(j)} \bigr)
 = e \ps \bigl( c^{(j)} \bigr) e
 = \pi_A (q) \pi_A (b) \pi_A (q)
 = \pi_A (q b q)
 \in D.
\]
The claim is proved.

For $k \in \N$ let $J_k \S B$ be the ideal consisting of all
sequences $(a_m)_{m \in \N} \in B$ such that $a_m = 0$ for all $m > k$.
Then
\[
J_1 \S J_2 \S \cdots
\andeqn
{\overline{\bigcup_{k = 1}^{\I} J_k}}
 = B \cap c_0 (\N, A) = {\operatorname{Ker}} (\pi_A |_B).
\]
Let $\kp_k \colon B \to B / J_k$ be the quotient map.
Use equivariant semiprojectivity of $C (G)$
(Theorem~4.4 of~\cite{Gdla}) to find $m_0 \in \N$ such that
$\ps \colon C (G) \to D = B / {\operatorname{Ker}} (\pi_A |_B)$
lifts to an equivariant unital \hm{} $\rh \colon C (G) \to B / J_{m_0}$,
that is, $\kp_{m_0} \circ \rh = \ps$.
We may require $m_0 \geq N$.
There is an obvious identification of $B / J_{m_0}$ with the \ca{}
of all bounded sequences $(a_m)_{m > m_0}$
such that $a_m \in q_m A_{l (m)} q_m$ for all $m > m_0$
and the map $g \mapsto (\af_g (a_m))_{m > m_0}$ is \ct.
Under this identification, there are equivariant unital \hm{s}
$\rh_m \colon C (G) \to q_m A_{l (m)} q_m$ such that
$\rh (f) = (\rh_m (f))_{m > m_0}$ for all $f \in C (G)$.

Since $\kp_{m_0} \circ \rh = \ps$, which has range contained in
$A_{\I, \af} \cap A'$,
for all $a \in A$ and $f \in C (G)$ we have
$\lim_{m \to \I} \| \rh_m (f) a - a \rh_m (f) \| = 0$.
In particular, there is $m_1 \geq m_0$ such that for all $m \geq m_1$,
all $f \in S$, and all $a \in F$,
we have $\| \rh_m (f) a - a \rh_m (f) \| < \ep$.
Since $\| e x e \| = 1$ and $\pi_A (q) = e$,
there is $m_2 \in \N$ such that for all $m \geq m_2$,
we have $\| q_m x q_m \| > 1 - \ep$.
Since $1 - e$ is $\af$-small and small in $(A^{\af})_{\I}$,
there is $m_3 \in \N$ such that for all $m \geq m_3$,
we have $1 - q_m \precsim_A x$ and $1 - q_m \precsim_{A^{\af}} y$.
Since $1 - e \precsim_{(A^{\af})_{\I}} e$,
there is $v \in (A^{\af})_{\I}$ such that $v^* v = 1 - e$ and
$v v^* \leq e$.
Then $1 - e = v^* e v$.
Choose $w = (w_{m})_{m \in \N} \in l^{\I} (\N, A^{\af})$
such that $\pi_A (w) = v$.
Then $\lim_{m \to \infty} \| w_m^* q_m w_m - (1 - q_m) \| = 0$.
Therefore there is $m_4 \in \N$ such that for every $m \geq m_4$
we have $\| w_m^* q_m w_m - (1 - q_m) \| < 1$.
Set $m = \max (m_1, m_2, m_3, m_4)$,
take the number $n$ in the statement to be $\max (N, l (m))$,
and set $p = q_m$ and $\ph = \rh_m$.
Lemma~\ref{L_1X09_CSbPj} implies that
$1 - q_m \precsim_{A^{\alpha}} q_m$,
and the rest of the conclusion is clear.
\end{proof}

Let $A$ be a unital \ca.
It is known (in~\cite{garduhfabs} see Theorem 2.17, Example 3.22, and
Example 3.23)
that the existence of an action of $S^1$ on~$A$
with the Rokhlin property implies severe restrictions on~$A$.
One can in fact rule out at least direct limit actions
on a simple unital AF~algebra
with even the naive tracial Rokhlin property.

\begin{prp}\label{P_1921_NoAction}
Let $A$ be an in\fd{} simple unital AF~algebra
and let $G$ be a one dimensional compact abelian Lie group.
There is no direct limit action of $G$ on~$A$,
with respect to any realization of $A$ as a direct limit
of \fd{} \ca{s}, which has the naive tracial Rokhlin property
(Definition~\ref{D_1920_NTRP}).
\end{prp}

\begin{proof}
Let $\bigl( (A_n)_{n \in \Nz}, \, (\nu_{n, m})_{m \leq n} \bigr)$
be an equivariant direct system of \fd{} \ca{s}
with actions $\af^{(n)} \colon G \to \Aut (A)$
and such that the direct limit action $\af = \dirlim \af^{(n)}$
has the naive tracial Rokhlin property.
Choose any \pj{} $x \in A \SM \{ 0, 1 \}$.
Apply Proposition \ref{P_1920_S1DLim}(\ref{Item_1920_S1DL_NTRP})
with $m = 0$, $F = \{ 1 \}$, $S = \{ 1 \}$, $\ep = \frac{1}{2}$,
and $x$ as given.
We get $n \in \Nz$, a nonzero \pj{} $p \in (A_n)^{\alpha^{(n)}}$,
and a unital equivariant \hm{} $\ph \colon C (G) \to p A_n p$.
Since $C (G)$ is $G$-simple and $\ph$ is equivariant,
$\ph$ must be injective.
This is a contradiction because $G$ is infinite and $A_n$ is \fd.
\end{proof}

\begin{cor}\label{C_1921_NoTRPC}
Let $A$ be an in\fd{} simple unital AF~algebra
and let $G$ be a one dimensional compact abelian Lie group.
There is no direct limit action of $G$ on~$A$,
with respect to any realization of $A$ as a direct limit
of \fd{} \ca{s}, which has the tracial Rokhlin property with comparison.
\end{cor}

\begin{proof}
This is immediate from Proposition~\ref{P_1921_NoAction}, because
the tracial Rokhlin property with comparison
implies the naive tracial Rokhlin property.
\end{proof}

\end{document}